\pdfoutput=1

\documentclass[preprint,12pt,A4]{elsarticle}
\usepackage[top=2.5cm,bottom=2.5cm,left=2.5cm,right=2.5cm]{geometry}

\usepackage{amssymb}
\usepackage{latexsym}
\usepackage{graphicx}
\usepackage{amsmath,amsthm}
\usepackage{amssymb}
\usepackage{epsfig}
\usepackage{float,lscape}
\usepackage{rotating}
\usepackage{graphicx}
\usepackage{sidecap}
\usepackage{multirow}
\usepackage{float}
\usepackage{lipsum}
\usepackage{rotating}
\newtheorem{thm}{Theorem}[section]

\theoremstyle{definition}
\newtheorem{defn}{Definition}[section]

\theoremstyle{remark}
\newtheorem{rem}{Remark}[section]

\journal{}

\begin{document}

\begin{frontmatter}


\author[label1,label2]{Rabia Hameed}
\ead{rabiahameedrazi@hotmail.com}
\author[label1]{Ghulam Mustafa \corref{cor1}}
\ead{Corresponding author: ghulam.mustafa@iub.edu.pk}
\address[label1]{Department of Mathematics, The Islamia University of Bahawalpur \fnref{label3}}
\address[label2]{Department of Mathematics, The Government Sadiq College Women University Bahawalpur \fnref{label3}}
\title{Recursive process for constructing the refinement rules of new combined subdivision schemes and its extended form}




\begin{abstract}
In this article, we present a new method to construct a family of $(2N+2)$-point binary subdivision schemes with one tension parameter where $N$ is a non-negative integer. The construction of the family of schemes is based on repeated local translation of points by certain displacement vectors. Therefore, the refinement rules of a $(2N+2)$-point scheme for $N=M$ are recursively obtained from the refinement rules of the $(2N+2)$-point schemes for $N=0,1,2,\ldots,M-1$. The complexity, polynomial reproduction and polynomial generation of these schemes are increased by two for the successive values of $N$. Furthermore, we modify this family of schemes to a family of $(2N+3)$-point schemes with two tension parameters.
Moreover, a family of interproximate subdivision schemes with tension parameters is also introduced, which allows a different tension value for each edge and vertex of the initial control polygon. Interproximate schemes generate curves and surfaces such that some initial control points are interpolated and others are approximated.
\end{abstract}


\end{frontmatter}


\section{Introduction}
Subdivision schemes are efficient tools for generating smooth curves/surfaces as the limit of an iterative process based on simple refinement rules starting from certain control points defining a control polygon/mesh. In recent years, subdivision schemes have been an important research area. These schemes provide an efficient way to describe curves, surfaces and related geometric objects. Generally, subdivision schemes are classified as interpolatory or approximating, depending on whether the limit curve passes through all the given initial control points or not. Although approximating schemes yield smoother curves with higher order continuity, interpolating schemes are more useful for engineering applications as they preserve the shape of the coarse mesh.
The special family of interpolatory schemes consists of the schemes with refinement rules that preserve the points associated with the coarse mesh and only generate new points related to the additional vertices of the refined mesh. An important family of interpolatory schemes was introduced by Deslauriers and Dubuc \cite{Dubuc} and latest tools for its analysis were introduced in \cite{Amat}. Whereas an important family of approximating subdivision schemes that is the dual counterparts of the schemes of Deslauriers and Dubuc \cite{Dubuc} was proposed in \cite{Dyn7}. However, there also exists the family of combined subdivision schemes, which can be used either as an approximating scheme or as an interpolatory scheme for a special choice of the tension parameter. See the surveys \citep{Pan, Novara}. Furthermore, combined subdivision schemes can be converted to shape controlled subdivision schemes by defining local tension parameters. In the literature, there are several shape controlled subdivision algorithms that have been constructed by combining the refinement rules of interpolatory and approximating subdivision schemes. Beccari et al. \cite{Beccari1} proposed a method in which the refinement rules of an interpolating univariate subdivision scheme can be derived from the refinement rules of an approximating subdivision scheme by applying certain simple operations on the mask coefficients. Li and Zheng \cite{Li} combined the $4$-point scheme of Dyn et al. \cite{Dyn1} and the cubic B-spline binary refinement scheme to construct a shape controlled subdivision scheme. Tan et al. \cite{Tan} combined the $4$-point scheme of Dyn et al. \cite{Dyn1} and a $2$-point corner cutting scheme to construct another shape controlled subdivision scheme. In this paper, we present a recursive method to construct the $(2N+2)$-point combined subdivision schemes which is a new trend in the construction of subdivision schemes. We also present an extended form of this family of combined schemes by defining shape controlled parameters to increase the flexibility in curves and surfaces fitting.

The following motivation plays an important role in the construction of the proposed families of schemes.
\subsection{Motivation}
There are many algorithms for constructing families of binary dual subdivision schemes. One family of primal binary approximating schemes is constructed in \cite{Dyn5}, but its continuity is only $C^{1}$ and it reproduces only linear polynomials. However, there exists an algorithm that constructs both primal and dual schemes, called Refine-Smooth algorithm. Some of the latest Refine-Smooth algorithms are presented in \citep{Hameed, Hameed1, Mustafa10, Romani}. In this algorithm, generally the smoothness of the schemes is increased by increasing the number of smoothing stages but the degree of polynomial reproduction remains the same. We propose an algorithm for constructing binary primal combined subdivision schemes. These schemes not only give optimal smoothness in limit curves and surfaces but also give the good reproduction degree. Moreover, our families of schemes give both interpolatory and approximating behavior for specific values of the tension parameter. Furthermore, we convert our schemes to interproximate schemes that generate curves and surfaces such that some initial control points are interpolated and others are approximated.

This article is organized as follows. Section 2 deals with some basic definitions and results. In Section 3, we construct three families of primal subdivision schemes. Section 4 deals with some important properties of the proposed families of schemes. In Section 5, numerical examples and comparisons are presented. A family of interproximate subdivision schemes and associated numerical examples are presented in Section 6. Conclusions are given in Section 7.
\section{Preliminaries}
A general compact form of linear, uniform and stationary binary univariate subdivision scheme $S_{a}$ which maps a polygon $f^{k}=\{f^{k}_{i}, i \in \mathbb{Z}\}$ to a refined polygon $f^{k+1}=\{f^{k+1}_{i}, i \in \mathbb{Z}\}$ is defined as
\begin{eqnarray}\label{subdivisionscheme}
  f_{i}^{k+1} &=& \sum \limits_{j \in \mathbb{Z}}a_{i -2 j} f_{j}^{k}, \quad i \in \mathbb{Z}.
\end{eqnarray}
Since the subdivision scheme (\ref{subdivisionscheme}) is a binary scheme, the two rules for defining the new control points are:
\begin{eqnarray}\label{binary-ss}
\nonumber  f_{2i}^{k+1} &=& \sum \limits_{j \in \mathbb{Z}}a_{2i -2 j} f_{j}^{k}=\sum \limits_{\gamma \in \mathbb{Z}}a_{2\gamma} f_{i-\gamma}^{k}, \quad i \in \mathbb{Z},\\
  f_{2i+1}^{k+1} &=& \sum \limits_{j \in \mathbb{Z}}a_{2i+1 -2 j} f_{j}^{k}=\sum \limits_{\gamma \in \mathbb{Z}}a_{2\gamma+1} f_{i-\gamma}^{k}, \quad i \in \mathbb{Z}.
\end{eqnarray}
The symbol of the above subdivision scheme is given by the Laurent polynomial
\begin{eqnarray}\label{polynomial}
  a(z) &=& \sum \limits_{i \in \mathbb{Z}}a_{i}z^{i}, \quad z \in \mathbb{C}\setminus \{0\},
\end{eqnarray}
where $a= \{a_{i}, i \in \mathbb{Z}\}$ is called the mask of the subdivision scheme.
Detailed information about refinement rules, Laurent polynomials and convergence of a subdivision scheme can be found in \cite{Charina, Dyn6, Dyn2}.
The necessary condition for the convergence of the subdivision scheme (\ref{binary-ss}) is that $\sum \limits_{\gamma \in \mathbb{Z}}a_{2\gamma} =\sum \limits_{\gamma \in \mathbb{Z}}a_{2\gamma+1}=1 $.
The continuity of the subdivision schemes can be analyzed by the following theorems.

\begin{thm}\label{thmcontinuity}
\cite{Dyn6} A convergent subdivision scheme $S_{a}$ corresponding to the symbol
\begin{eqnarray*}
a(z)&=&\left(\frac{1+z}{2\, z}\right)^{n}b(z),
\end{eqnarray*}
 is $C^{n}$-continuous iff the subdivision scheme $S_{b}$ corresponding to the symbol $b(z)$ is convergent.
\end{thm}
\begin{thm}\label{thmcontinuity--1}
The scheme $S_{b}$ corresponding to the symbol $b(z)$ is convergent iff its difference scheme $S_{c}$ corresponding to the symbol $c(z)$ is contractive, where $b(z)=(1+z)c(z)$. The scheme $S_{c}$ is contractive if
\begin{eqnarray*}
&&||c^{l}||_{\infty}=\mbox{max}\left\{\sum\limits_{i}|c^{l}_{j-2^{l_{i}}}|:0\leq j <2^{l}\right\}<1, \,\ l \in \mathbb{N},
\end{eqnarray*}
where $c^{l}_{i}$ are the coefficients of the scheme $S^{l}_{c}$ with symbol
\begin{eqnarray*}
&&c^{l}(z)=c(z)c(z^2)\ldots c(z^{2^{l-1}}).
\end{eqnarray*}
\end{thm}
In a geometric context, subdivision schemes are further categorized into primal and dual subdivision schemes. The primal binary subdivision schemes are the schemes that leave or modify the old vertex points and create one new point at each old edge. Primal schemes can be interpolatory, approximating or combined. Dual binary subdivision schemes on the other hand, are the schemes that create two new points at the old edges and discard the old points. Dual schemes are always approximating subdivision schemes. Detailed information about the primal and dual subdivision schemes can be found in \cite{Charina}. Furthermore, if one refinement rule of a binary subdivision scheme (\ref{binary-ss}) uses the affine combination of $\xi=\{\xi_{1} > 2: \xi_{1} \in \mathbb{Z}\}$ control points at level $k$ to get a new control point at level $k+1$ whereas the other refinement rule of (\ref{binary-ss}) uses the affine combination of control points less than $\xi$ and at level $k$ to get a new point at level $k+1$, then that binary scheme is called the primal binary scheme. Moreover, every primal binary scheme is the relaxed subdivision scheme. Mathematical definition of primal are dual subdivision schemes is presented below.
\begin{defn}\label{primal-dual}
\cite{Conti} Let the symbol of the subdivision scheme (\ref{binary-ss}) defined in (\ref{polynomial}) can particularly be written as $a(z)=\ldots +a_{-3}z^{-3}+a_{-2}z^{-2}+a_{-1}z^{-1}+a_{0}z^{0}+a_{1}z^{1}+a_{2}z^{2}+a_{3}z^{3}+\ldots$. If the symbol $a(z)$ defined in (\ref{polynomial}) corresponding to the scheme $S_{a}$ satisfy the following condition
\begin{eqnarray*}
a(z)&=&a(z^{-1}),
\end{eqnarray*}
then $S_{a}$ is said to be a primal subdivision scheme.
On the other hand if it satisfy following condition
\begin{eqnarray*}
z a(z)&=&a(z^{-1}),
\end{eqnarray*}
then $S_{a}$ is said to be a dual subdivision scheme.
\end{defn}

The combined subdivision schemes are the schemes which depend on one or more tension parameters. Moreover, at the specific values of these
parameters, these schemes can be regarded either as an approximating subdivision scheme or an interpolatory one. Interproximate subdivision schemes (see \cite{Li}) are the schemes which generate the limit curves that interpolate some of the vertices of the given control polygons, while approximate the other vertices of the given control polygons.

Generation and reproduction degrees are used to check the behaviors of a subdivision scheme when the original data points lie on the graph of a polynomial. Suppose that the original data points are taken from a polynomial of degree $d$. If the control points of the limit curve lie on graph of the polynomial having same degree (i.e. $d$) then we say that the subdivision scheme generates polynomials of degree $d$. If the control points of the limit curve lie on graph of the same polynomial then we say that the subdivision scheme reproduces polynomials of degree $d$. Mathematically, let $\Pi_{d}$ denote the space of polynomials of degree $d$ and $g$, $h$ $\in$ $\Pi_{d}$, an approximation operator $\mathbf{O}$ generates polynomials of degree $d$ if $\mathbf{O}g=h$ $\forall$ $g, h \in \Pi_{d}$, whereas $\mathbf{O}$ reproduces polynomials of degree $d$ if $\mathbf{O}g=g$ $\forall$ $g \in \Pi_{d}$. Furthermore, the generation degree of a subdivision scheme is the maximum degree of polynomials that can theoretically be generated by the scheme, provided that the initial data is taken correctly. Evidently, it is not less than the reproduction degree. For exact definitions of polynomial generation and reproduction the readers can consult \cite{Charina,Conti}. The following theorem is used to check the generation and reproduction degrees of the subdivision schemes in this paper.

\begin{thm}\label{thm-gd-rd}
\cite{Conti} A univariate binary subdivision scheme $S_{a}$
\begin{description}
  \item[(i)] Generates polynomials up to degree $d$ if and only if
  \begin{eqnarray*}
&& a(1)=2, \,\ a(-1)=0, \,\ \left.D^{(m)}a(z)\right|_{z=-1}=0, \,\ m=1,2,\ldots,d,
  \end{eqnarray*}
  where $\left.D^{(m)}a(z)\right|_{z=-1}$ denotes the $m$-th derivative of $a(z)$ with respect to $z$ evaluated at a point $z=-1$.
  \item[(ii)] Reproduces polynomials up to degree $d$ with respect to the parametrization $\{t_i^{(k)} =\frac{i+\tau}{2^{k}}\}_{i \in \mathbb{Z}}$ with $\tau=\frac{1}{2}\left.D^{(1)}(z)\right|_{z=1}$  if and only if it generates polynomials of degree d and
      \begin{eqnarray*}
      &&\left.D^{(m)}a(z)\right|_{z=1}=2\prod\limits_{h=0}^{m-1}(\tau-h), \,\ m=1,2,\ldots,d.
      \end{eqnarray*}
\end{description}
\end{thm}
The support of a basic limit function and a subdivision scheme is the area of the limit curve that will be affected by the displacement of a single control point from its initial place. The part which is dependent on that given control point is called the support width of the given subdivision scheme. By following the approach of \cite{Beccari}, we give following theorem to calculate the support width of a relaxed binary combined scheme or an interpolatory binary scheme.
\begin{thm}\label{support-theorem}
The support width of a $\xi$-point binary relaxed subdivision scheme $S_{a}$ is $2 \xi$ where $\xi=\{\xi_{1} > 2: \xi_{1} \in \mathbb{Z}\}$, which implies that it vanishes outside the interval $[-\xi,\xi]$. The support width of a $\xi$-point interpolatory binary scheme $S_{a}$ is $2 \xi -2$, which implies that it vanishes outside the interval $[-\xi+1,\xi-1]$.
\end{thm}

\section{Construction of the families of subdivision schemes}
In this section, we present a family of $(2N+2)$-point relaxed combined subdivision schemes that is based on repeated local translation of points by using certain displacement vectors. Thus, the refinement rules of a member of the proposed family is recursively obtained by the refinement rules of one other member of this family, i.e. the refinement rules of a $(2N+2)$-point scheme for $N=M$ are recursively obtained from the refinement rules of the $(2N+2)$-point scheme for $N=M-1$.
We propose a new family of $(2N+3)$-point relaxed combined schemes with two tension parameters by extending the points of the family of $(2N+2)$-point relaxed combined schemes. Then we modify the family of $(2N+3)$-point relaxed schemes to a family of $(2N+4)$-point interpolatory schemes by removing one of its tension parameters.
Construction process for the family of $(2N+2)$-point relaxed combined schemes with one tension parameter is given below:
\subsection{Framework for the construction of a family of $(2N+2)$-point relaxed schemes}\label{FSS1}

The family of $(2N+2)$-point combined subdivision schemes $S_{a_{2N+2}}$ which maps the polygon $f^{k}_{N+1}=\{f^{k}_{i,N+1}:{i \in \mathbb{Z}}\}$ to the refined polygon $f^{k+1}_{N+1}=\{f^{k+1}_{i,N+1}:{i \in \mathbb{Z}}\}$ is defined by the set of following refinement rules
\begin{eqnarray}\label{schm1}
\left\{\begin{array}{ccccccc}
&&f^{k+1}_{2i,N+1}=\sum\limits_{j=-N-1}^{N+1}a_{2j,N+1}f^{k}_{i+j,N+1},\\ \\
&&f^{k+1}_{2i+1,N+1}=\sum\limits_{j=-N-1}^{N}a_{2j+1,N+1}f^{k}_{i+j+1,N+1},
\end{array}\right.
\end{eqnarray}
where $N \in \mathbb{N}_{0}=\mathbb{N} \cup \{0\}$ is used to calculate the complexity (number of control points at $k$-th subdivision level used in the insertion of a new point at $(k+1)$-th subdivision level is called the complexity) of the subdivision schemes and $k \in \mathbb{N}$ denote the number of times subdivision is applied on the original data points. Hence for each $N$, the set $\{f_{i,N+1}^{k+1}:{i \in \mathbb{Z}}\}$ represents the $(k+1)$-th level subdivided points obtained by applying $(k+1)$-times the $(2N+2)$-point relaxed subdivision scheme (\ref{schm1}) on the initial data points $\{f_{i}^{0}=f_{i,N+1}^{0}:{i \in \mathbb{Z}}\}$ and $a=[a_{j,N+1}:j=-2(N+1),\ldots,2(N+1)]$ is mask of the scheme (\ref{schm1}) which is same at each level of refinement for a fix value of $N$. The schematic sketches of both rules defined in (\ref{schm1}) are presented in Figure \ref{P5Schematic1}.
\begin{figure}[htb] 
\begin{center}
\begin{tabular}{cc}
\epsfig{file=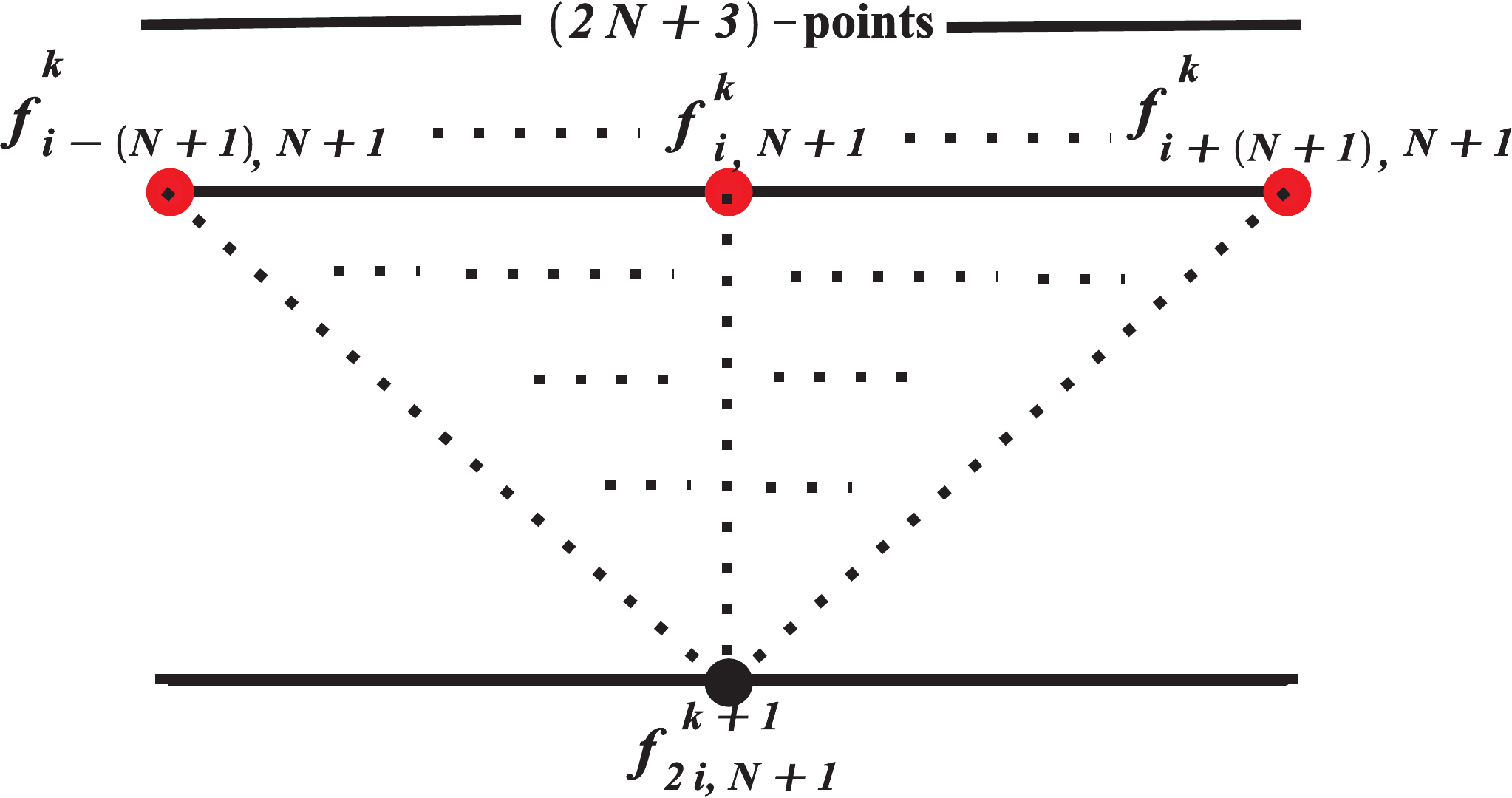, width=2.5 in} & \epsfig{file=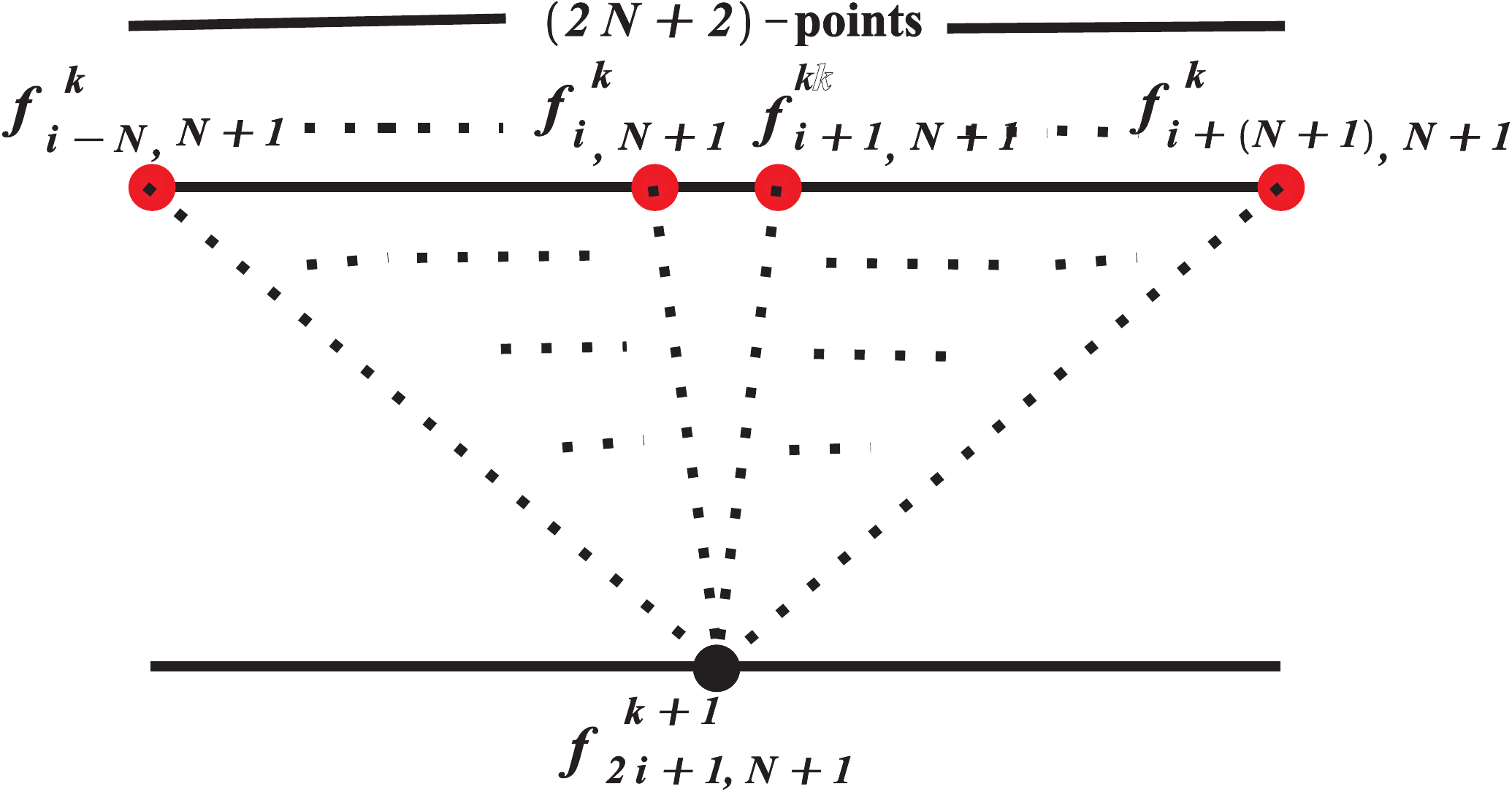, width=2.5 in}\\
(a) Vertex rule & (b) Edge rule
\end{tabular}
\end{center}
 \caption[Graphical sketches of the rules of $(2N+2)$-point primal schemes.]{\label{P5Schematic1}\emph{Graphical sketches of the rules of $(2N+2)$-point primal schemes.}}
\end{figure}
The construction process of these rules is given below:

If $N=0$, the two refinement rules of the $2$-point relaxed scheme are obtained from (\ref{1}). Hence $\{f^{1}_{i,1}:{i \in \mathbb{Z}}\}$ are the control points at first subdivision level obtained by the $2$-point relaxed subdivision scheme on the initial control points $\{f^{0}_{i}=f^{0}_{i,1}:{i \in \mathbb{Z}}\}$. These two refinement rules are the initial refinement rules used to calculate the other refinement rules of proposed family of schemes for each successive value of $N$. The initial refinement rules are defined as
\begin{eqnarray}\label{1}
\left\{\begin{array}{ccccccc}
&&f^{1}_{2i,1}=f_{i,1}^{0}+ \alpha_{0} (f^{0}_{i-1,1}-2 f^{0}_{i,1}+f^{0}_{i+1,1}),\\ \\
&&f^{1}_{2i+1,1}=\frac{1}{2}f_{i,1}^{0}+\frac{1}{2}f_{i+1,1}^{0},
\end{array}\right.
\end{eqnarray}

Now we calculate points $\{f_{2i+\ell,N+1}^{1}:\ell =0,1\}_{i \in \mathbb{Z}}$ of the $(2N+2)$-point relaxed subdivision scheme for $N \geqslant 1$. Hence for a fix value of $N$, the points $\{f_{2i+\ell,N+1}^{1}:\ell =0,1\}_{i \in \mathbb{Z}}$ of the $(2N+2)$-point relaxed subdivision scheme are obtained by moving the points $\{f_{2i+\ell,N}^{1}:\ell =0,1\}_{i \in \mathbb{Z}}$ to the new position according to the displacement vectors $\{\alpha_{\ell} \vec{C}_{2i+\ell,N+1}:\ell =0,1\}_{i \in \mathbb{Z}}$, where $\alpha_{\ell}$ is the tension parameter with $\alpha_{1}=1$ and $\alpha_{0}=\alpha$.
Mathematically, for $N \geq 1$, the two refinement rules of the family of $(2N+2)$-point relaxed subdivision schemes at first level of subdivision are obtained by the following recurrence relation
\begin{eqnarray}\label{scheme}
f^{1}_{2i+\ell,N+1}&=&f_{2i+\ell,N}^{1}+\alpha_{\ell} \vec{C}_{2i+\ell,N+1}, \,\ \ell=0,1,
\end{eqnarray}
where the vectors $\vec{C}_{2i+\ell,N+1}: \ell=0,1$ are calculated by the following recurrence relation
\begin{eqnarray}\label{vec-C}
\left\{\begin{array}{ccccccc}
\vec{C}_{2i,N+1}&=&(\vec{C}_{2i,N}-\vec{C}_{2(i+1),N})+ (\vec{C}_{2i,N}-\vec{C}_{2(i-1),N}), \\ \\
\vec{C}_{2i+1,N+1}&=&\frac{1}{N}\left(\frac{N}{4}-\frac{1}{8}\right)\left((\vec{C}_{2i+1,N}-\vec{C}_{2(i+1)+1,N})+ (\vec{C}_{2i+1,N}-\vec{C}_{2(i-1)+1,N})\right),
\end{array}\right.
\end{eqnarray}
and initial values for relation (\ref{vec-C}) are
\begin{eqnarray}\label{3}
\left\{\begin{array}{ccccccc}
&&\vec{C}_{2i,1}=f^{0}_{i-1,1}-f^{0}_{i,1}+f^{0}_{i+1,1},\\ \\
&&\vec{C}_{2i+1,1}=\frac{1}{2}\left(f_{i,1}^{0}+f_{i+1,1}^{0}\right).
\end{array}\right.
\end{eqnarray}
Here $f^{0}_{i}$ $=$ $f^{0}_{i,1}$ $=$ $f^{0}_{i,2}$ $=$ $f^{0}_{i,3}$ $=$ $\ldots$ $=$ $f^{0}_{i,N+1}$ are the initial control points.
%
While the points $f^{1}_{i,1}$, $f^{1}_{i,2}$, $f^{1}_{i,3}$ and $f^{1}_{i,4}$ are the initial points of the $4$-point, $6$-point and $8$-point relaxed subdivision schemes obtained by substituting the values of $N$ equal to $1$, $2$ and $3$ respectively in (\ref{scheme})-(\ref{vec-C}). Since the proposed subdivision schemes are stationary, so the refinement rules are same at each level of subdivision. Therefore, for other subdivision levels, we apply (\ref{schm1}) while the coefficients of points $f^{k}_{i,N+1}$ remain same as the coefficients of points $f^{0}_{i,N+1}=f^{0}_{i}$ obtained from (\ref{scheme}). Also $f^{k+1}_{i,1}$, $f^{k+1}_{i,2}$, $f^{k+1}_{i,3}$ and $f^{k+1}_{i,4}$ are the control points at $(k+1)$-th subdivision level obtained by applying the $2$-point, $4$-point, $6$-point and $8$-point relaxed subdivision schemes on the $k$-th level points $f^{k}_{i,1}$, $f^{k}_{i,2}$, $f^{k}_{i,3}$ and $f^{k}_{i,4}$ respectively. Moreover, the points other than the initial control points hold the relation $f^{k}_{i,N} \neq f^{k}_{i,N+1}$ $\forall$ $N \in \mathbb{N}$.

At each iteration, i.e. by substituting $N=1,2,3,\ldots$ in (\ref{schm1}) and (\ref{scheme})-(\ref{vec-C}), we get a new binary primal $(2N+2)$-point subdivision scheme. The masks of these $(2N+2)$-point schemes by defining $\alpha_{0}=\alpha$ are tabulated in Table \ref{Mask-S-2N+2}.

\begin{table}[htb] 
 \caption[Mask of the $(2N+2)$-point schemes $S_{a_{2N+2}}$.]{\label{Mask-S-2N+2}\emph{Mask of the $(2N+2)$-point schemes $S_{a_{2N+2}}$.}}
\begin{center}
\begin{tabular}{||c|c||}
  \hline \hline
  N & Mask   \\
  \hline \hline
  0 &  $\frac{1}{2}[2\alpha, 1, 2- 4\alpha, 1, 2 \alpha]$\\
  \hline
  1 &  $\frac{1}{16}[-16 \alpha, -1, 64 \alpha, 9, 16-96 \alpha, 9, 64 \alpha, -1, -16 \alpha]$\\
  \hline
  2 &  $\frac{1}{256}[256 \alpha, 3, -1536 \alpha, -25, 3840 \alpha, 150, 256-5120 \alpha, 150, 3840 \alpha, $\\
    &  $  -25, -1536 \alpha, 3, 256 \alpha]$ \\
    \hline
  3 & $\frac{1}{2048}[-2048 \alpha, -5, 16384 \alpha, 49, -57344 \alpha, -245, 114688 \alpha, 1225, 2048-$  \\
    & $ 143360 \alpha, 1225, 114688 \alpha, -245, -57344 \alpha, 49, 16384 \alpha, -5, -2048 \alpha]$ \\
    \hline
    & $\frac{1}{65536}[65536 \alpha, 35, -655360 \alpha, -405, 2949120 \alpha, 2268, -7864320 \alpha,$  \\
  4 & $  -8820, 13762560 \alpha, 39690, 65536-16515072 \alpha, 39690, 13762560 \alpha, $ \\
    & $  -8820, -7864320 \alpha, 2268, 2949120 \alpha, -405, -655360 \alpha, 35, 65536 \alpha]$\\
    \hline
    & $\frac{1}{524288} [-524288 \alpha, -63, 6291456 \alpha, 847, -34603008 \alpha, -5445, 115343360 \alpha,$ \\
  5 & $ 22869, -259522560 \alpha, -76230, 415236096 \alpha, 320166, 524288$\\
    & $-484442112 \alpha, 320166, 415236096 \alpha, -76230, -259522560 \alpha, 22869,$ \\
    & $ 115343360 \alpha, -5445, -34603008 \alpha, 847, 6291456 \alpha, -63, -524288 \alpha]$\\
  \hline \hline
\end{tabular}
\end{center}
\end{table}

\begin{rem}
If $\alpha=0$, the family of schemes (\ref{schm1}) reduces to the family of $(2N+2)$-point interpolatory schemes with symbol
\begin{eqnarray*}
  a(z) &=& 1+\sum\limits_{j=-N-1}^{N}a_{2j+1,N+1}z^{2j+1},
\end{eqnarray*}
proposed by Deslauriers and Dubuc \cite{Dubuc}. The continuity of $(2N+2)$-point interpolatory schemes is $C^{N}$ for $0 \leq N \leq 4$ and $C^{\approx \frac{83}{200}(N+1)}$ for $N \geq 5$.
\end{rem}

\subsection{Interpretation of Framework \ref{FSS1} for $N=1$}
The refinement rules of the initial subdivision scheme defined in (\ref{1}) are
\begin{eqnarray}\label{4}
\left\{\begin{array}{ccccccc}
&&f^{1}_{2i,1}=\alpha f_{i-1}^{0}+(1-2 \alpha) f_{i}^{0}+\alpha f^{0}_{i+1},\\ \\
&&f^{1}_{2i+1,1}=\frac{1}{2}f_{i}^{0}+\frac{1}{2}f_{i+1}^{0}.
\end{array}\right.
\end{eqnarray}
Moreover, the initial values $\{\vec{C}_{2i+\ell,1}:\ell=0,1\}$ which will be use to calculate the vectors $\{\vec{C}_{2i+\ell,N+1}:\ell=0,1\}_{N \in \mathbb{N}}$ defined in (\ref{3}) are
\begin{eqnarray}\label{5}
\left\{\begin{array}{ccccccc}
&&\vec{C}_{2i,1}=f^{0}_{i-1}-f^{0}_{i}+f^{0}_{i+1},\\ \\
&&\vec{C}_{2i+1,1}=\frac{1}{2} f_{i}^{0}+ \frac{1}{2} f_{i+1}^{0}.
\end{array}\right.
\end{eqnarray}
Now we use initial relations (\ref{4})-(\ref{5}) to calculate the refinement rules of $4$-point relaxed scheme which will be obtain by putting $N=1$ in (\ref{scheme})-(\ref{vec-C}). Hence for $N=1$ we get
\begin{eqnarray}\label{6}
\left\{\begin{array}{ccccccc}
f^{1}_{2i,2}&=&f_{2i,1}^{1}+\alpha_{0} \vec{C}_{2i,2},\\ \\
f^{1}_{2i+1,2}&=&f_{2i+1,1}^{1}+\alpha_{1} \vec{C}_{2i+1,2},
\end{array}\right.
\end{eqnarray}
where
\begin{eqnarray}\label{A}
\left\{\begin{array}{ccccccc}
\vec{C}_{2i,2}&=&(\vec{C}_{2i,1}-\vec{C}_{2(i+1),1})+(\vec{C}_{2i,1}-\vec{C}_{2(i-1),1}), \\ \\
\vec{C}_{2i+1,2}&=&\frac{1}{8}[(\vec{C}_{2i+1,1}-\vec{C}_{2(i+1)+1,1})+(\vec{C}_{2i+1,1}-\vec{C}_{2(i-1)+1,1})].
\end{array}\right.
\end{eqnarray}
By using (\ref{5}) in (\ref{A}), we get
\begin{eqnarray}\label{B}
\left\{\begin{array}{ccccccc}
\vec{C}_{2i,2}&=& -f^{0}_{i-2}+3f^{0}_{i-1}-4f^{0}_{i}+3f^{0}_{i+1}-f^{0}_{i+2}, \\ \\
\vec{C}_{2i+1,2}&=&\frac{1}{16}(-f^{0}_{i-1}+f^{0}_{i}+f^{0}_{i+1}-f^{0}_{i+2}).
\end{array}\right.
\end{eqnarray}
Now firstly we put $\alpha_{0}=\alpha$ and $\alpha_{1}=1$ in (\ref{6}) and then we use (\ref{4}) $\&$ (\ref{B}) in (\ref{6}). Hence we get the following $4$-point relaxed subdivision scheme
\begin{eqnarray}\label{C}
\left\{\begin{array}{ccccccc}
f^{1}_{2i,2}&=&-\alpha f^{0}_{i-2}+4 \alpha f^{0}_{i-1}+(1-6 \alpha)f^{0}_{i}+4 \alpha f^{0}_{i+1}-\alpha f^{0}_{i+2},\\ \\
f^{1}_{2i+1,2}&=&\frac{1}{16}(-f^{0}_{i-1}+9 f^{0}_{i}+9 f^{0}_{i+1}-f^{0}_{i+2}).
\end{array}\right.
\end{eqnarray}
By using (\ref{C}) in (\ref{schm1}), we get
\begin{eqnarray}\label{D}
\left\{\begin{array}{ccccccc}
f^{k+1}_{2i,2}&=&-\alpha f^{k}_{i-2}+4 \alpha f^{k}_{i-1}+(1-6 \alpha)f^{k}_{i}+4 \alpha f^{k}_{i+1}-\alpha f^{k}_{i+2},\\ \\
f^{k+1}_{2i+1,2}&=&\frac{1}{16}(-f^{k}_{i-1}+9 f^{k}_{i}+9 f^{k}_{i+1}-f^{k}_{i+2}).
\end{array}\right.
\end{eqnarray}
Step by step graphical representations of the above procedure are shown in Figures \ref{Geometrical}-\ref{Geometrical1}.
\begin{figure}[p] 
\begin{center}
\begin{tabular}{c}
\epsfig{file=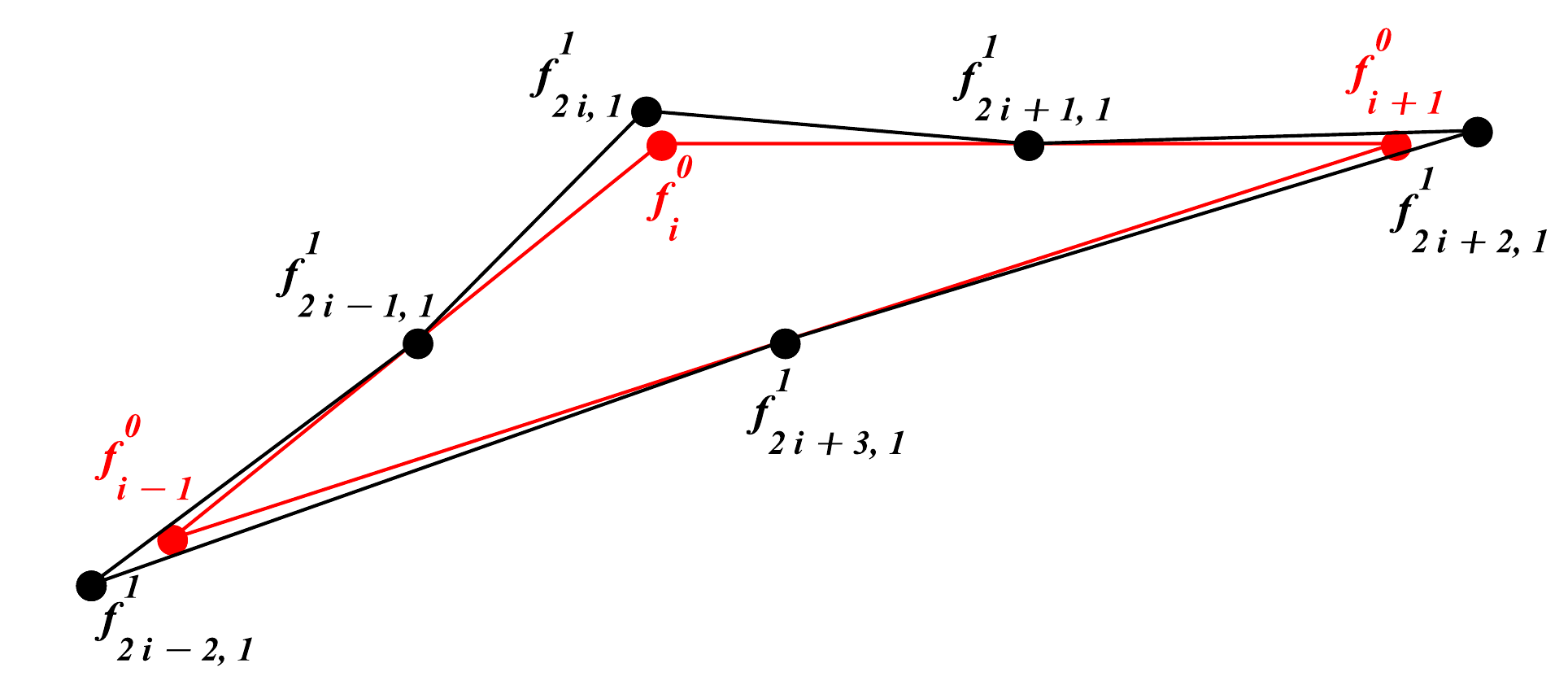, width=4.8 in} \\
(a) Implementation of the initial subdivision scheme defined in (\ref{4}). \\
\epsfig{file=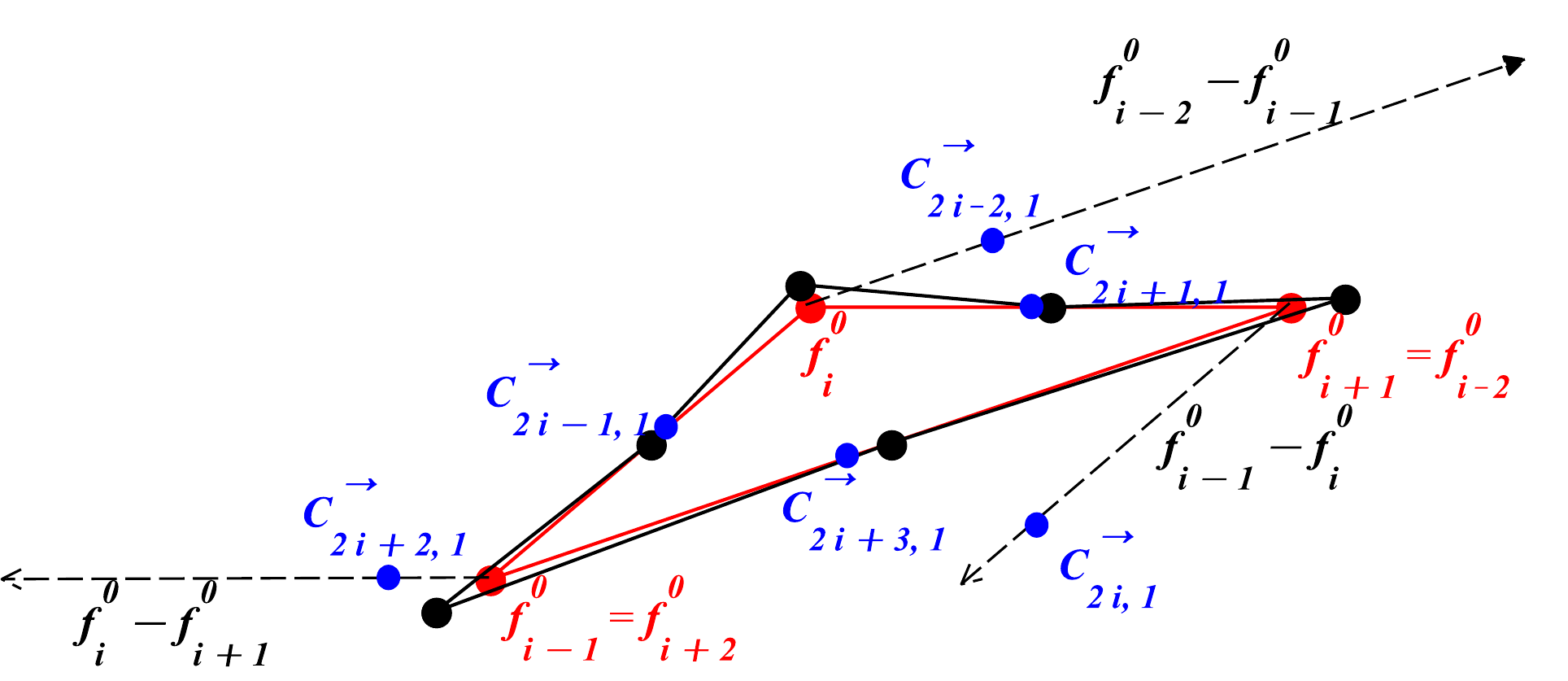, width=4.8 in} \\
(b) Blue bullets show the points obtained by the relations defined in (\ref{5}). Here \\$\vec{C}_{2i-2,1}=f^{0}_{i-2}-f^{0}_{i-1}+f^{0}_{i}$, $\vec{C}_{2i-1,1}=\frac{1}{2}(f^{0}_{i-1}+f^{0}_{i})$, $\vec{C}_{2i,1}=f^{0}_{i-1}-f^{0}_{i}+f^{0}_{i+1}$,\\ $\vec{C}_{2i+1,1}=\frac{1}{2}(f^{0}_{i}+f^{0}_{i+1})$, $\vec{C}_{2i+2,1}=f^{0}_{i}-f^{0}_{i+1}+f^{0}_{i+2}$, $\vec{C}_{2i+3,1}=\frac{1}{2}(f^{0}_{i+1}+f^{0}_{i+2})$.\\
\epsfig{file=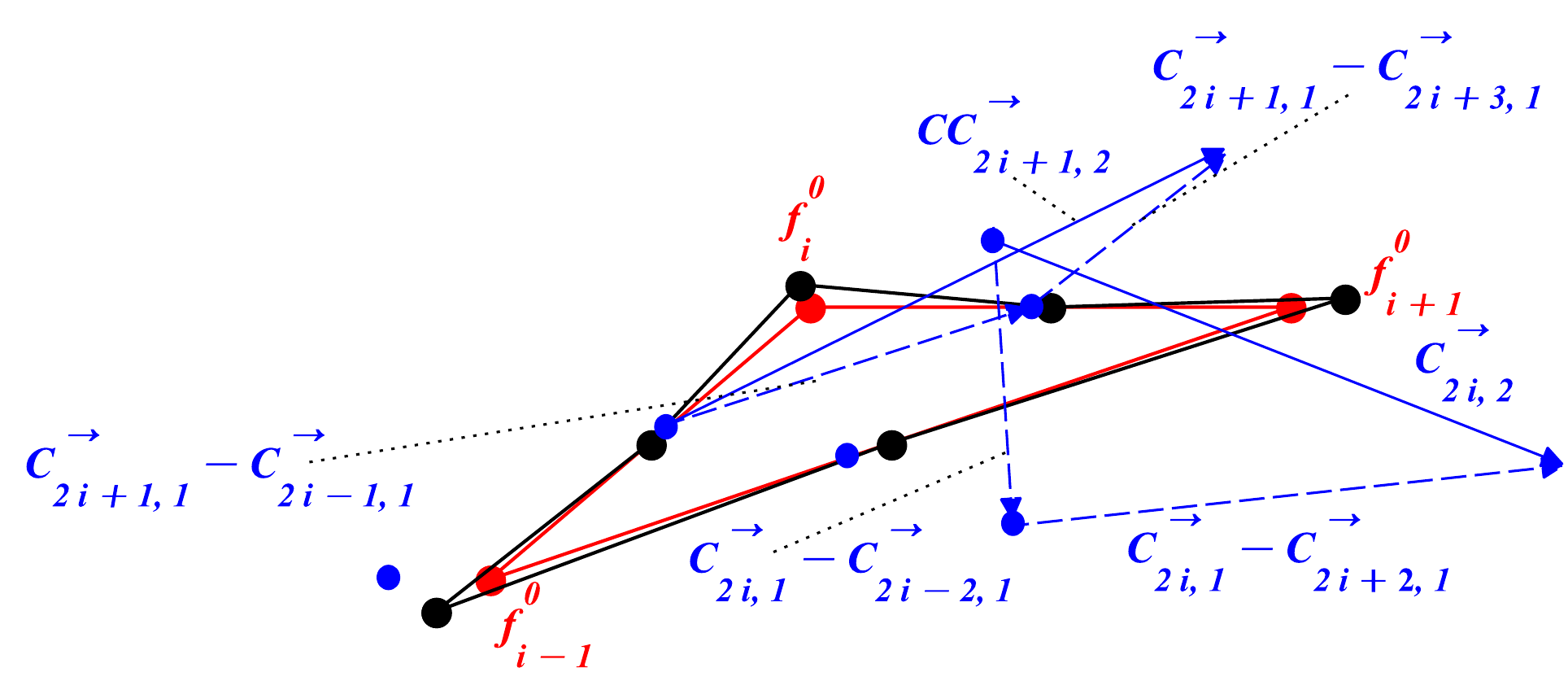, width=4.8 in} \\
(c) Here $\vec{C}_{2i,2}$ is the resultant vector of two vectors $\vec{C}_{2i,1}-\vec{C}_{2(i+1),1}$ and $\vec{C}_{2i,1}-\vec{C}_{2(i-1),1}$.\\ Similarly, $\vec{CC}_{2i+1,2}$ is the resultant vector of $\vec{C}_{2i+1,1}-\vec{C}_{2(i+1)+1,1}$ and $\vec{C}_{2i+1,1}-\vec{C}_{2(i-1)+1,1}$. \\
The resultant vectors are obtained by adding two vectors with head to tail rule. The \\resultant vectors are denoted by blue solid lines while the other vectors are denoted by \\ blue dashed lines.
 \end{tabular}
\end{center}
 \caption{\label{Geometrical}\emph{Geometrical interpretation of the proposed framework. Red bullets and red lines represent initial points and initial polygons respectively. Black bullets and black lines represent the points and the polygons obtained by the subdivision scheme (\ref{4}).}}
\end{figure}
\begin{figure}[p] 
\begin{center}
\begin{tabular}{c}
\epsfig{file=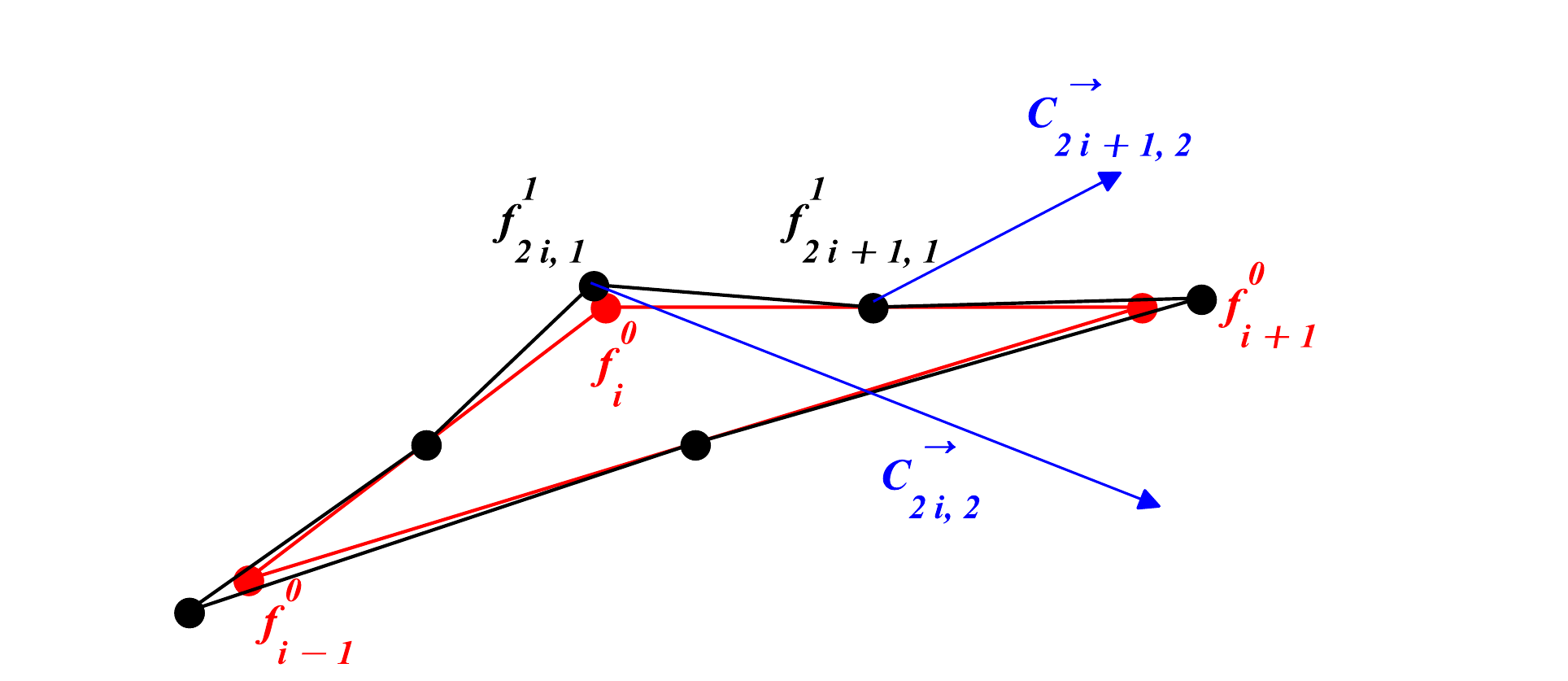, width=4.8 in} \\
(a) Geometrical representation of the vectors defined in (\ref{A}). Here $\vec{C}_{2i+1,2}=\frac{1}{8}\vec{CC}_{2i+1,2}$. \\
\epsfig{file=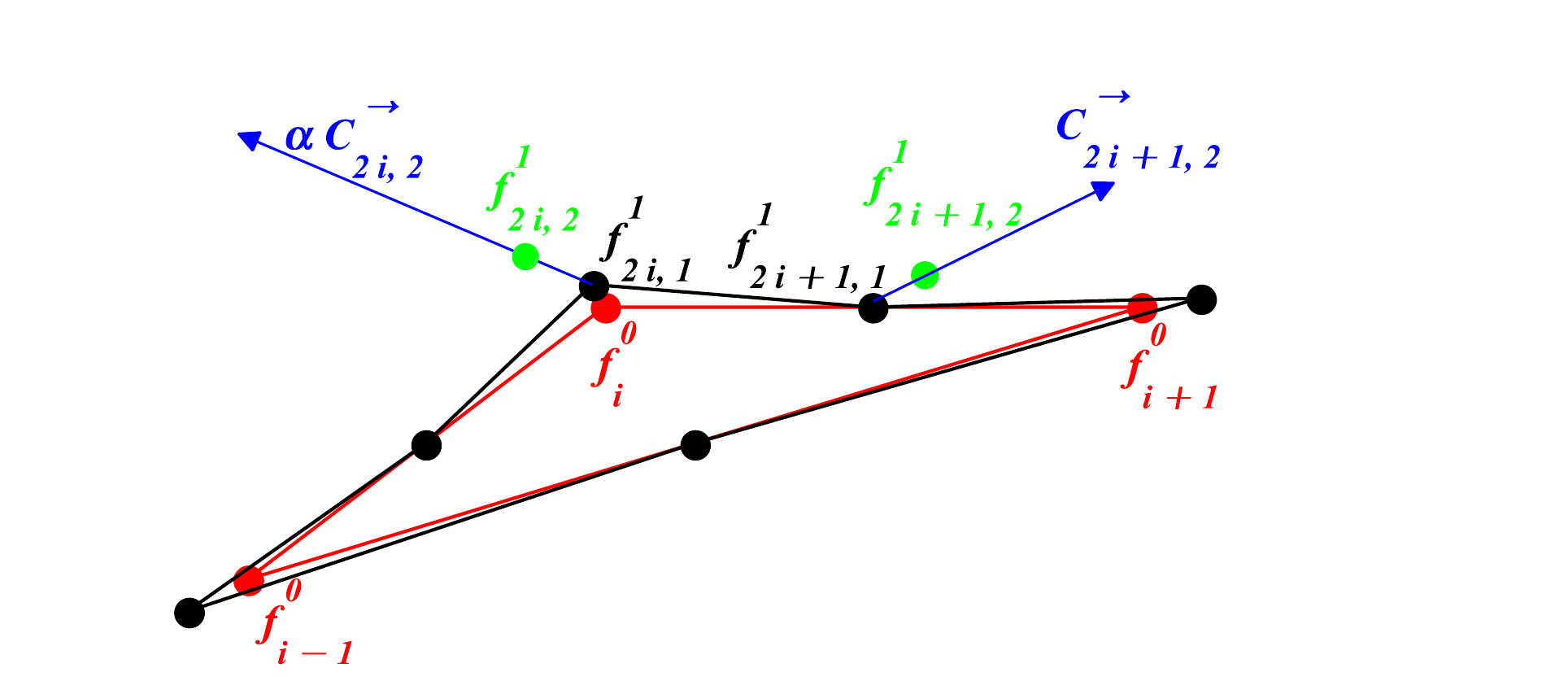, width=4.8 in} \\
(b) Translation of the points $f^{1}_{2i,1}$ and $f^{1}_{2i+1,1}$ by using vectors $\alpha \vec{C}_{2i,2}$ and $\vec{C}_{2i+1,2}$
to \\ obtain the points $f^{1}_{2i,2}$ and $f^{1}_{2i+1,2}$ denoted by green bullets.
Geometrical \\ representation of points obtained by the refinement rules defined in (\ref{B}). \\ Here $\alpha$ has a negative value, hence direction of $\alpha \vec{C}_{2i,2}$ is changed.\\
\epsfig{file=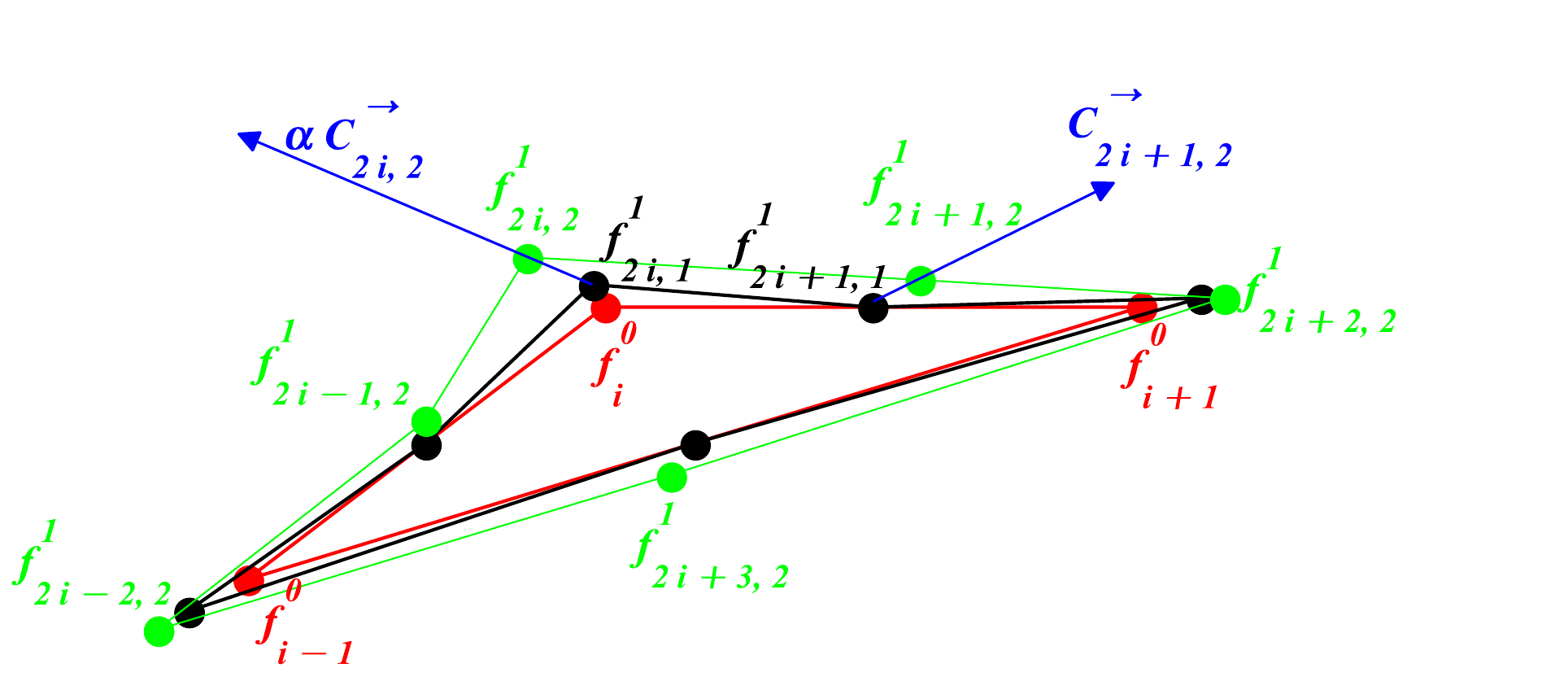, width=4.8 in} \\
(c) Green bullets show the points of the subdivision scheme (\ref{B}) constructed by\\ the proposed framework.
 \end{tabular}
\end{center}
 \caption{\label{Geometrical1}\emph{Geometrical interpretation of the proposed framework. Red bullets and red lines represent initial points and initial polygons respectively. Black bullets and black lines represent the points and the polygons obtained by the subdivision scheme (\ref{4}). Green bullets and green lines represent the points and the polygons obtained by the subdivision scheme (\ref{C}).}}
\end{figure}
\subsection{Extended form of Framework \ref{FSS1} for constructing a family of $(2N+3)$-point relaxed schemes}\label{FSS2}
When $\ell=1$, we add these weights
\begin{eqnarray}\label{weights}
\nonumber &&\sum\limits_{j=-N-1}^{N}(-1)^{j+N+1}\left(\frac{2j+1}{N-j+1}\right)\left(\begin{array}{c}2N+2\\N+j+2\end{array}\right)\beta(1-\alpha)f^{0}_{i+j+1,N+1} +\beta(1-\alpha) (f^{0}_{i-N-1,N+1}\\&&+f^{0}_{i+N+2,N+1}),
\end{eqnarray}
in (\ref{scheme}), where $-1<\alpha<1$ $\&$ $-1<\beta<1$. Hence we get the family of $(2N+3)$-point combined relaxed schemes $S_{a_{2N+3}}$ associated with the following refinement rules
\begin{eqnarray}\label{schm2}
\left\{\begin{array}{ccccccc}
&&f^{k+1}_{2i,N+1}=\sum\limits_{j=-N-1}^{N+1}b_{2j}f^{k}_{i+j,N+1},\\ \\
&&f^{k+1}_{2i+1,N+1}=\sum\limits_{j=-N-2}^{N+1}b_{2j+1}f^{k}_{i+j+1,N+1},
\end{array}\right.
\end{eqnarray}
where
\begin{eqnarray}\label{b2j+1}
\nonumber &&b_{2j}=a_{2j,N+1}\,\ \mbox{for} \,\ j=-(N+1),-N,\ldots,(N+1),\\
\nonumber &&b_{2j+1}=a_{2j+1,N+1}+(-1)^{j+N+1}\left(\frac{2j+1}{N-j+1}\right)\left(\begin{array}{c}2N+2\\N+j+2\end{array}\right)\beta (1-\alpha)\\&&
\mbox{for} \,\ j=-(N+1),-N,\ldots,N \,\ \mbox{and} \\&&
\nonumber b_{2j+1}=\beta (1-\alpha) \,\ \mbox{for} \,\ j=-(N+2)\,\ \& \,\ (N+1).
\end{eqnarray}

\begin{figure}[htb] 
\begin{center}
\begin{tabular}{cc}
\epsfig{file=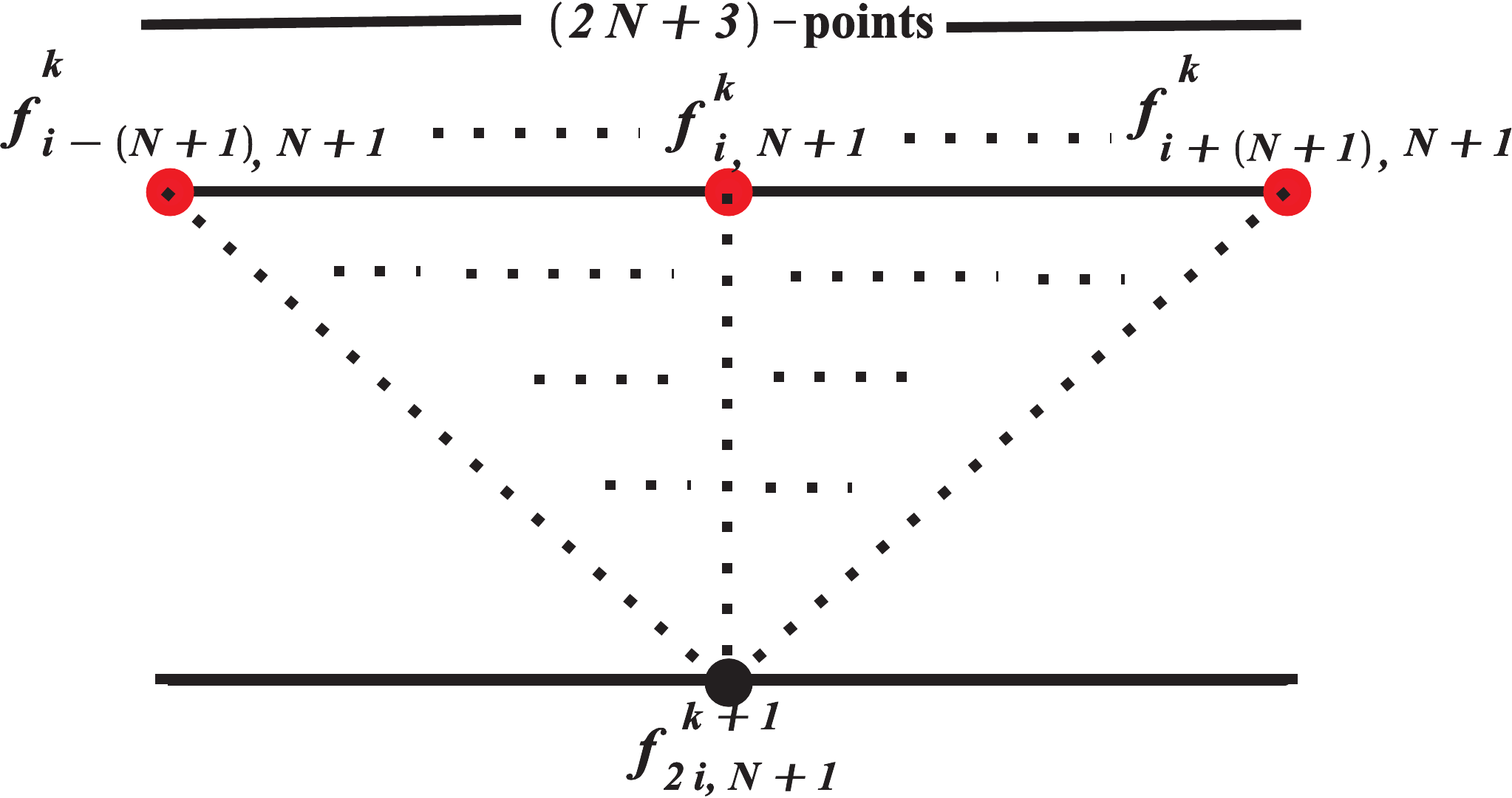, width=2.5 in} & \epsfig{file=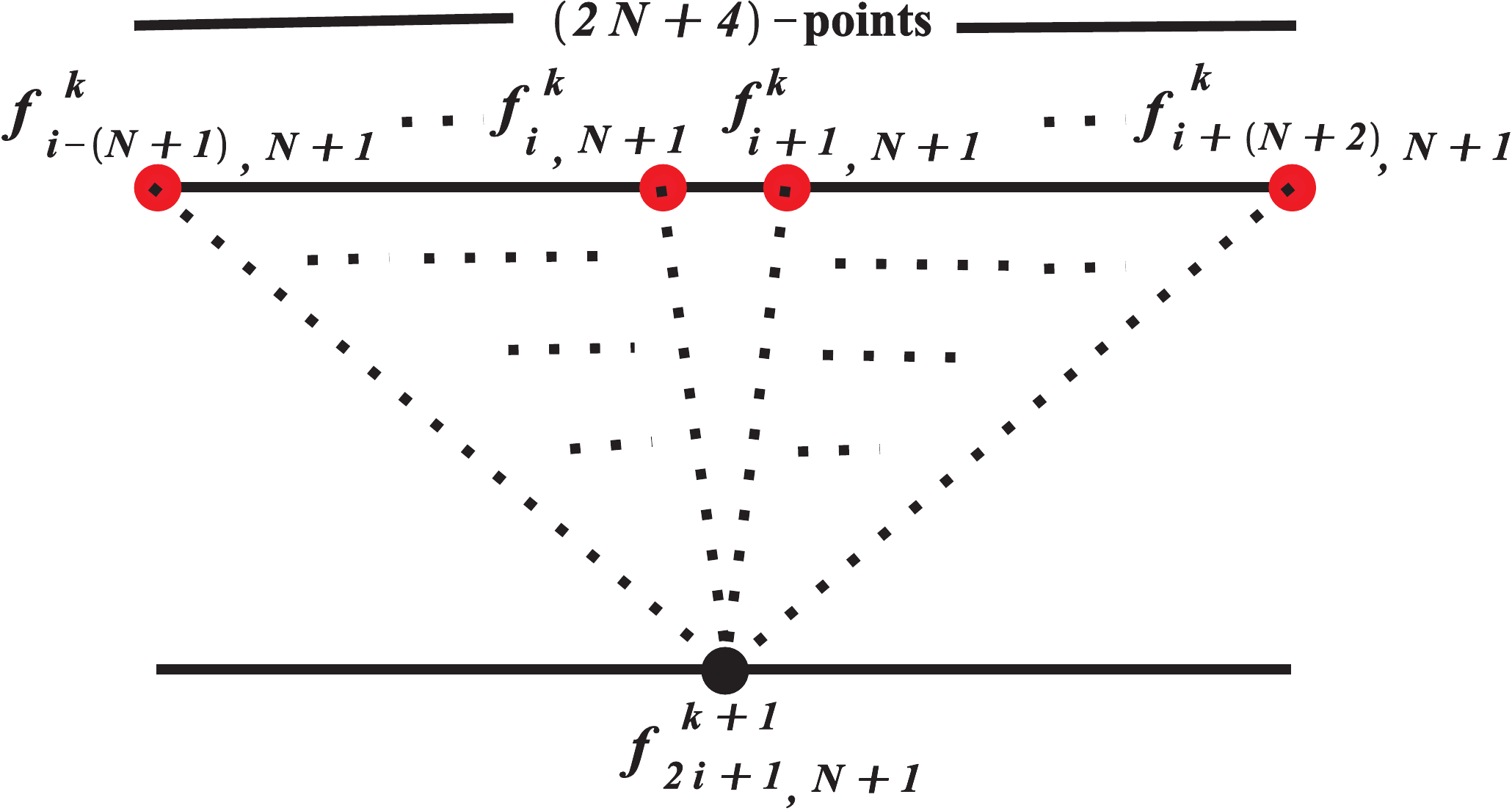, width=2.5 in}\\
(a) Vertex rule & (b) Edge rule
\end{tabular}
\end{center}
 \caption[Graphical sketches of the rules of $(2N+3)$-point primal schemes.]{\label{P5Schematic2}\emph{Graphical sketches of the rules of $(2N+3)$-point primal schemes.}}
\end{figure}

The schematic sketches of these rules are given in Figure \ref{P5Schematic2} and mask of the first three members of this family of schemes are given below:


\begin{itemize}
  \item When $N=0$, (\ref{schm2}) gives the primal $3$-point relaxed scheme with mask
  \begin{eqnarray*}
&&\left[\beta (1-\alpha),\alpha,\frac{1}{2}-\beta (1-\alpha),1-2\alpha,\frac{1}{2}-\beta (1-\alpha),\alpha,\beta (1-\alpha)\right].
  \end{eqnarray*}
  \item When $N=1$, (\ref{schm2}) gives the primal $5$-point relaxed scheme with mask
  \begin{eqnarray*}
&& \left[\beta (1-\alpha),-\alpha,-\frac{1}{16}-3 \beta (1-\alpha),4 \alpha,\frac{9}{16}+2 \beta (1-\alpha),1-6 \alpha,\frac{9}{16}+\right.\\&&\left.2 \beta (1-\alpha),4 \alpha,-\frac{1}{16}-3 \beta (1-\alpha),-\alpha,\beta (1-\alpha)\right].
  \end{eqnarray*}
  \item When $N=2$, (\ref{schm2}) gives the primal $7$-point relaxed scheme with mask
  \begin{eqnarray*}
  &&\left[\beta (1-\alpha),\alpha,\frac{3}{256}-5 \beta(1- \alpha),-6 \alpha,-\frac{25}{256}+9 \beta (1-\alpha), 15 \alpha,\frac{150}{256} -\right.\\&&\left.
5 \beta (1-\alpha),1-20 \alpha, \frac{150}{256}-5 \beta (1-\alpha), 15 \alpha,-\frac{25}{256}+9 \beta (1-\alpha), -6 \alpha,\right.\\&&\left.
\frac{3}{256}-5 \beta(1- \alpha),\alpha,\beta (1-\alpha)\right].
  \end{eqnarray*}
\end{itemize}

\subsection{Interpretation of Extended form \ref{FSS2} for $N=1$}
When $N=1$, (\ref{weights}) gives
\begin{eqnarray}\label{weights1}
\nonumber &&\sum\limits_{j=-2}^{1}(-1)^{j+2}\left(\frac{2j+1}{2-j}\right)\left(\begin{array}{c}4\\3+j\end{array}\right)\beta(1-\alpha)f^{0}_{i+j+1,2} +\beta(1-\alpha) (f^{0}_{i-2,2}+f^{0}_{i+3,2})\\&=&
\beta(1-\alpha)\left[f^{0}_{i-2,2}-3 f^{0}_{i-1,2}+ 2 f^{0}_{i,2}+2 f^{0}_{i+1,2}-3f^{0}_{i+2,2}+f^{0}_{i+3,2}\right].
\end{eqnarray}
Adding weights which are defined in (\ref{weights1}) in the edge rule of (\ref{C}), we get
\begin{eqnarray}\label{E}
\left\{\begin{array}{ccccccc}
f^{1}_{2i,2}&=&-\alpha f^{0}_{i-2}+4 \alpha f^{0}_{i-1}+(1-6 \alpha)f^{0}_{i}+4 \alpha f^{0}_{i+1}-\alpha f^{0}_{i+2},\\ \\
f^{1}_{2i+1,2}&=&\beta(1-\alpha)f^{0}_{i-2}-\left(\frac{1}{16}+3\beta(1-\alpha)\right)f^{0}_{i-1}+\left(\frac{9}{16}+2\beta(1-\alpha)\right)f^{0}_{i}\\&&
+\left(\frac{9}{16}+2\beta(1-\alpha)\right)f^{0}_{i+1}-\left(\frac{1}{16}+3\beta(1-\alpha)\right)f^{0}_{i+2}+\beta(1-\alpha)f^{0}_{i+3}.
\end{array}\right.
\end{eqnarray}
Since the above scheme is stationary, hence the refinement rules of proposed $5$-point relaxed scheme with two tension parameters are
\begin{eqnarray}\label{E}
\left\{\begin{array}{ccccccc}
f^{k+1}_{2i,2}&=&-\alpha f^{k}_{i-2}+4 \alpha f^{k}_{i-1}+(1-6 \alpha)f^{k}_{i}+4 \alpha f^{k}_{i+1}-\alpha f^{k}_{i+2},\\ \\
f^{k+1}_{2i+1,2}&=&\beta(1-\alpha)f^{k}_{i-2}-\left(\frac{1}{16}+3\beta(1-\alpha)\right)f^{k}_{i-1}+\left(\frac{9}{16}+2\beta(1-\alpha)\right)f^{k}_{i}\\&&
+\left(\frac{9}{16}+2\beta(1-\alpha)\right)f^{k}_{i+1}-\left(\frac{1}{16}+3\beta(1-\alpha)\right)f^{k}_{i+2}+\beta(1-\alpha)f^{k}_{i+3}.
\end{array}\right.
\end{eqnarray}
\begin{rem}
The family of $(2N+4)$-point interpolatory schemes $S_{a^{I}_{2N+4}}$ with one tension parameter is obtained by putting $\alpha=0$ in (\ref{schm2}). The refinement rules of these schemes are given below

\begin{eqnarray}\label{schm3}
\left\{\begin{array}{ccccccc}
&&f^{k+1}_{2i,N+1}=f^{k}_{i,N+1},\\ \\
&&f^{k+1}_{2i+1,N+1}=\sum\limits_{j=-N-2}^{N+1}b_{2j+1}f^{k}_{i+j+1,N+1},
\end{array}\right.
\end{eqnarray}
where $b_{2j+1}$ is defined in (\ref{b2j+1}).
\end{rem}
\section{Analysis of the families of schemes}\label{Analysis of the families of schemes1}
In this section, we present properties of the proposed families of schemes. In Table \ref{Continuity-S-2N+2}, we present "ranges of tension parameter $\alpha$" for which the first six members of "the family of $(2N+2)$-point relaxed schemes" are $C^{n}$-continuous. The continuity of the proposed schemes is analyzed with a computer algebra system like Mathematica/Maple by using Theorems \ref{thmcontinuity}-\ref{thmcontinuity--1}. Similarly, continuity of the first three members of $(2N+3)$-point relaxed schemes and $(2N+4)$-point interpolatory schemes is presented in Tables \ref{Continuity-S-2N+3} and \ref{table-S-inter-2N+4} respectively for specific ranges of tension parameters.

In Tables \ref{reproduction-S-2N+2}, \ref{reproduction-S-2N+3} and \ref{table-S-inter-2N+4}, we tabulate the support widths, degrees of polynomial generation and degrees of polynomial reproduction of the proposed families of $(2N+2)$-point relaxed, $(2N+3)$-point relaxed and $(2N+4)$-point interpolatory schemes respectively. The generation and reproduction degrees of the proposed schemes are analyzed by using Theorem \ref{thm-gd-rd}, while the support widths of the proposed schemes are calculated by using Theorem \ref{support-theorem}. Properties of some special members of the family of schemes (\ref{schm1}) are presented in Table \ref{reproduction-SS-2N+2}.
Let $a_{2N+2}(z)$, $a_{2N+3}(z)$ and $a^{I}_{2N+4}(z)$ denote the symbols of the schemes (\ref{schm1}), (\ref{schm2}) and (\ref{schm3}) respectively. By Definition \ref{primal-dual}, the following theorem can be easily proved.
\begin{thm}\label{primaltheorem}
The families of subdivision schemes (\ref{schm1}), (\ref{schm2}) and (\ref{schm3}) are the families of primal schemes.
\end{thm}
\begin{proof}
Since the symbols associated with the schemes (\ref{schm1}), (\ref{schm2}) and (\ref{schm3}) satisfy the relations $a_{2N+2}(z)=a_{2N+2}(z^{-1})$, $a_{2N+3}(z)=a_{2N+3}(z^{-1})$ and $a^{I}_{2N+4}(z)=a^{I}_{2N+4}(z^{-1})$ respectively for all $N \in \mathbb{N}_{0}$, then by Definition \ref{primal-dual} these schemes are primal. This completes the proof.
\end{proof}
Now we check the property of polynomial generation and reproduction of the families of schemes $S_{a_{2N+2}}$, $S_{a_{2N+3}}$ and $S_{a^{I}_{2N+4}}$ by using Theorem \ref{thm-gd-rd}. We prove the following theorems for this purpose.

\begin{sidewaystable}[p] 
\caption[Continuity of the proposed schemes $S_{a_{2N+2}}$.]{\label{Continuity-S-2N+2}\emph{Let MC be the maximum continuity of the proposed schemes $S_{a_{2N+2}}$.}}
\centering

    \setlength{\tabcolsep}{0.5pt}
\begin{tabular}{||c||c|c|c|c|c|c|c||}
  \hline \hline
   &&&&&&&\\
 N & MC & $C^{0}$                                   & $C^{1}$                                    & $C^{2}$                                   & $C^{3}$ & $C^{4}$ & $C^{5}$\\  &&&&&&&\\
 \hline \hline
  &&&&&&&\\
 0 & 1  & $-\frac{1}{4}<\alpha<\frac{3}{4}$     &       $0<\alpha<\frac{1}{4}$           &  $\alpha=\frac{1}{8}$                 &&&\\ &&&&&&& \\
 \hline
&&&&&&& \\
 1 & 2  & $-\frac{3}{25}<\alpha<\frac{6}{25}$   & $-\frac{1}{35}<\alpha<\frac{27}{200}$  & $0<\alpha<\frac{1}{16}$               &&&\\ &&&&&&&\\
\hline
&&&&&&& \\
 2 & 3  & $-\frac{1}{21}<\alpha<\frac{1}{14}$   & $-\frac{1}{64}<\alpha<\frac{1}{25}$    & $-\frac{1}{80}<\alpha<\frac{1}{48}$   & $\frac{33}{5000}<\alpha<\frac{19}{1250}$ &&\\ &&&&&&&\\
\hline
&&&&&&& \\
 3 & 4  & $-\frac{1}{256}<\alpha<\frac{1}{116}$ & $-\frac{1}{151}<\alpha<\frac{1}{80}$   & $-\frac{1}{384}<\alpha<\frac{1}{160}$ & $-\frac{1}{21000}<\alpha<\frac{1}{246}$ & $\frac{29}{50000}<\alpha<\frac{41}{50000}$    &\\ &&&&&&&\\
\hline
&&&&&&& \\
 4 & 5  & $-\frac{1}{70}<\alpha<\frac{1}{129}$  & $-\frac{1}{416}<\alpha<\frac{1}{255}$  & $-\frac{1}{838}<\alpha<\frac{1}{503}$ & $-\frac{1}{3560}<\alpha<\frac{1}{828}$         & $-\frac{1}{200000}<\alpha<\frac{1}{1354}$  & $\frac{1}{4099}<\alpha<\frac{1}{1717}$ \\ &&&&&&&\\
\hline
&&&&&&& \\
 5 & 5  & $-\frac{1}{509}<\alpha<\frac{1}{400}$ & $-\frac{1}{1208}<\alpha<\frac{1}{804}$  & $-\frac{1}{2212}<\alpha<\frac{1}{1556}$ & $-\frac{1}{7440}<\alpha<\frac{1}{2716}$        & $-\frac{1}{26106}<\alpha<\frac{1}{4446}$ & $\frac{1}{10160}<\alpha<\frac{1}{7240}$\\  &&&&&&&\\
\hline \hline
\end{tabular}
\end{sidewaystable}

\begin{table}[htb] 
\caption[Properties of schemes $S_{a_{2N+2}}$ for $\forall$ $\alpha$.]{\label{reproduction-S-2N+2}\emph{Properties of schemes $S_{a_{2N+2}}$ for $\forall$ $\alpha$.}}
\begin{center}
\begin{tabular}{||c||cccccc||}
  \hline
  \hline
 N            & 0       & 1          & 2       & 3        & 4           & 5  \\
\hline \hline
 Support      & 4       & 8          & 12      & 16       & 20          & 24              \\ \hline
 Generation   & 1       & 3          & 5       & 7        & 9           & 11              \\ \hline
 Reproduction & 1       & 3          & 5       & 7        & 9           & 11              \\
\hline
\hline
\end{tabular}
\end{center}
\end{table}

\begin{table}[htb] 
\caption[Properties of schemes $S_{a_{2N+2}}$ for specific values of $\alpha$.]{\label{reproduction-SS-2N+2}\emph{Properties of schemes $S_{a_{2N+2}}$ for specific values of $\alpha$.}}
\begin{center}
\begin{tabular}{||c||cccccc||}
  \hline
  \hline
 N            & 0       & 1          & 2       & 3        & 4           & 5  \\
\hline \hline
 $\alpha$     & $\frac{1}{8}$ & $\frac{3}{128}$ & $\frac{5}{1024}$ & $\frac{35}{32768}$ & $\frac{63}{262144}$ & $\frac{231}{4194304}$ \\ \hline
 Continuity & $C^{2}$ & $C^{2}$ & $C^{2}$ & $C^{3}$ & $C^{5}$ & $C^{3}$ \\ \hline
 Generation   & 3       & 5          & 7       & 9        & 11          & 13              \\ \hline
 Reproduction & 1       & 3          & 5       & 7        & 9           & 11              \\
\hline
\hline
\end{tabular}
\end{center}
\end{table}

\begin{sidewaystable}[p] 
\caption[Continuity of the proposed schemes $S_{a_{2N+3}}$.]{\label{Continuity-S-2N+3}\emph{Let MC be the maximum continuity of the proposed scheme $S_{a_{2N+3}}$.}}
\centering

    \setlength{\tabcolsep}{0.5pt}
\begin{small}
\begin{tabular}{||c|c|c|c|c|c|c|c||}
  \hline \hline
   &&&&&&&\\
 N & $\alpha / \beta$ & MC & $C^{0}$                                   & $C^{1}$                                    & $C^{2}$                                   & $C^{3}$ & $C^{4}$ \\  &&&&&&&\\
 \hline \hline
  &&&&&&&\\
    & $\alpha=\frac{1}{8}\left(\frac{16 \beta+1}{2 \beta+1}\right)$ &   & $-\frac{3}{20}<\beta<\frac{5}{4}$     &       $-\frac{1}{16}<\beta<\frac{3}{8}$           &  $-\frac{1}{16}<\beta<\frac{3}{8}$                 & $0<\beta<\frac{1}{12}$ & $\beta=\frac{1}{26}$\\ $0$ && $4$&&&&& \\
    & $\beta=-\frac{1}{16}\left(\frac{8 \alpha-1}{\alpha-1}\right)$ &   & $-\frac{1}{4}<\alpha<\frac{3}{4}$     &       $0<\alpha<\frac{1}{2}$           &  $0<\alpha<\frac{1}{2}$                 & $\frac{1}{8}<\alpha<\frac{1}{4}$ & $\alpha=\frac{3}{16}$\\ &&&&&&& \\
 \hline
    &&&&&&&\\
     & $\alpha=\frac{1}{128}\left(\frac{256 \beta-3}{2 \beta-1}\right)$ &  & $-\frac{21}{208}<\beta<\frac{11}{272}$ & $-\frac{13}{224}<\beta<\frac{3}{256}$ & $-\frac{13}{224}<\beta<\frac{3}{256}$ & $-\frac{2}{121}<\beta<-\frac{1}{248}$ & $-\frac{2}{121}<\beta<-\frac{1}{248}$  \\
  $1$ && $4$ &&&&&\\
   & $\beta=\frac{1}{256}\left(\frac{128\alpha-3}{\alpha-1}\right)$ &  & $-\frac{1}{16}<\alpha<\frac{3}{16}$ & $0<\alpha<\frac{1}{8}$ & $0<\alpha<\frac{1}{8}$ & $\frac{1}{32}<\alpha<\frac{7}{128}$ & $\frac{1}{32}<\alpha<\frac{7}{128}$  \\
   &&&&&&&\\
    \hline
    &&&&&&&\\
     & $\alpha=\frac{1}{1024}\left(\frac{2048 \beta+5}{2 \beta+1}\right)$ &  & $-\frac{21}{2080}<\beta<\frac{113}{5888}$ & $-\frac{5}{2048}< \beta < \frac{11}{997}$ & $-\frac{5}{2048}< \beta < \frac{11}{997}$ & $\frac{3}{2032} < \beta < \frac{8}{3041}$ & $\frac{3}{2032} < \beta < \frac{8}{3041}$  \\
 $2$  && $4$ &&&&&\\
   & $\beta=-\frac{1}{2048}\left(\frac{1024 \alpha-5}{\alpha-1}\right)$ &  & $-\frac{1}{64}<\alpha<\frac{1}{24}$ & $0< \alpha < \frac{27}{1024}$ & $0< \alpha < \frac{27}{1024}$ & $\frac{1}{128} < \alpha < \frac{31}{3072}$ & $\frac{1}{128} < \alpha < \frac{31}{3072}$  \\
   &&&&&&&\\
\hline \hline
\end{tabular}
\end{small}
\end{sidewaystable}

\begin{sidewaystable}[p] 
\caption[Properties of the proposed schemes $S_{a_{2N+3}}$.]{\label{reproduction-S-2N+3}\emph{Properties of schemes $S_{a_{2N+3}}$. Here Sp, GD and RD represent support width, generation degree and reproduction degree respectively.}}
\centering \setlength{\tabcolsep}{0.5pt}
\begin{small}
\begin{tabular}{||c||c|c|c||}
  \hline
  \hline
 N            & 0       & 1          & 2        \\
\hline
 Sp      & 6       & 10         & 14       \\
    \hline
      &&&\\
              & $1$ $\forall$ $\alpha$ $\&$ $\beta$ & $3$ $\forall$ $\alpha$ $\&$ $\beta$ & $5$ $\forall$ $\alpha$ $\&$ $\beta$\\ &&& \\
&$3$ $\forall$ $\alpha$ where $\beta=-\frac{1}{16}\left(\frac{8 \alpha-1}{\alpha-1}\right)$      & $5$ $\forall$ $\alpha$ where $\beta=\frac{1}{256}\left(\frac{128 \alpha-3}{\alpha-1}\right)$           & $7$ $\forall$ $\alpha$ where $\beta=-\frac{1}{2048}\left(\frac{1024 \alpha-5}{\alpha-1}\right)$ \\
 GD    &$3$ $\forall$ $\beta$ where $\alpha=\frac{1}{8}\left(\frac{16 \beta+1}{2 \beta+1}\right)$      & $5$ $\forall$ $\beta$ where $\alpha=\frac{1}{128}\left(\frac{256 \beta-3}{2 \beta-1}\right)$           & $7$ $\forall$ $\beta$ where $\alpha=\frac{1}{1024}\left(\frac{2048 \beta+5}{2 \beta+1}\right)$ \\  &&& \\

 &$5$ for $\alpha=\frac{3}{16}$ $\&$ $\beta=-\frac{1}{16}\left(\frac{8 \alpha-1}{\alpha-1}\right)$ & $7$ for $\alpha=\frac{5}{128}$ $\&$ $\beta=\frac{1}{256}\left(\frac{128 \alpha-3}{\alpha-1}\right)$  &  $9$ for $\alpha=\frac{35}{4096}$ $\&$ $\beta=-\frac{1}{2048}\left(\frac{1024 \alpha-5}{\alpha-1}\right)$     \\
 &$5$ for $\beta=\frac{1}{26}$ where $\alpha=\frac{1}{8}\left(\frac{16 \beta+1}{2 \beta+1}\right)$      & $7$ for $\beta=-\frac{1}{123}$ where $\alpha=\frac{1}{128}\left(\frac{256 \beta-3}{2 \beta-1}\right)$           & $9$ for $\beta=\frac{15}{8122}$ where $\alpha=\frac{1}{1024}\left(\frac{2048 \beta+5}{2 \beta+1}\right)$ \\
                &&&\\
                 \hline
                   &&&\\
 &  $1$ $\forall$ $\alpha$ $\&$ $\beta$        & $3$ $\forall$ $\alpha$ $\&$ $\beta$          & $5$ $\forall$ $\alpha$ $\&$ $\beta$          \\ &&& \\

 RD &  $3$ for $\alpha=0$ $\&$ $\beta=-\frac{1}{16}$ & $5$ for $\alpha=0$ $\&$ $\beta=\frac{3}{256}$  & $7$ for $\alpha=0$ $\&$ $\beta=-\frac{5}{2048}$ \\
              &  $3$ for $\alpha=-\frac{1}{8}$ $\&$ $\beta=0$ & $5$ for $\alpha=-\frac{3}{128}$ $\&$ $\beta=0$  & $7$ for $\alpha=-\frac{5}{1024}$ $\&$ $\beta=0$ \\
                &&&\\
\hline
\hline
\end{tabular}
\end{small}
\end{sidewaystable}

\begin{table}[htb] 
\caption[Properties of the proposed schemes $S_{a^{I}_{2N+4}}$.]{\label{table-S-inter-2N+4}\emph{Let GD and RD be the degree of polynomial generation and degree of polynomial reproduction of the proposed schemes $S_{a^{I}_{2N+4}}$.}}
\centering \setlength{\tabcolsep}{0.5pt}
\begin{tabular}{||c||c|c|c||}
  \hline \hline
N & $0$ & $1$ & $2$ \\
\hline \hline
& & & \\
$C^{0}$ & $-\frac{1}{8}< \beta < 0$ & $-\frac{5}{48}<\beta<\frac{1}{16}$ & $-\frac{39}{2560}<\beta<\frac{89}{2560}$ \\
& & & \\
\hline
& & & \\
$C^{1}$ & $\frac{2-\sqrt{6}}{8}<\beta < \frac{1-\sqrt{3}}{8}$ & $\frac{15-\sqrt{273}}{96}<\beta<\frac{-5+3\sqrt{37}}{224}$ & $\frac{219-\sqrt{661133}}{31232}< \beta <\frac{-171+\sqrt{56713}}{6144}$ \\
& & & \\
\hline
& & & \\
$C^{2}$ &  & $0<\beta<\frac{1}{36}$ & $\frac{16-\sqrt{286}}{96}<\beta<\frac{-675+\sqrt{637161}}{31744}$ \\
& & & \\
\hline
& & & \\
$C^{3}$ &  &  & $-\frac{157}{20000}<\beta<-\frac{43}{20000}$ \\
& & & \\
\hline
& & & \\
GD & 1 $\forall$ $\beta$ & 3 $\forall$ $\beta$ & 5 $\forall$ $\beta$ \\
& 3 for $\beta=-\frac{1}{16}$ & 5 for $\beta=\frac{3}{256}$ & 7 for $\beta=-\frac{5}{2048}$ \\
& & & \\
\hline
& & & \\
RD & 1 $\forall$ $\beta$ & 3 $\forall$ $\beta$ & 5 $\forall$ $\beta$ \\
& 3 for $\beta=-\frac{1}{16}$ & 5 for $\beta=\frac{3}{256}$ & 7 for $\beta=-\frac{5}{2048}$ \\
& & & \\
\hline \hline
\end{tabular}
\end{table}

\begin{thm}\label{GDtheorem}
The families of subdivision schemes (\ref{schm1}), (\ref{schm2}) and (\ref{schm3}) generate polynomials up to degree $2N+1$ for all $\alpha$, $\beta$ and for all $N \in \mathbb{N}_{0}$.
\end{thm}
\begin{proof}
It is to be noted that for all $\alpha$ and $\beta$ the results $a_{2N+2}(1)=a_{2N+3}(1)=a^{I}_{2N+4}(1)=2$ and $a_{2N+2}(-1)=a_{2N+3}(-1)=a^{I}_{2N+4}(-1)=0$ are trivial for all $N \in \mathbb{N}_{0}$. Furthermore, $\left. D^{(m)} a_{2N+2}(z)\right|_{z=-1}=\left. D^{(m)} a_{2N+3}(z)\right|_{z=-1}=\left. D^{(m)}a^{I}_{2N+4}(z)\right|_{z=-1}=0$ for $m=0,1,\ldots,2N+1$ and $N \in \mathbb{N}_{0}$. Therefore, by Theorem \ref{thm-gd-rd} the result proved.
\end{proof}

\begin{thm}\label{RDtheorem}
The families of subdivision schemes (\ref{schm1}), (\ref{schm2}) and (\ref{schm3}) reproduce polynomials up to degree $2N+1$ for all $\alpha$, $\beta$ and $N \in \mathbb{N}_{0}$.
\end{thm}
\begin{proof}
Since by Theorem \ref{primaltheorem}, the families of schemes (\ref{schm1}), (\ref{schm2}) and (\ref{schm3}) are primal, hence the parameter $\tau=0$ for all $\alpha$, $\beta$ and $N \in \mathbb{N}_{0}$. Moreover, $\left. D^{(m)} a_{2N+2}(z)\right|_{z=1}=\left. D^{(m)} a_{2N+3}(z)\right|_{z=1}=\left. D^{(m)}a^{I}_{2N+4}(z)\right|_{z=1}=2 \prod\limits_{h=0}^{m-1}(\tau-h)$ for $m=0,1,\ldots,2N+1$ and $N \in \mathbb{N}_{0}$. Now by combining these conditions with Theorem \ref{GDtheorem} and then by using Theorem \ref{thm-gd-rd}, we get the required result.
\end{proof}
\begin{rem}
Theorems \ref{GDtheorem}-\ref{RDtheorem} show that the schemes defined in (\ref{schm1}), (\ref{schm2}) and (\ref{schm3}) generate and reproduce polynomials up to degree $2N+1$ for all $\alpha$ and $\beta$. However, by choosing specific values of $\alpha$ and $\beta$, the polynomial generation and polynomial reproduction of these schemes can be increased. Tables \ref{reproduction-SS-2N+2}, \ref{reproduction-S-2N+3} and \ref{table-S-inter-2N+4} summarize the generation and reproduction degrees of proposed schemes for specials choices of tension parameters.
\end{rem}
In the coming theorem, we give the support width of the proposed families of schemes.
\begin{thm}
The support width of the family of schemes (\ref{schm1}) is $4N+4$ where $N \in \mathbb{N}_{0}$. The support width of each family of subdivision schemes defined in (\ref{schm2}) and (\ref{schm3}) is $4N+6$.
\end{thm}
\begin{proof}
By Theorem \ref{support-theorem} the result is trivial.
\end{proof}

We prove the following theorem by using Theorems \ref{thmcontinuity}-\ref{thmcontinuity--1}.

\begin{thm}\label{}
The subdivision schemes $S_{a_{3}}$, $S_{a_{5}}$ and $S_{a_{7}}$ are $C^{1}$, $C^{3}$ and $C^{5}$ respectively for some special conditions.
\end{thm}

\begin{proof}
The conditions for the $C^{1}$, $C^{3}$ and $C^{5}$ continuities of the schemes $S_{a_{3}}$, $S_{a_{5}}$ and $S_{a_{7}}$ respectively are given below:

The subdivision scheme $S_{a_{3}}$ is $C^{1}$ if
 \begin{eqnarray*}
&&  \mbox{max} \{\gamma_{1},\gamma_{2}+\gamma_{3}\}<1,
 \end{eqnarray*}
 where $\gamma_{1}=2\left|2\,\alpha-4 \beta\,(1-\alpha)\right|$, $\gamma_{2}=2\left|2 \beta\,(1-\alpha)\right|$ and $\gamma_{3}=\left|1+4\,\beta\,(1-\right.$ $\left.\alpha)- 4\,\alpha \right|$.\\
The scheme $S_{a_{5}}$ is $C^{3}$ if
\begin{eqnarray*}
&& \mbox{max}\{\eta_{1} +\eta_{2},  \eta_{3}+ \eta_{4}\}<1,
\end{eqnarray*}
where $\eta_{1}=16\, \left| \beta\, \left( 1-\alpha \right)  \right|$, $\eta_{2}=2\, \left| -56\,\beta\, \left( 1-\alpha \right) -32\,\alpha+\frac{1}{2} \right|$,\\ $\eta_{3}=2\, \left| -32\,\beta\, \left( 1-\alpha \right) -8\,\alpha \right|$ and $\eta_{4}=\left| -48\,\alpha-64\,\beta\, \left( 1-\alpha \right) +2 \right|$.

Similarly, the subdivision scheme $S_{a_{7}}$ is $C^{5}$ if
 \begin{eqnarray*}
&& \mbox{max}\{\chi_{1} + \chi_{2} + \chi_{3}, \chi_{4} + \chi_{5}\}<1,
 \end{eqnarray*}
      where $\chi_{1}=64\, \left| \beta\, \left( 1-\alpha \right)  \right|$, $\chi_{2}=2\, \left| -512 \,\beta\, \left( 1-\alpha \right) +192\,\alpha-\frac{3}{8} \right|$,\\ $\chi_{3}=\left| -{\frac {19}{4}}-960\,\beta\, \left( 1-\alpha \right) +640\,\alpha \right|$, $\chi_{4}=2\, \left| -192\,\beta\, \left( 1-\alpha \right) +32\,\alpha \right|$ \\ and $\chi_{5}=2\, \left| 480\,\alpha-\frac{9}{4}-832\,\beta\, \left( 1-\alpha \right) \right|$.
\end{proof}

\begin{rem}
Throughout the article, red bullets and red lines represent initial points and initial polygons/meshes respectively.
\end{rem}
\begin{figure}[!h] 
\begin{center}
\begin{tabular}{ccc}
\epsfig{file=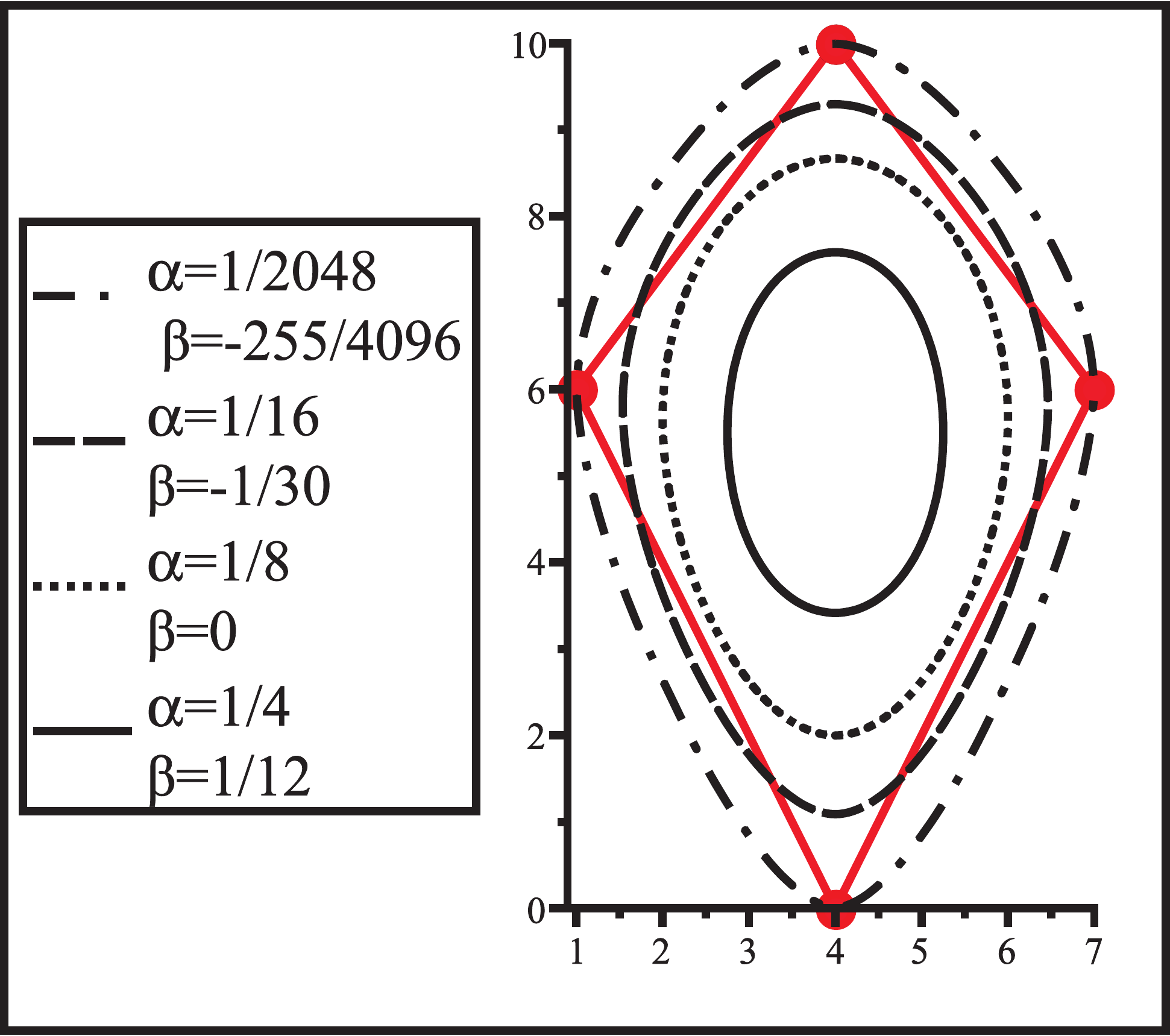, width=1.8 in} & \epsfig{file=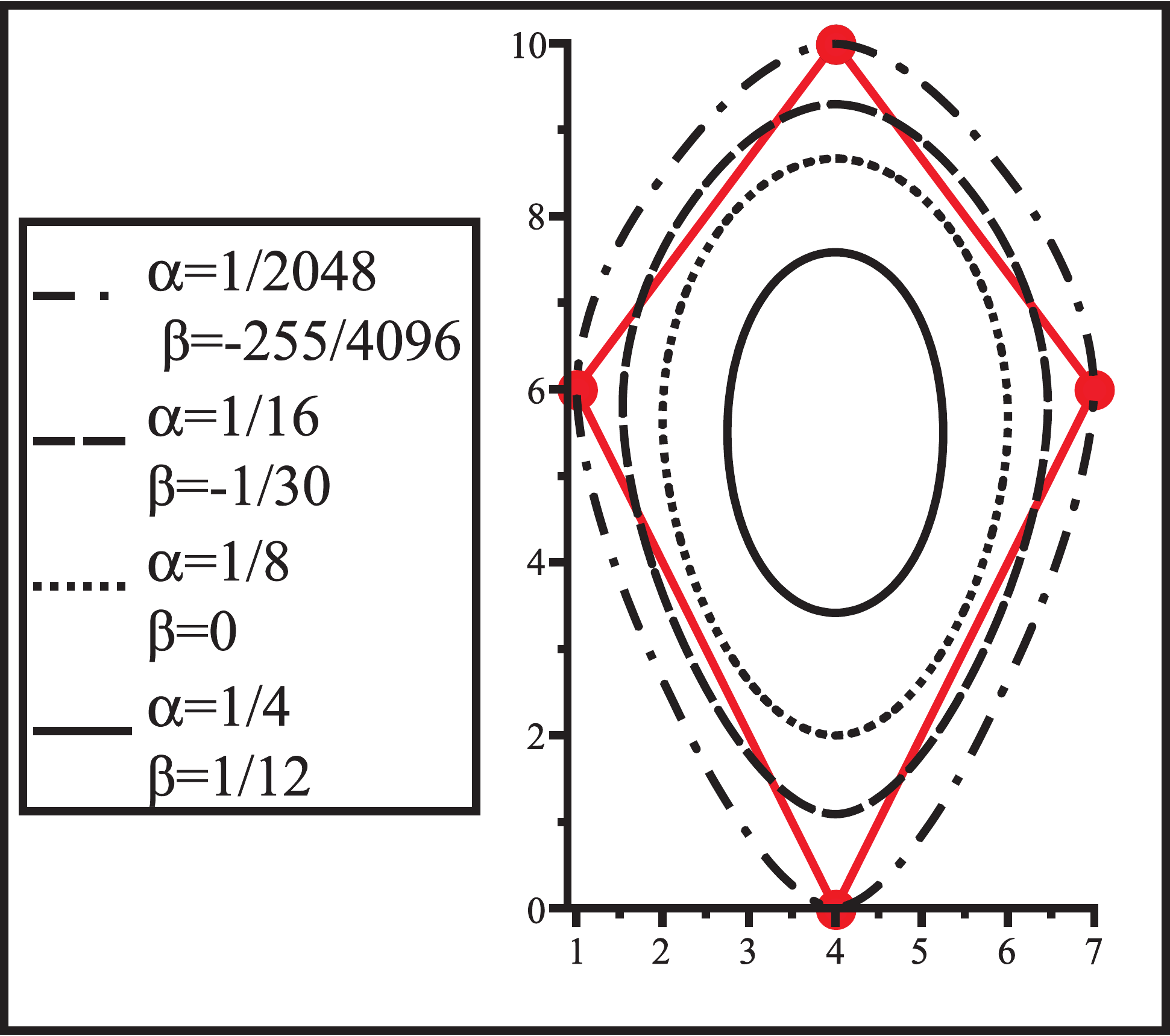, width=1.8 in} & \epsfig{file=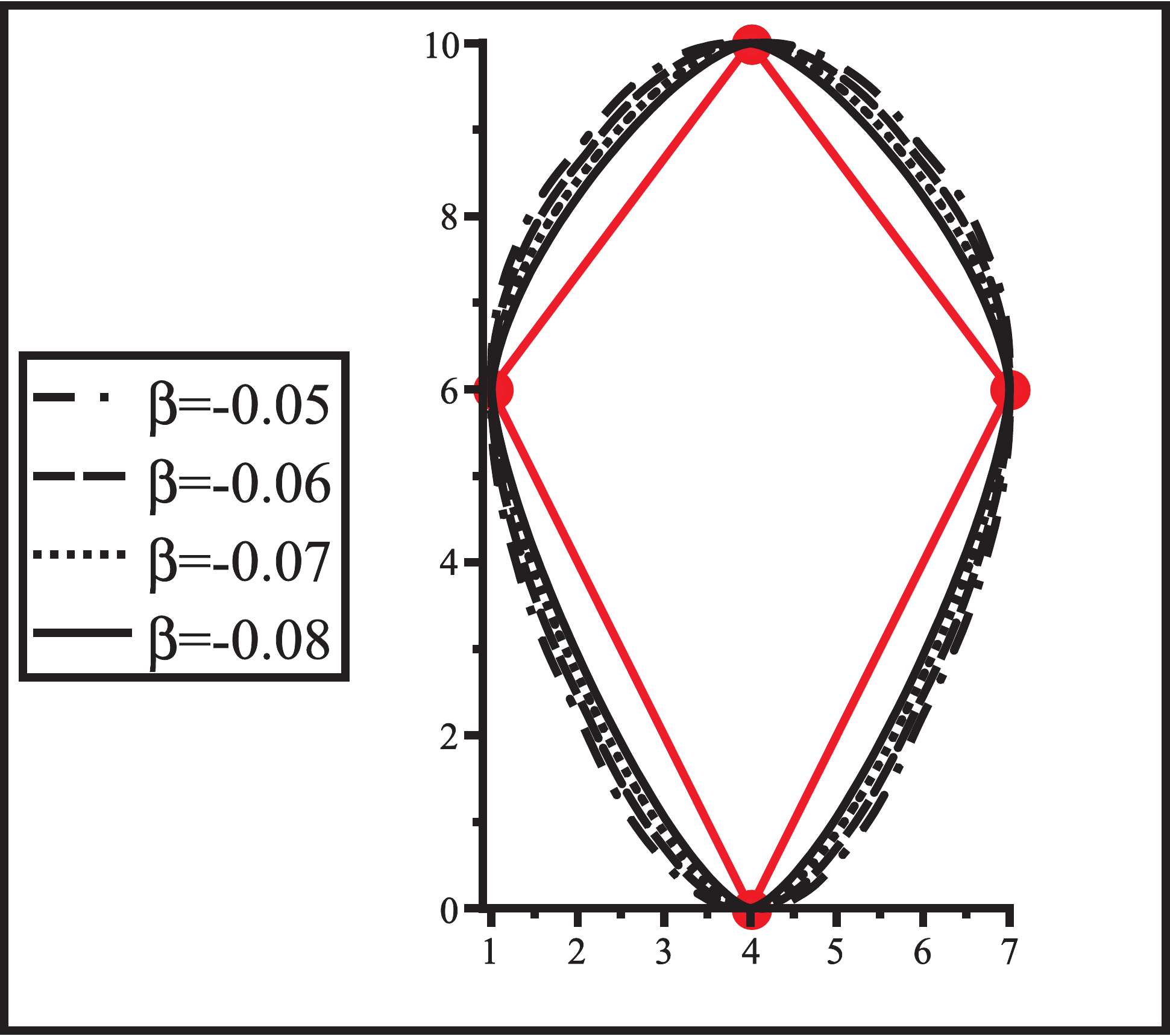, width=1.8 in}\\
(a) & (b) & (c)
\end{tabular}
\end{center}
 \caption[Limit curves obtained by the subdivision schemes $S_{a_{2}}$, $S_{a_{3}}$ $\&$ $S_{a^{I}_{4}}$.]{\label{p5-3-point}\emph{Black curves are the limit curves obtained by the subdivision schemes (a) $S_{a_{2}}$ (b) $S_{a_{3}}$ $\&$ (c) $S_{a^{I}_{4}}$ respectively.}}
\end{figure}
\begin{figure}[!h] 
\begin{center}
\begin{tabular}{ccc}
\epsfig{file=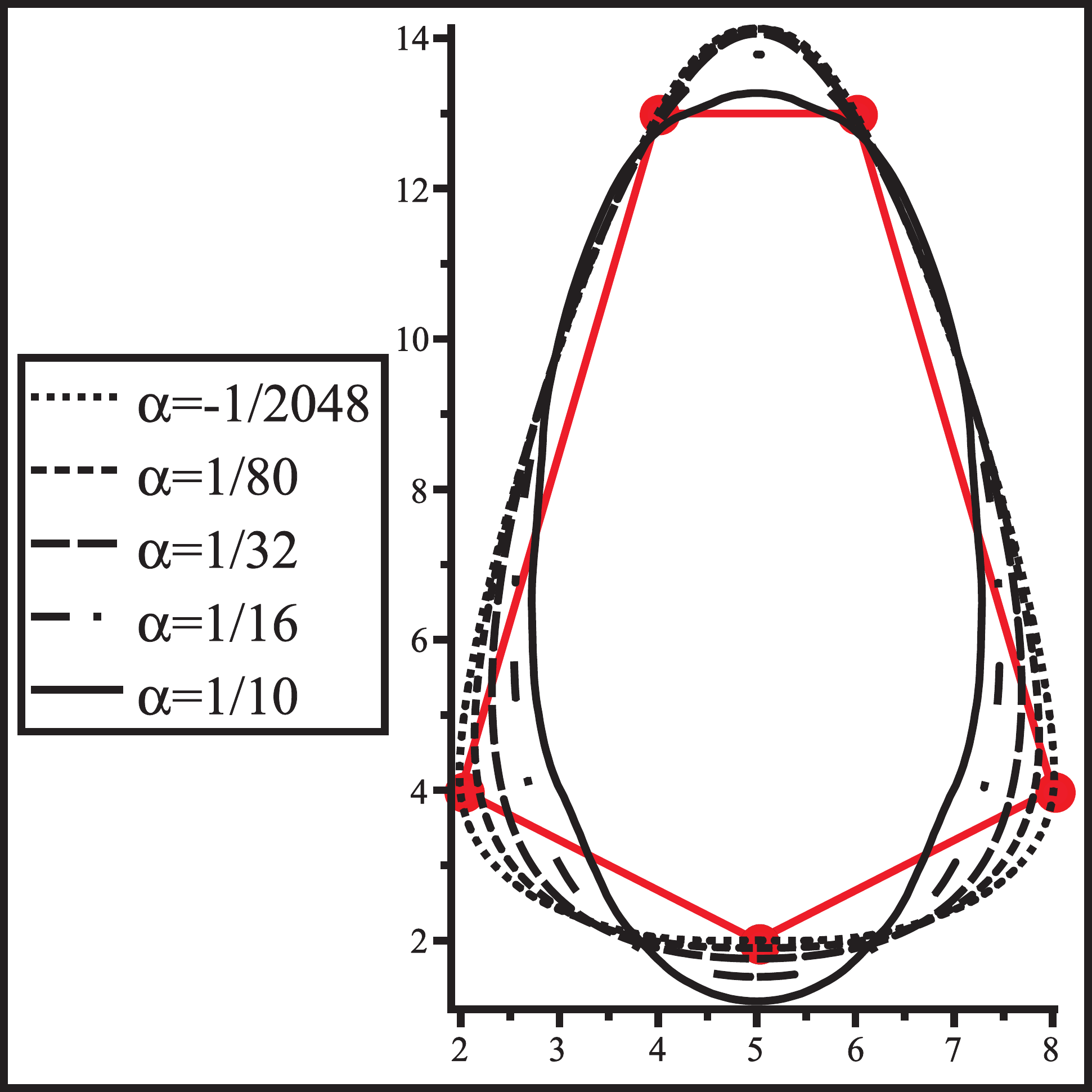, width=1.8 in} & \epsfig{file=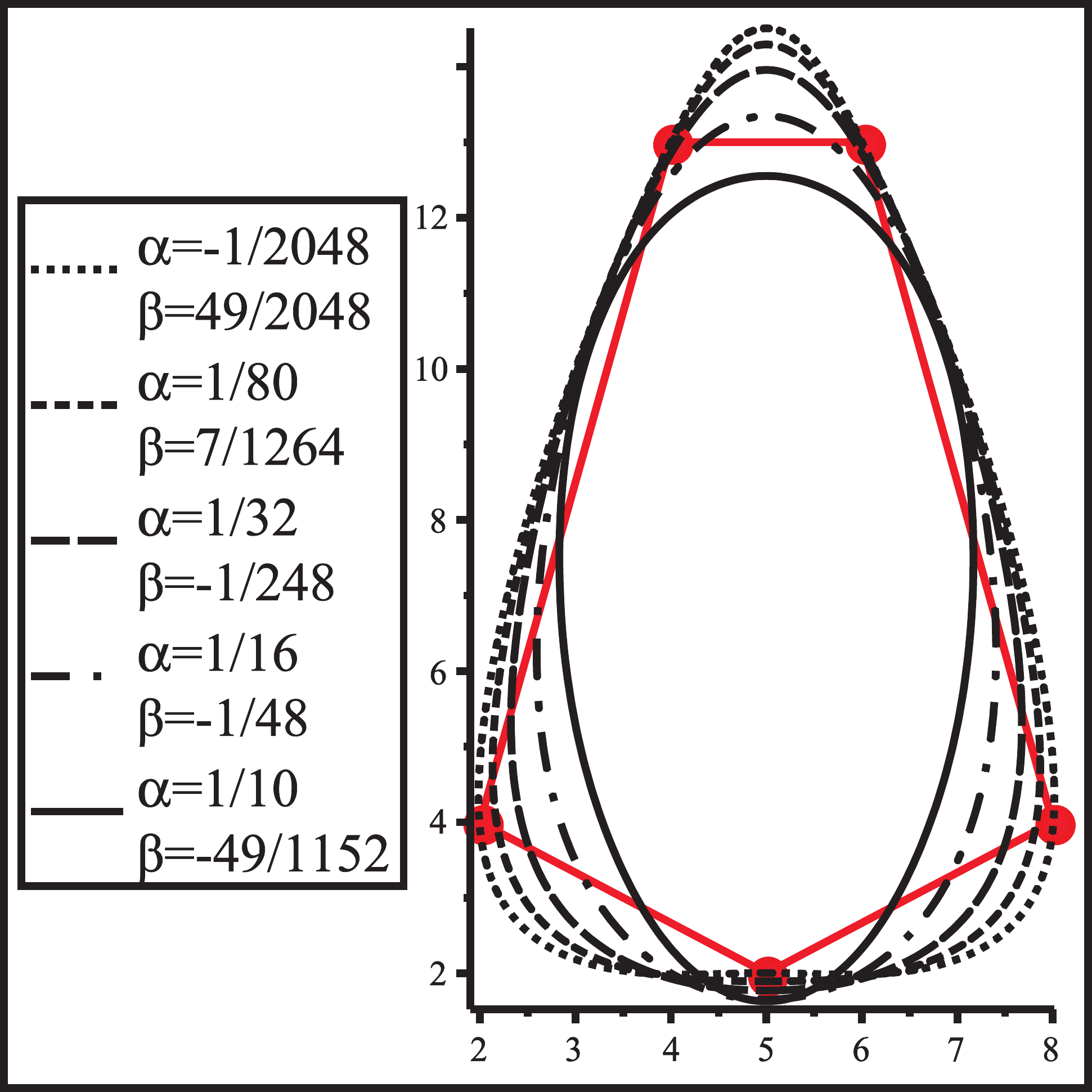, width=1.8 in} & \epsfig{file=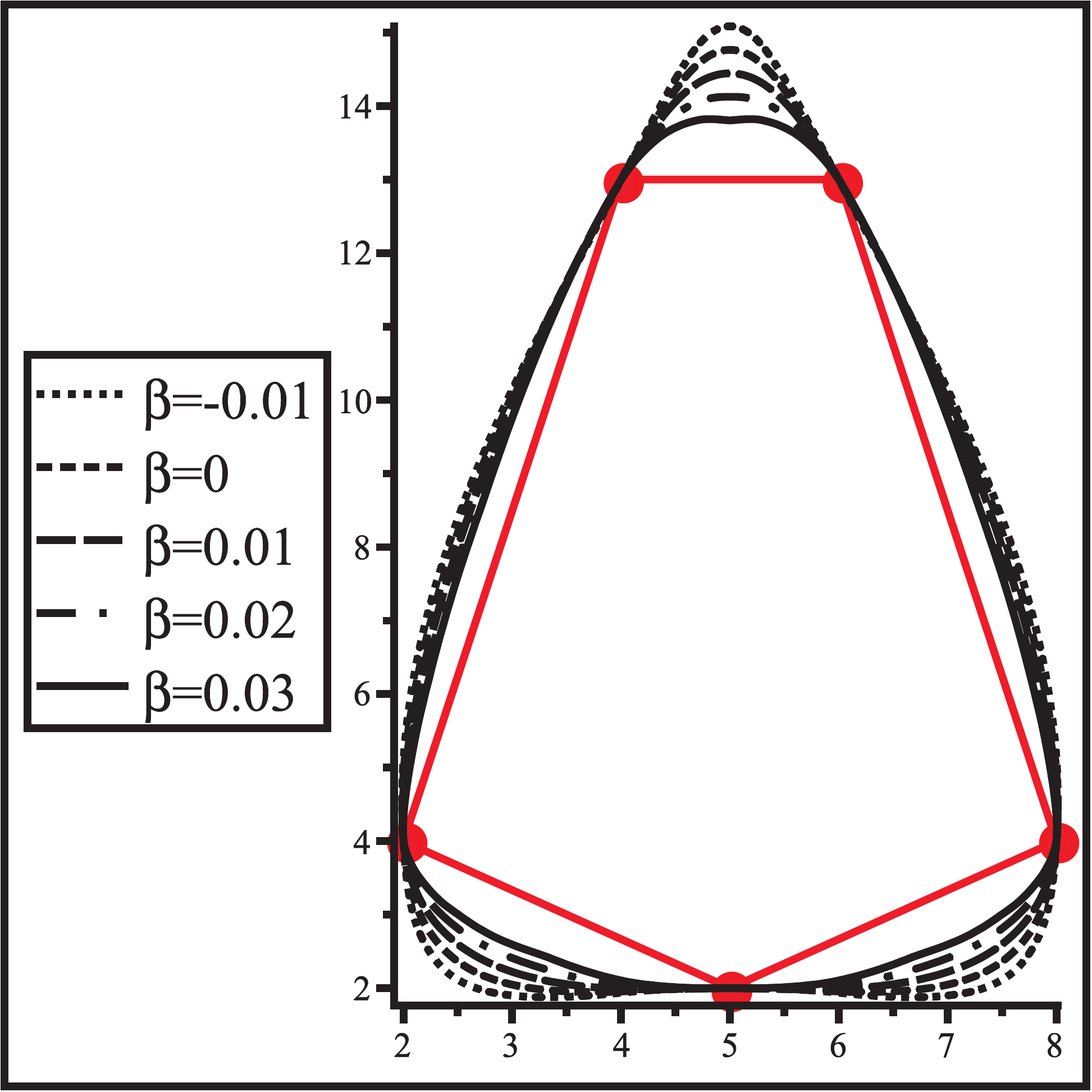, width=1.8 in}\\
(a) & (b) & (c)
\end{tabular}
\end{center}
 \caption[Limit curves obtained by the subdivision schemes $S_{a_{4}}$, $S_{a_{5}}$ $\&$ $S_{a^{I}_{6}}$.]{\label{p5-5-point}\emph{Black curves are the limit curves obtained by the subdivision schemes (a) $S_{a_{4}}$ (b) $S_{a_{5}}$ $\&$ (c) $S_{a^{I}_{6}}$ respectively.}}
\end{figure}
\begin{figure}[!h] 
\begin{center}
\begin{tabular}{ccc}
\epsfig{file=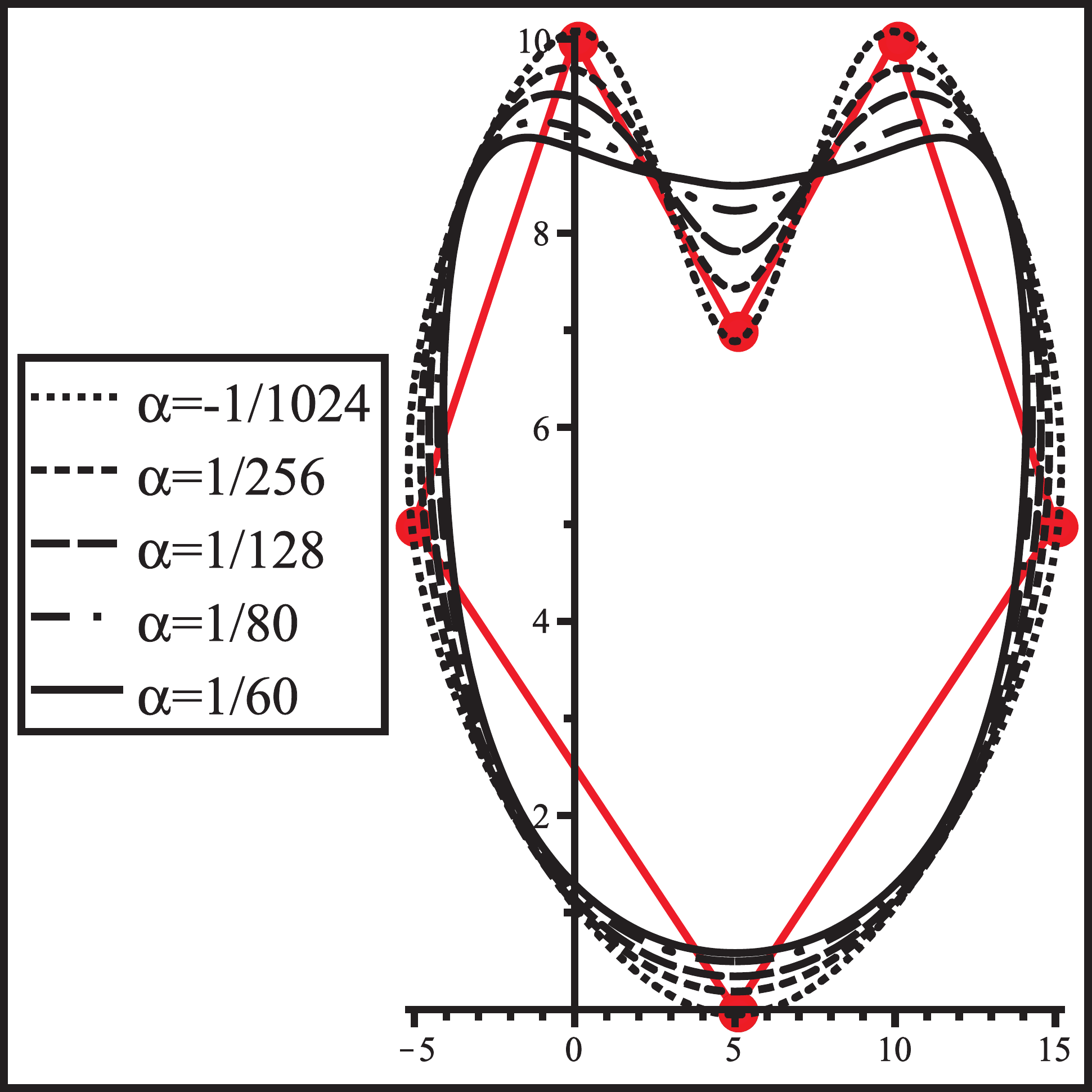, width=1.8 in} & \epsfig{file=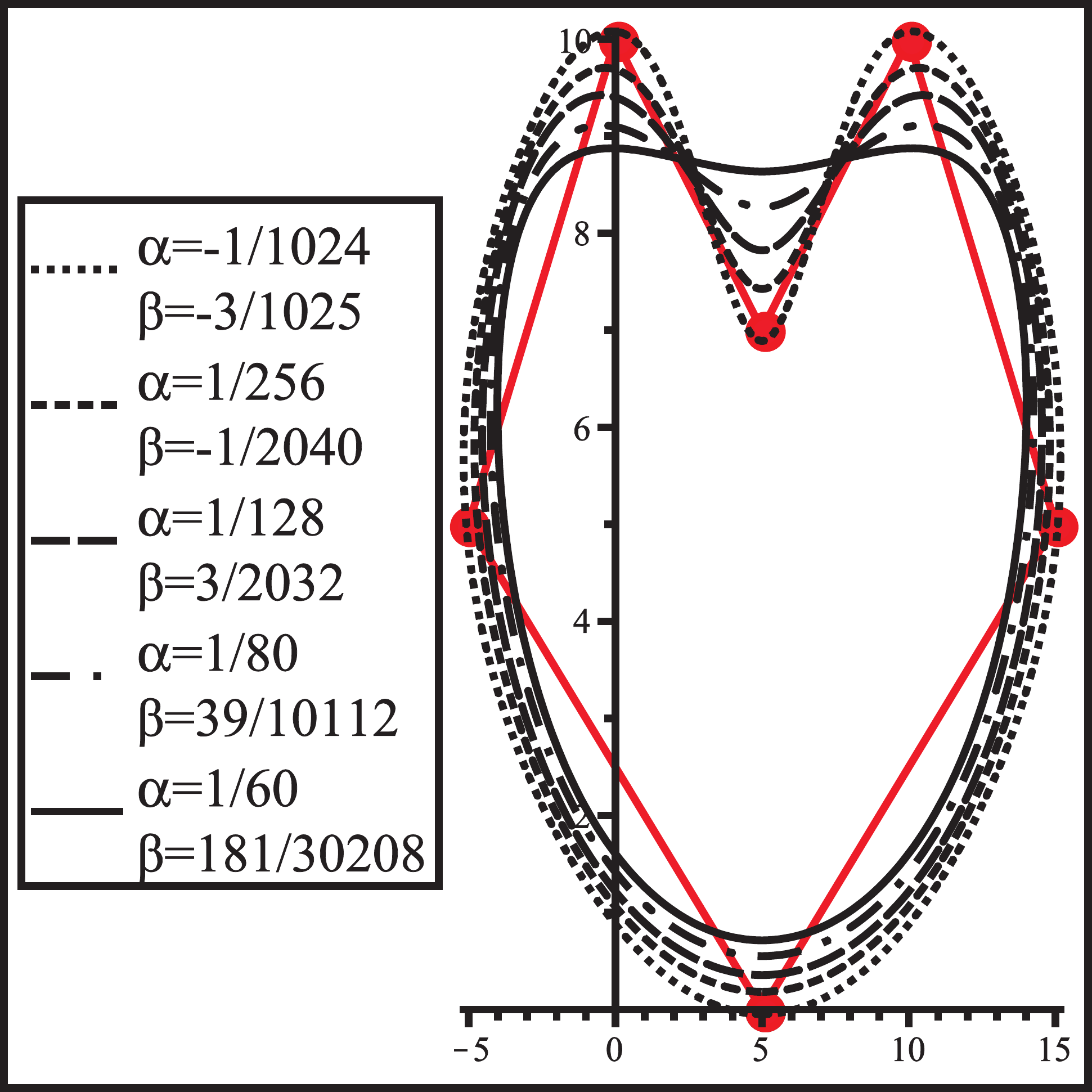, width=1.8 in} & \epsfig{file=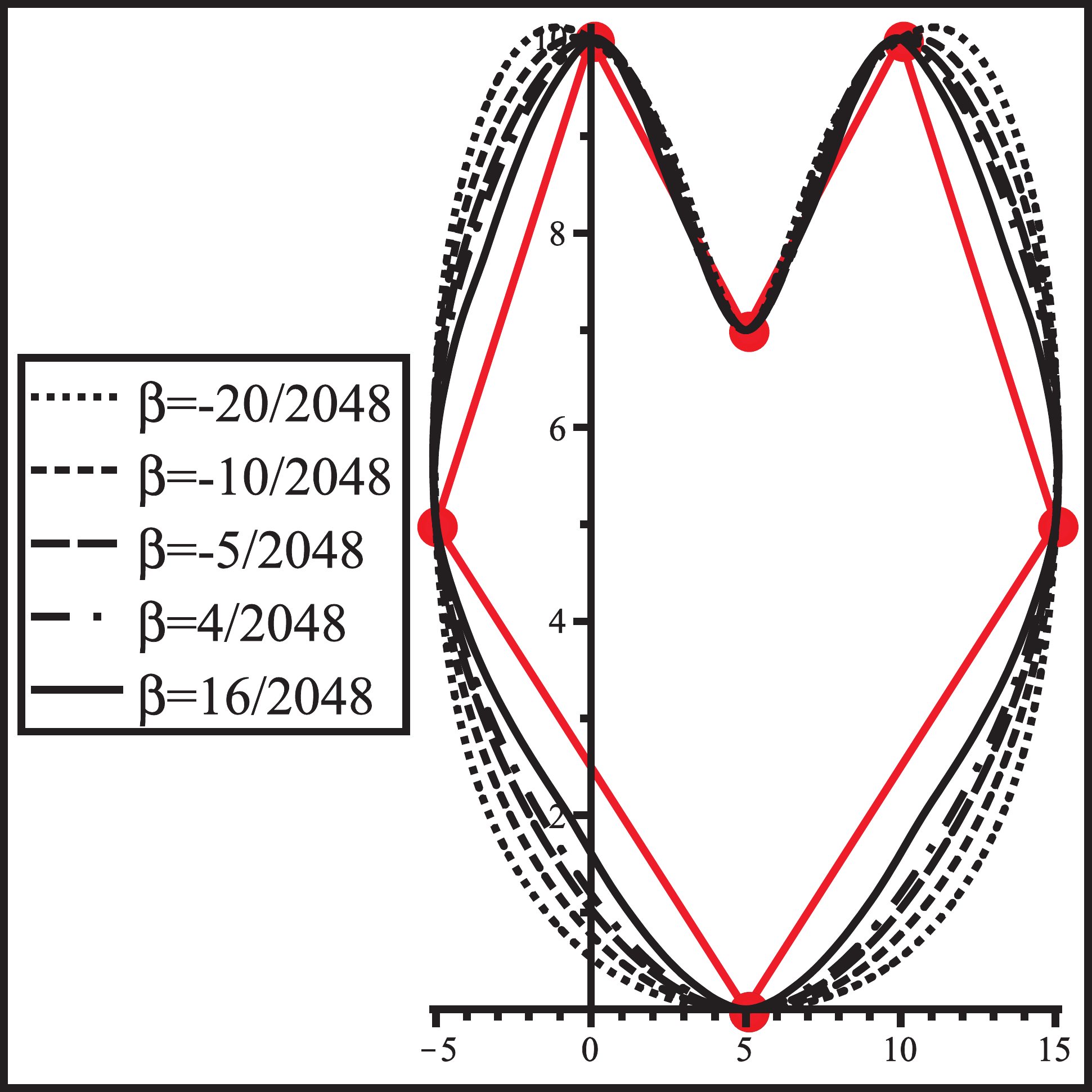, width=1.8 in}\\
(a) & (b) & (c)
\end{tabular}
\end{center}
 \caption[Limit curves obtained by the subdivision schemes $S_{a_{6}}$, $S_{a_{7}}$ $\&$ $S_{a^{I}_{8}}$.]{\label{p5-7-point}\emph{Black curves are the limit curves obtained by the subdivision schemes (a) $S_{a_{6}}$ (b) $S_{a_{7}}$ $\&$ (c) $S_{a^{I}_{8}}$ respectively.}}
\end{figure}

\begin{figure}[!h] 
\begin{center}
\begin{tabular}{cccc}
\epsfig{file=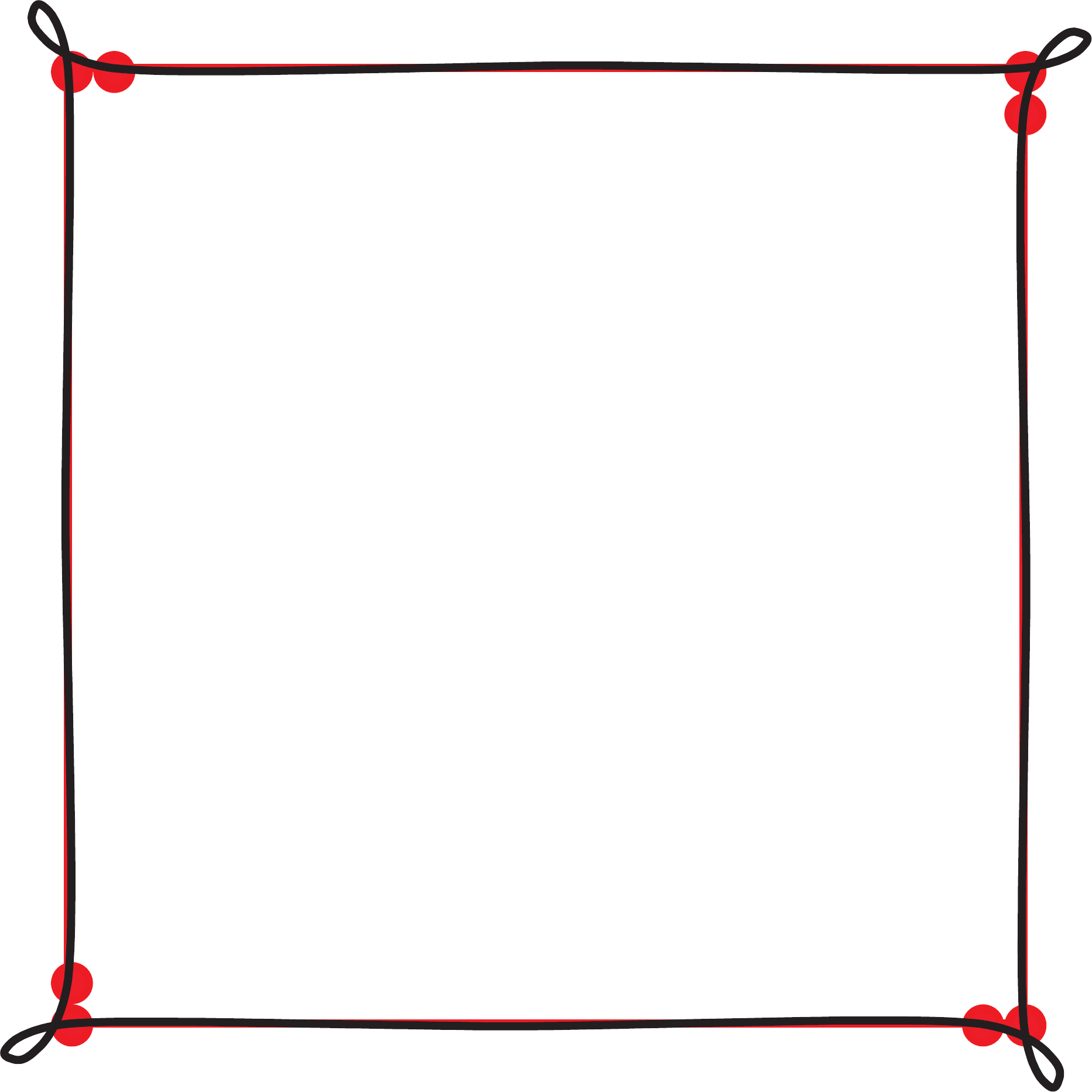, width=1.4 in} & \epsfig{file=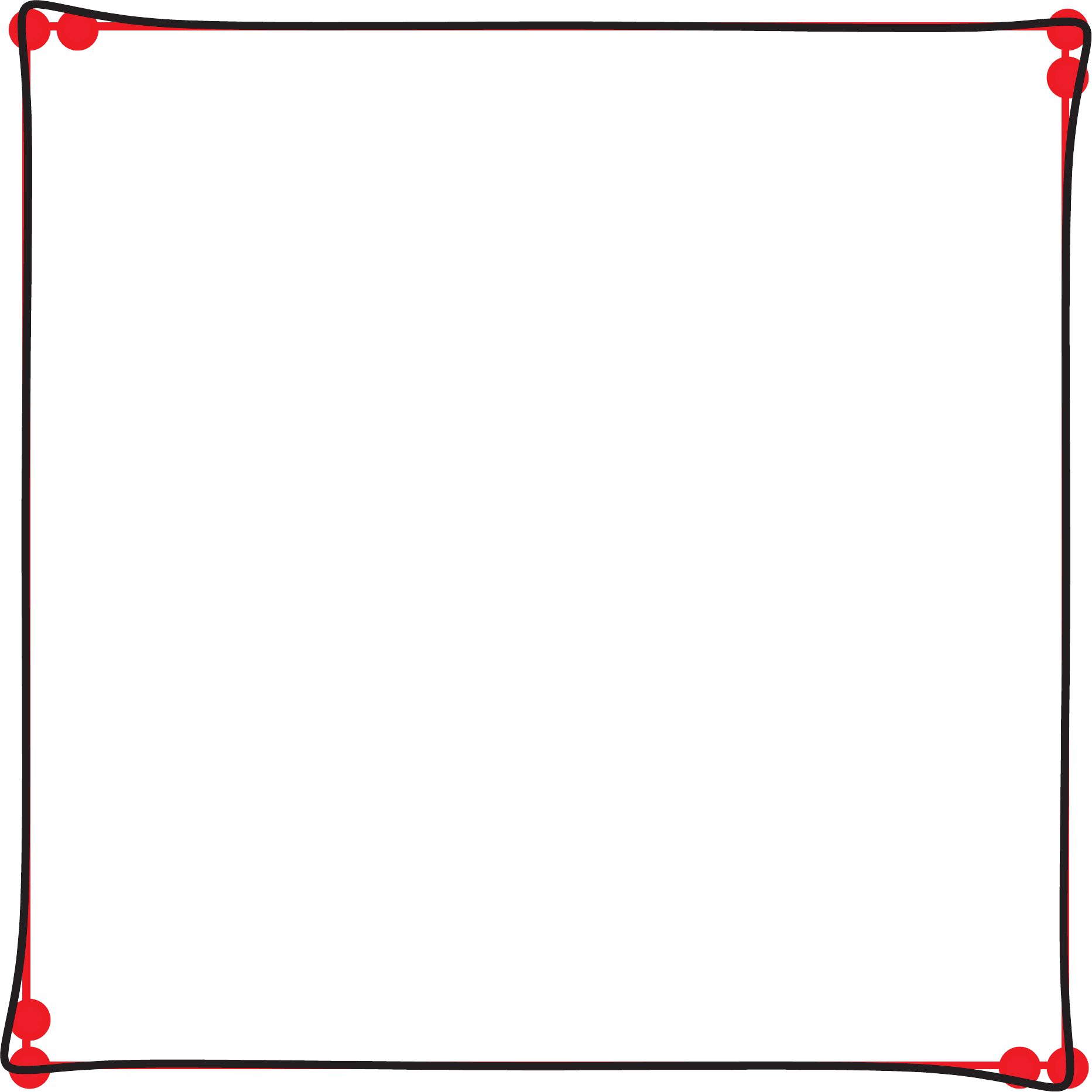, width=1.4 in}  & \epsfig{file=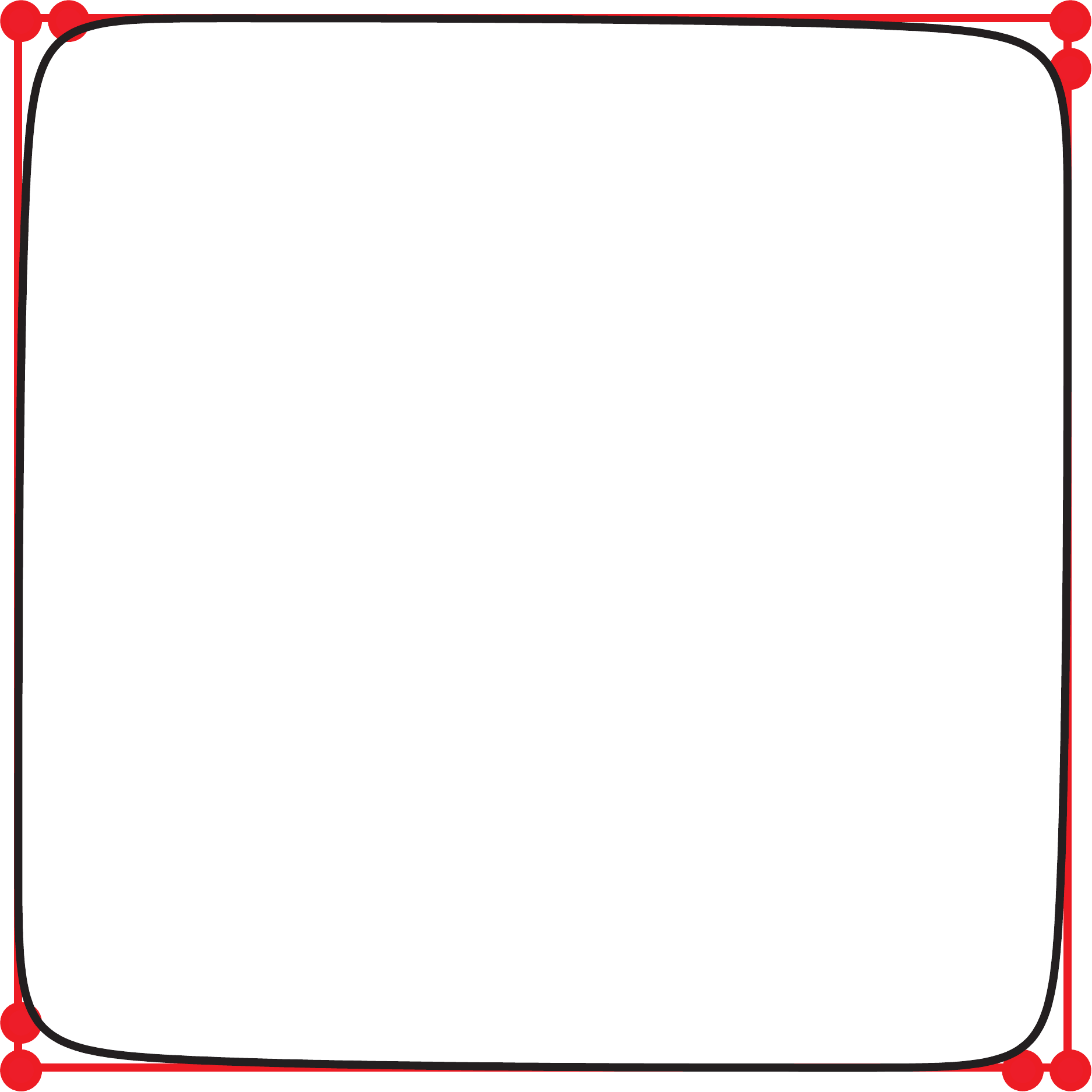, width=1.4 in} & \epsfig{file=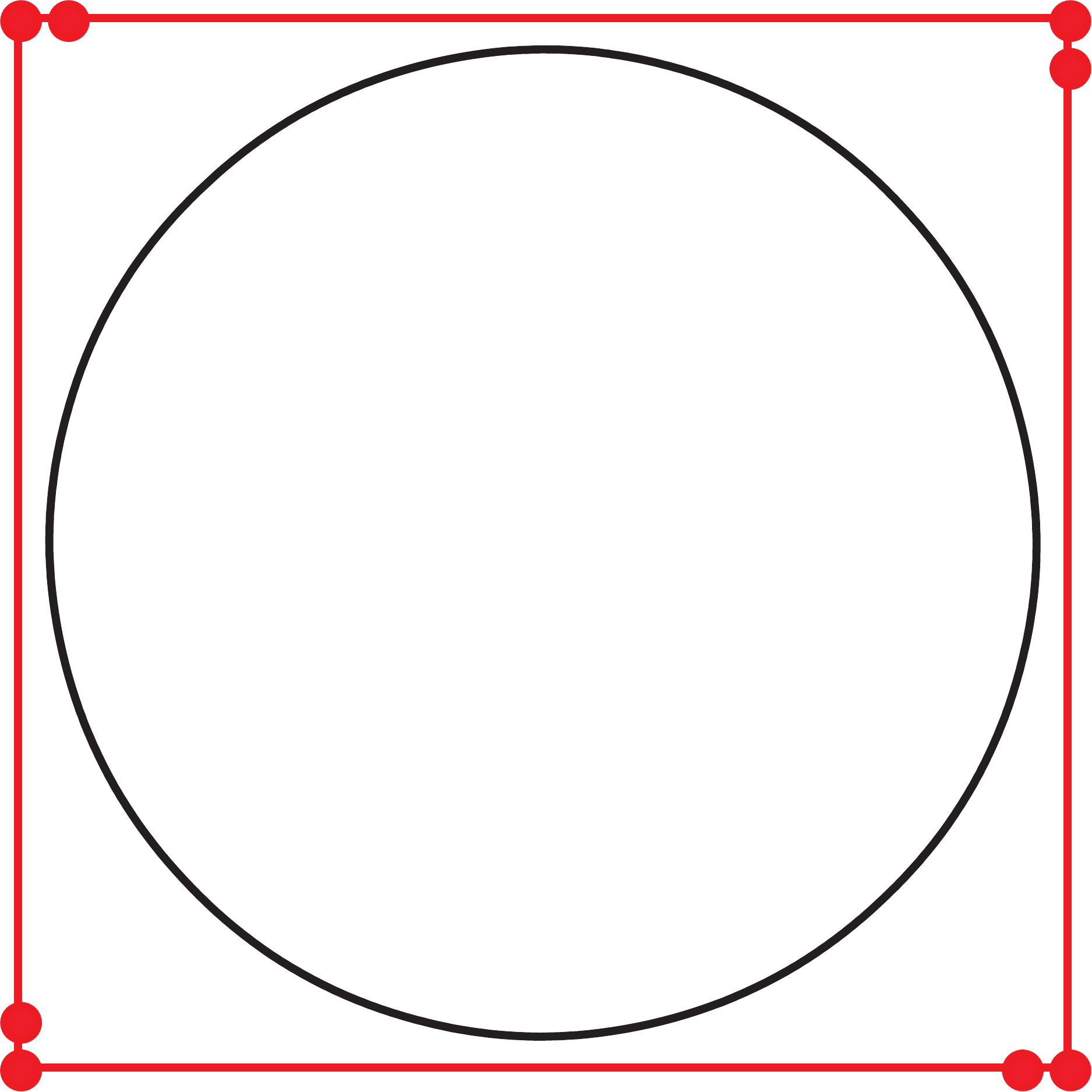, width=1.4 in}\\
\epsfig{file=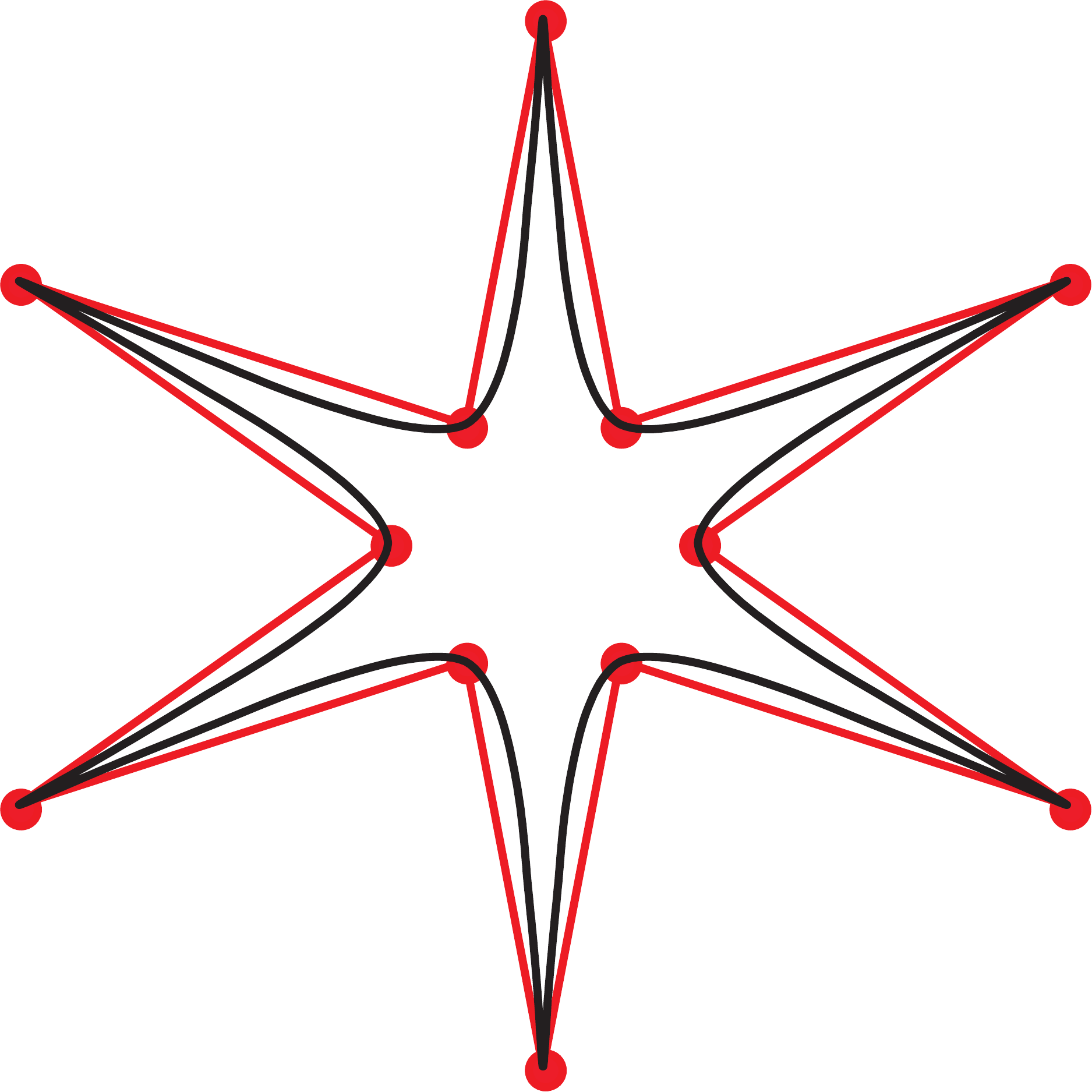, width=1.4 in} & \epsfig{file=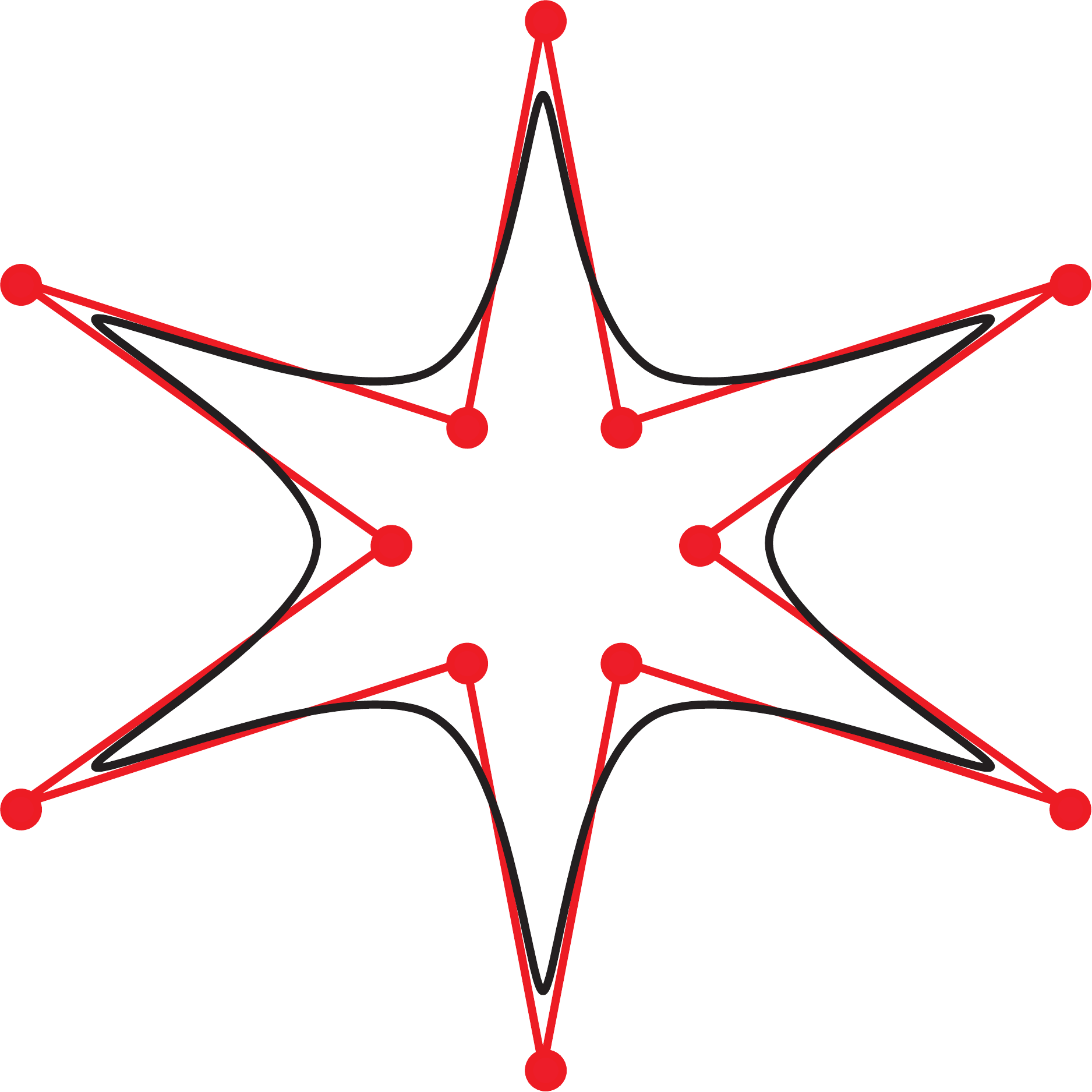, width=1.4 in}  & \epsfig{file=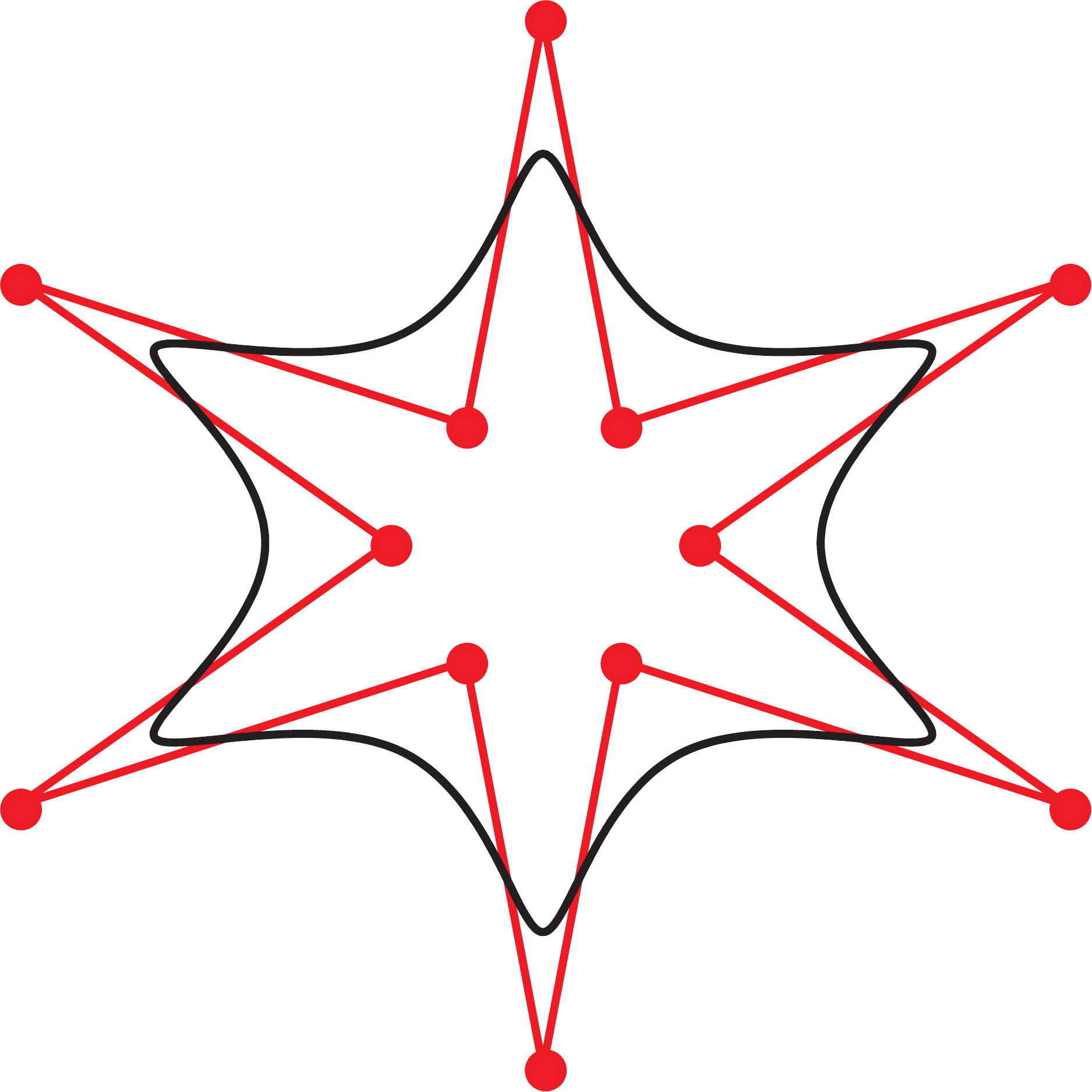, width=1.4 in} & \epsfig{file=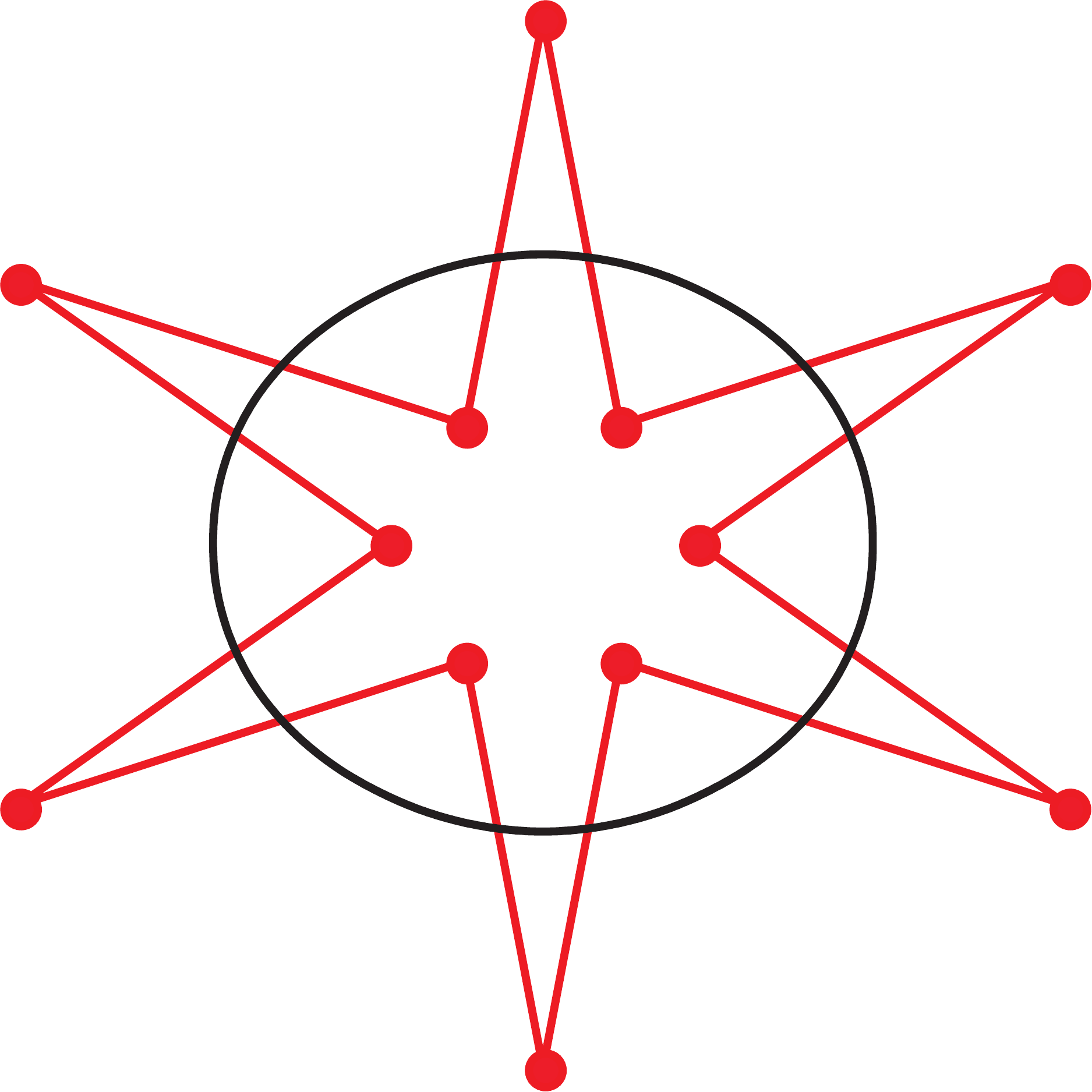, width=1.4 in}\\
(a) $\alpha=\frac{1}{2048}$ &(b) $\alpha=\frac{1}{16}$  & (c) $\alpha=\frac{1}{8}$ & (d) $\alpha=\frac{1}{4}$\\
\qquad $\beta=-\frac{255}{4096}$ & \qquad $\beta=-\frac{1}{30}$ & \qquad  $\beta=0$ & \qquad  $\beta=\frac{1}{4}$
 \end{tabular}
\end{center}
 \caption[Limit curves obtained by the scheme $S_{a_{3}}$.]{\label{p5-3-point-1}\emph{Black curves are the limit curves obtained by subdivision scheme $S_{a_{3}}$.}}
\end{figure}
\begin{figure}[!h] 
\begin{center}
\begin{tabular}{cccc}
\epsfig{file=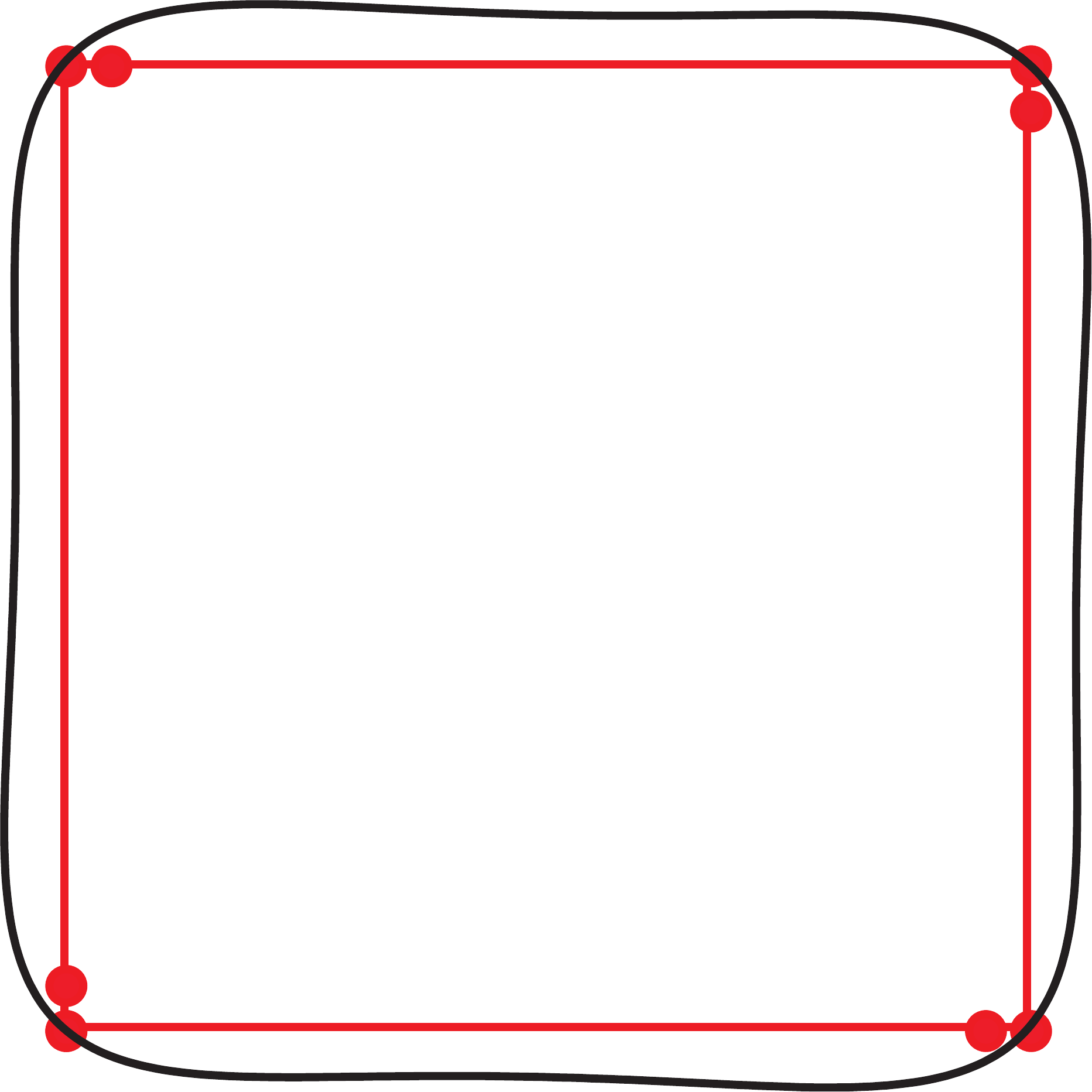, width=1.4 in} & \epsfig{file=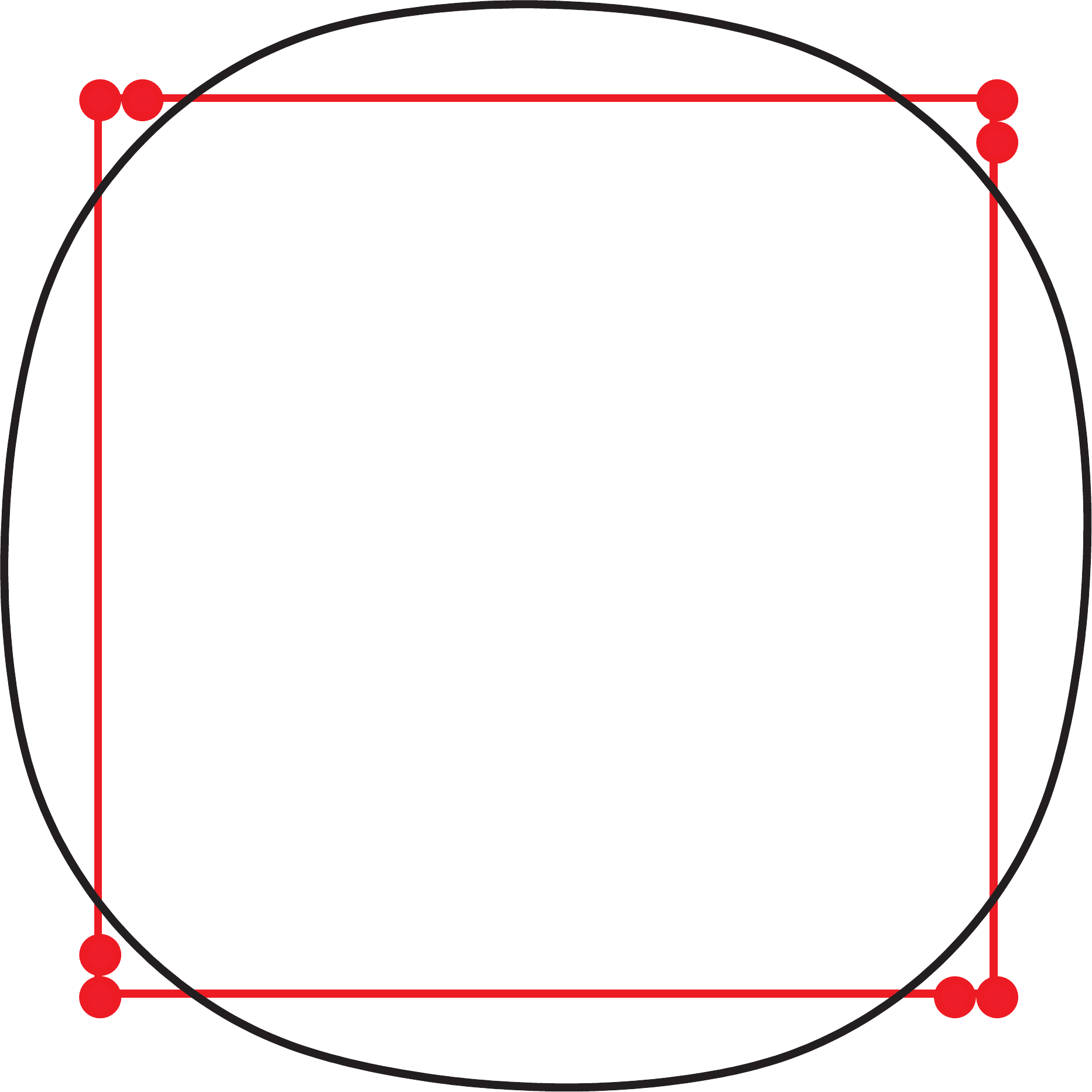, width=1.4 in}  & \epsfig{file=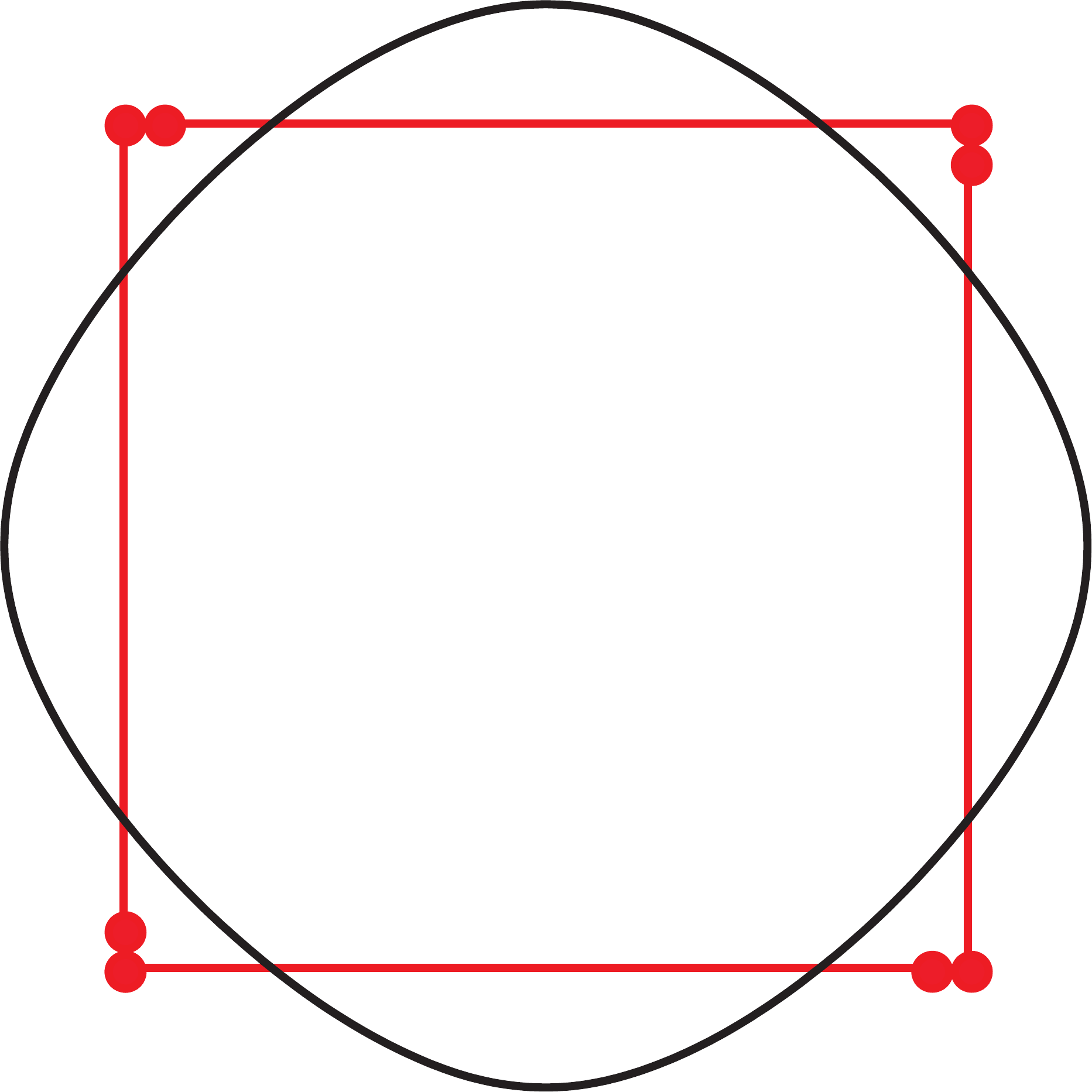, width=1.4 in} & \epsfig{file=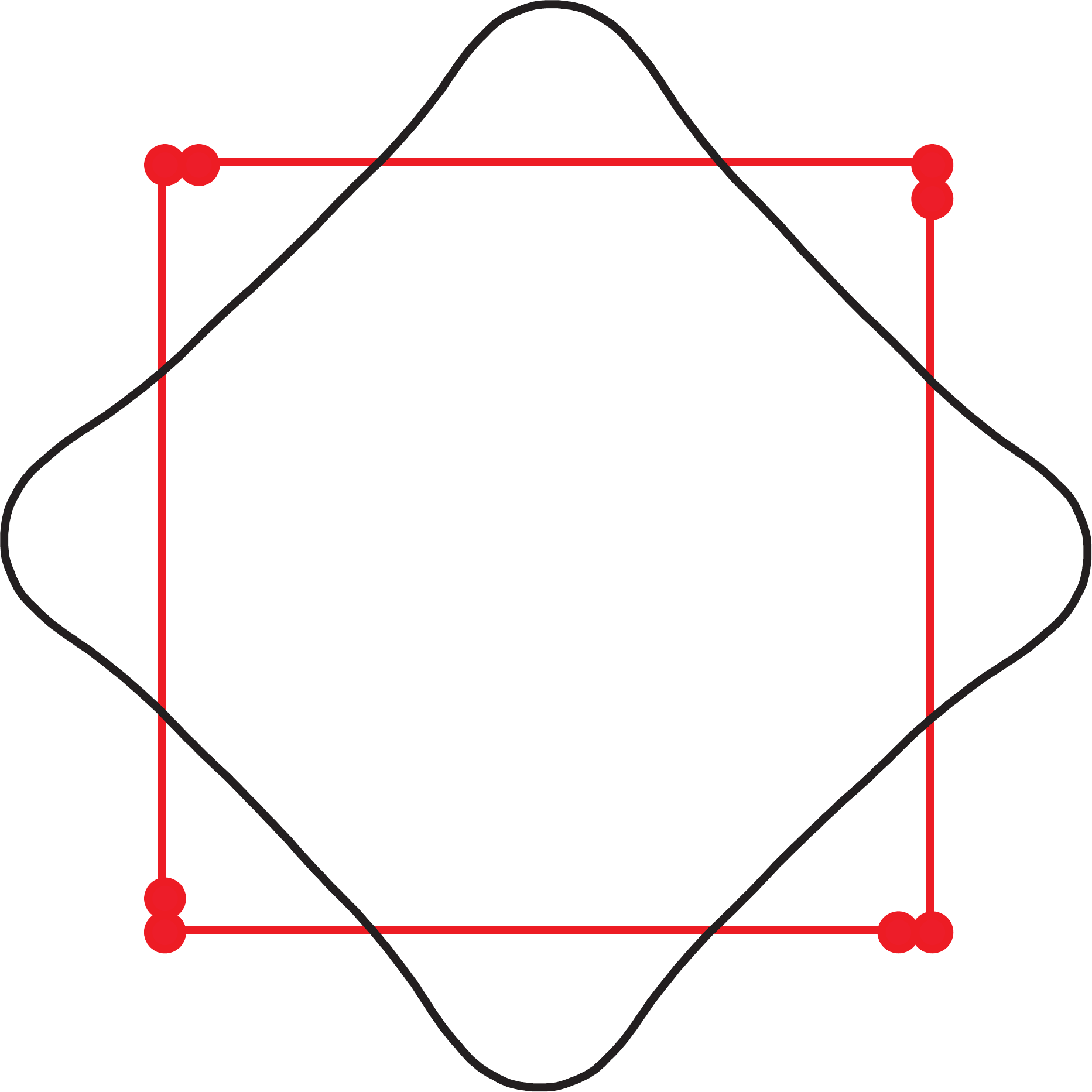, width=1.4 in}\\
\epsfig{file=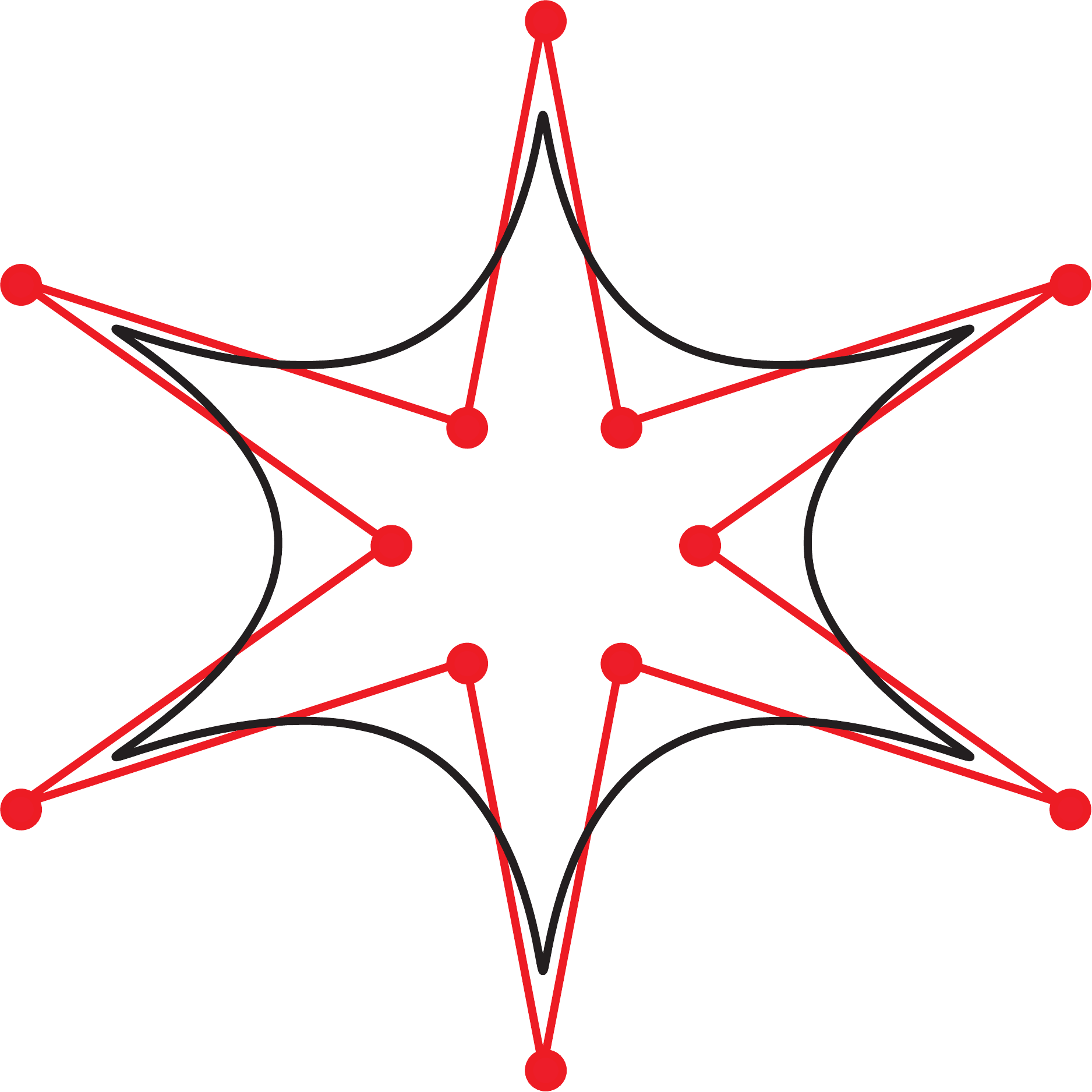, width=1.4 in} & \epsfig{file=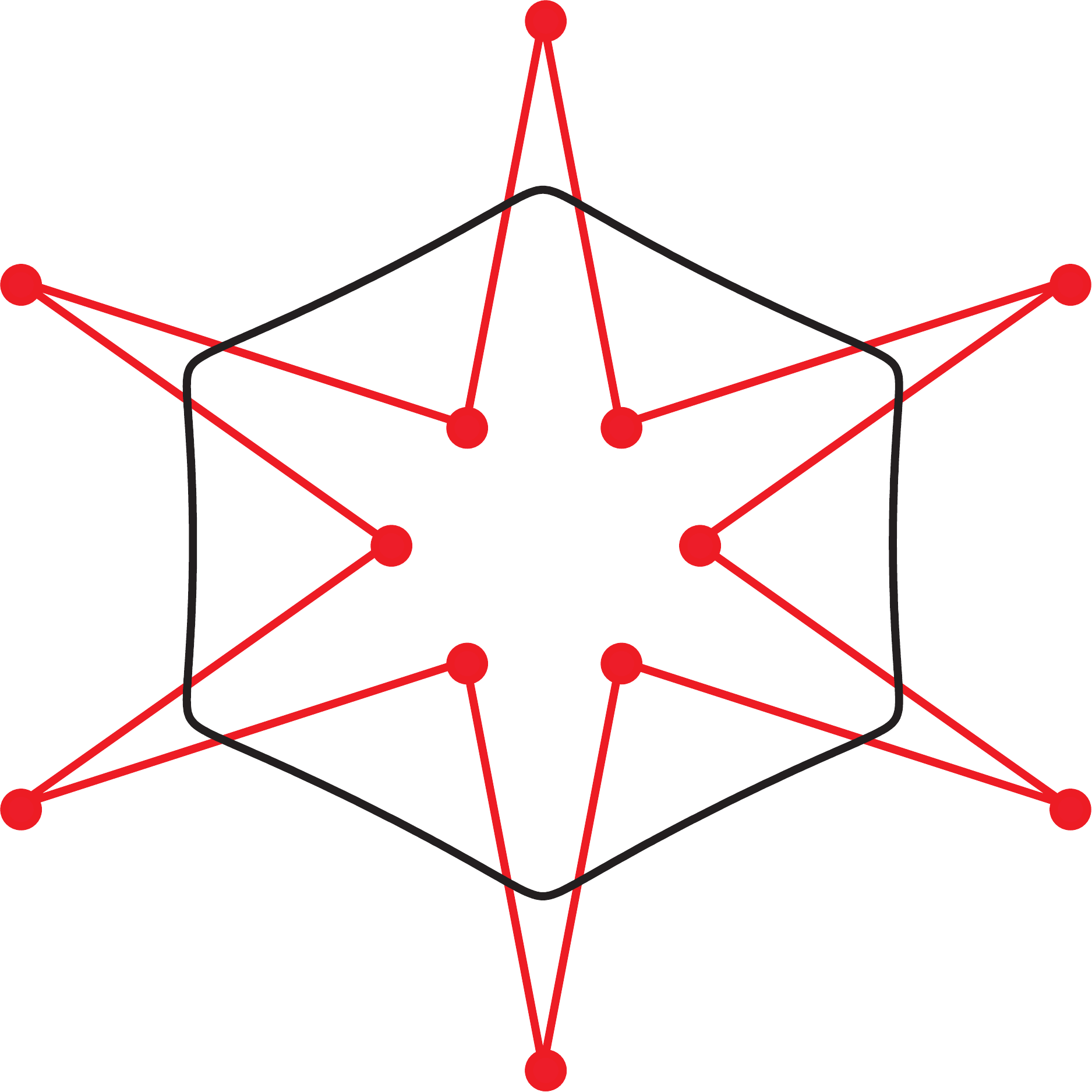, width=1.4 in}  & \epsfig{file=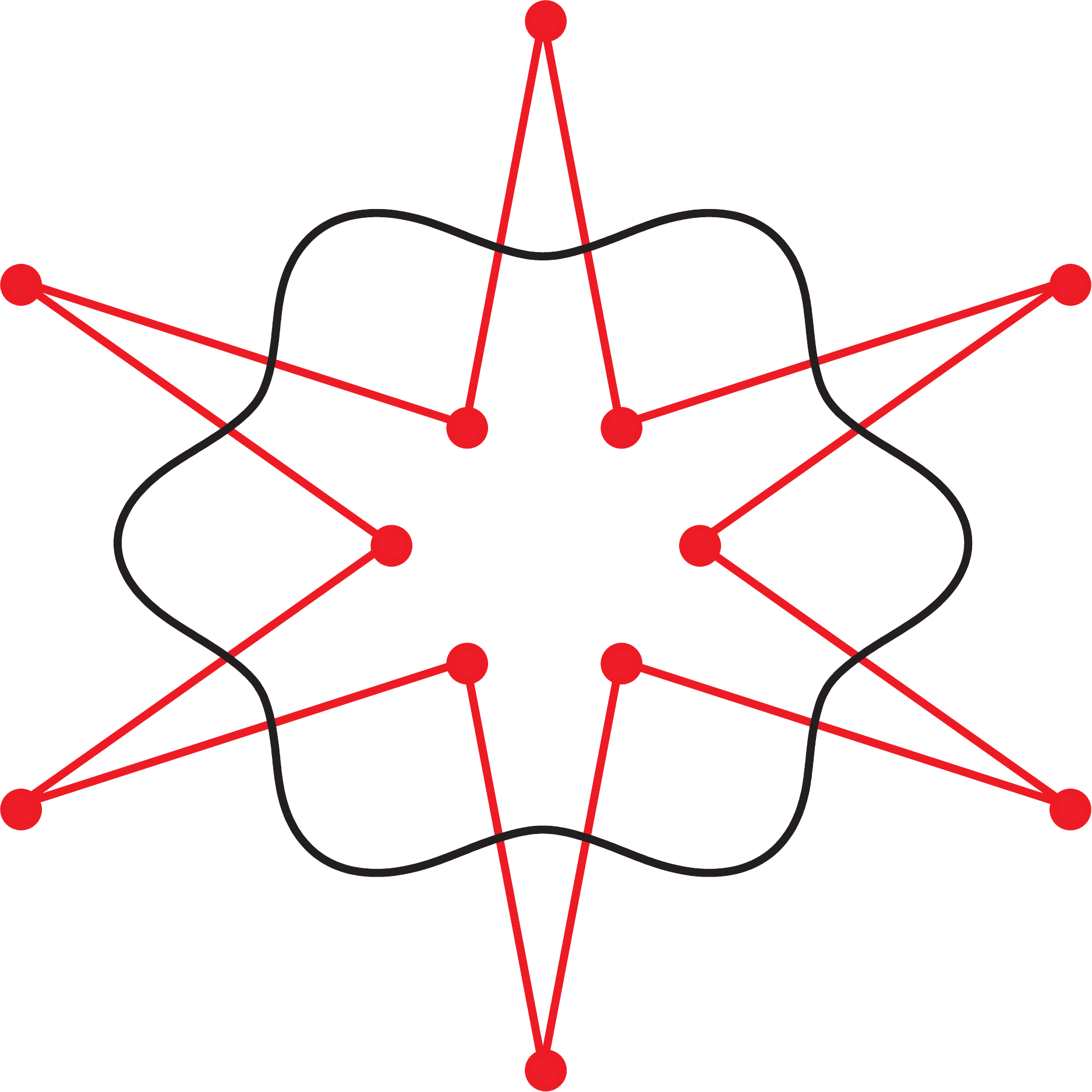, width=1.4 in} & \epsfig{file=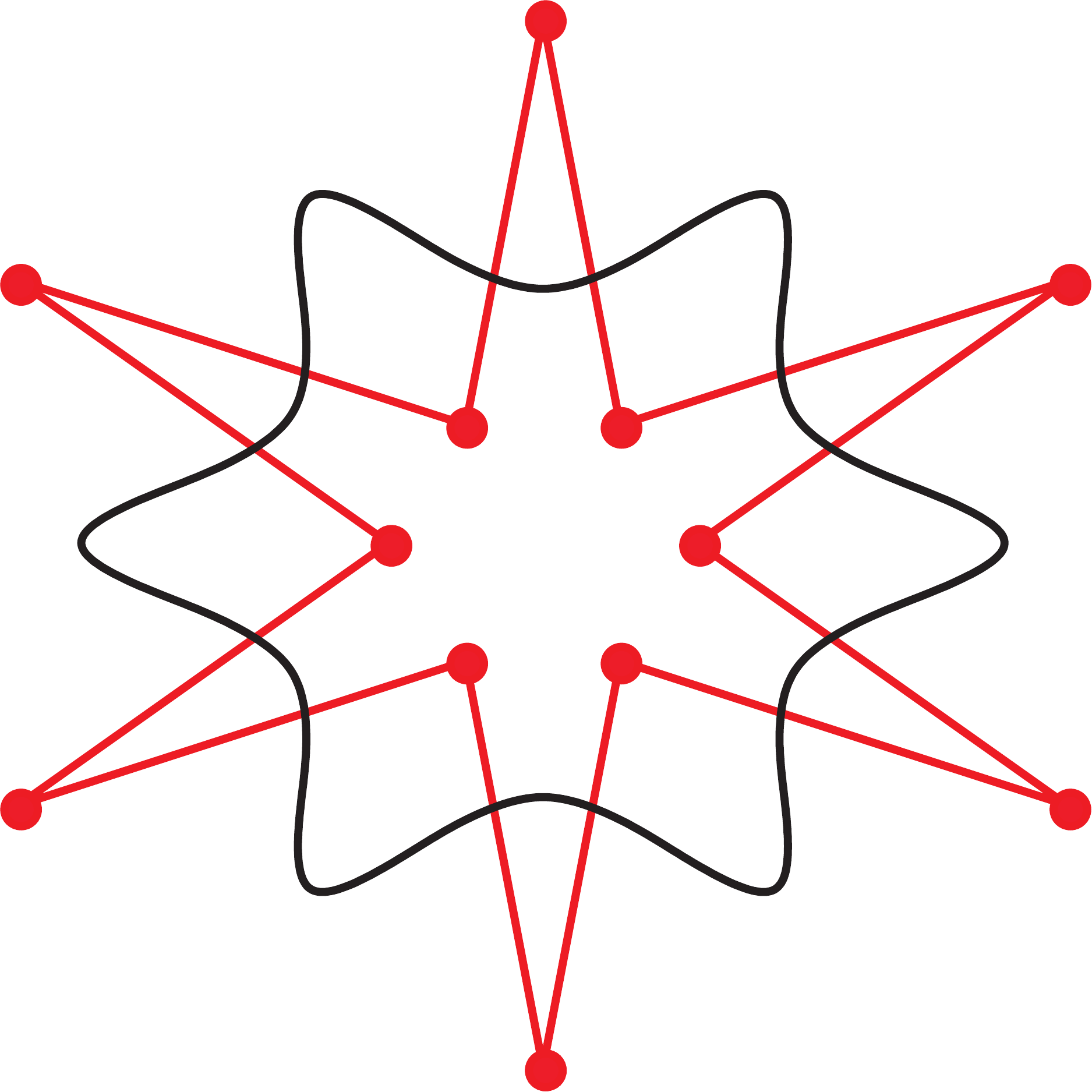, width=1.4 in}\\
(a) $\alpha=\frac{1}{32}$ &(b) $\alpha=\frac{1}{16}$  & (c) $\alpha=\frac{1}{10}$ & (d) $\alpha=\frac{1}{8}$\\
\qquad     $\beta=-\frac{3}{62}$ & \qquad \quad $\beta=-\frac{1}{48}$ & \qquad \quad $\beta=-\frac{49}{1152}$ & \qquad \quad $\beta=-\frac{11}{128}$
 \end{tabular}
\end{center}
 \caption[Limit curves obtained by the scheme $S_{a_{5}}$.]{\label{p5-5-point-1}\emph{Black curves are the limit curves obtained by subdivision scheme $S_{a_{5}}$.}}
\end{figure}

\begin{figure}[!h] 
\begin{center}
\begin{tabular}{cccc}
\epsfig{file=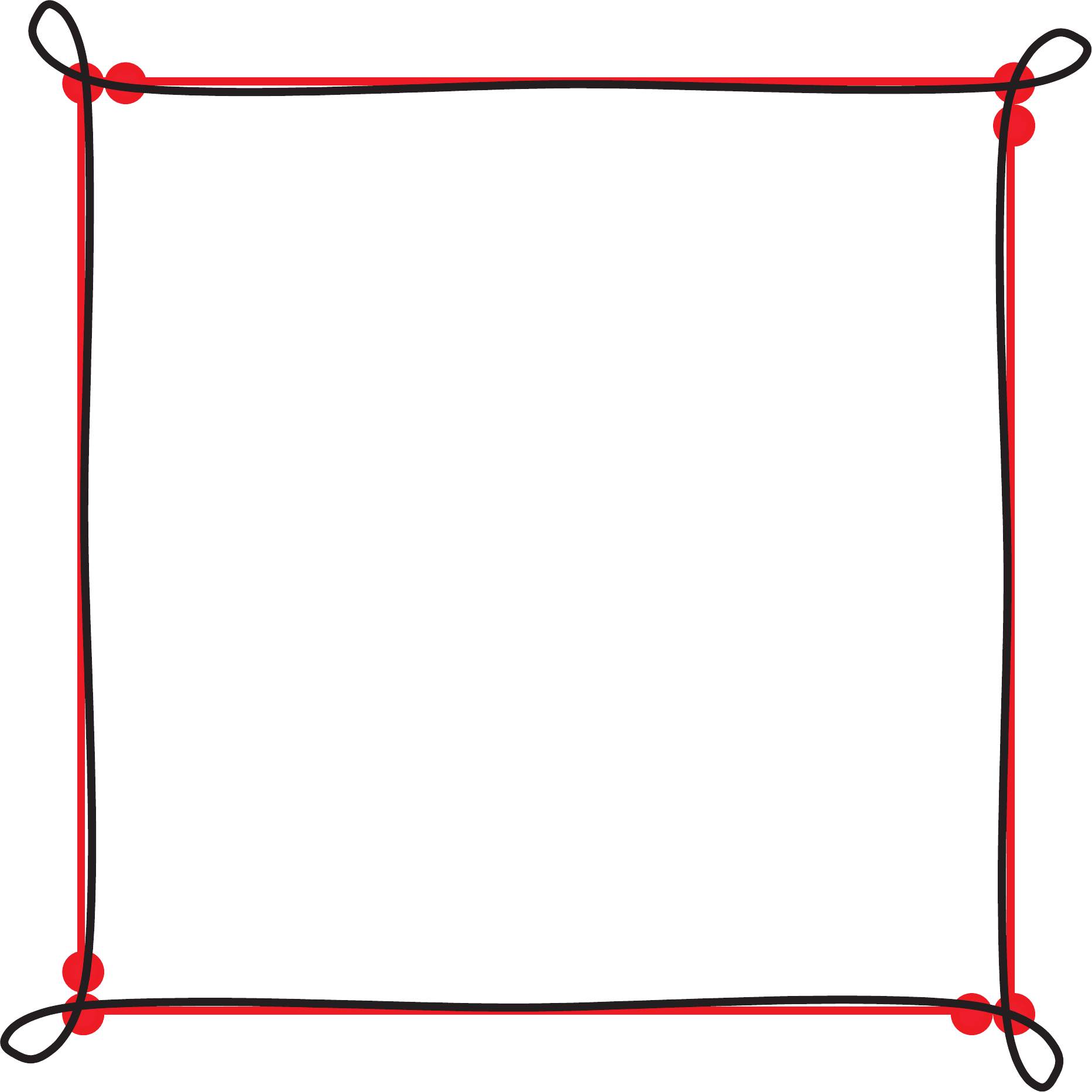, width=1.4 in} & \epsfig{file=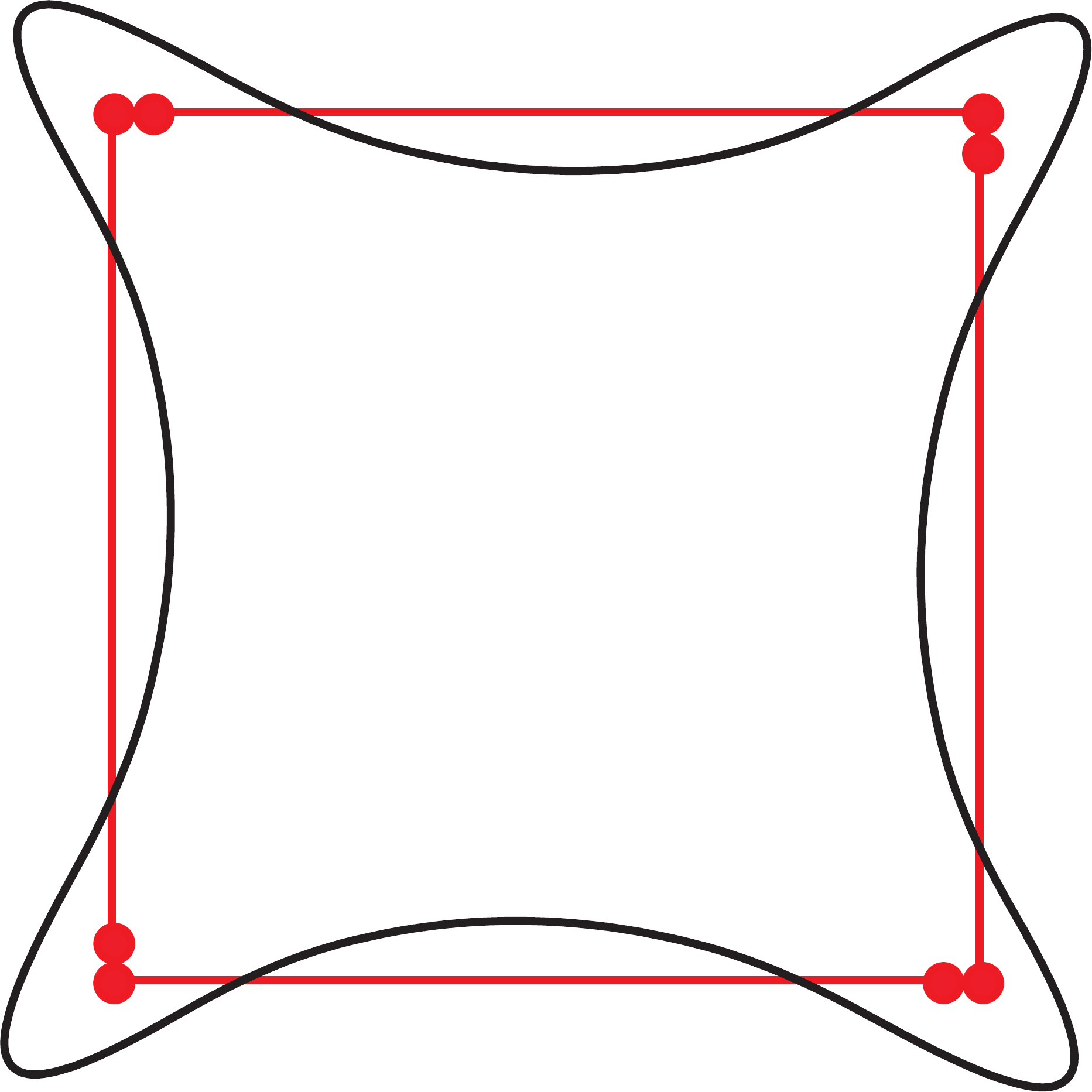, width=1.4 in}  & \epsfig{file=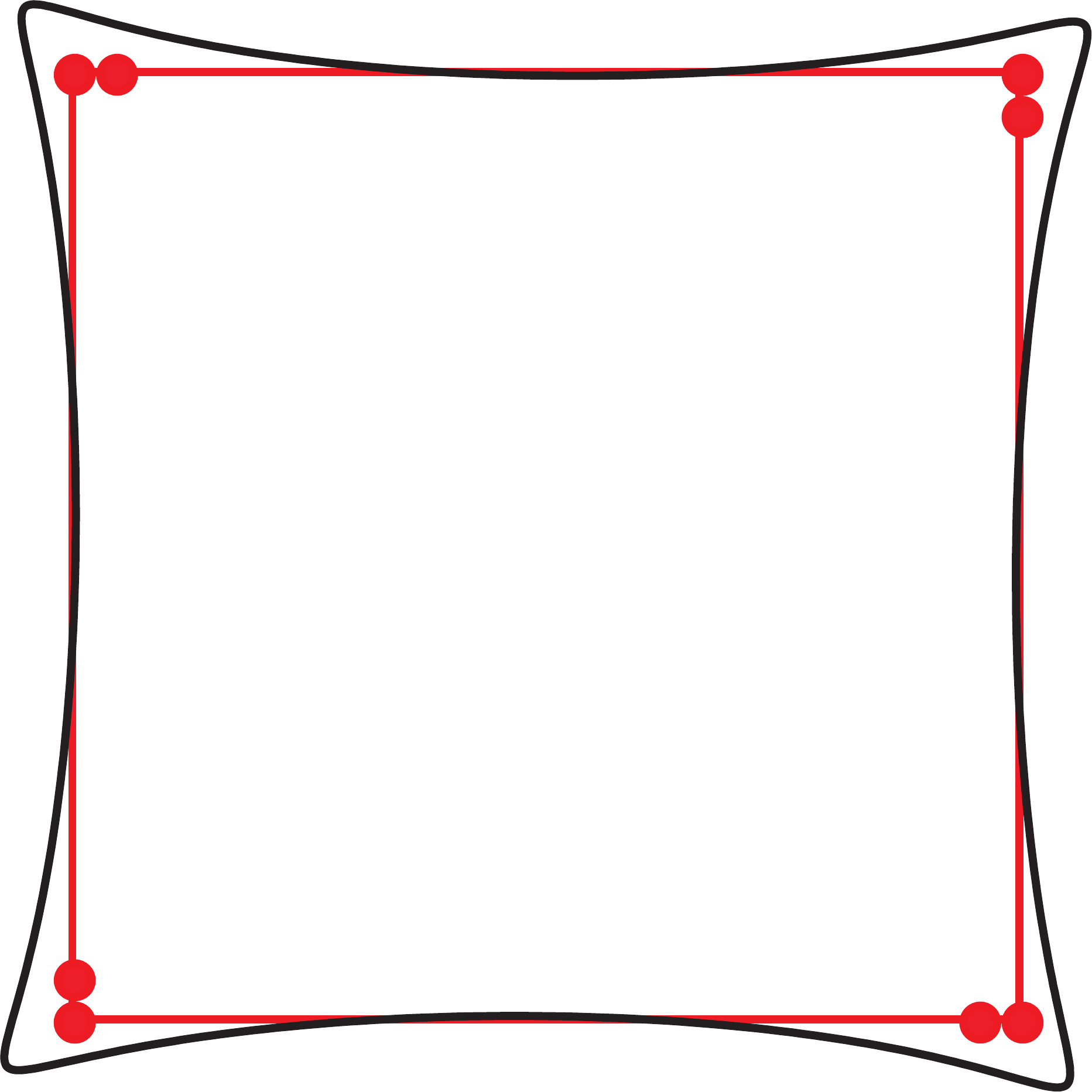, width=1.4 in} & \epsfig{file=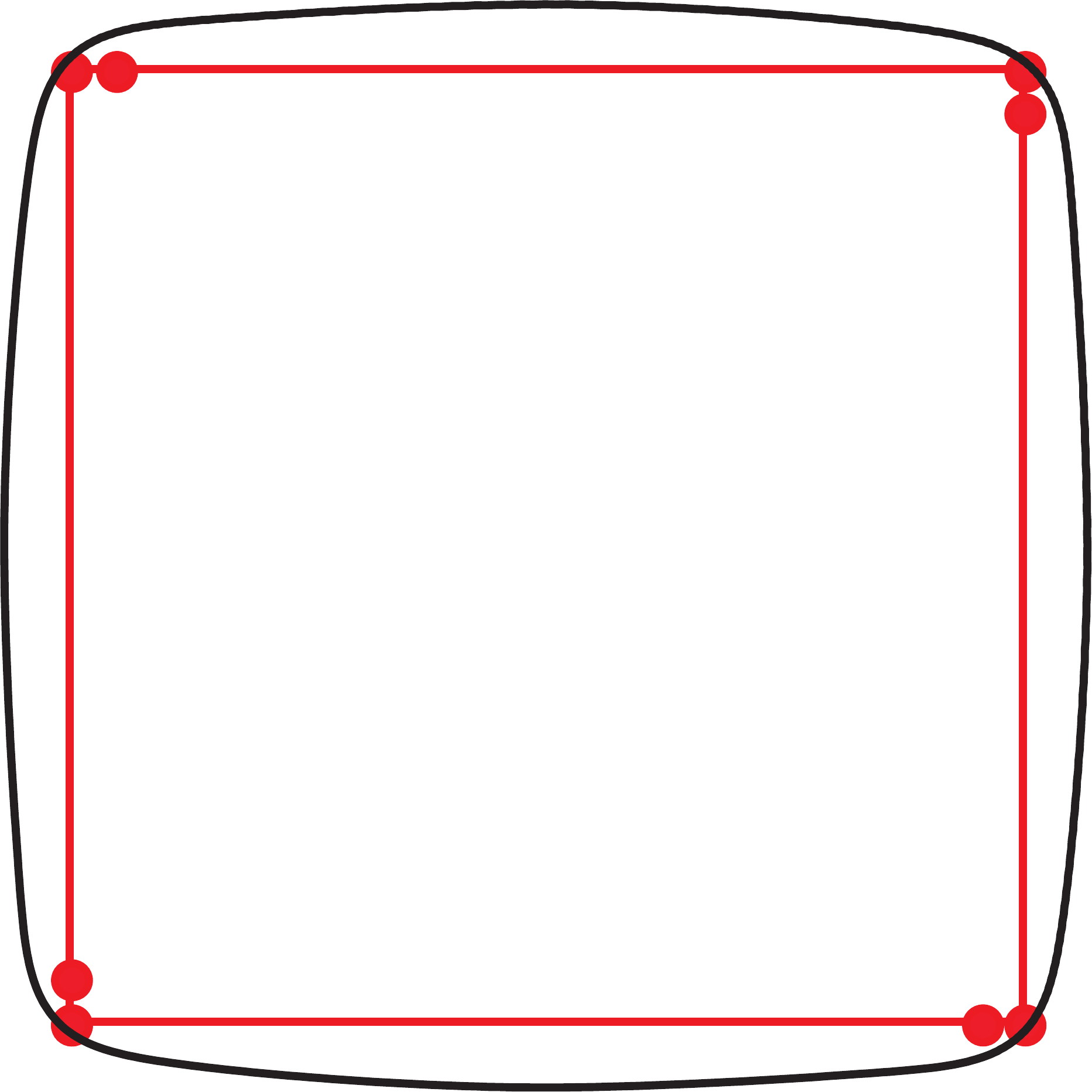, width=1.4 in}\\
\epsfig{file=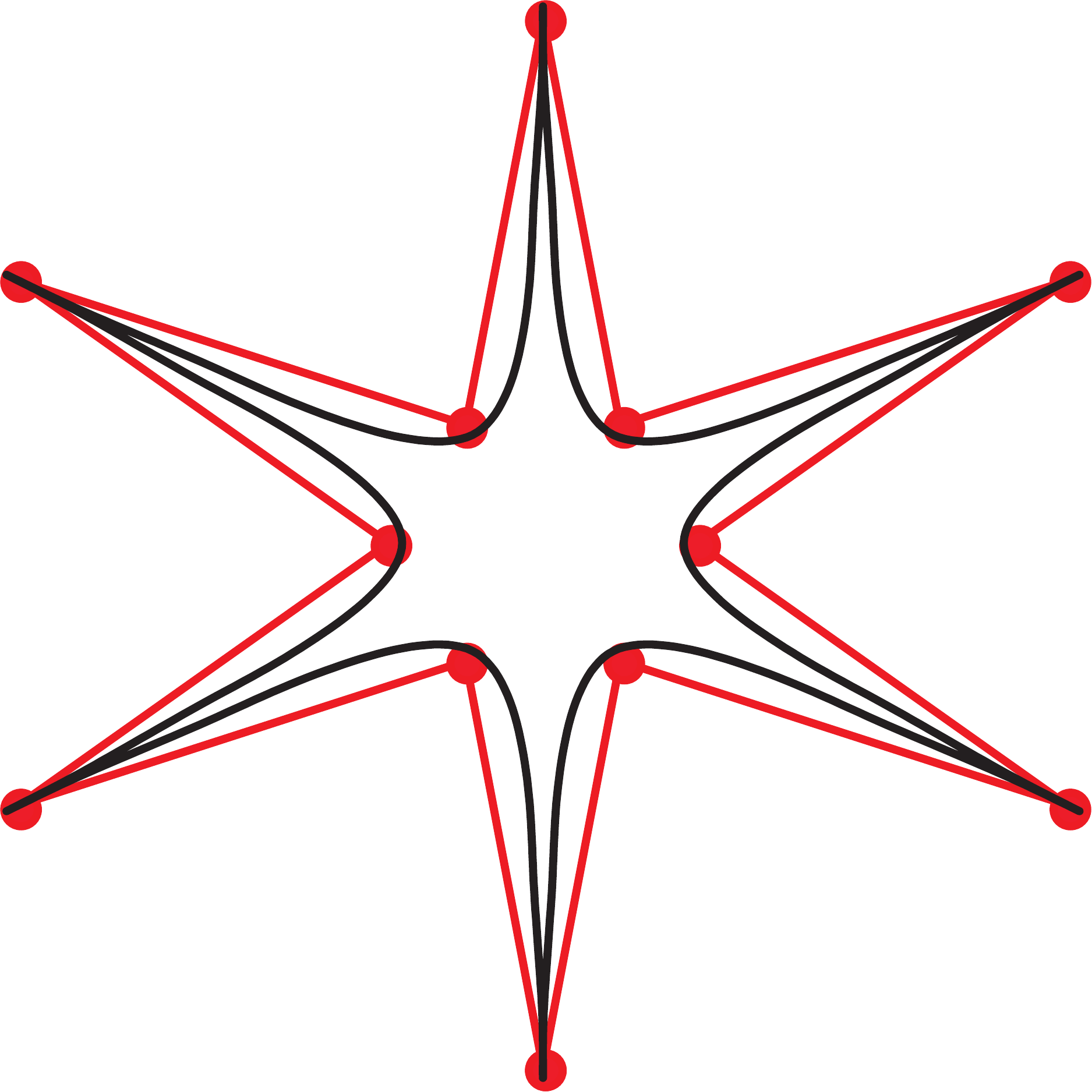, width=1.4 in} & \epsfig{file=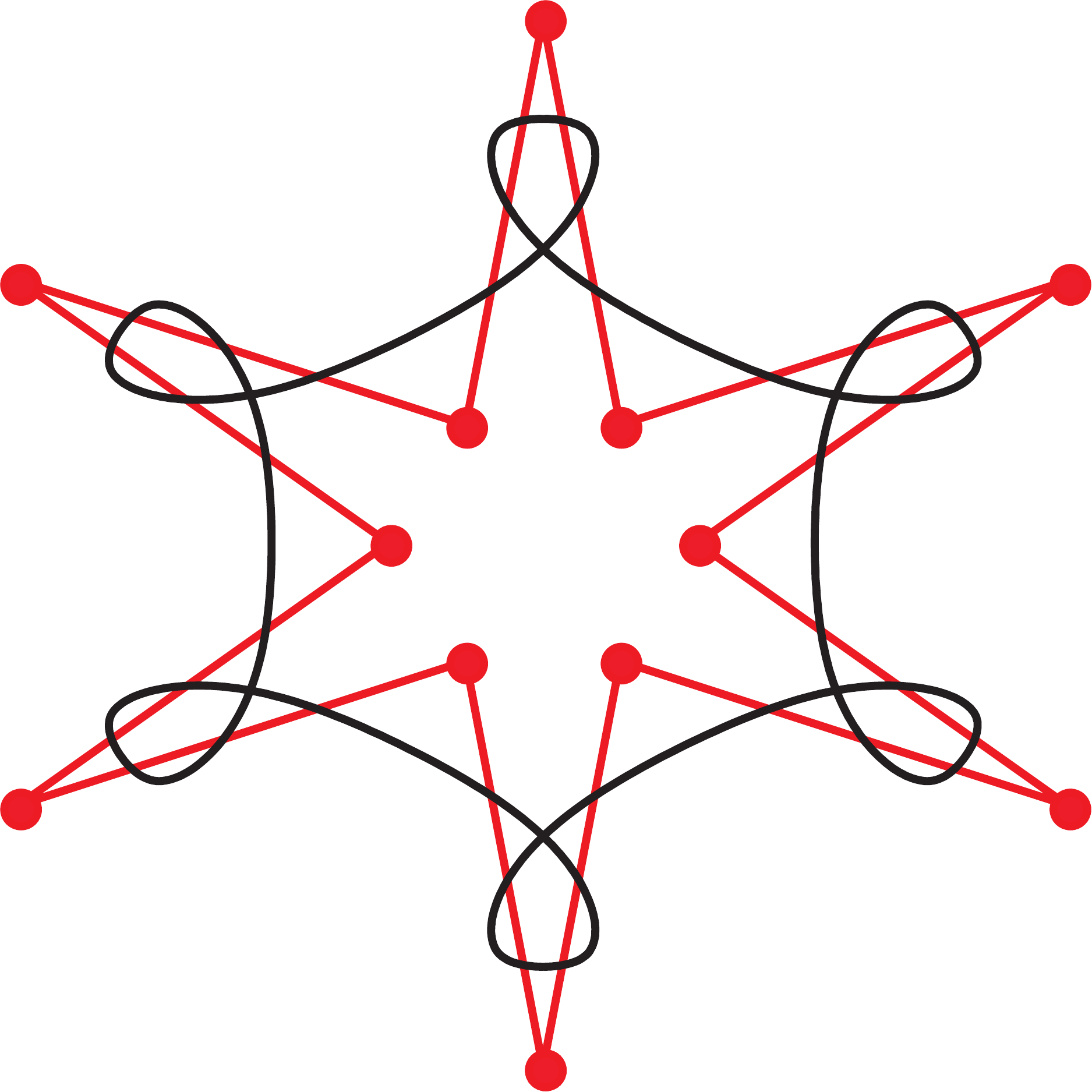, width=1.4 in}  & \epsfig{file=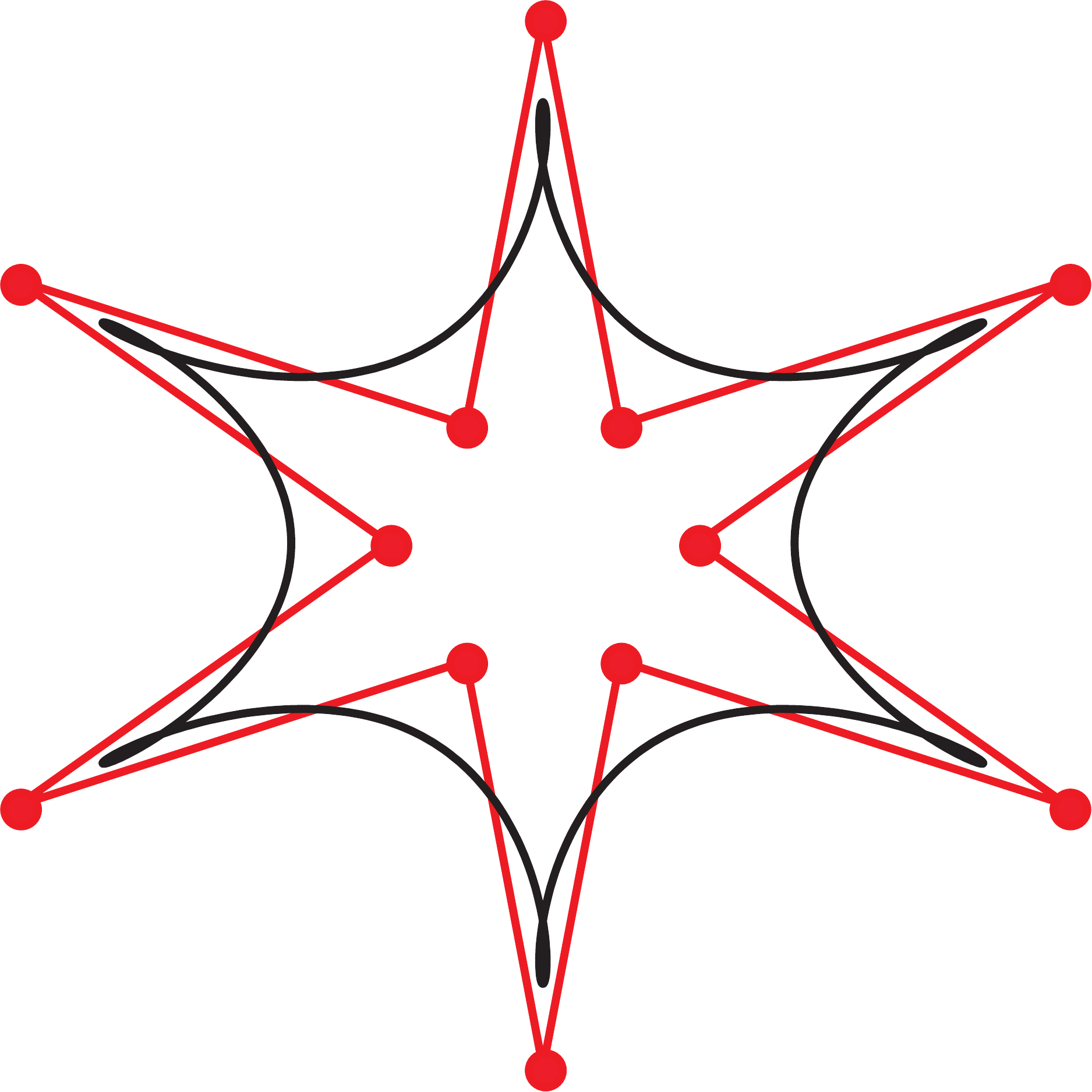, width=1.4 in} & \epsfig{file=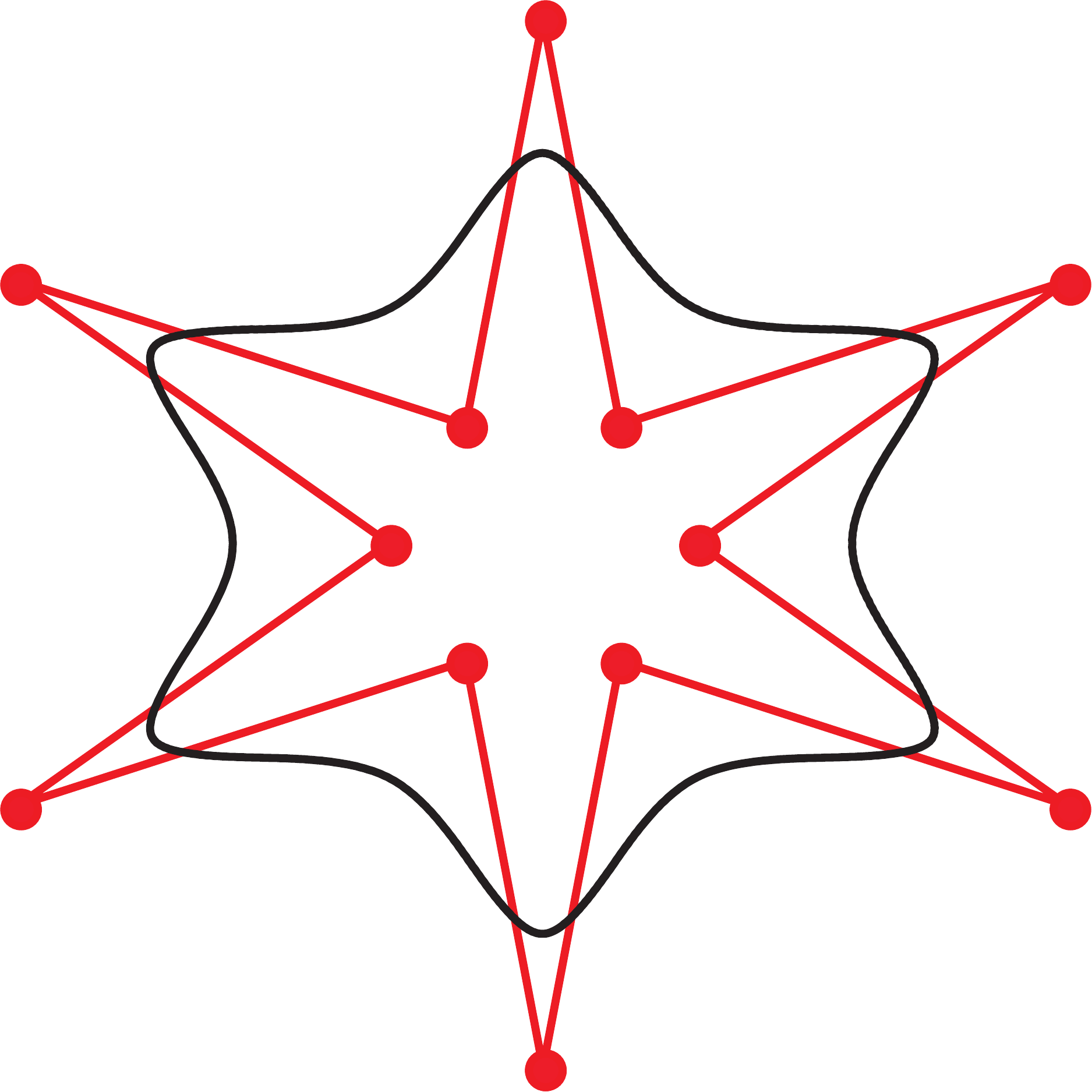, width=1.4 in}\\
(a) $\alpha=-\frac{1}{2048}$ &(b) $\alpha=\frac{1}{100}$  & (c) $\alpha=\frac{1}{128}$ & (d) $\alpha=\frac{1}{80}$\\
\,\ $\beta=\frac{1}{512}$ & \qquad  $\beta=-\frac{1}{128}$ & \quad \,\ $\beta=\frac{1}{512}$ & \qquad  $\beta=\frac{1}{80}$
 \end{tabular}
\end{center}
 \caption[Limit curves obtained by the scheme $S_{a_{7}}$.]{\label{p5-7-point-1}\emph{Black curves are the limit curves obtained by subdivision scheme $S_{a_{7}}$.}}
\end{figure}
\begin{figure}[!h] 
\begin{center}
\begin{tabular}{cccc}
\epsfig{file=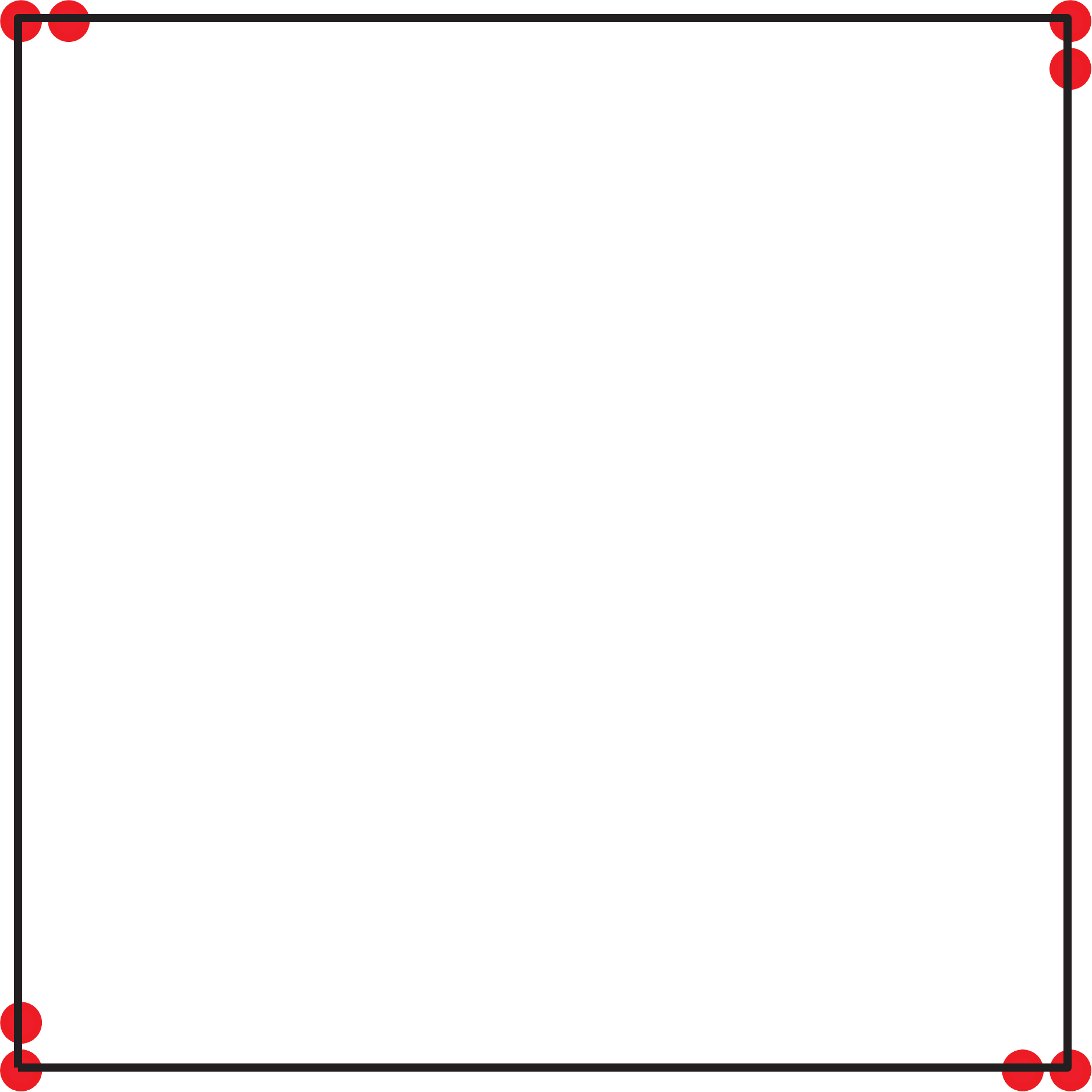, width=1.3 in} & \epsfig{file=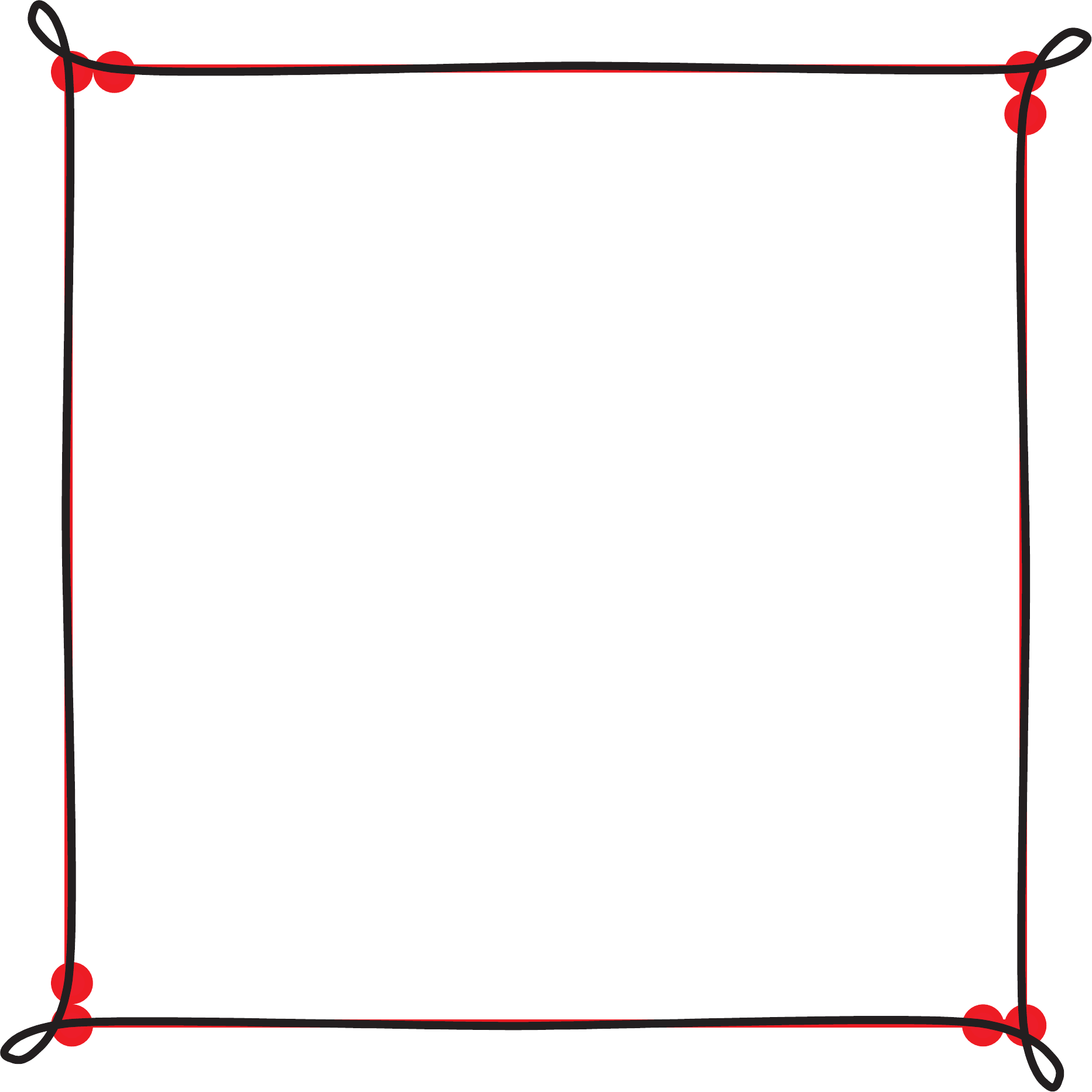, width=1.3 in}  & \epsfig{file=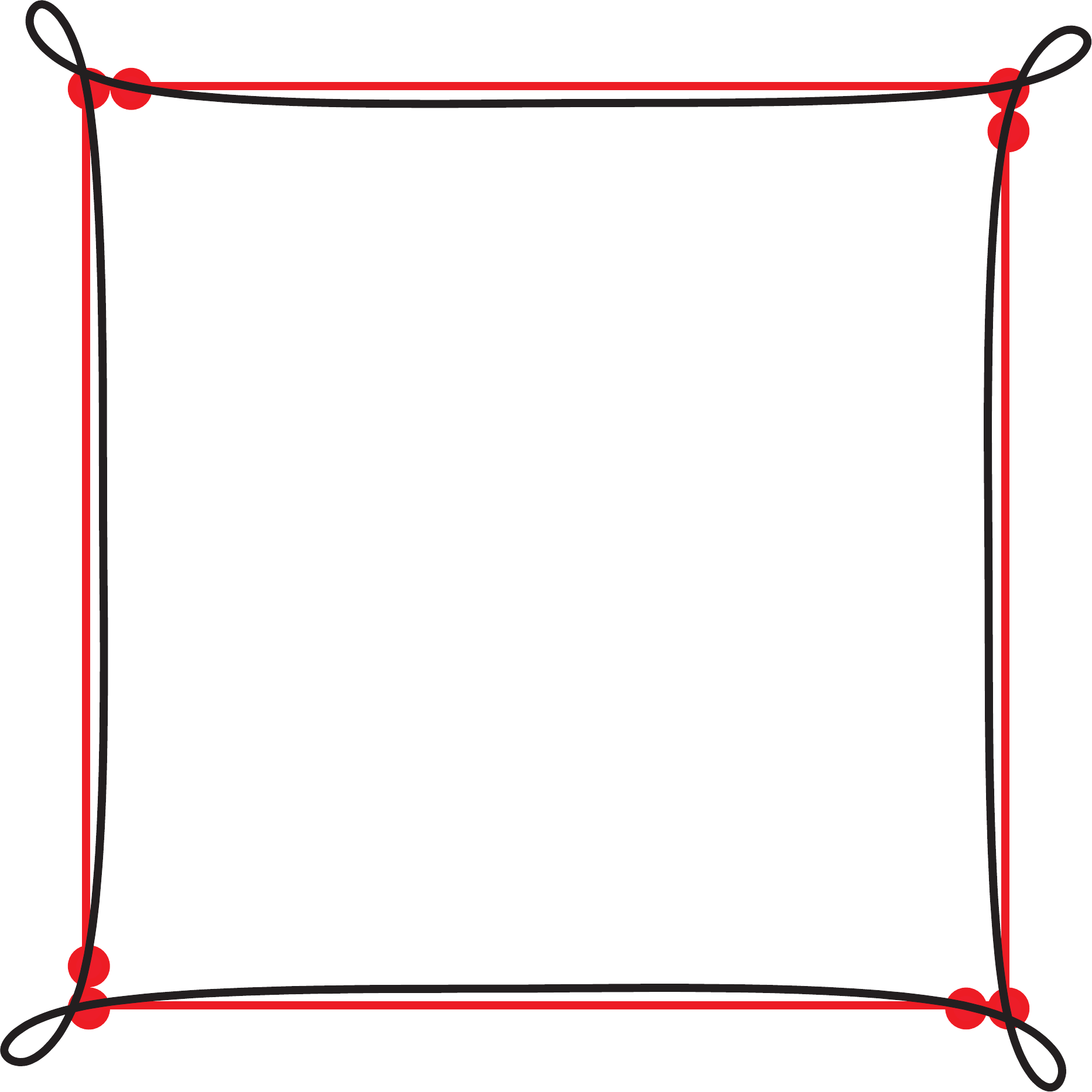, width=1.3 in} & \epsfig{file=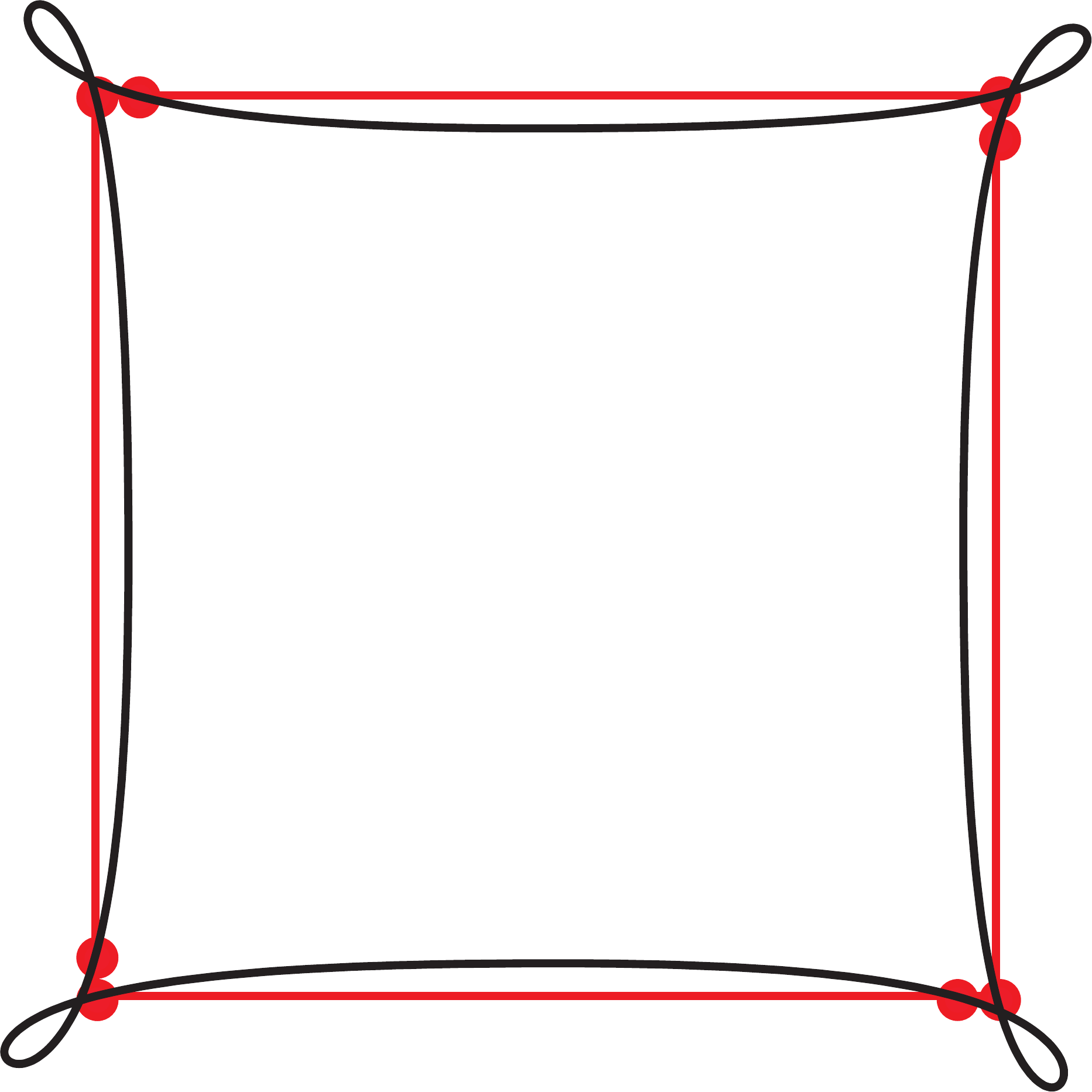, width=1.3 in}  \\
\epsfig{file=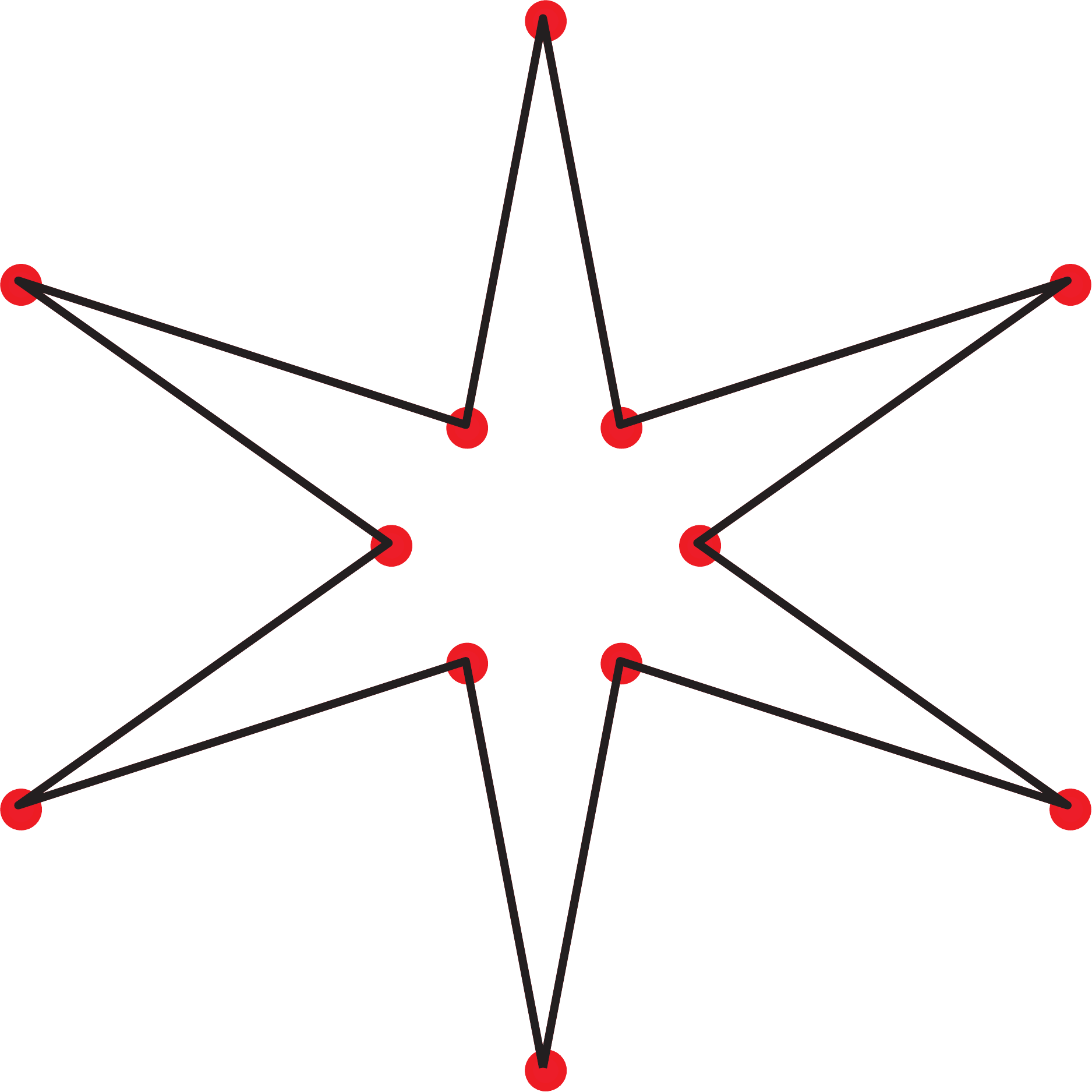, width=1.3 in} & \epsfig{file=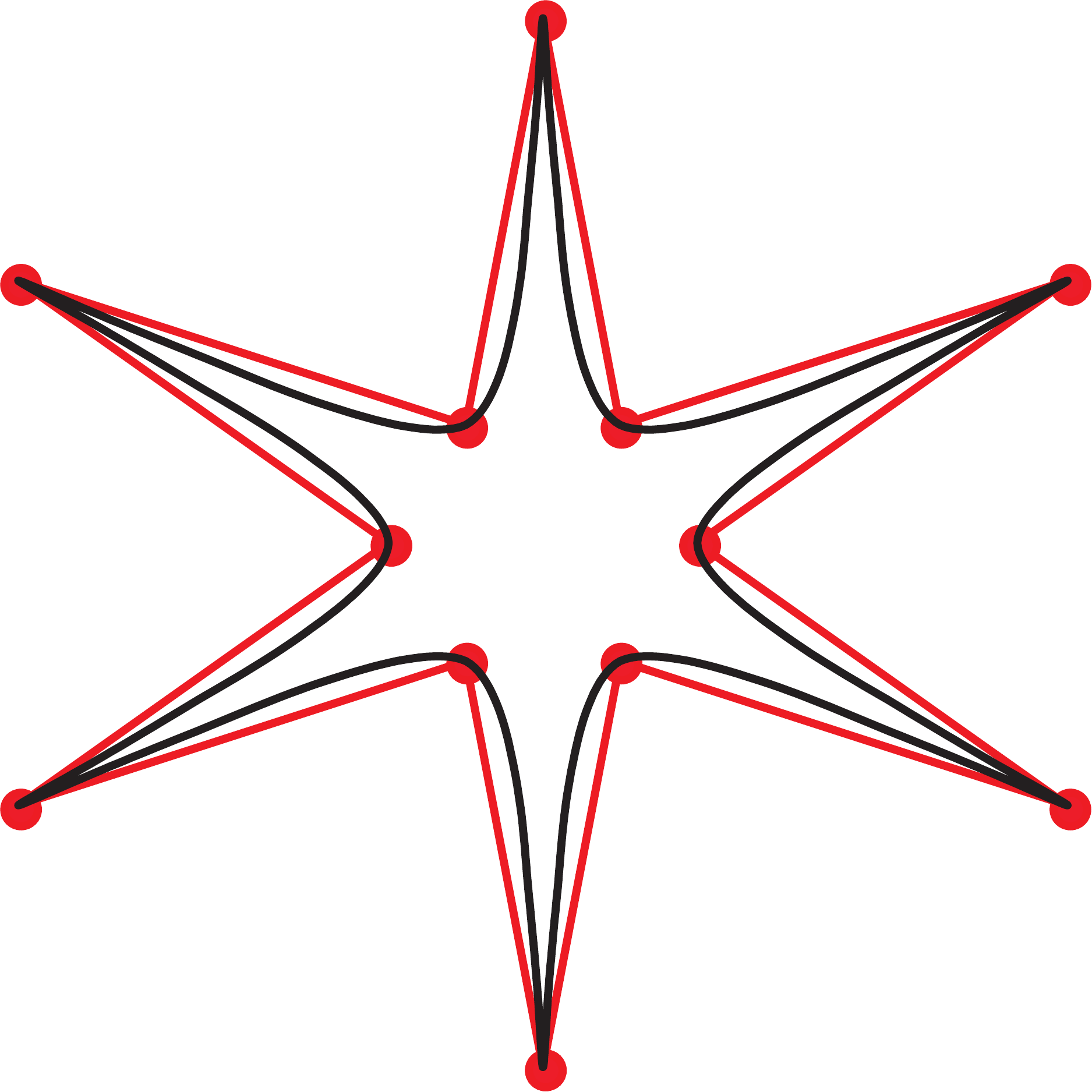, width=1.3 in}  & \epsfig{file=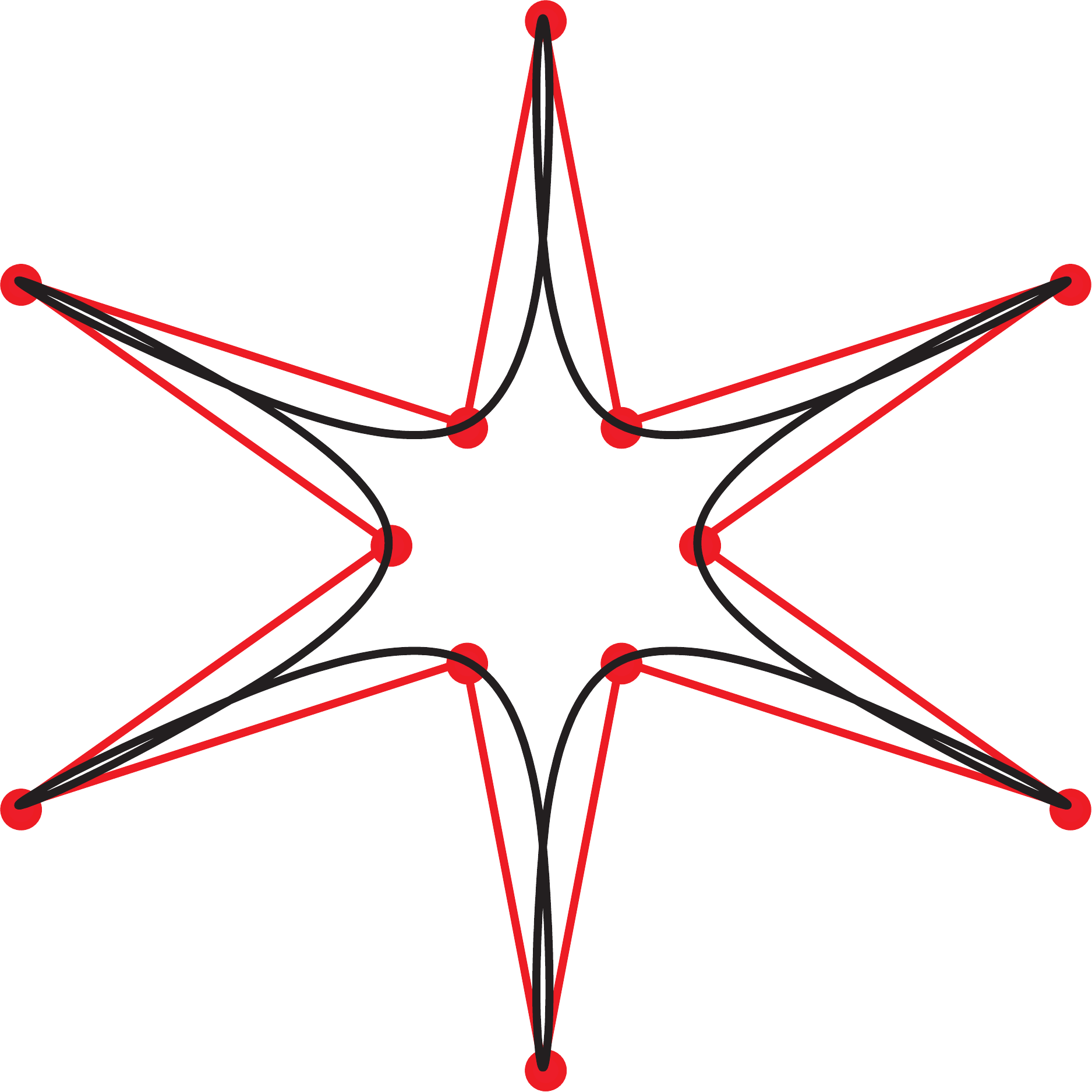, width=1.3 in} & \epsfig{file=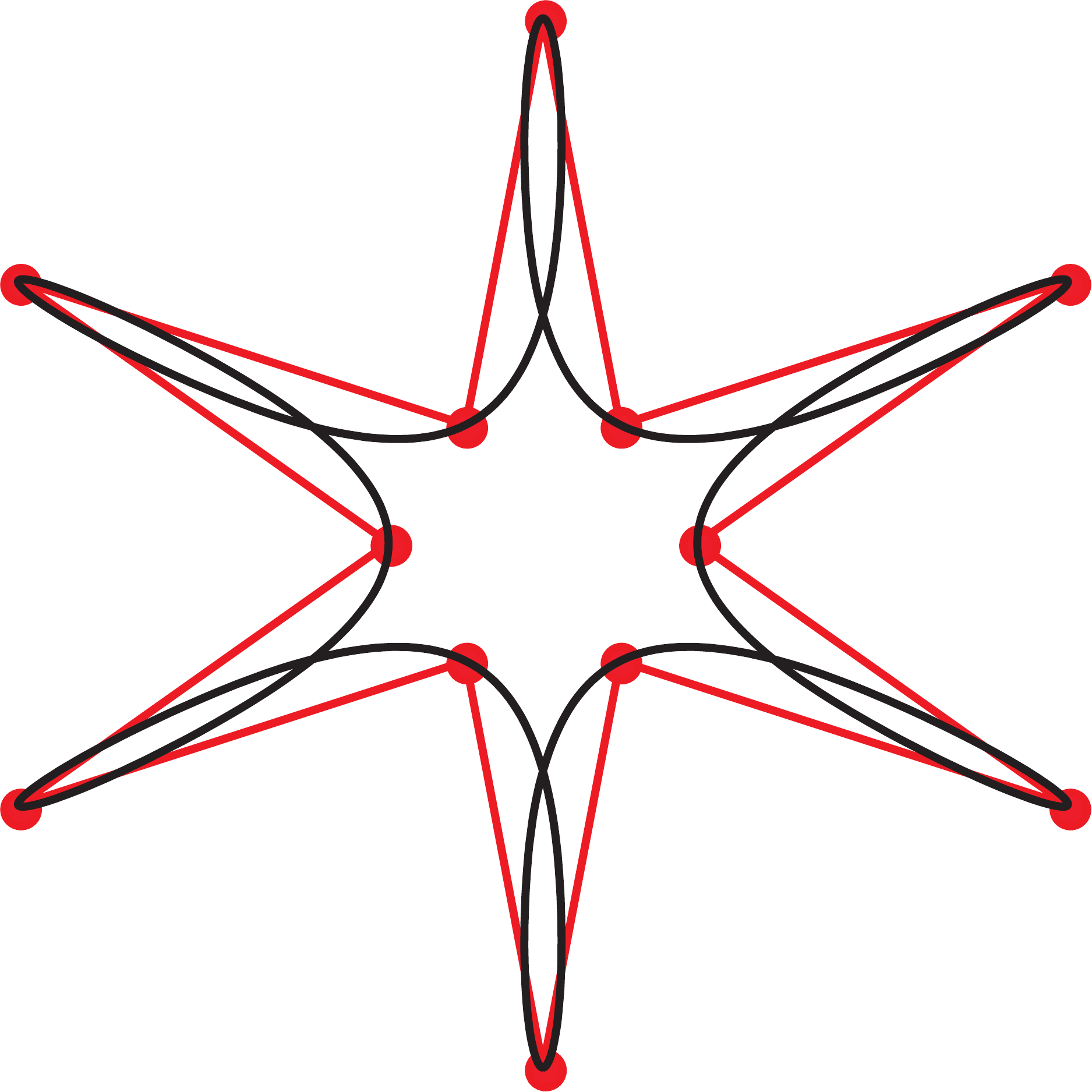, width=1.3 in}  \\
(a) $2$-point scheme & (b) $4$-point scheme & (c) $6$-point scheme & (c) $8$-point scheme
 \end{tabular}
\end{center}
 \caption[Limit curves produced by primal schemes of Deslauriers and Dubuc \cite{Dubuc}.]{\label{p5-DD-u}\emph{Solid black lines represent limit curves produced by primal schemes of Deslauriers and Dubuc \cite{Dubuc}.}}
\end{figure}

\begin{figure}[!h] 
\begin{center}
\begin{tabular}{cccc}
\epsfig{file=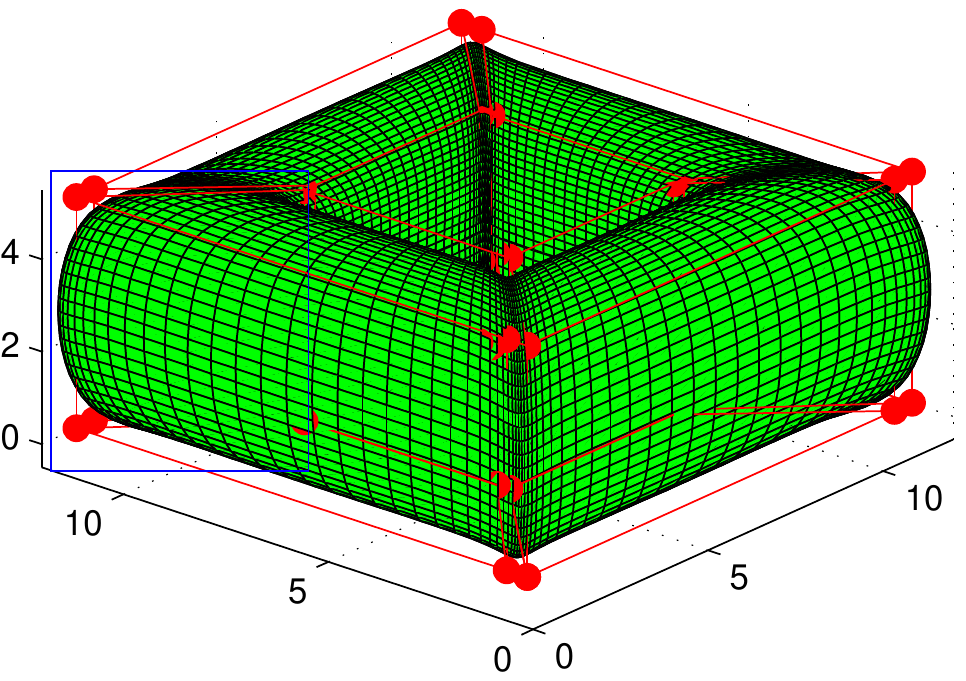, width=1.9 in} & \epsfig{file=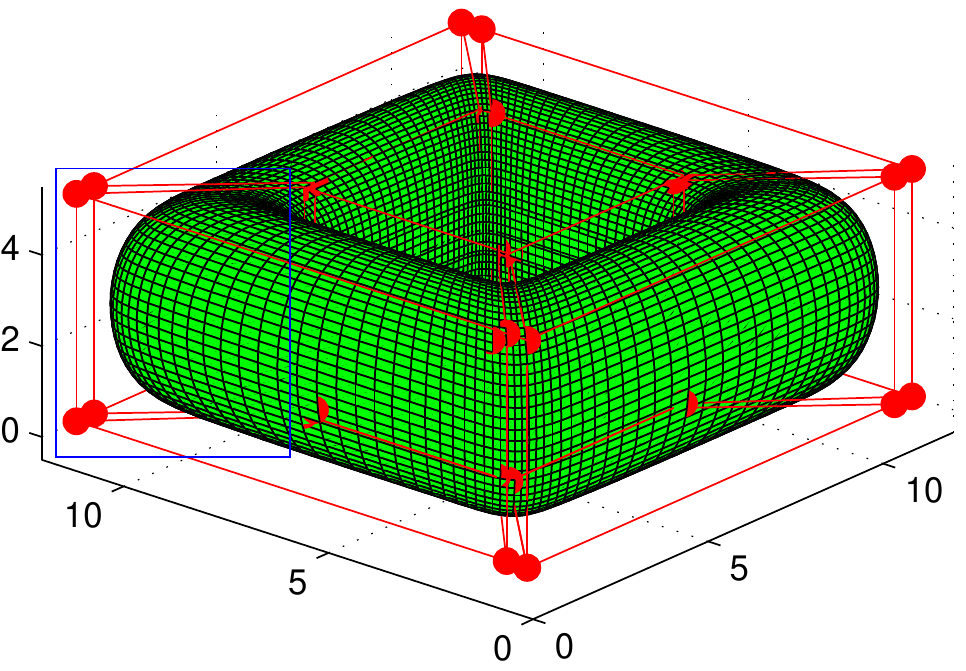, width=1.9 in}  & \epsfig{file=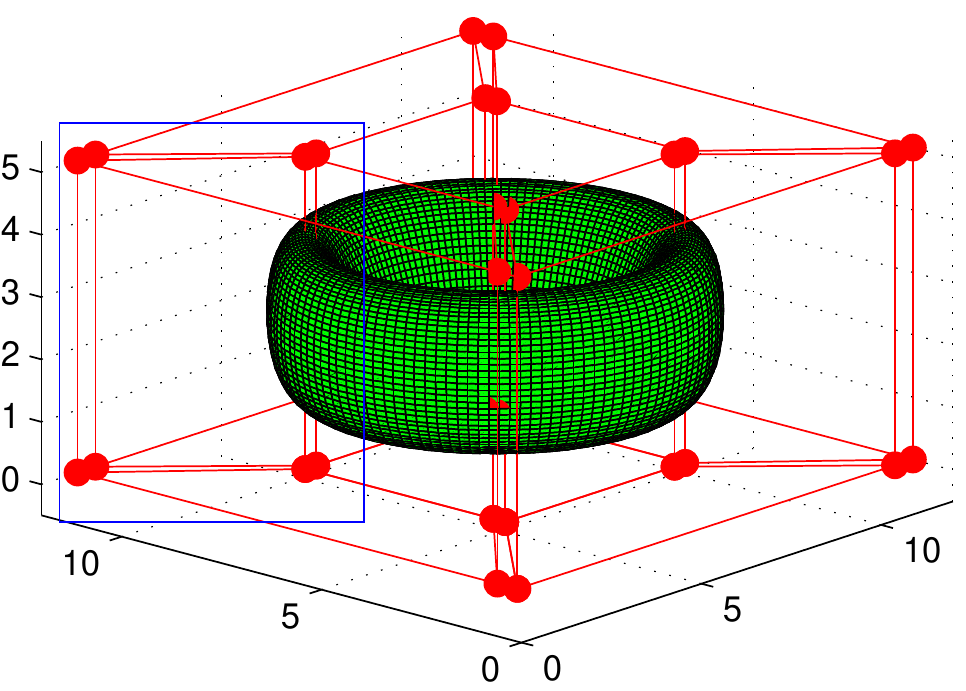, width=1.9 in}\\
(a) $(\alpha,\beta)=(\frac{1}{16},-\frac{1}{30})$ &(b) $(\alpha,\beta)=(\frac{1}{8},0)$  & (c) $(\alpha,\beta)=(\frac{1}{4},\frac{1}{4})$ & \\ \\
\epsfig{file=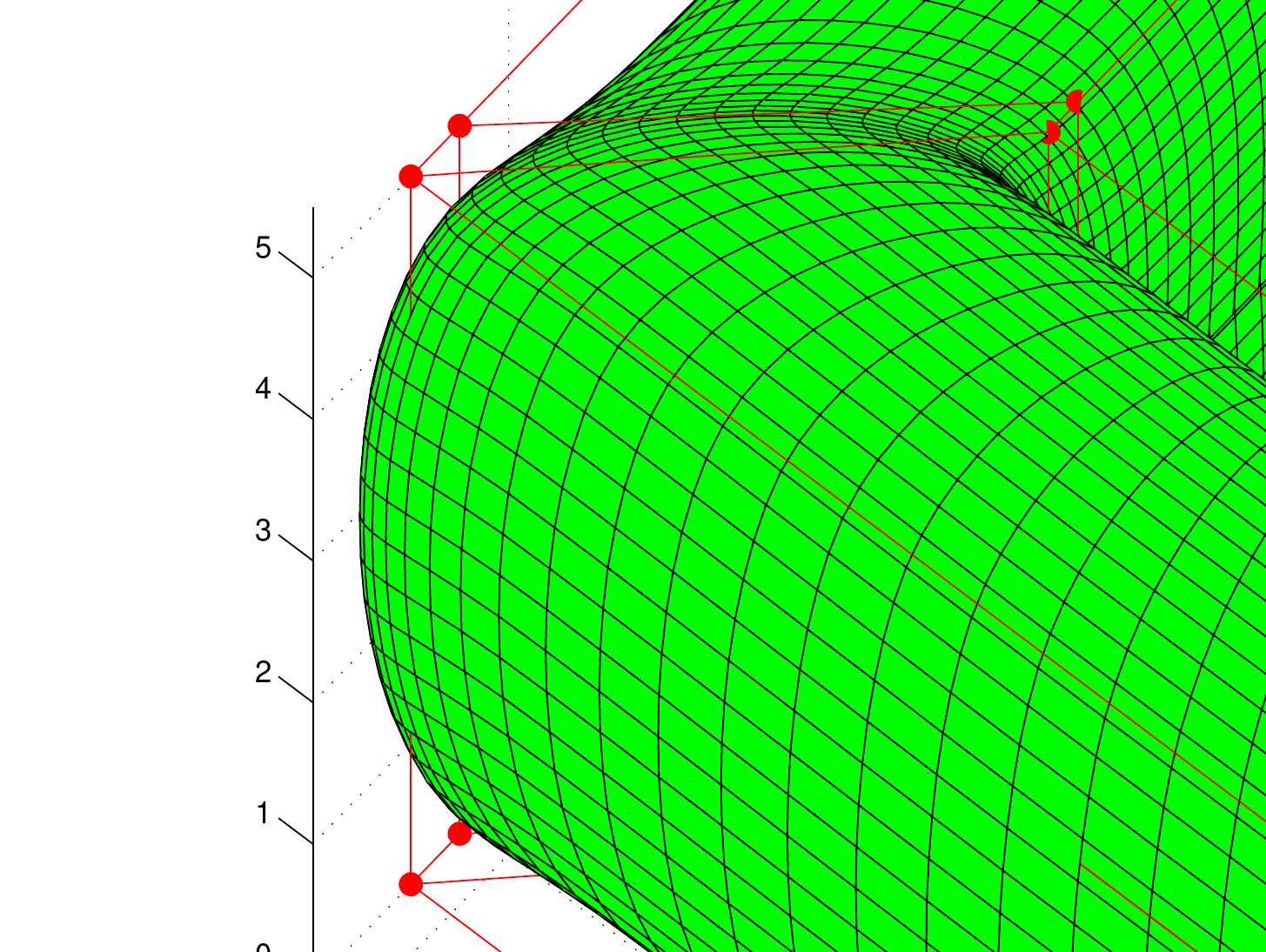, width=1.9 in} & \epsfig{file=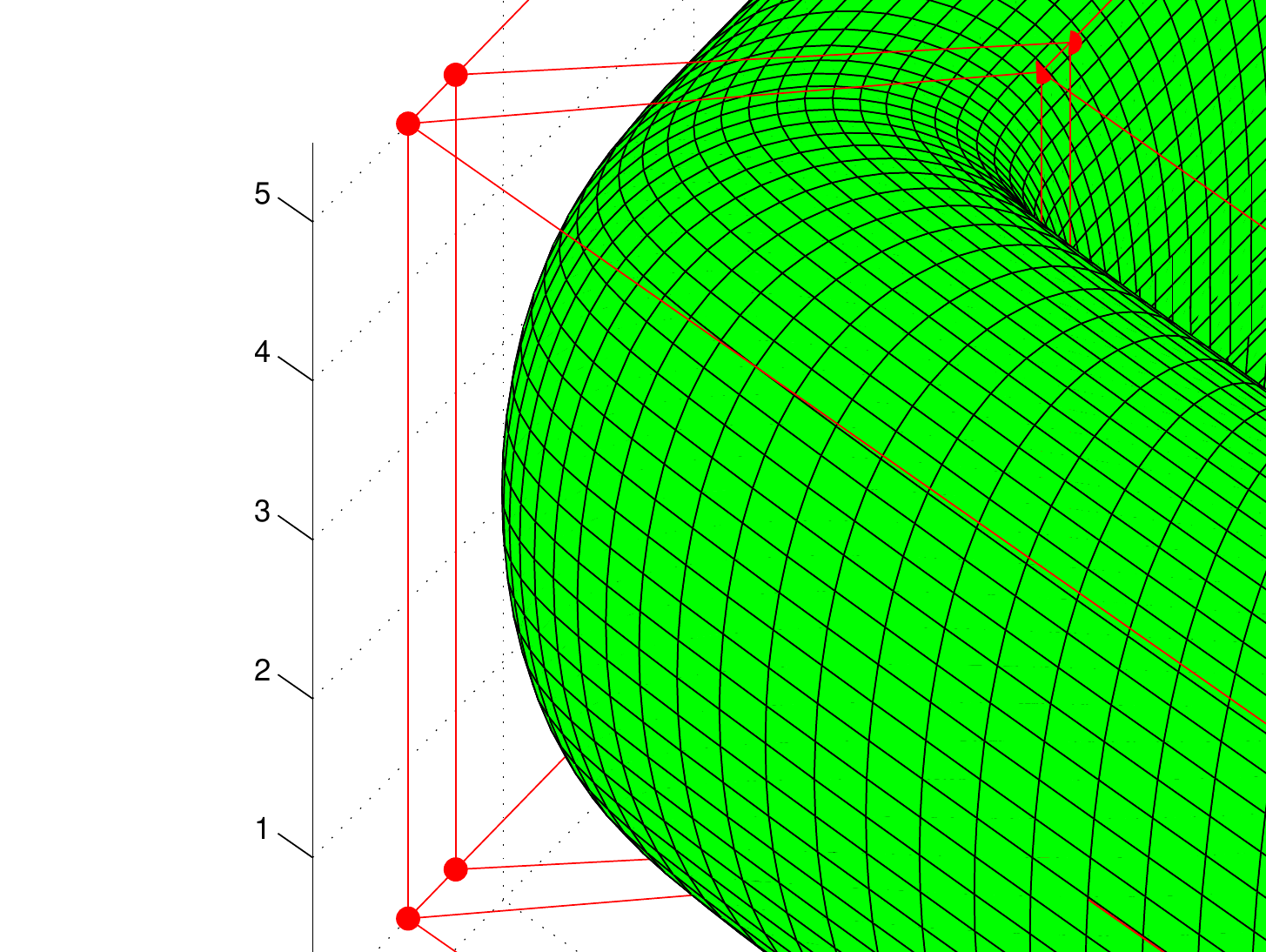, width=1.9 in}  & \epsfig{file=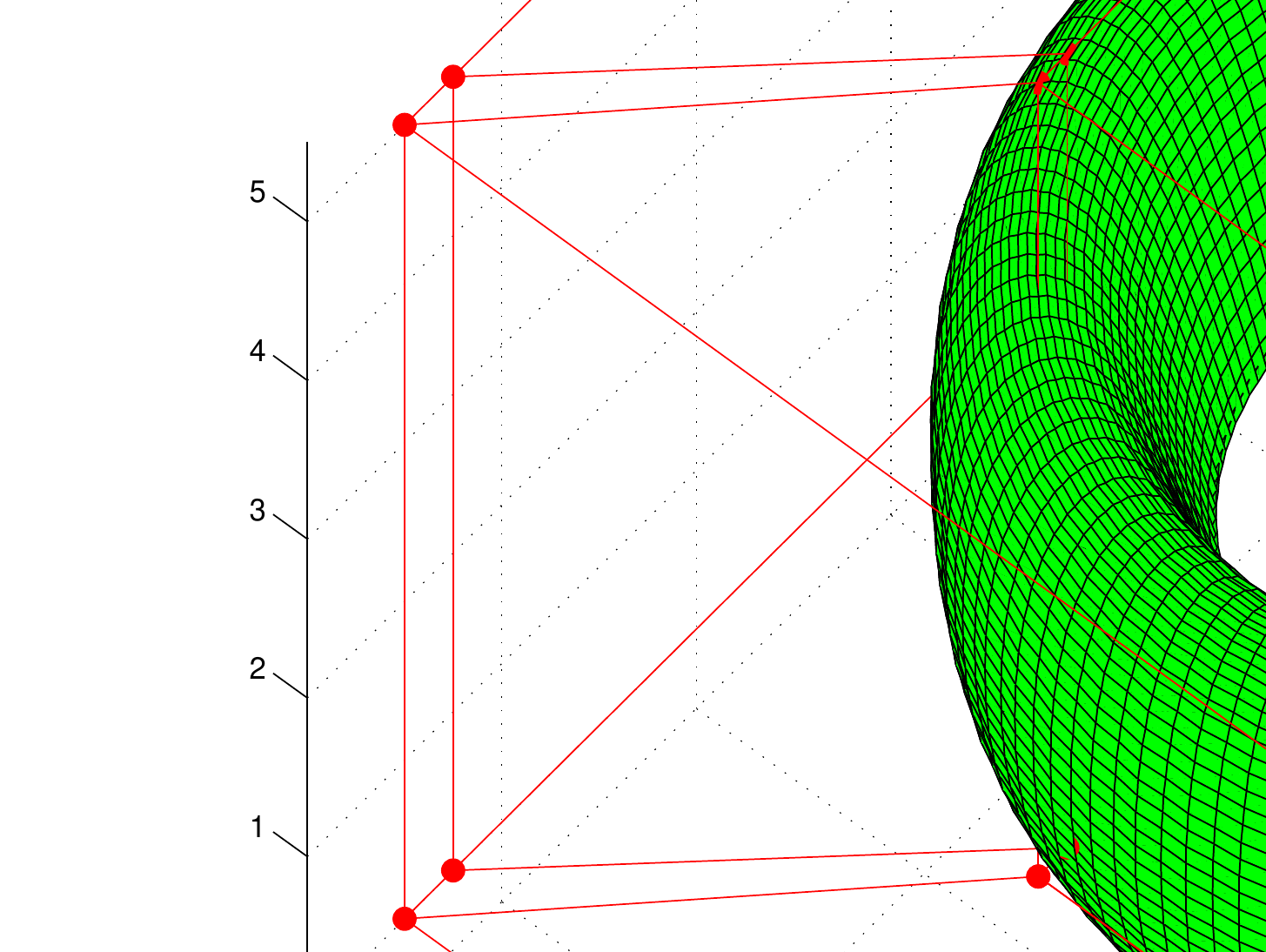, width=1.9 in}\\
(d) & (e) & (f) \\ \\
\epsfig{file=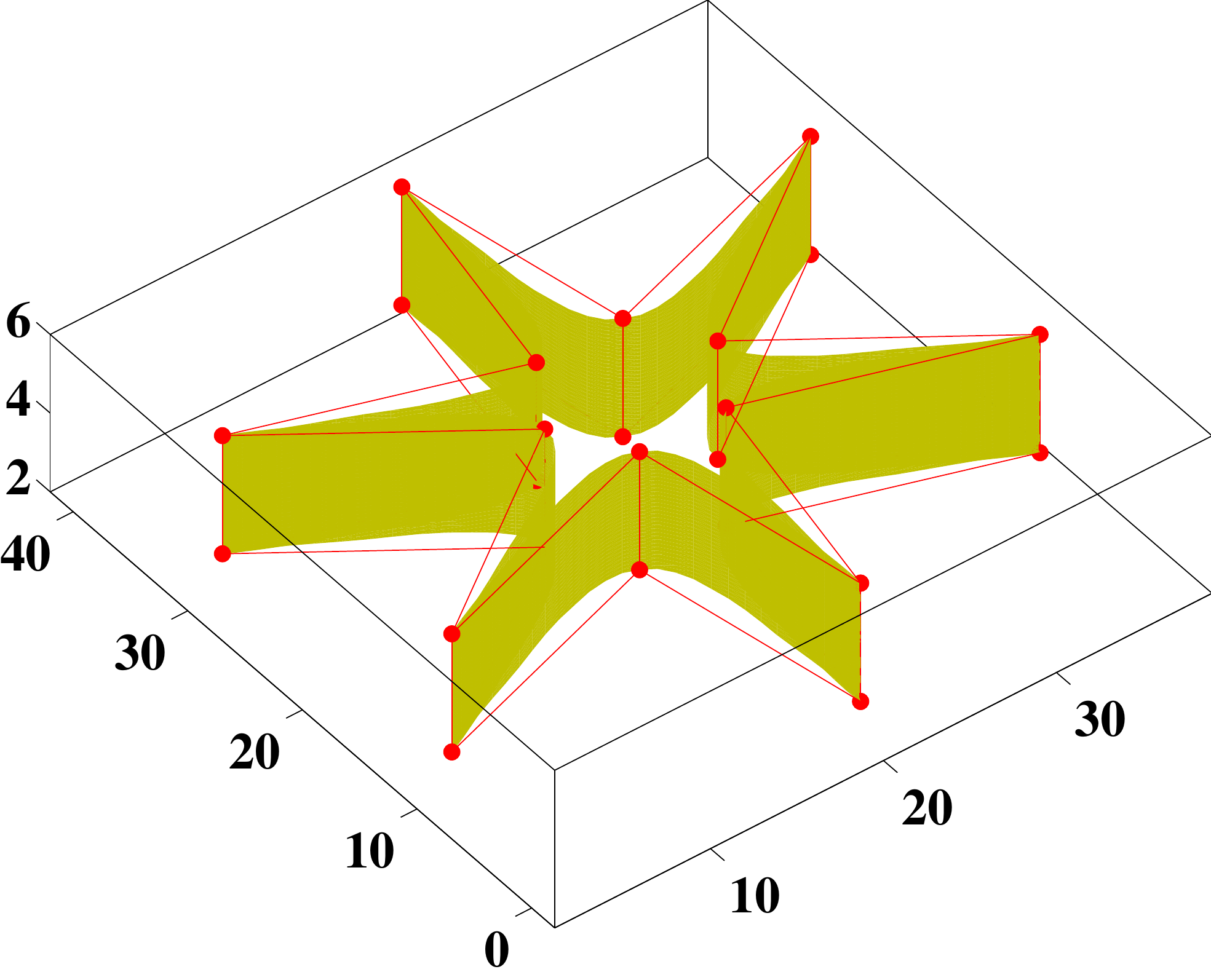, width=1.9 in} & \epsfig{file=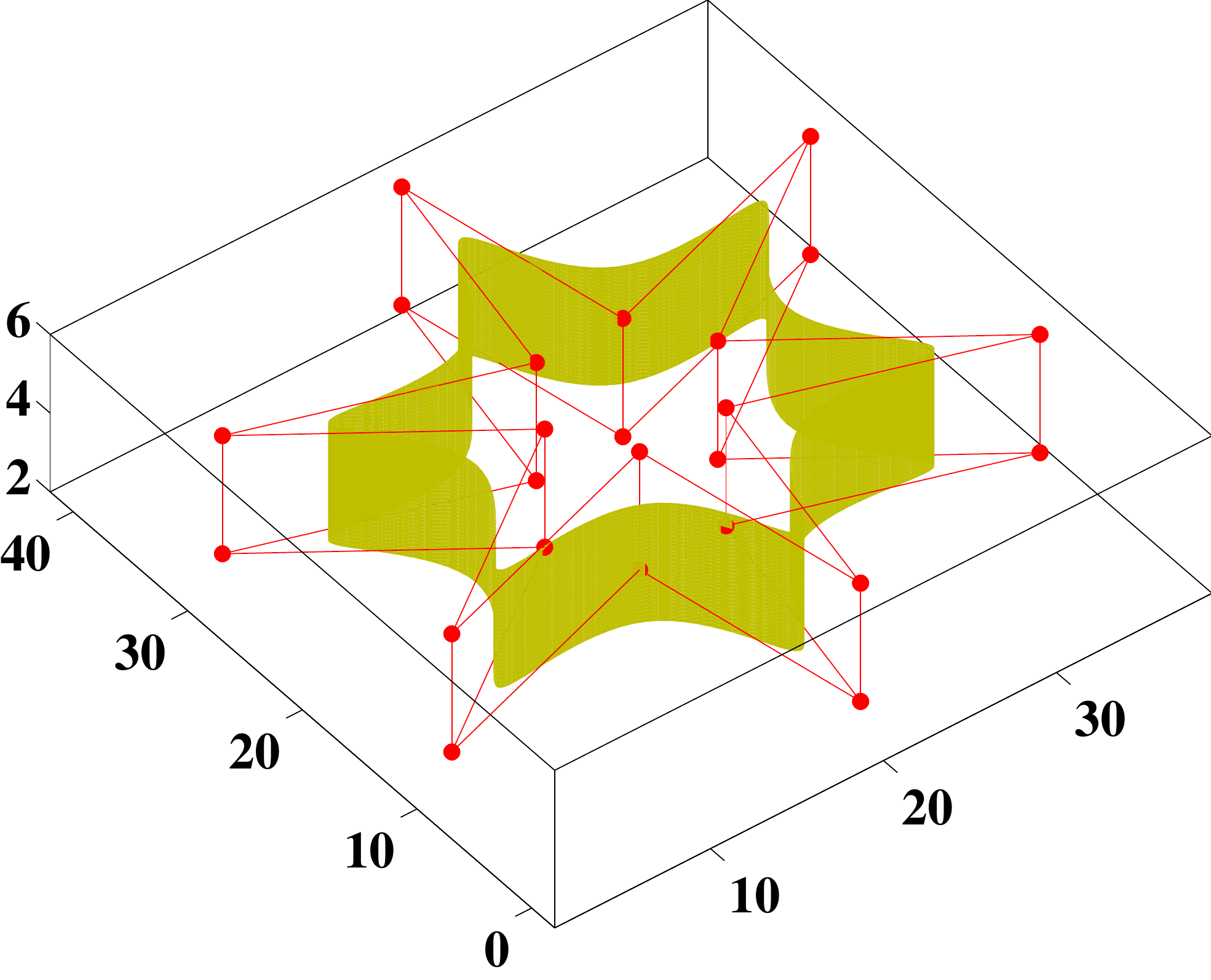, width=1.9 in}  & \epsfig{file=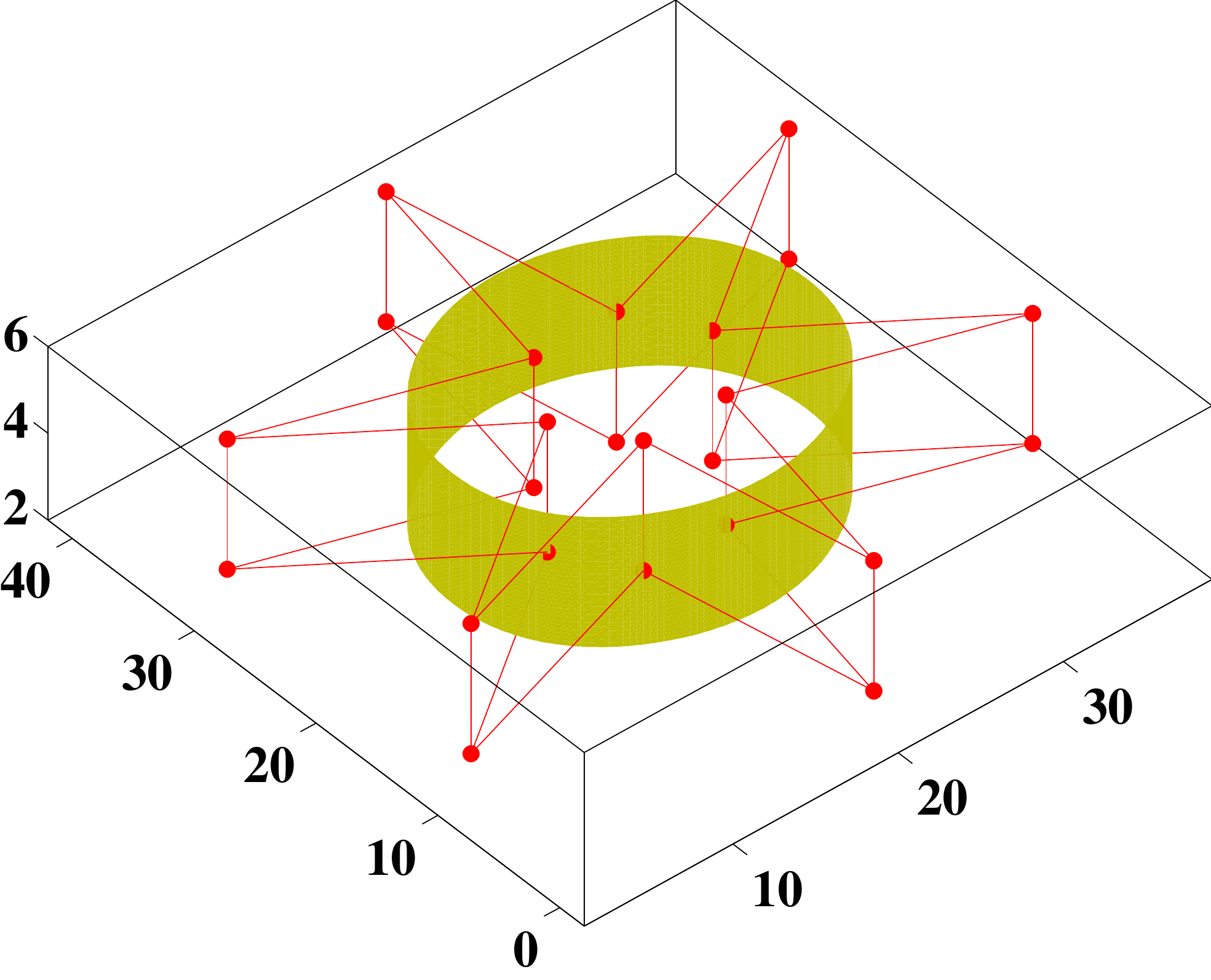, width=1.9 in}\\
(g) $(\alpha,\beta)=(0,-\frac{1}{12})$ &(h) $(\alpha,\beta)=(\frac{1}{8},0)$  & (i) $(\alpha,\beta)=(\frac{1}{4},\frac{1}{4})$ & \\ \\
\epsfig{file=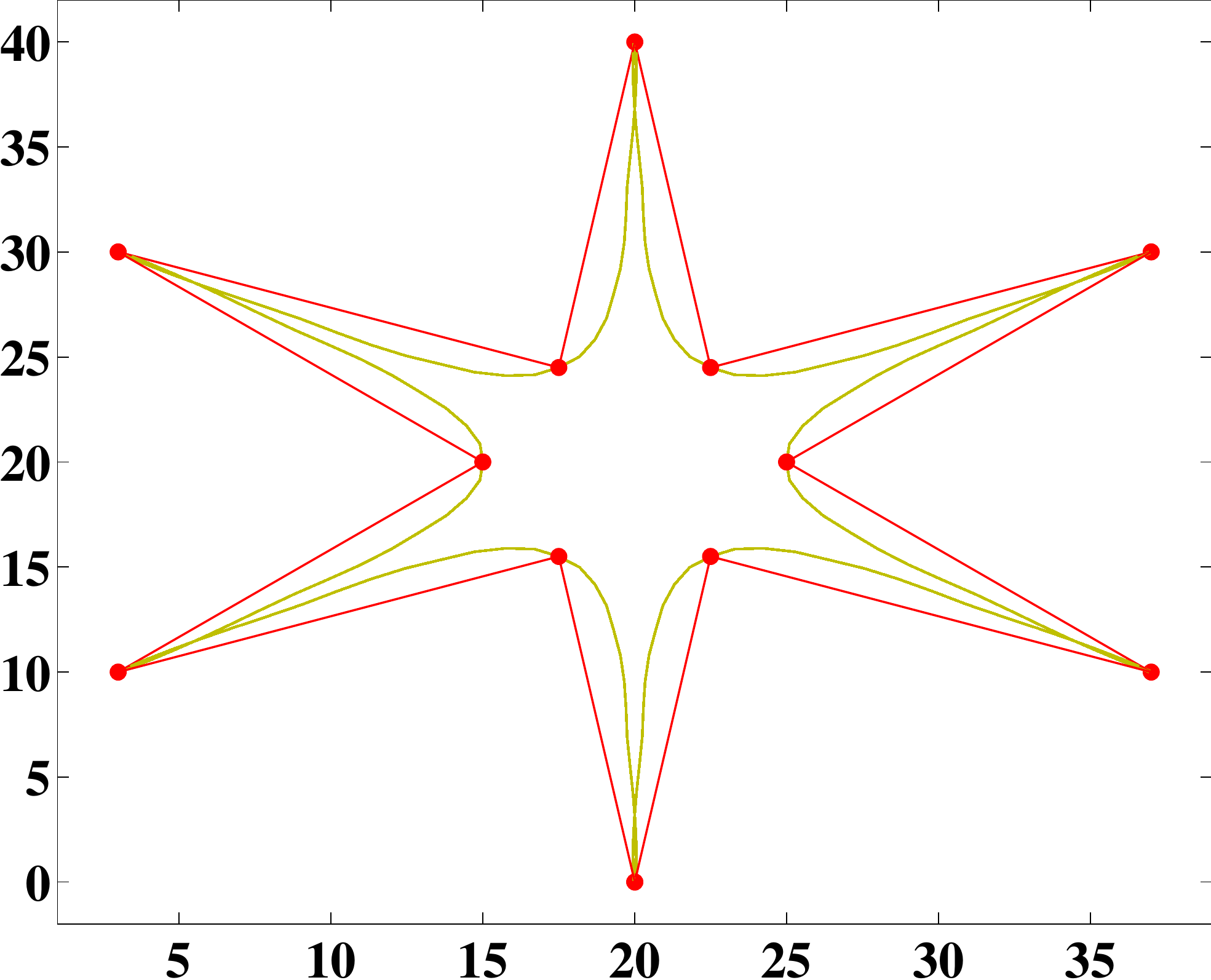, width=1.9 in} & \epsfig{file=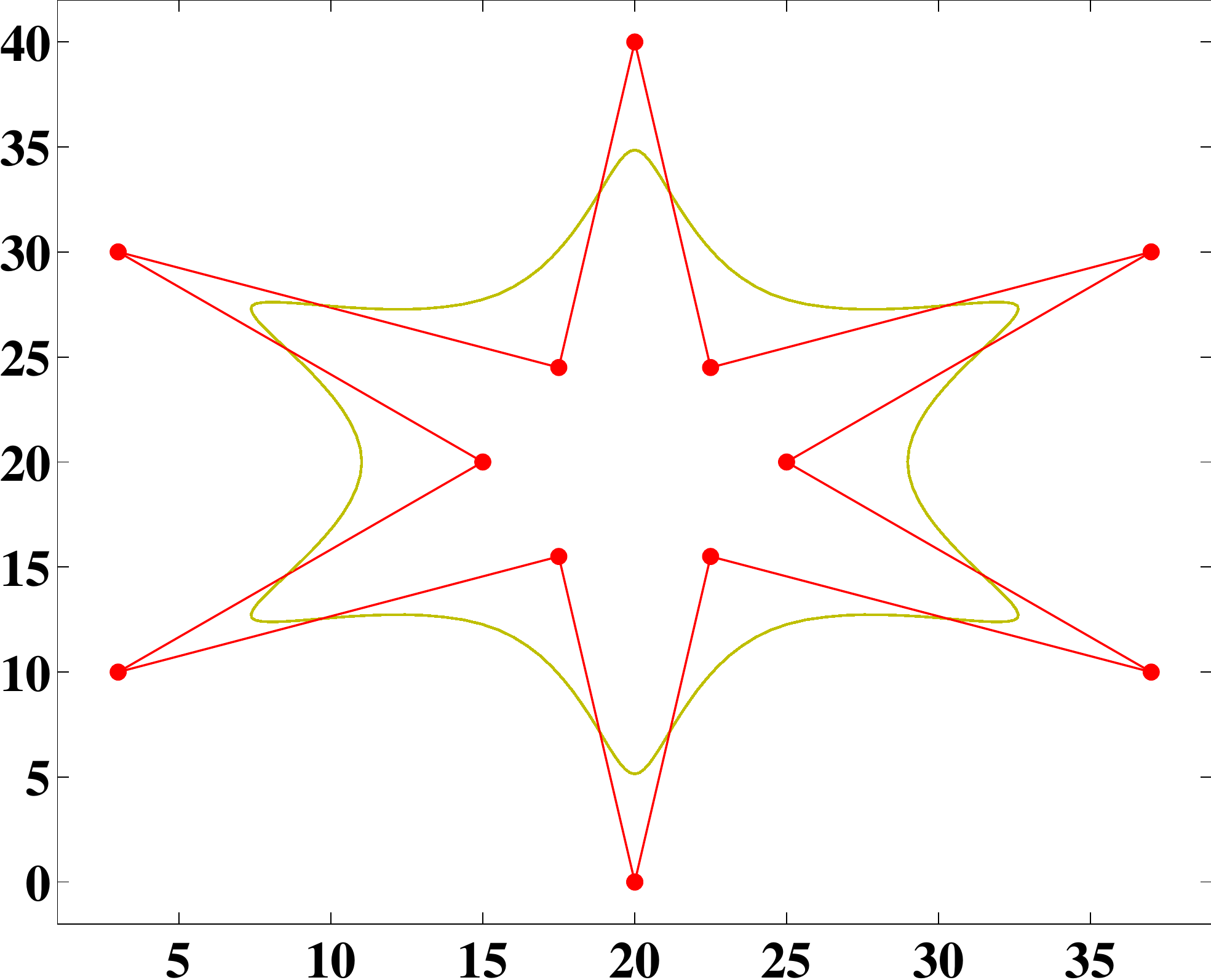, width=1.9 in}  & \epsfig{file=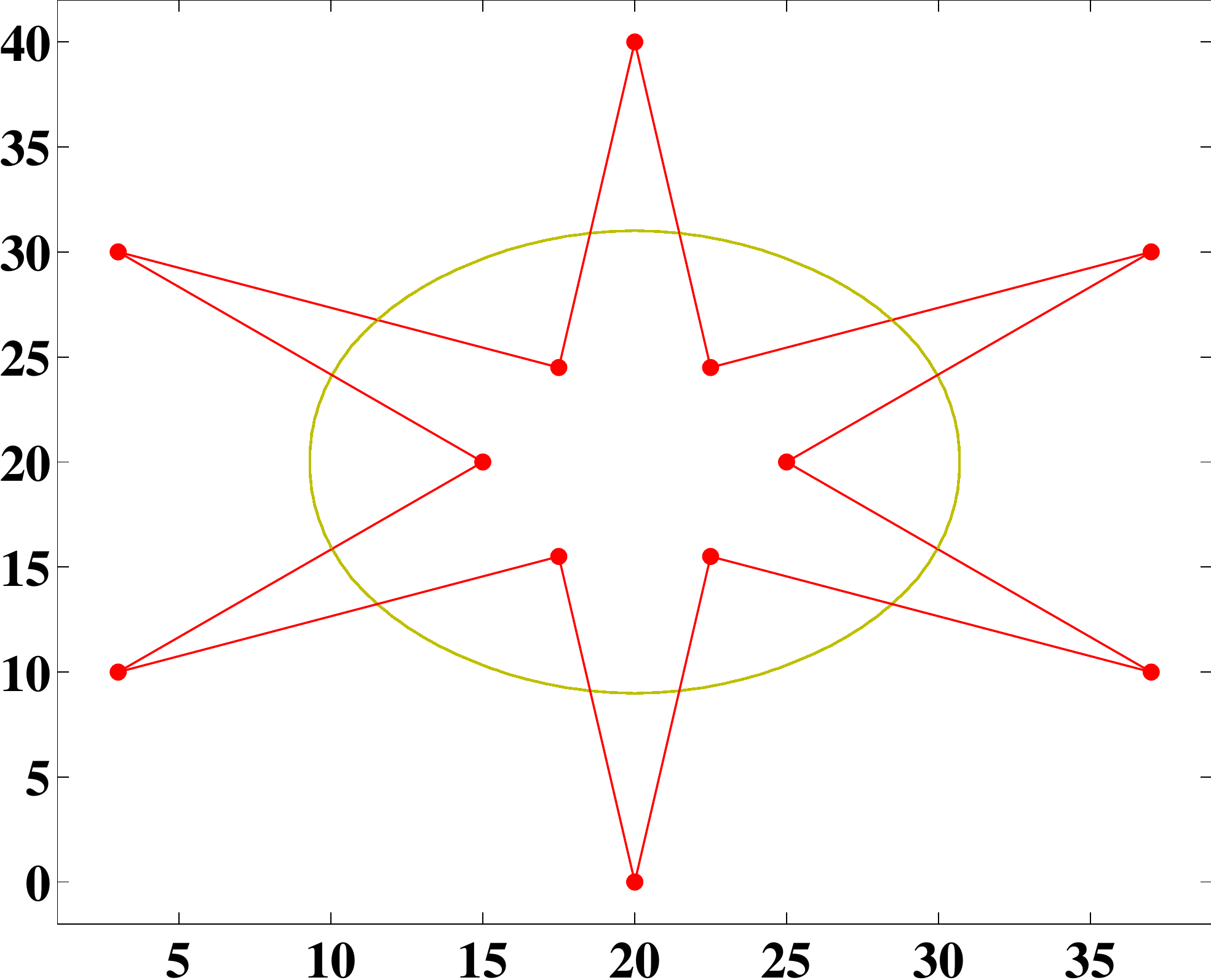, width=1.9 in}\\
(j) & (k) & (l)
 \end{tabular}
\end{center}
 \caption[Surfaces generated by the tensor product of scheme $S_{a_{3}}$.]{\label{p5-3-point-b}\emph{(a)-(c) $\&$ (g)-(i) are the surfaces generated by the tensor product of scheme $S_{a_{3}}$. (d)-(f) are the mirror images of the parts inside the blue rectangles of (a)-(c) respectively, whereas (j)-(l) are the 2-dimensional images in $xy$-planes of (g)-(i) respectively.}}
\end{figure}


\begin{figure}[!h] 
\begin{center}
\begin{tabular}{cccc}
\epsfig{file=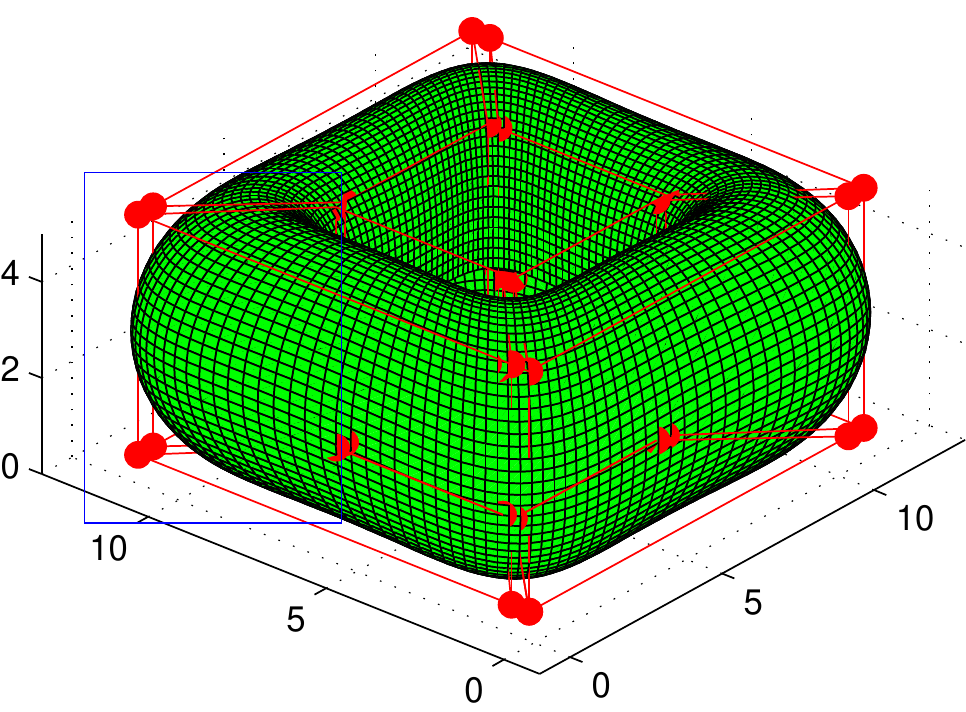, width=1.9 in} & \epsfig{file=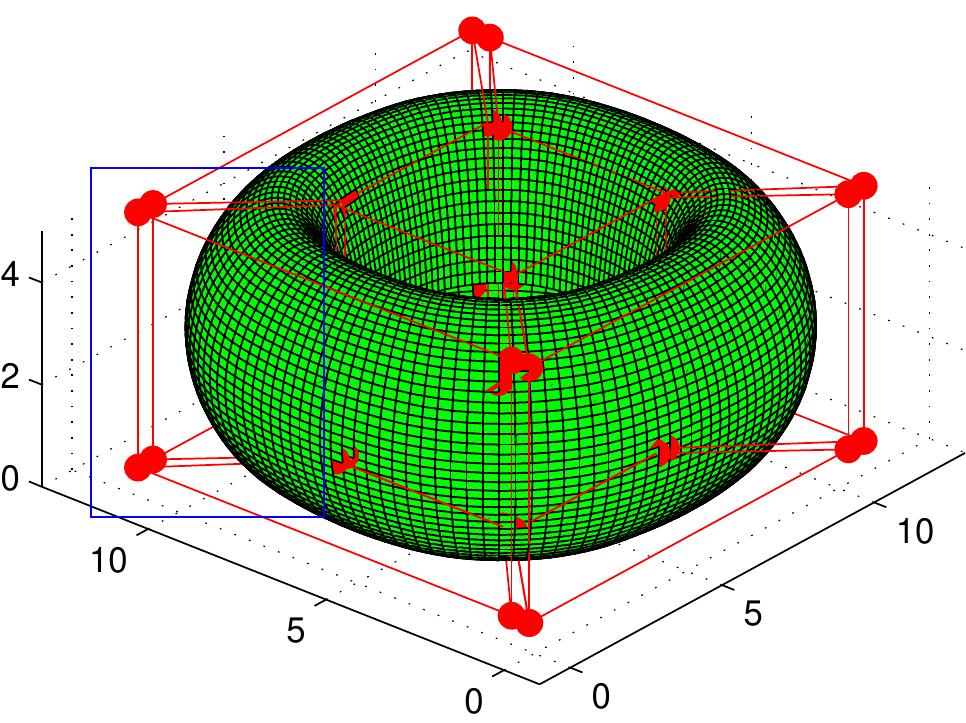, width=1.9 in}  & \epsfig{file=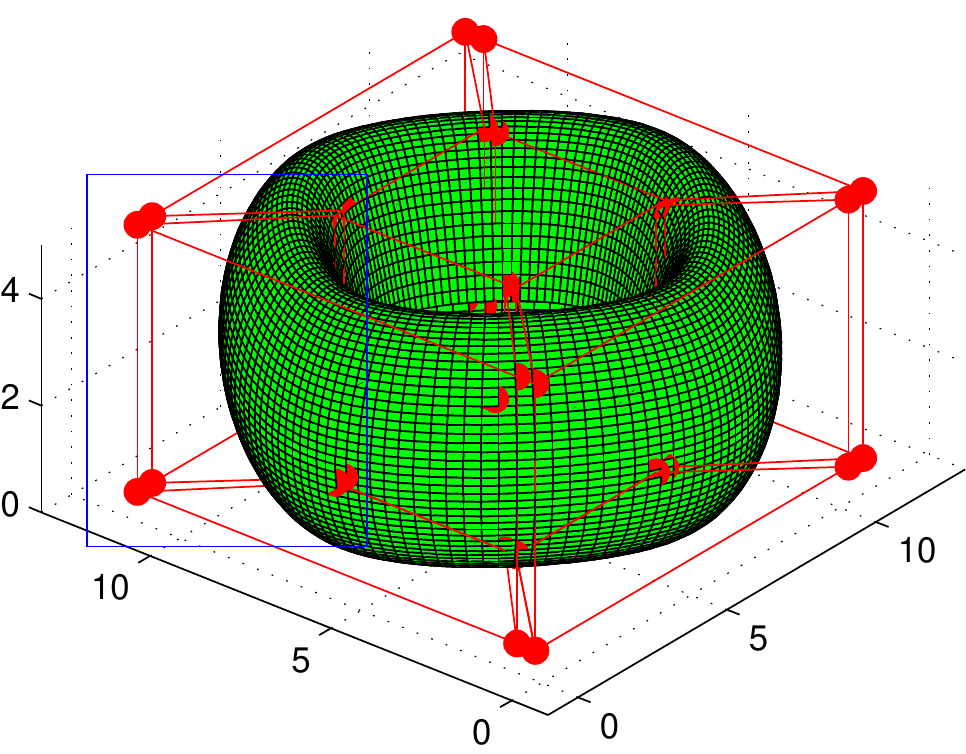, width=1.9 in}\\
(a) $(\alpha,\beta)=(\frac{1}{16},-\frac{1}{48})$ &(b) $(\alpha,\beta)=(\frac{1}{10},-\frac{49}{1152})$  & (c) $(\alpha,\beta)=(\frac{1}{8},-\frac{13}{224})$ & \\ \\
\epsfig{file=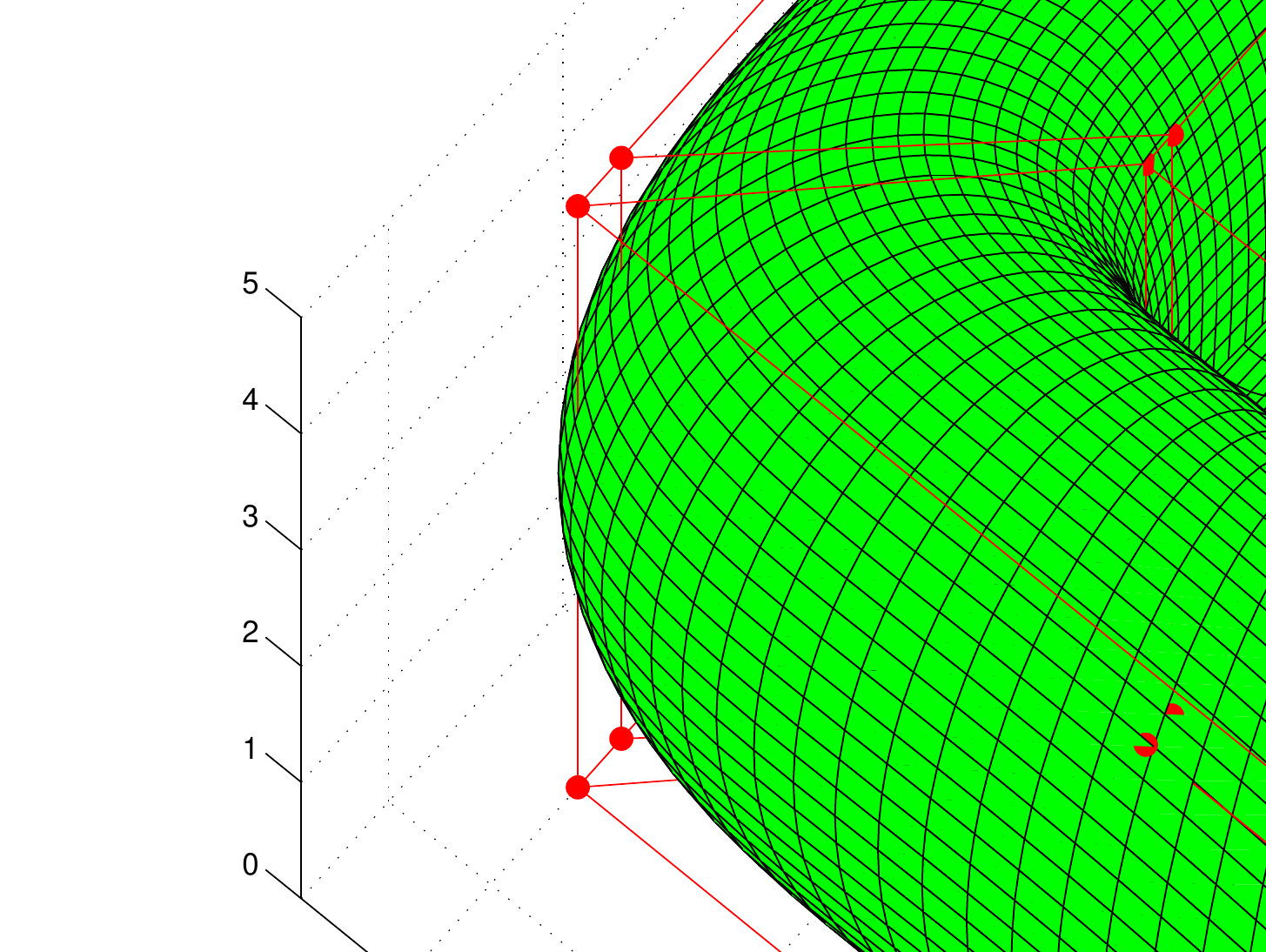, width=1.9 in} & \epsfig{file=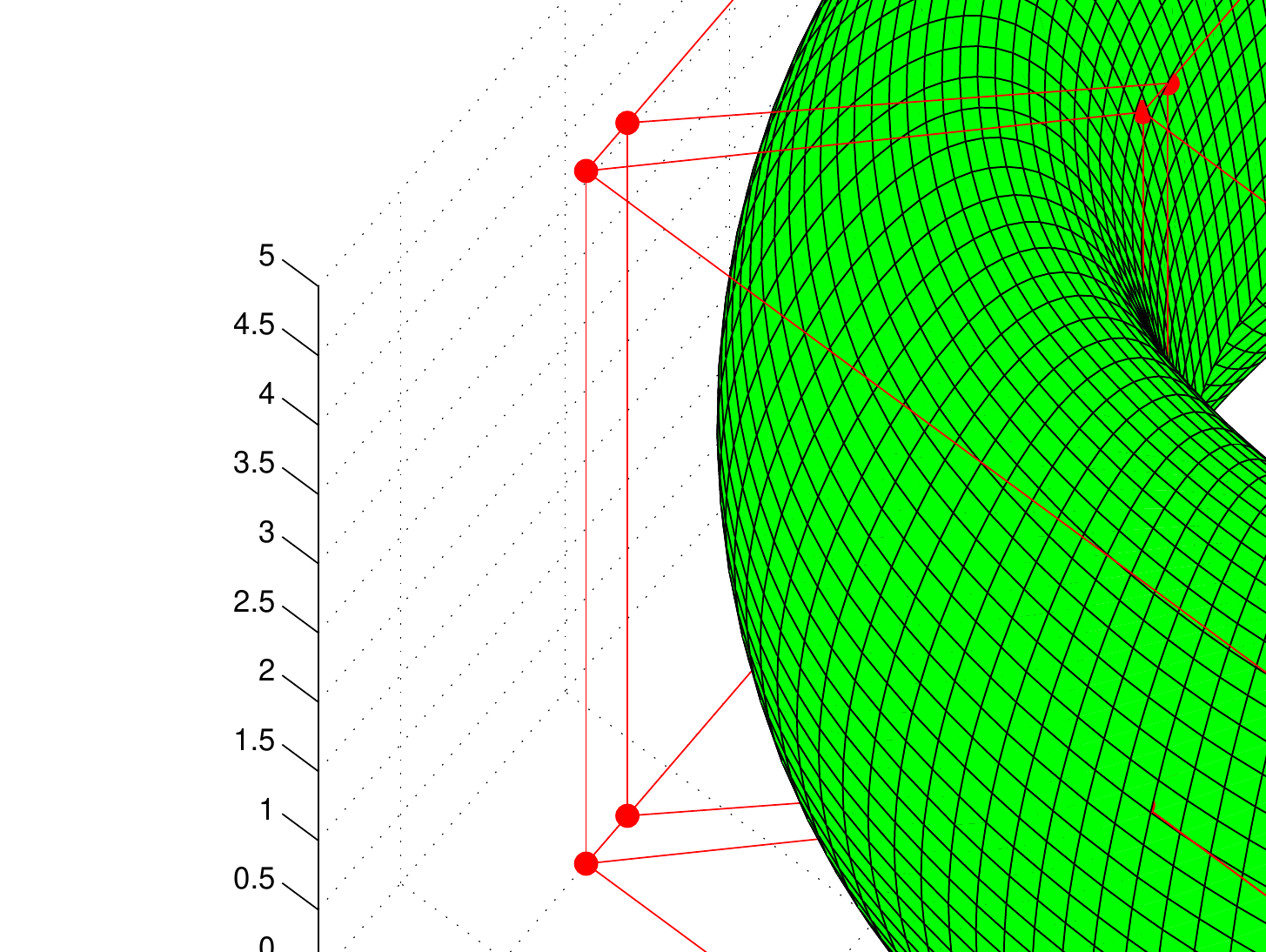, width=1.9 in}  & \epsfig{file=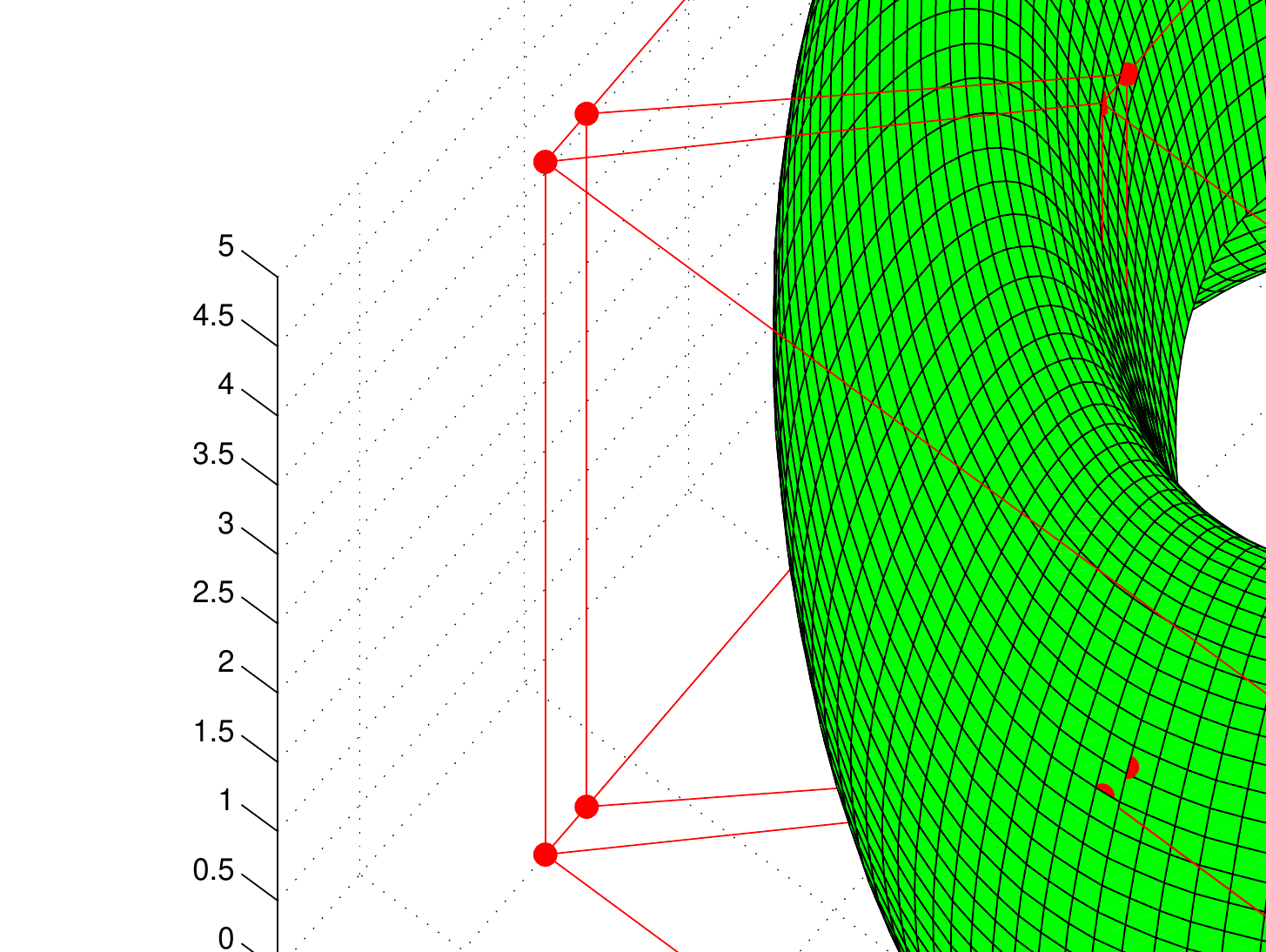, width=1.9 in}\\
(d) & (e) & (f) \\ \\
\epsfig{file=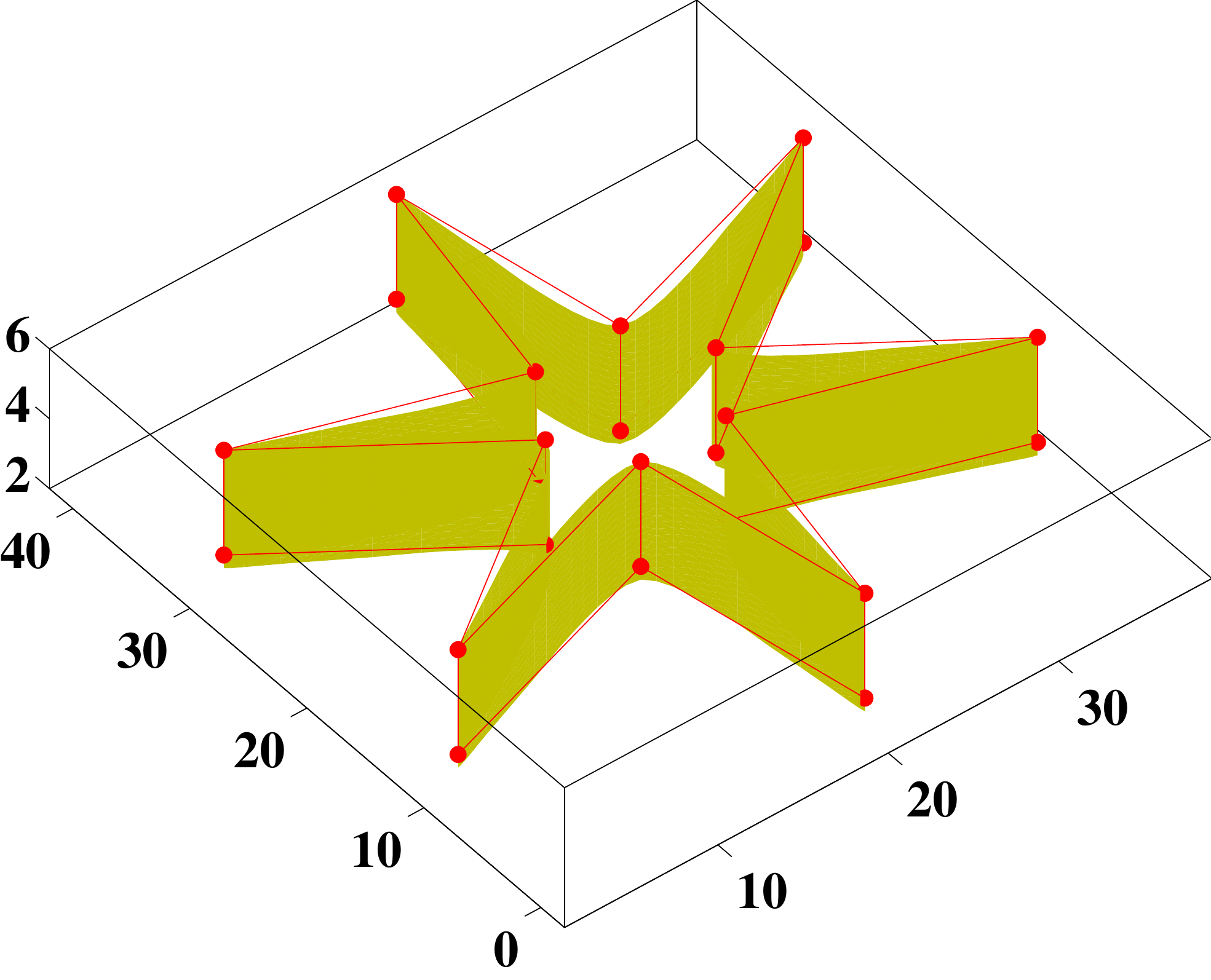, width=1.9 in} & \epsfig{file=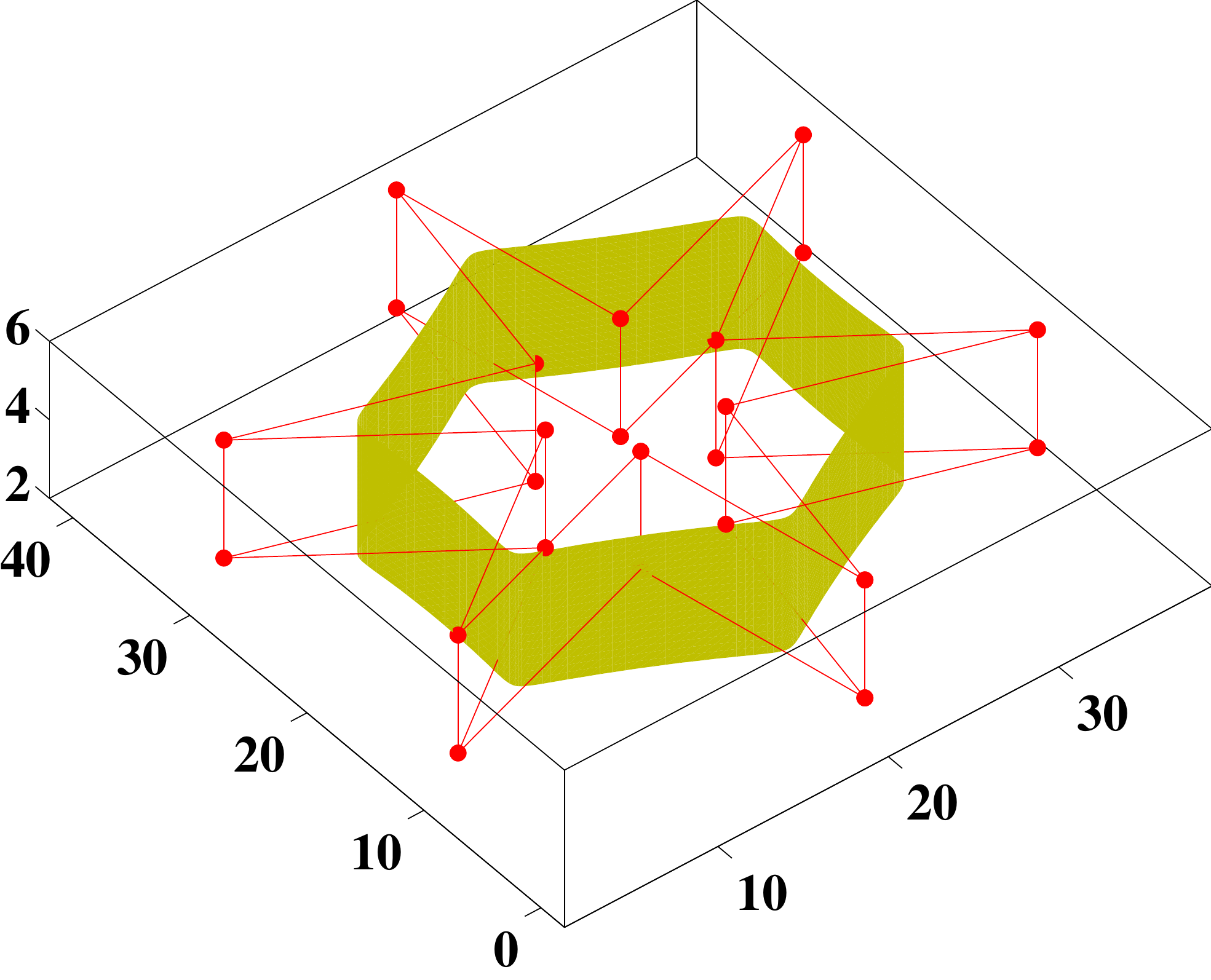, width=1.9 in}  & \epsfig{file=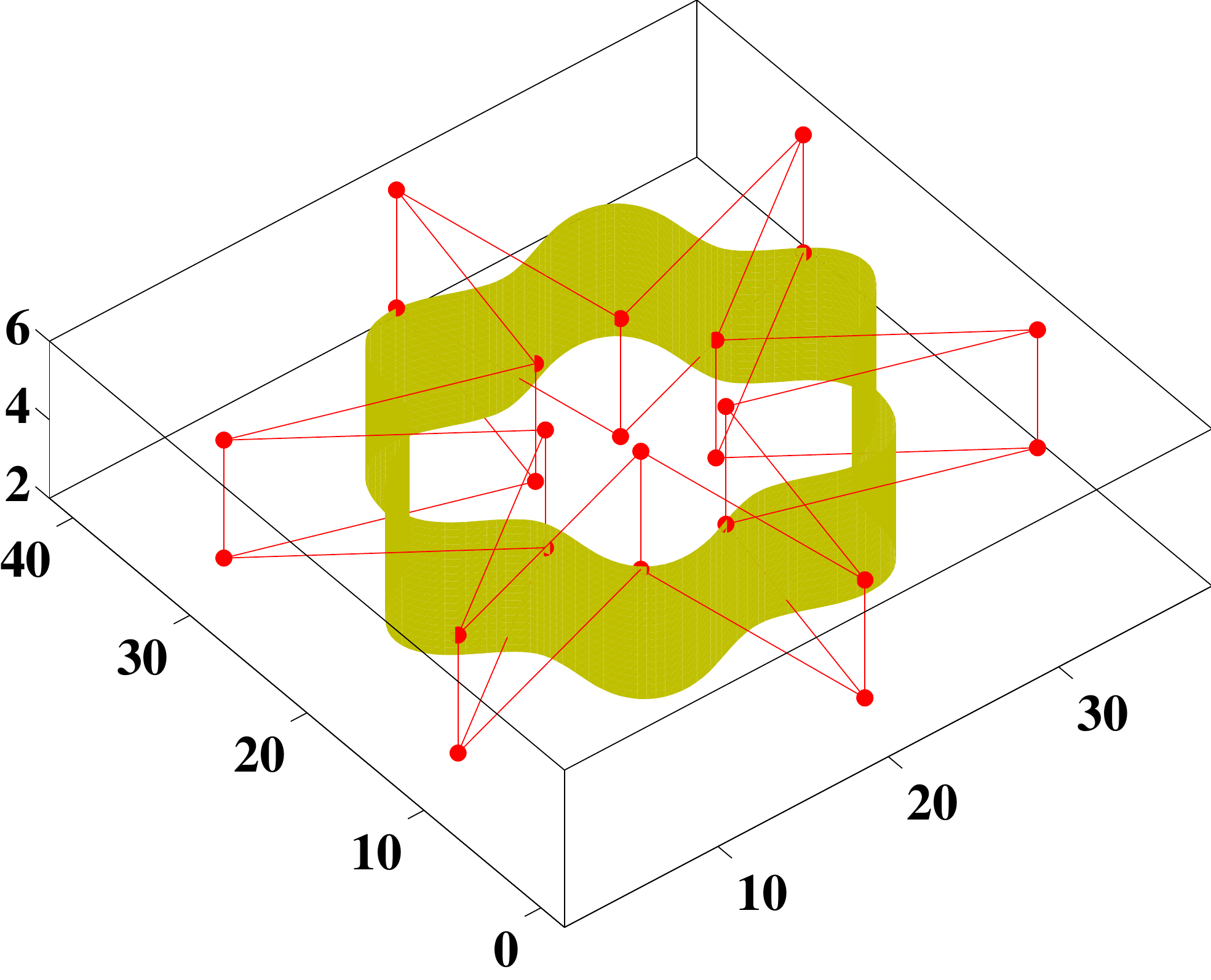, width=1.9 in}\\
(g) $(\alpha,\beta)=(0,-\frac{1}{256})$ &(h) $(\alpha,\beta)=(\frac{1}{16},-\frac{1}{48})$  & (i) $(\alpha,\beta)=(\frac{1}{10},-\frac{49}{1152})$ & \\ \\
\epsfig{file=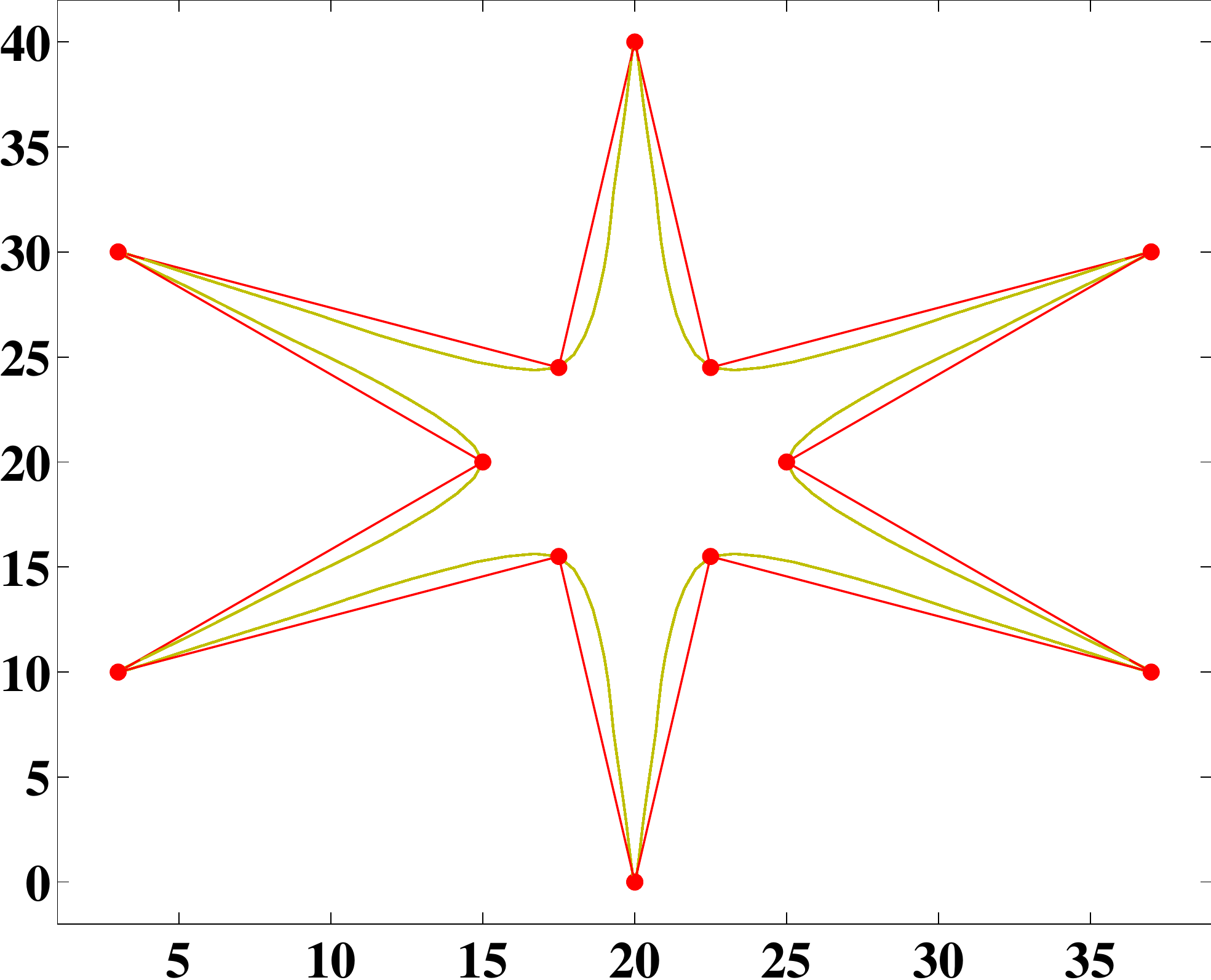, width=1.9 in} & \epsfig{file=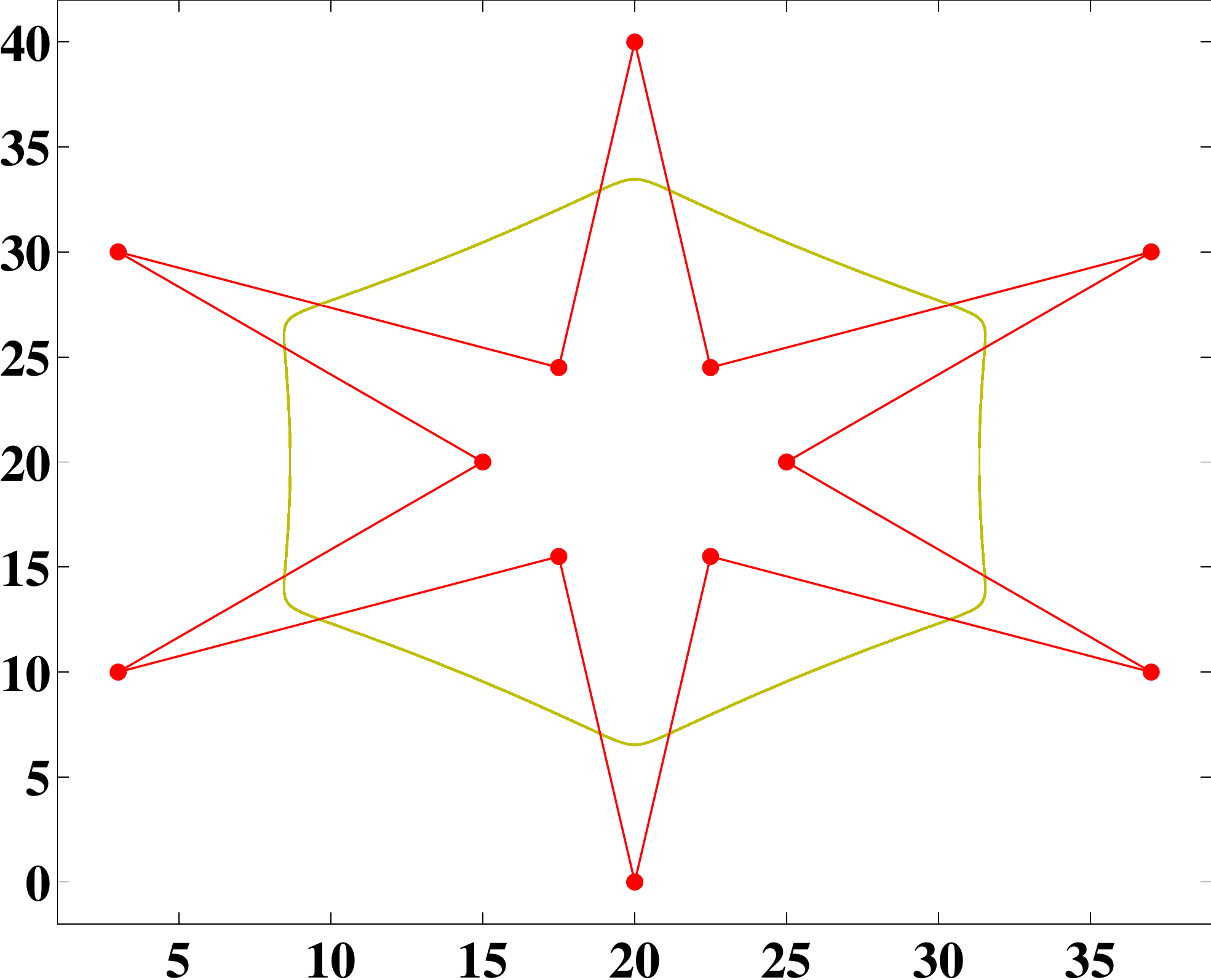, width=1.9 in}  & \epsfig{file=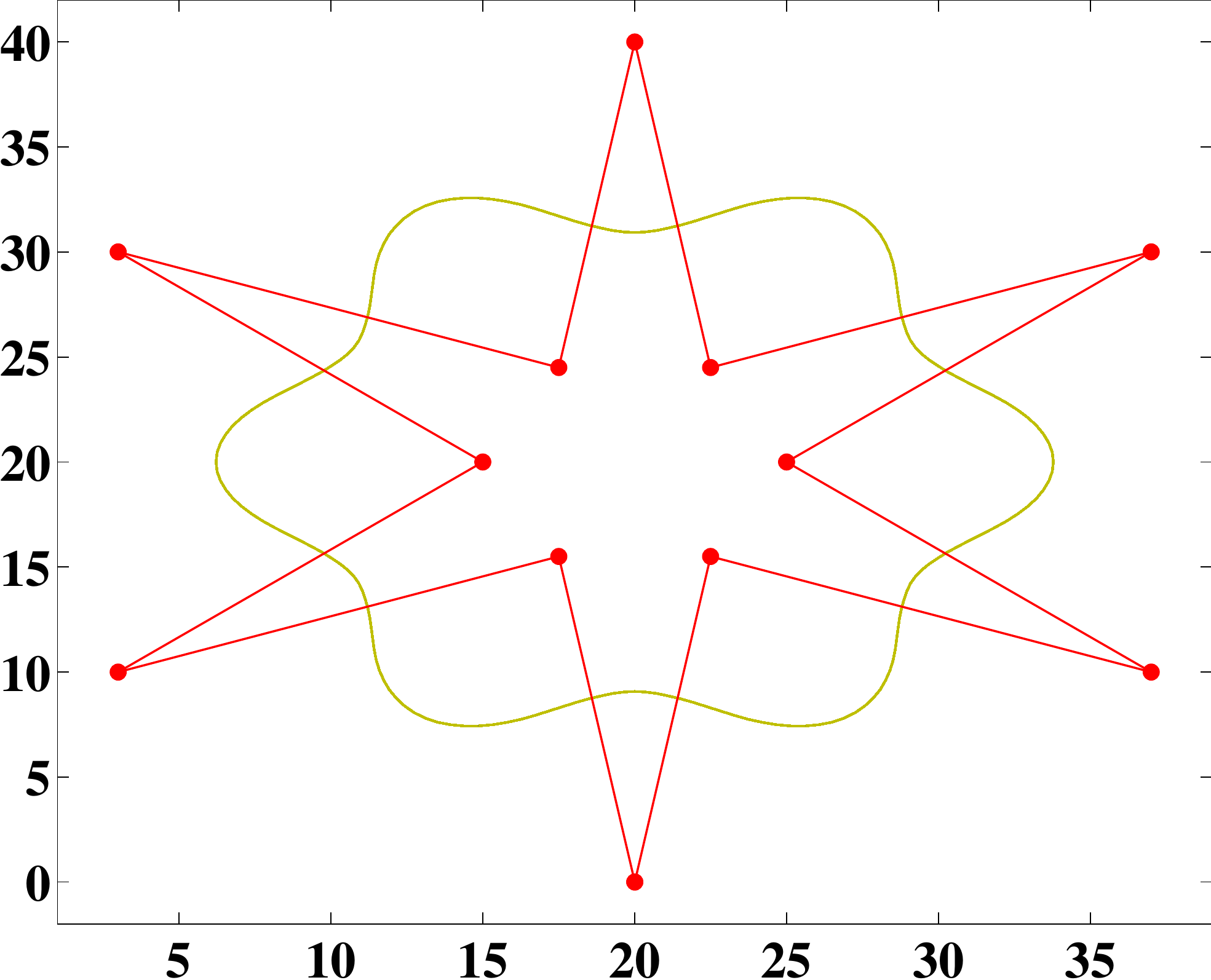, width=1.9 in}\\
(j) & (k) & (l)
 \end{tabular}
\end{center}
 \caption[Surfaces generated by the tensor product of scheme $S_{a_{5}}$.]{\label{p5-5-point-b}\emph{(a)-(c) $\&$ (g)-(i) are the surfaces generated by the tensor product of scheme $S_{a_{5}}$. (d)-(f) are the mirror images of the parts inside the blue rectangles of (a)-(c) respectively, whereas (j)-(l) are the 2-dimensional images in $xy$-planes of (g)-(i) respectively.}}
\end{figure}
\begin{figure}[!h] 
\begin{center}
\begin{tabular}{cccc}
\epsfig{file=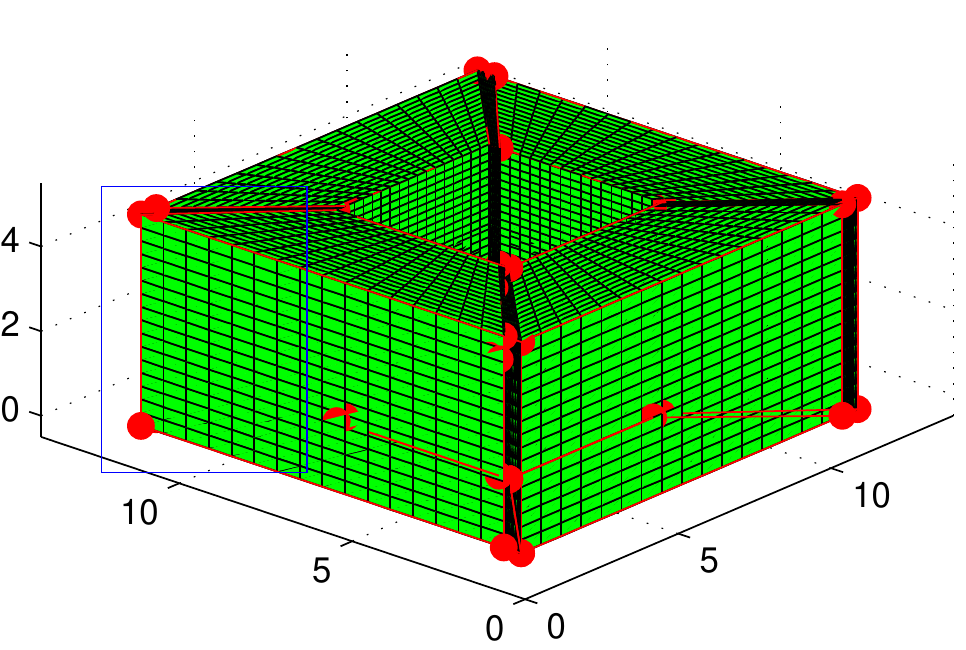, width=1.9 in} & \epsfig{file=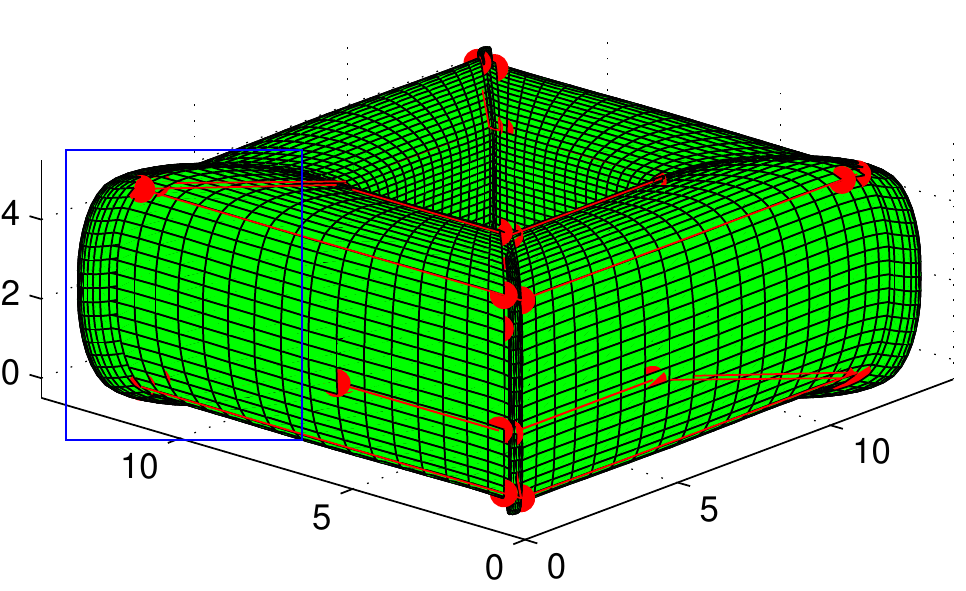, width=1.9 in}  & \epsfig{file=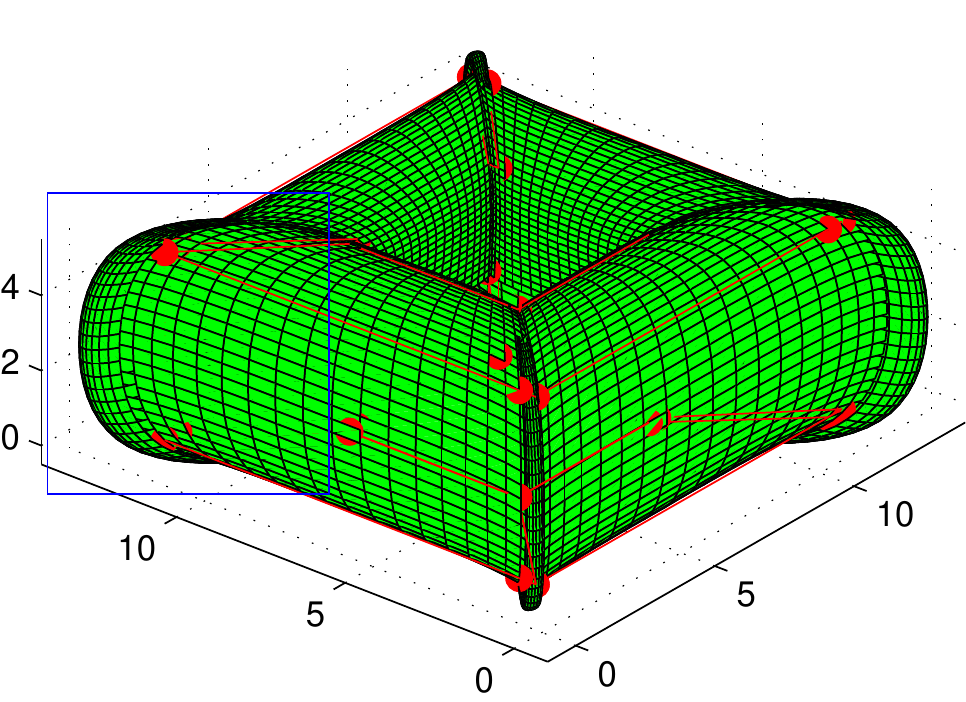, width=1.9 in}\\
(a) $2$-point scheme & (b) $4$-point scheme & (c) $6$-point scheme \\ \\
\epsfig{file=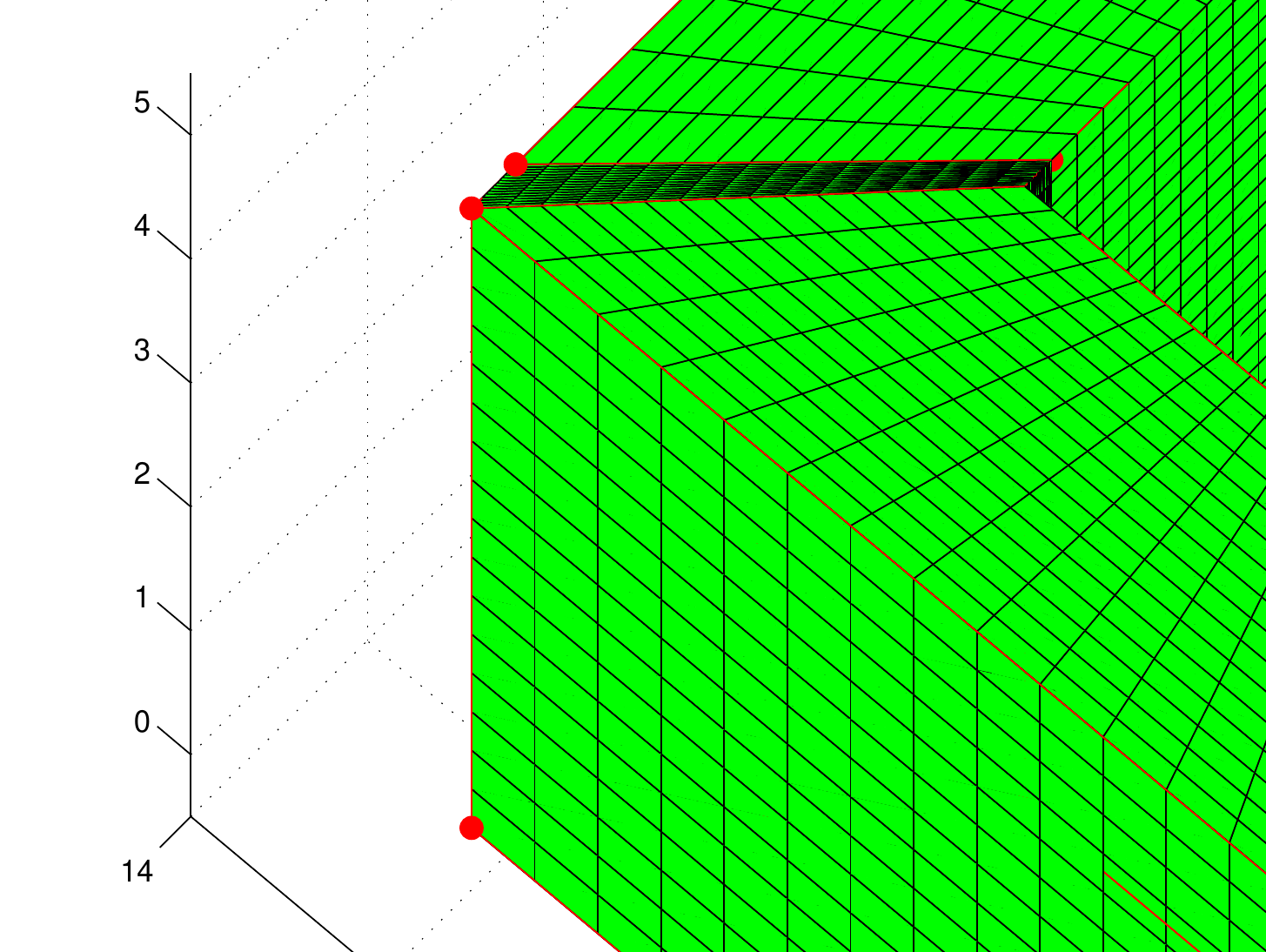, width=1.9 in} & \epsfig{file=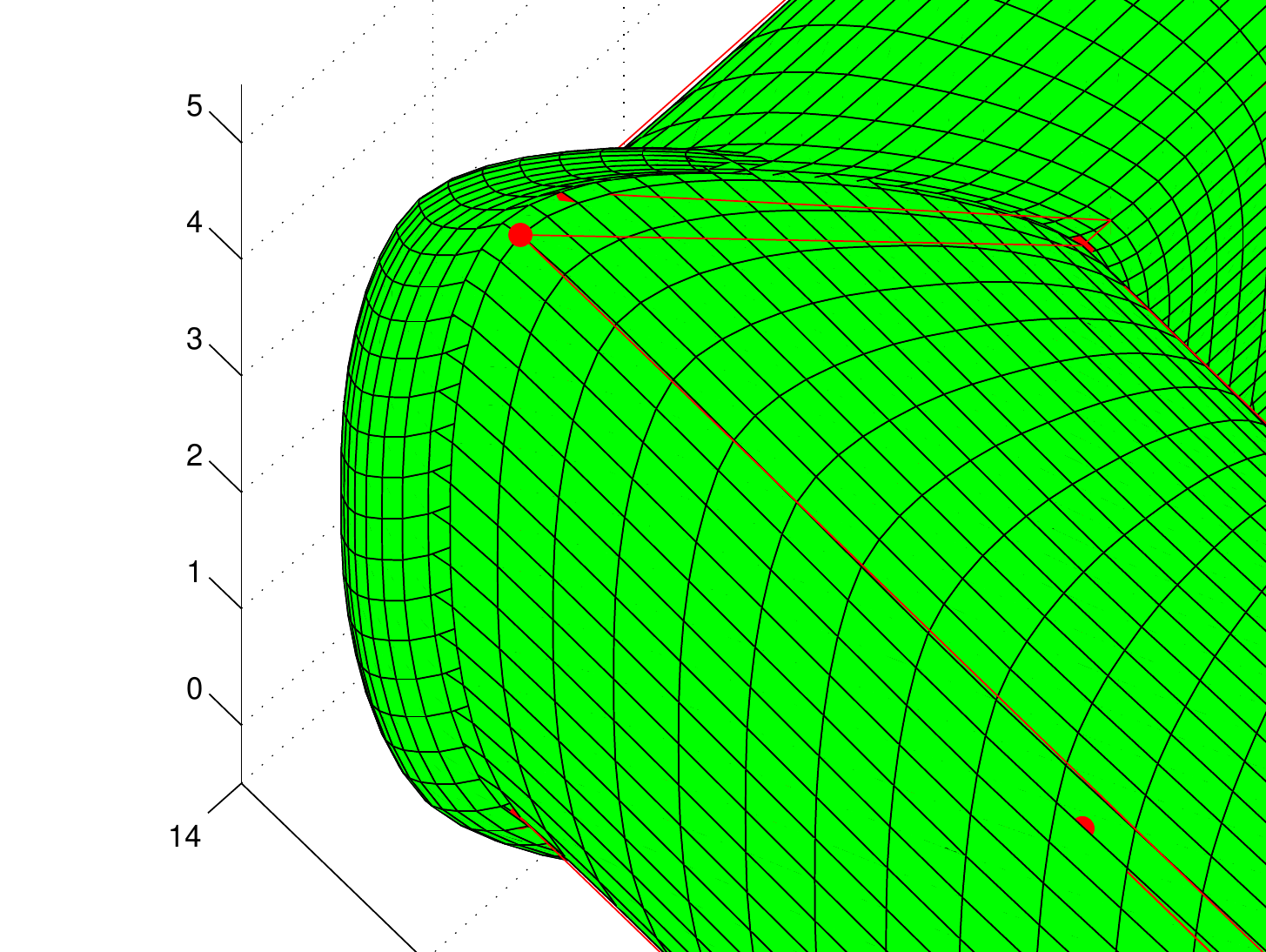, width=1.9 in}  & \epsfig{file=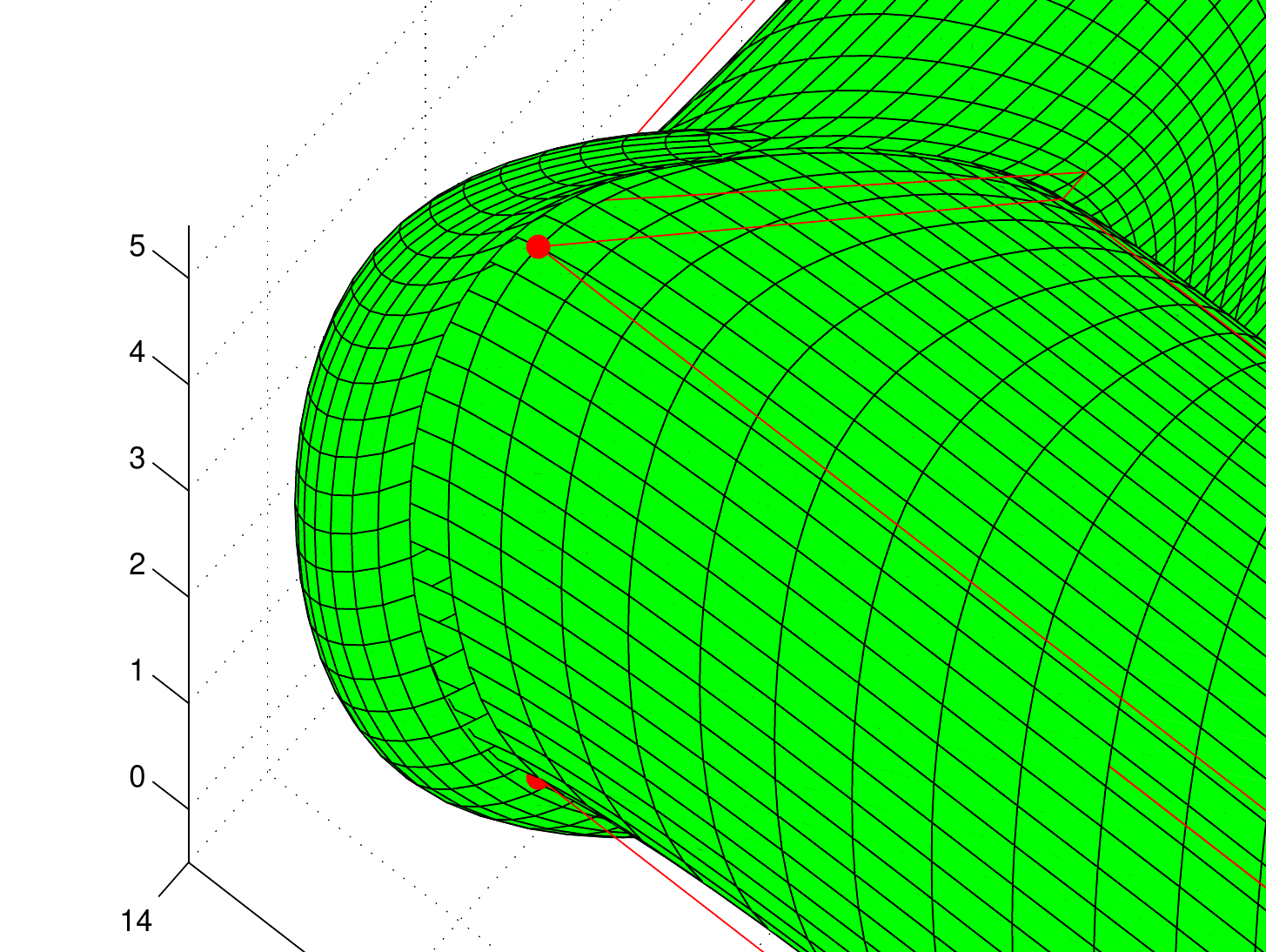, width=1.9 in}\\
(d) & (e) & (f) \\ \\
\epsfig{file=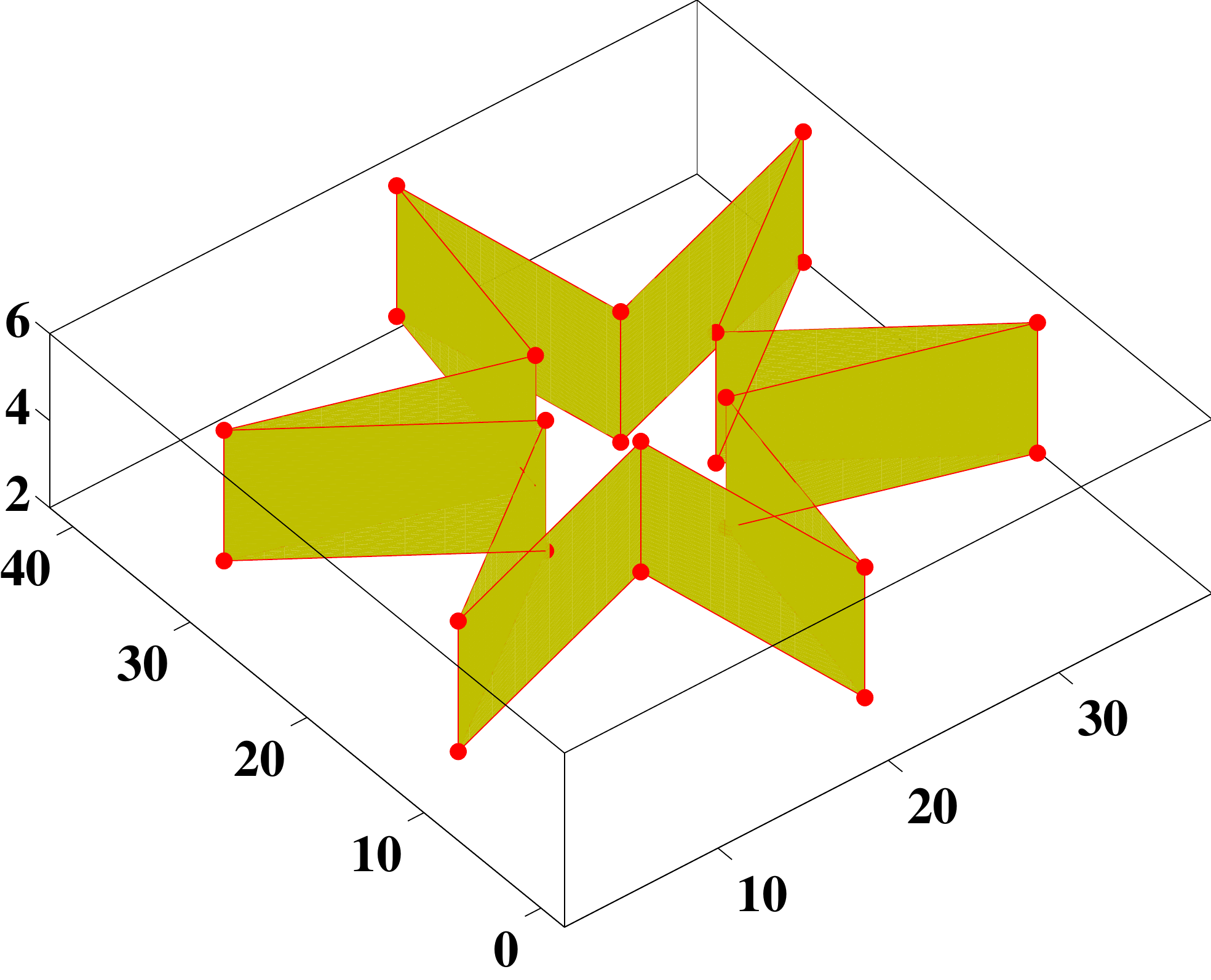, width=1.9 in} & \epsfig{file=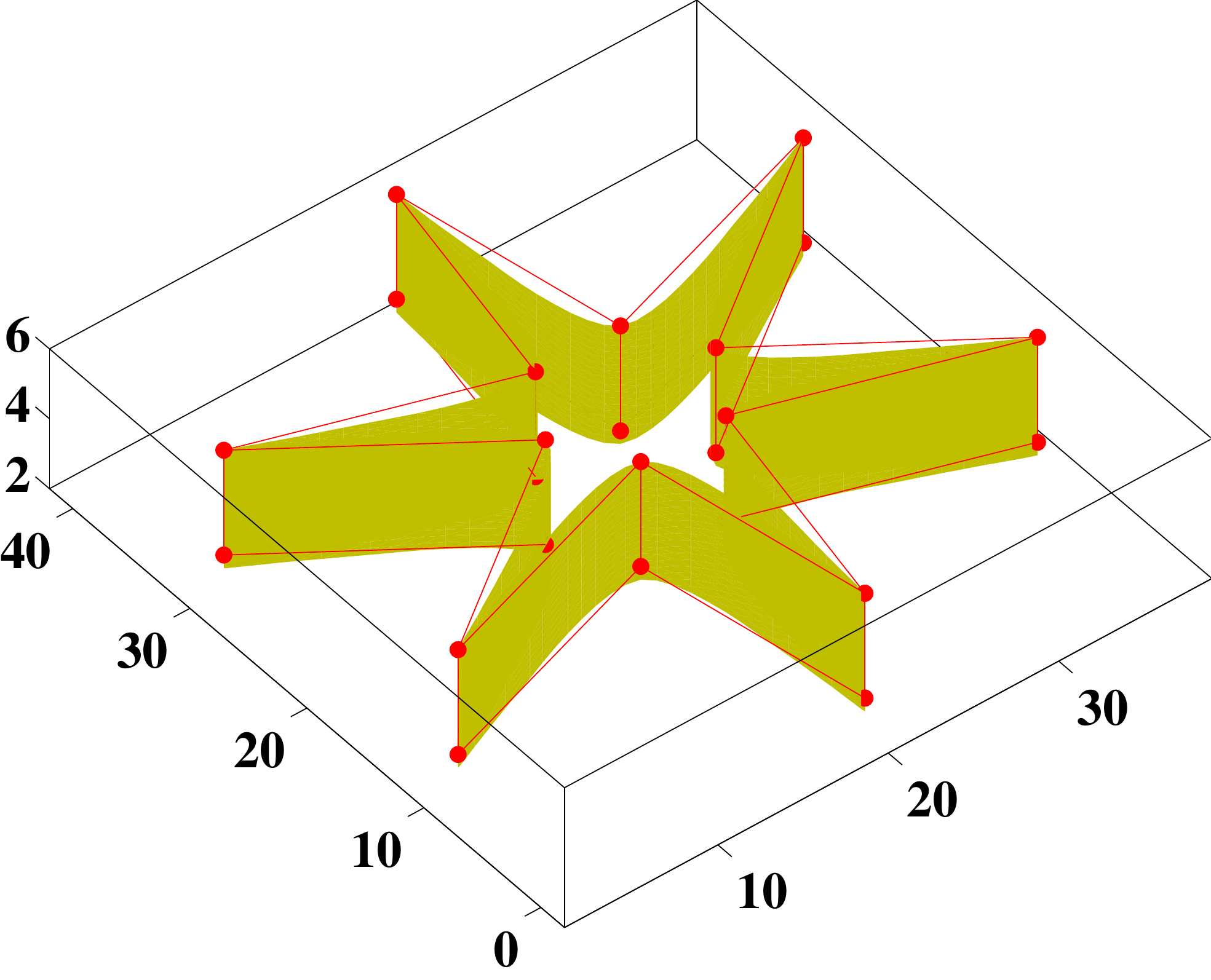, width=1.9 in}  & \epsfig{file=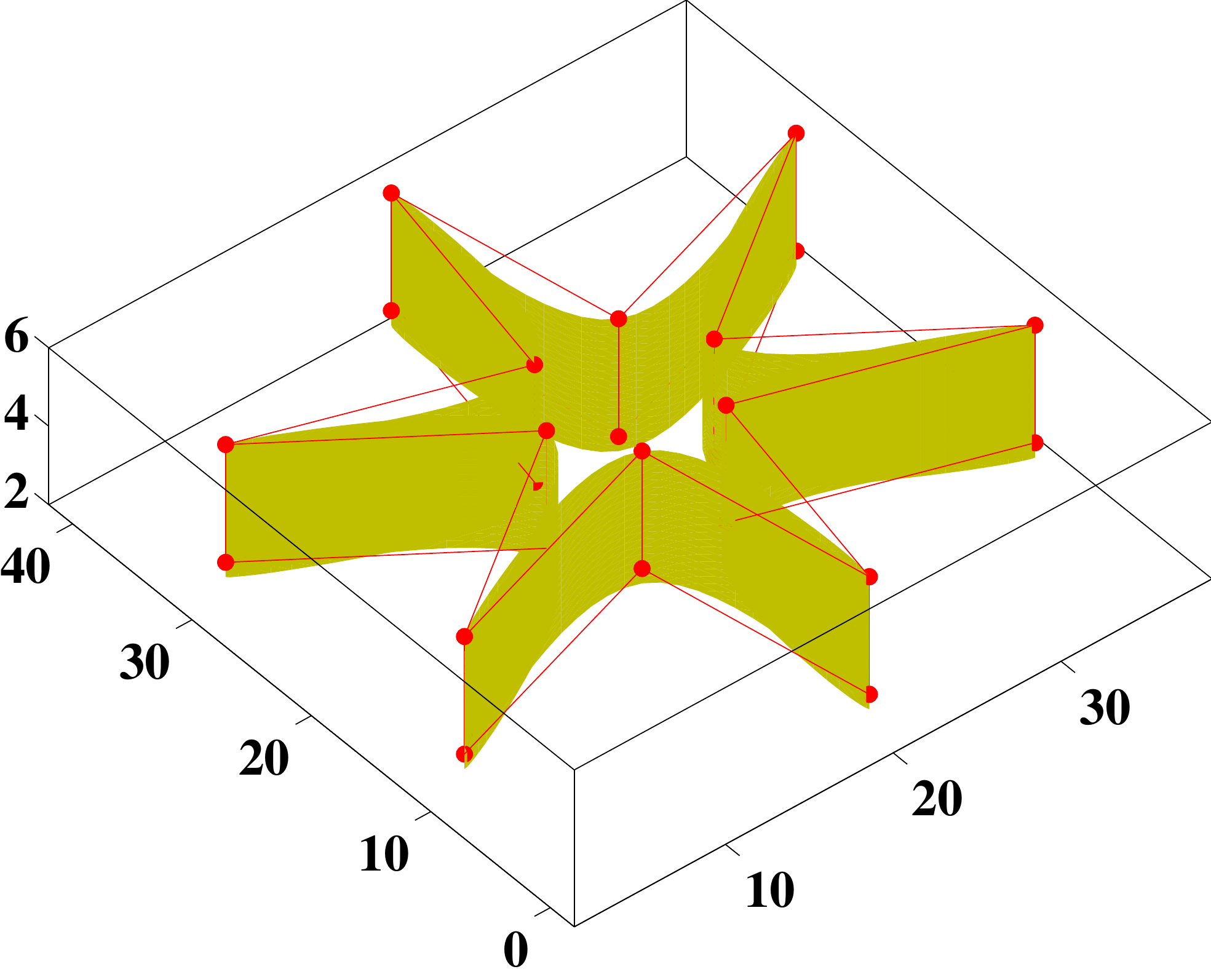, width=1.9 in}\\
(g) $2$-point scheme & (h) $4$-point scheme & (i) $6$-point scheme \\ \\
\epsfig{file=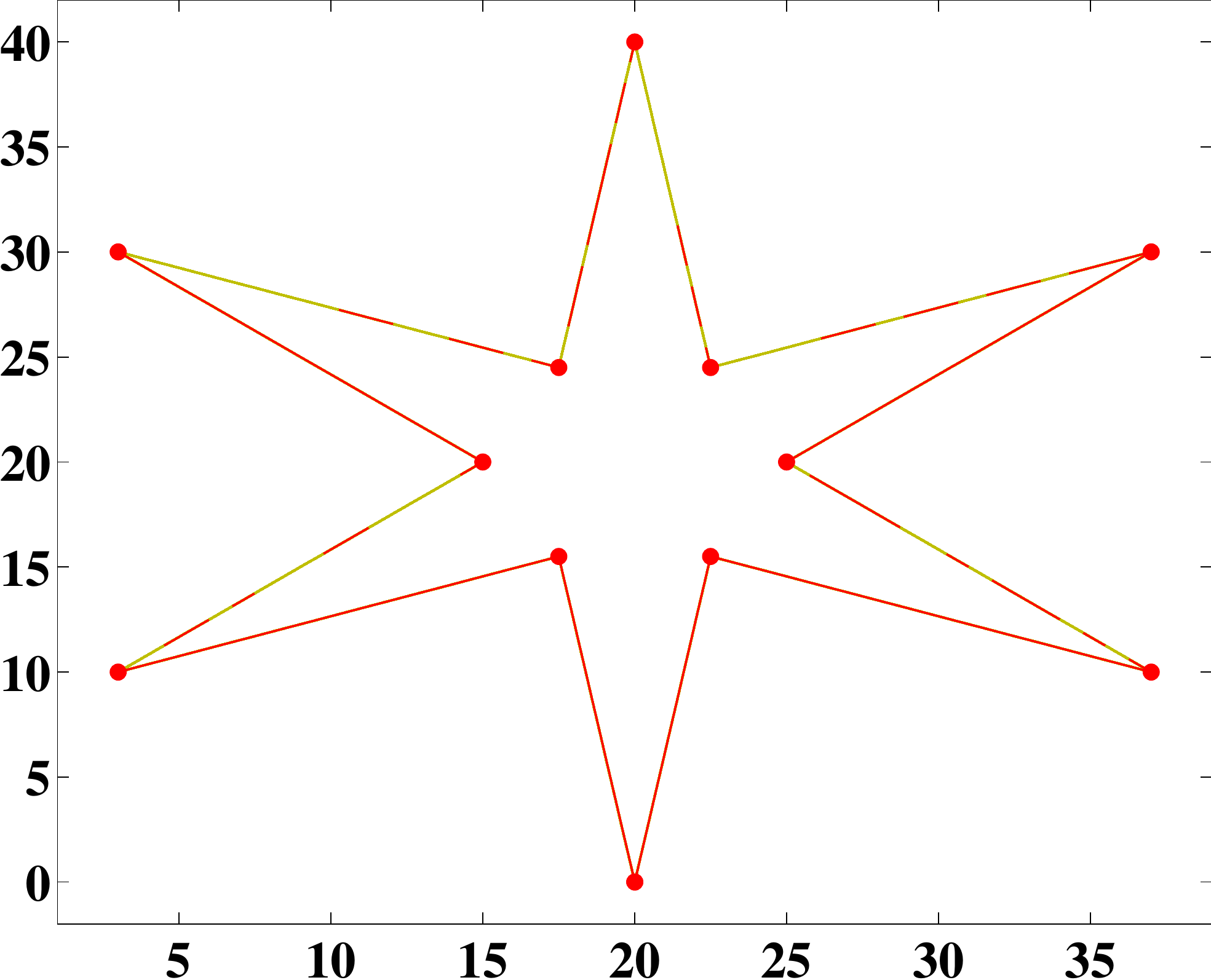, width=1.9 in} & \epsfig{file=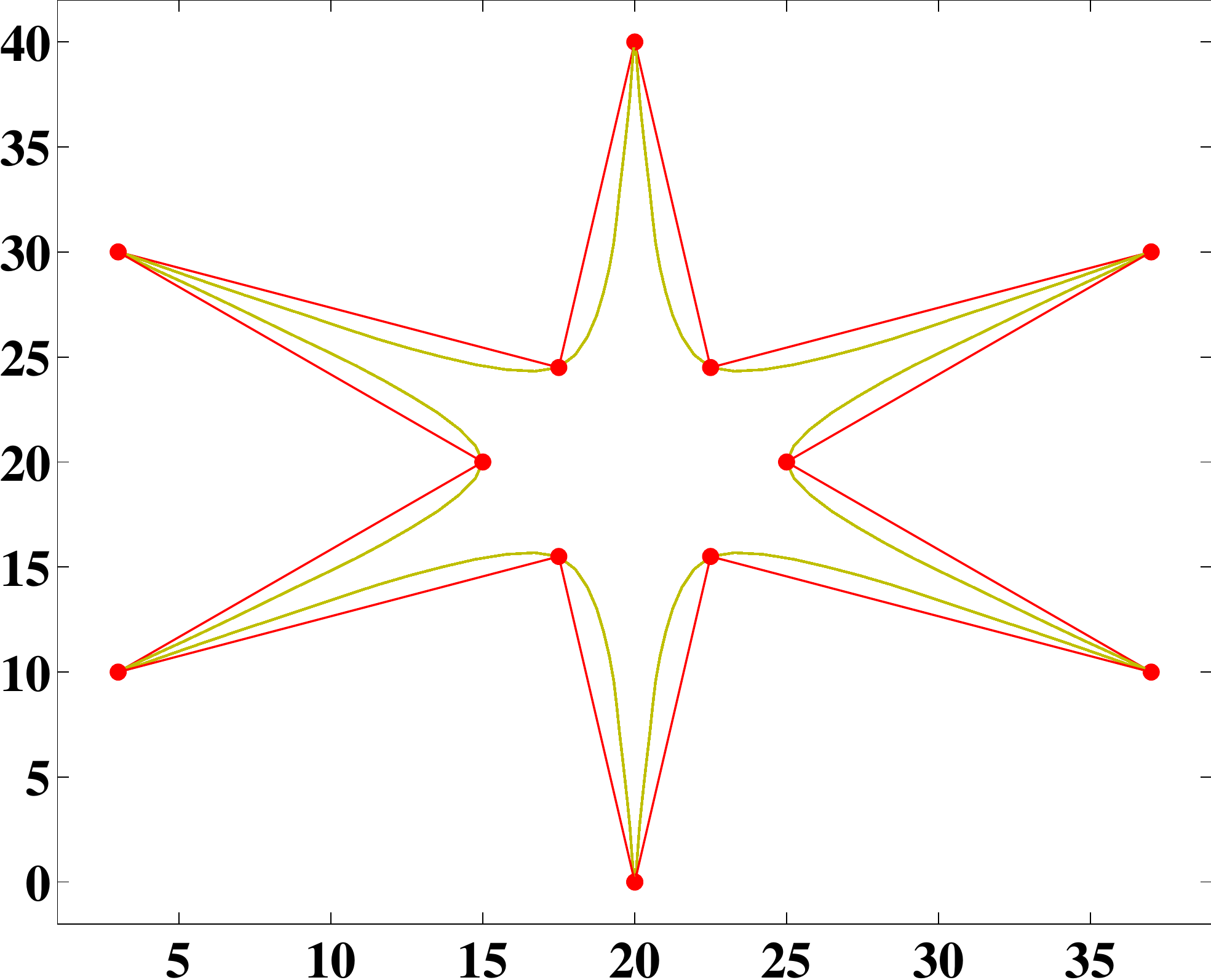, width=1.9 in}  & \epsfig{file=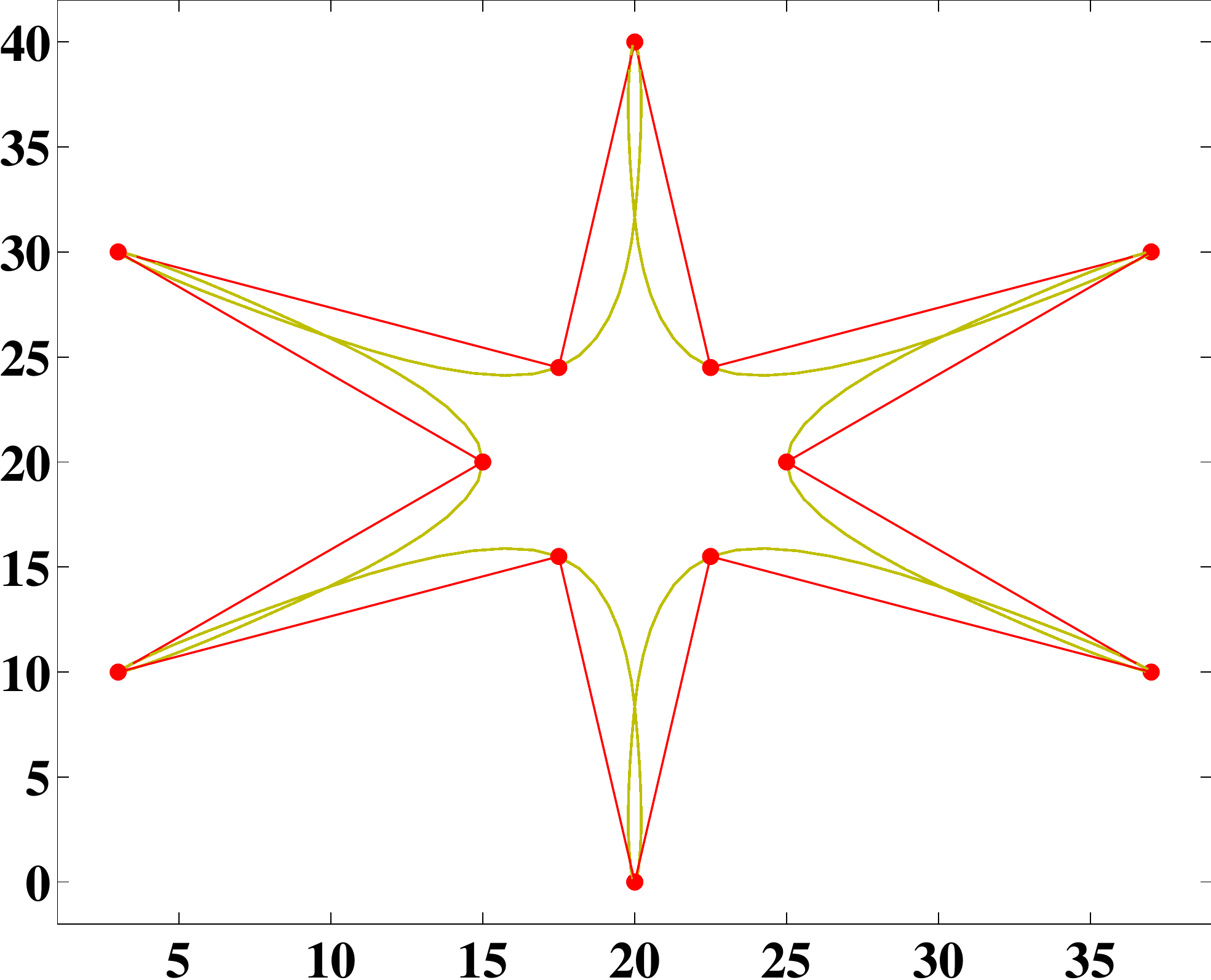, width=1.9 in}\\
(j) & (k) & (l)
 \end{tabular}
\end{center}
 \caption[Surfaces generated by the tensor product of the schemes of Deslauriers and Dubuc \cite{Dubuc}.]{\label{p5-DD-b}\emph{(a)-(c) $\&$ (g)-(i) are the surfaces generated by the tensor product schemes of Deslauriers and Dubuc \cite{Dubuc}. (d)-(f) are the mirror images of the parts inside the blue rectangles of (a)-(c) respectively, whereas (j)-(l) are the 2-dimensional images in $xy$-planes of (g)-(i) respectively.}}
\end{figure}

\section{Numerical examples and comparisons}\label{Numerical examples and comparisons}

This section deals with the numerical examples of the proposed families of subdivision schemes. We also give some numerical and mathematical companions of the proposed schemes with the existing schemes. Numerical performance of $2$-point, $3$-point, $4$-point, $5$-point, $6$-point and $7$-point relaxed subdivision schemes are shown in Figures \ref{p5-3-point}(a)-\ref{p5-3-point}(b), \ref{p5-5-point}(a)-\ref{p5-5-point}(b) and \ref{p5-7-point}(a)-\ref{p5-7-point}(b) respectively. From these figures, it is easy to see that the tension parameters which are involved in the proposed schemes allow us to draw different models from same initial points. Similarly, graphical behaviors of interpolatory $4$-point, $6$-point and $8$-point schemes for different values of tension parameters are shown in Figures \ref{p5-3-point}(c), \ref{p5-5-point}(c) and \ref{p5-7-point}(c) respectively. An interesting graphical behavior of $3$-point, $5$-point and $7$-point schemes is shown in Figures \ref{p5-3-point-1}-\ref{p5-7-point-1} respectively.

We apply these schemes to two different initial models with same values of the tension parameters. From these figures, we observe that the artifacts in the limit curve can be remove either by changing values of the tension parameters or by changing the initial polygons. While schemes proposed by Deslauriers and Dubuc \cite{Dubuc}, which are also the special cases of the proposed schemes for specific values of tension parameters $\alpha$ and $\beta$, do not have this characteristic (see Figure \ref{p5-DD-u}). Figures \ref{p5-3-point-b}, \ref{p5-5-point-b} and \ref{p5-DD-b} show that the surfaces produce by the tensor product schemes of proposed primal schemes also give better numerical results than that of the tensor product schemes of primal schemes \cite{Dubuc}. Table \ref{comparison-table-DD} gives comparisons between proposed families of primal schemes and the family of primal interpolatroy schemes of Deslauriers and Dubuc \cite{Dubuc}. The parameters used in our families of subdivision schemes not only increase the choices of drawing different shapes but also increase the polynomial reproduction, polynomial generation and smoothness of the proposed schemes. Moreover, we can draw interpolatory and approximating curves by using a single member of the family. Table \ref{comparison-table-all} gives the comparisons of the proposed primal schemes with the primal approximating schemes of \cite{Hameed, Mustafa10}. Since approximating subdivision schemes give more smoothness in limits curves as comparative to the interpolating and combined schemes, hence some of the proposed schemes give lower level of continuity than that of the schemes of \cite{Hameed, Mustafa10}. But the degrees of polynomial reproduction of the proposed schemes are higher than that of these approximating primal schemes.

\begin{table}[!h] 
\caption[Comparison of the proposed schemes with the existing schemes.]{\label{comparison-table-DD}\emph{Let NTP, RD, GD, MC, Pr, In, Cm, Re and DD-schemes be the number of tension parameter(s), Degree of polynomial reproduction, degree of polynomial generation, maximum continuity, primal, interpolatory, combined, relaxed and schemes of Deslauriers and Dubuc \cite{Dubuc} respectively. Where $N \in \mathbb{N}_{0}$ with $N < 5$. Here $\forall$ and $\forall^{'}$ stand for "for all value(s) of parameter(s)" and "not for all value(s) of parameter(s) (at specific values of parameter(s))" respectively.}}
\centering \setlength{\tabcolsep}{0.5pt}
\begin{tabular}{||c|c|c|c|c|c||}
  \hline \hline
  Family of schemes                & Type                 & NTP               & RD           & GD         & MC \\
\hline \hline
$(2N+2)$-point DD-schemes        & Pr/In                & $0$               & $2N+1$       & $2N+1$     & $C^{N}$     \\
  \hline
  $(2N+2)$-point schemes                     & Pr/Cm/Re            & $1$                & $2N+1$ $\forall$       & $2N+3$ $\forall^{'}$     & At least $C^{N+1}$    \\
  \hline
  $(2N+3)$-point schemes                     & Pr/Cm/Re            & $2$                & $2N+3$ $\forall^{'}$       & $2N+5$ $\forall^{'}$    & At least $C^{4}$ \\
  \hline
  $(2N+4)$-point schemes                     & Pr/In               & $1$                & $2N+3$ $\forall^{'}$       & $2N+3$ $\forall^{'}$     & $C^{N+1}$ \\
  \hline \hline
\end{tabular}
\end{table}

\begin{table}[!h] 
\caption[]{\label{comparison-table-all}\emph{Comparison of the proposed schemes with existing primal relaxed schemes having the same support width. Let SW, NTP, GD, GD1, RD and RD1 be the support width, number of tension parameter(s), degree of polynomial generation at all value(s) of tension parameter(s), degree of polynomial generation at specific value(s) of tension parameter(s), degree of polynomial reproduction for all value(s) of tension parameter(s) and degree of polynomial reproduction at specific value(s) of tension parameter(s).}}
\centering \setlength{\tabcolsep}{1.2pt}
\begin{tabular}{||c|c|c|c|c|c|c|c|c||}
\hline
\hline
Schemes & SW & NTP & Type & GD & GD1 & RD & RD1 & Continuity \\
\hline
\hline
Scheme $S_{a_{3}}$ & 6 & 2 & combined & 1 & 5 & 1 & 3 & $C^{4}$\\
\hline
Scheme \cite{Mustafa10} & 6 & 1 & approximating & 1 & 5 & 1 & 1 & $C^{1}$\\
\hline \hline
Scheme $S_{a_{4}}$ & 8 & 1 & combined & 3 & 5 & 3 & 3 & $C^{2}$ \\
\hline
Scheme \cite{Mustafa10} & 8 & 1 & approximating &3 & 7 & 1 & 1 & $C^{3}$\\
\hline
Scheme \cite{Hameed} & 8 & 1 & approximating & 1 & 3 & 1 & 1 & $C^{1}$\\
\hline \hline
Scheme $S_{a_{5}}$ & 10 & 2 & combined & 3 & 7 & 3 & 5 & $C^{4}$ \\
\hline
Scheme \cite{Mustafa10} & 10 & 1 & approximating & 5 & 9 & 1 & 1 & $C^{5}$\\
\hline \hline
Scheme $S_{a_{6}}$ & 12 & 1 & combined & 5 & 7 & 5 & 5 & $C^{3}$ \\
\hline
Scheme \cite{Mustafa10} & 12 & 1 & approximating & 7 & 11 & 1 & 1 & $C^{7}$\\
\hline \hline
Scheme $S_{a_{7}}$ & 14 & 2 & combined & 5 & 9 & 5 & 7 & $C^{4}$ \\
\hline
Scheme \cite{Hameed} & 14 & 1 & approximating & 3 & 5 & 1 & 3 & $C^{3}$\\
\hline
\hline
\end{tabular}
\end{table}

\begin{figure}[!h] 
\begin{center}
\begin{tabular}{cccc}
& & \\
& & \\
& & \\
\epsfig{file=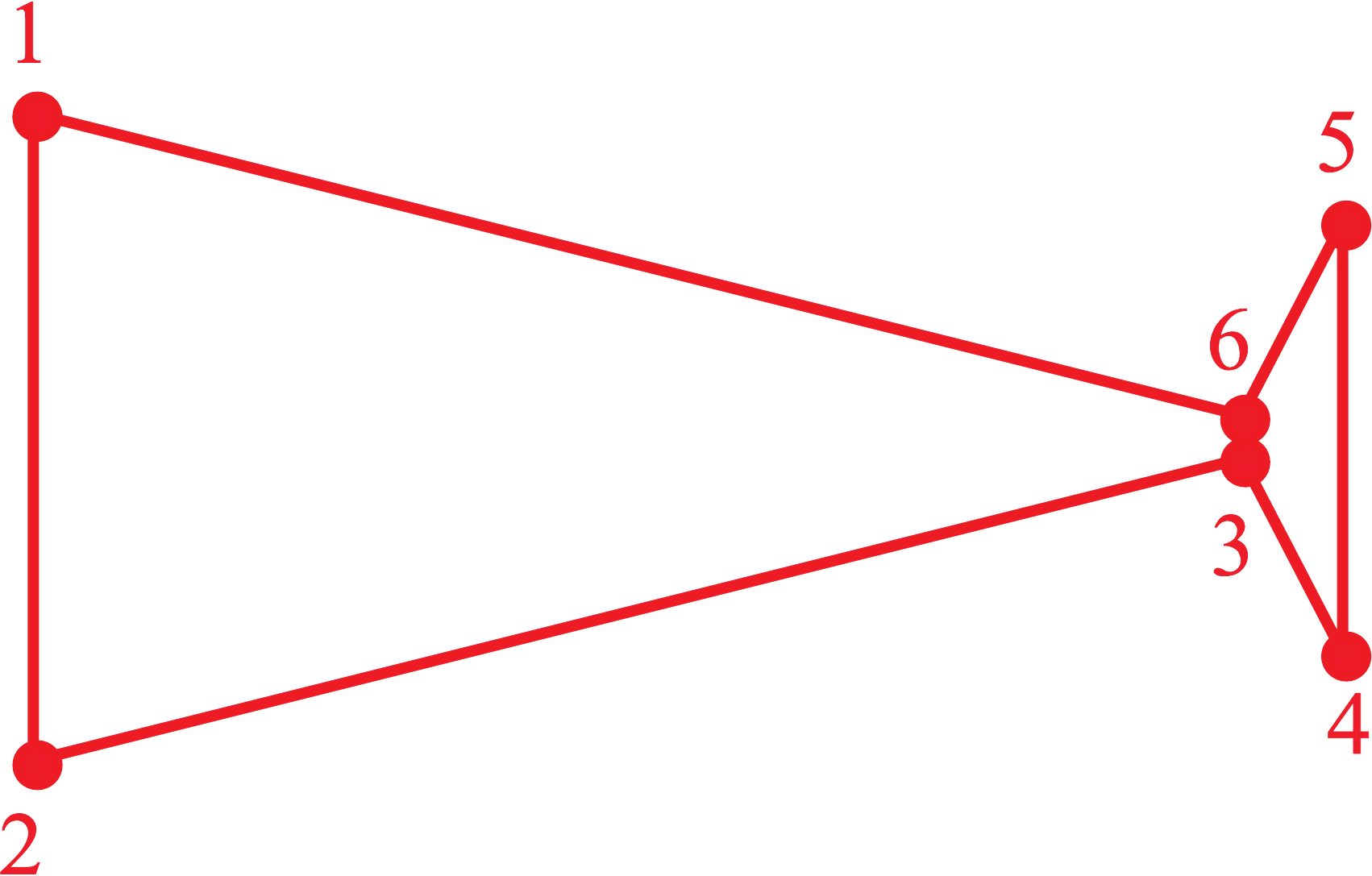, width=1.7 in} & \epsfig{file=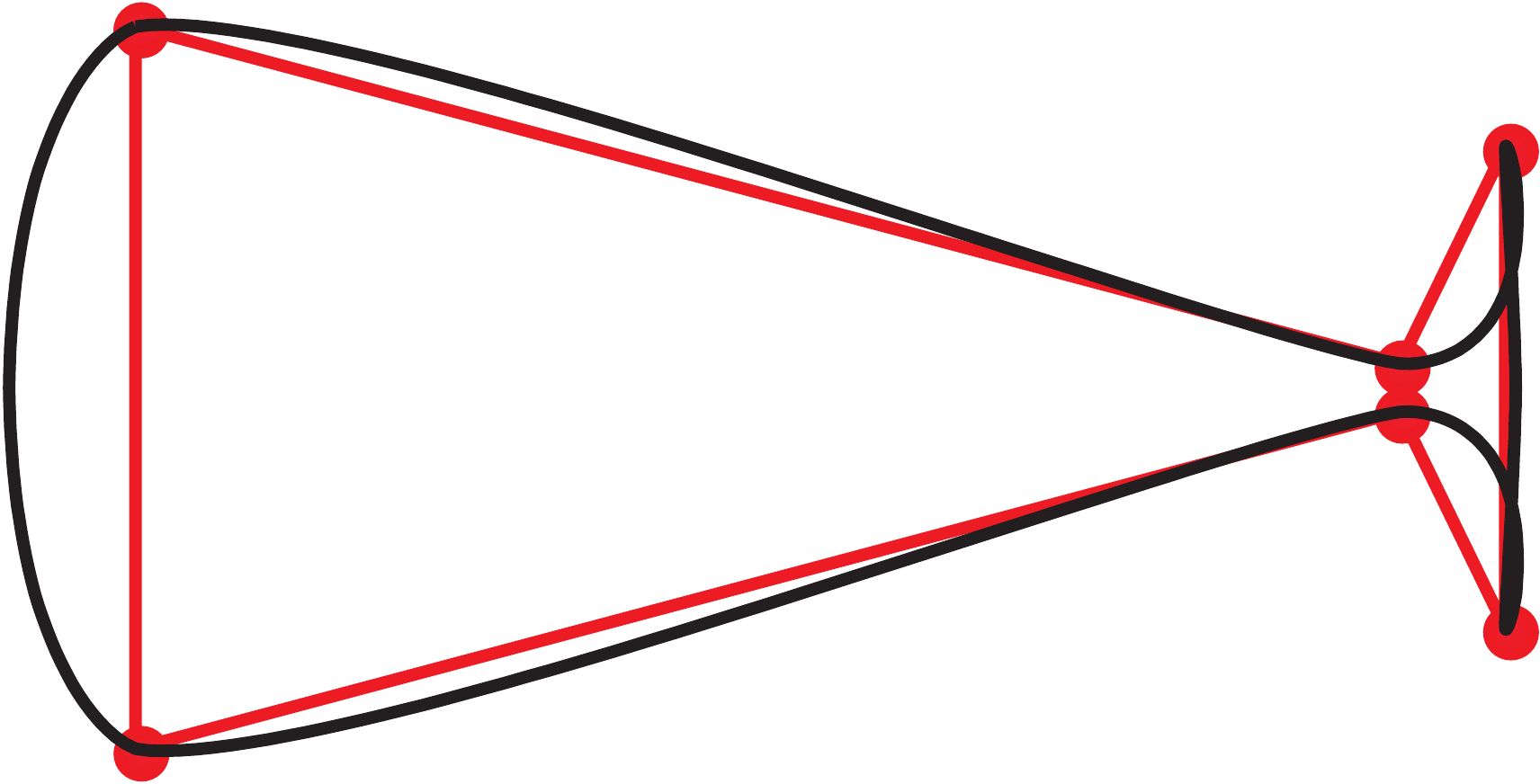, width=1.7 in}  & \epsfig{file=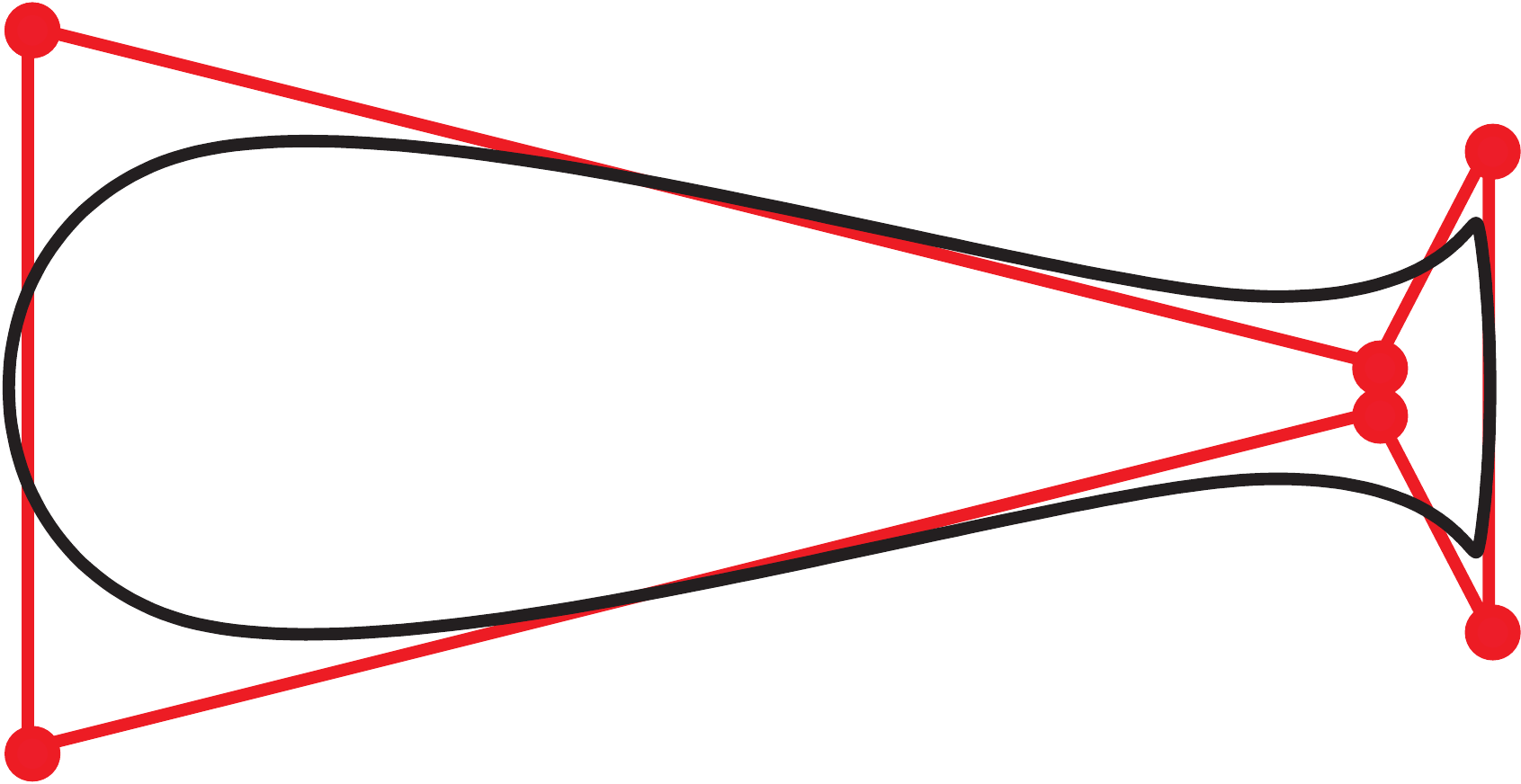, width=1.7 in}\\
\begin{scriptsize}(a) Initial control points with index\end{scriptsize} & \begin{scriptsize}(b) Interpolating all control points\end{scriptsize} & \begin{scriptsize}(c) Approximating all control points\end{scriptsize} \\
& \begin{scriptsize}$(\alpha,\beta)=(0,-0.05)$\end{scriptsize} & \begin{scriptsize}$(\alpha,\beta)=(\frac{1}{12},-\frac{1}{44})$ \end{scriptsize} \\
\epsfig{file=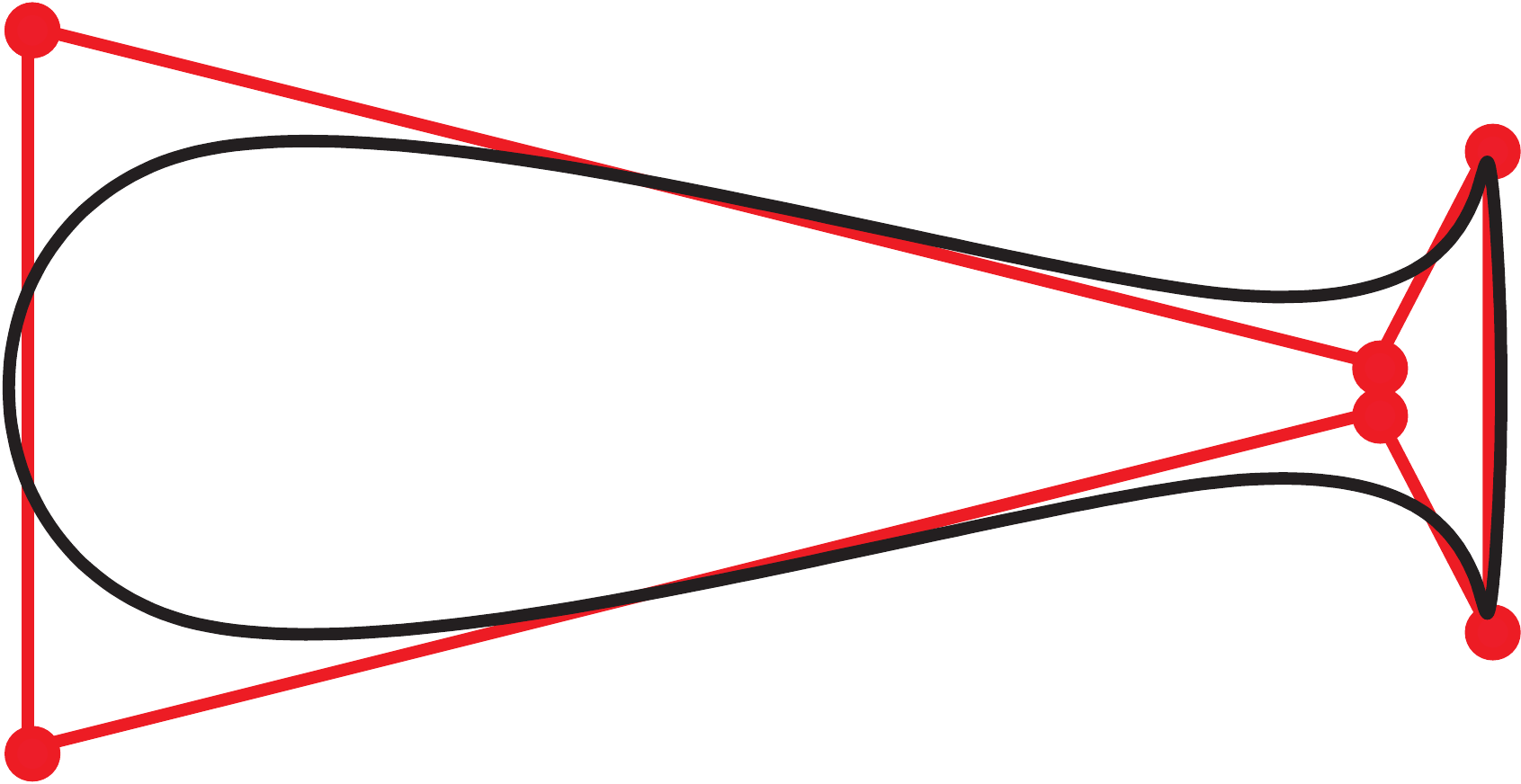,  width=1.7 in} & \epsfig{file=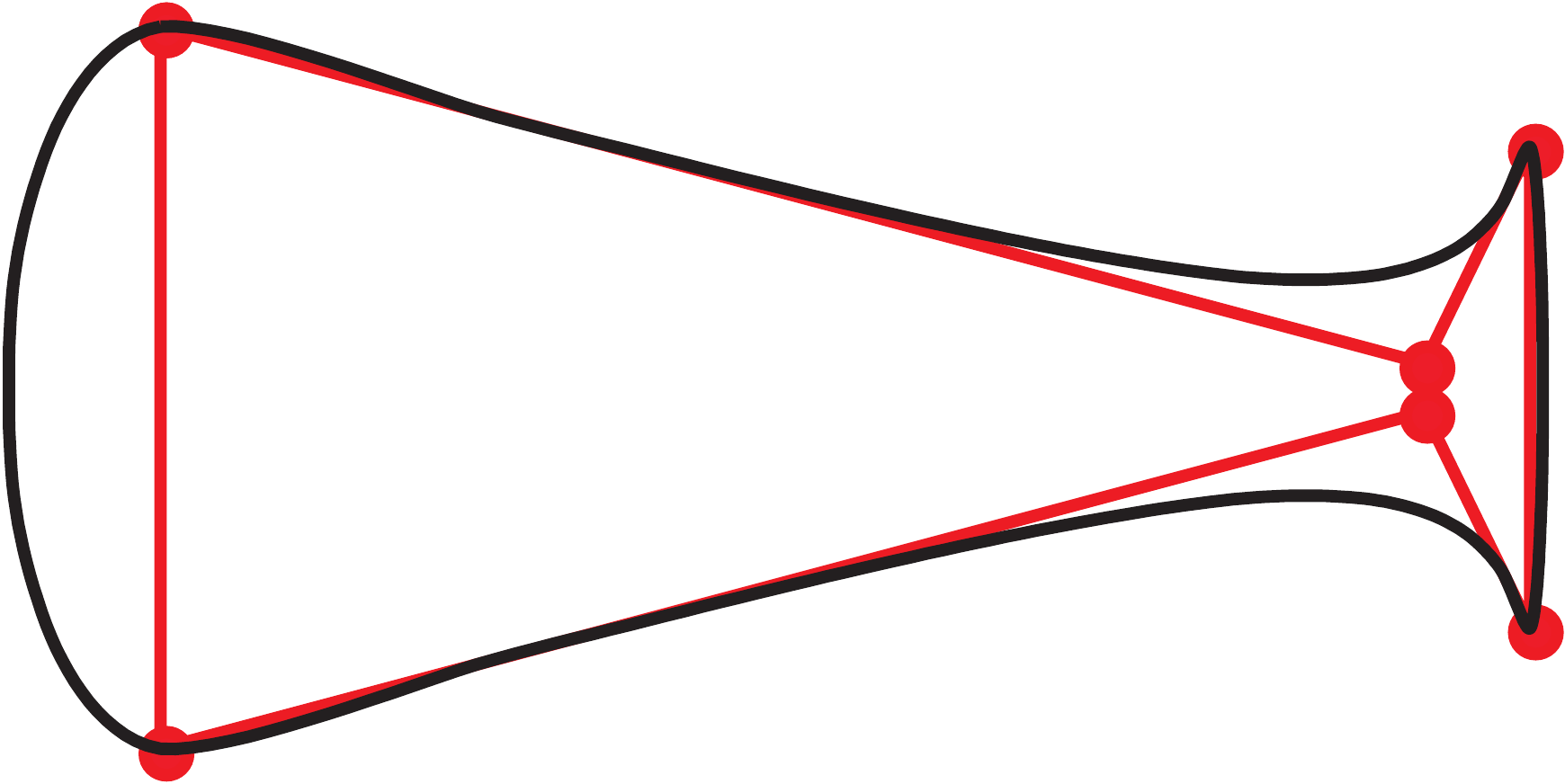, width=1.7 in}  & \epsfig{file=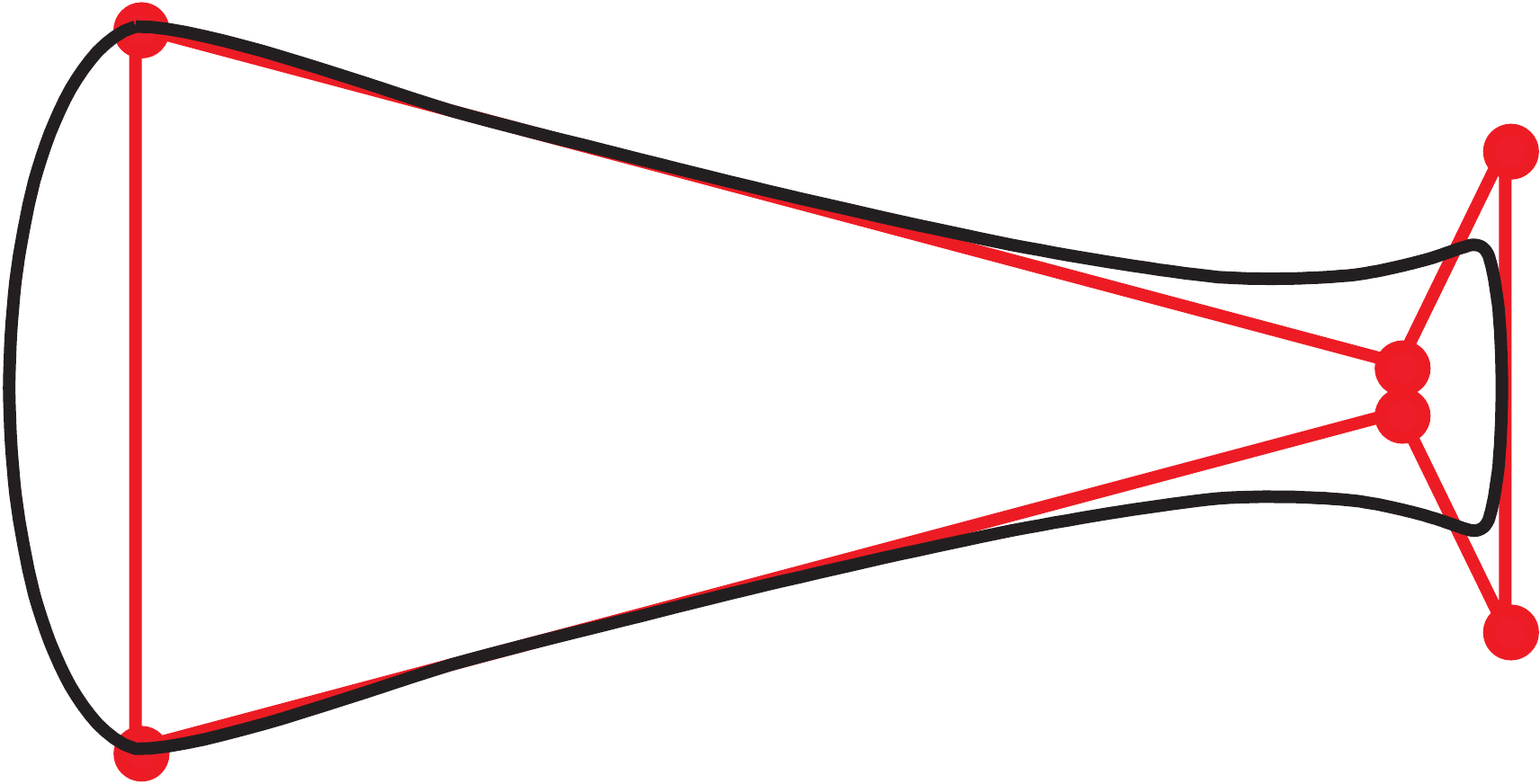, width=1.7 in}\\
\begin{scriptsize}(d) $(\alpha_{i},\beta_{i})=\left[(\frac{1}{12},-\frac{1}{44}),(\frac{1}{12},-\frac{1}{44}),(\frac{1}{12},\right.$\end{scriptsize}& \begin{scriptsize}(e) $(\alpha_{i},\beta_{i})=\left[(0,-\frac{3}{50}),(0,-\frac{3}{50})\right.$\end{scriptsize}& \begin{scriptsize}(f) $(\alpha_{i},\beta_{i})=\left[(0,-\frac{1}{20}),(0,-\frac{1}{20})\right.$\end{scriptsize}\\
\begin{scriptsize} \,\ $\left.-\frac{1}{44}),(0,-\frac{3}{50}),(0,-\frac{3}{50}),(\frac{1}{12},-\frac{1}{44})\right]$\end{scriptsize} & \begin{scriptsize} \,\ $\left.,(\frac{1}{8},0),(0,-\frac{3}{50}),(0,-\frac{3}{50}),(\frac{1}{8},0)\right]$\end{scriptsize}& \begin{scriptsize} \,\ $\left.,(\frac{1}{8},0),(\frac{1}{8},0),(\frac{1}{8},0),(\frac{1}{8},0)\right]$ \end{scriptsize} \\
\begin{scriptsize} at first subdivision level \end{scriptsize} & \begin{scriptsize} at first subdivision level \end{scriptsize} & \begin{scriptsize} at first subdivision level \end{scriptsize}
 \end{tabular}
\end{center}
\caption[Interproximate settings of scheme $S_{a_{3}}$.]{\label{interproximate-1-fig}\emph{The effect of local interpolation by our combined subdivision scheme $S_{a_{3}}$ with different $\alpha_{i}$ and $\beta_{i}$ where $i \in \mathbb{R}$ with $0< i < 7$.}}
\end{figure}
\begin{figure}[!h] 
\begin{center}
\begin{tabular}{cccc}
\epsfig{file=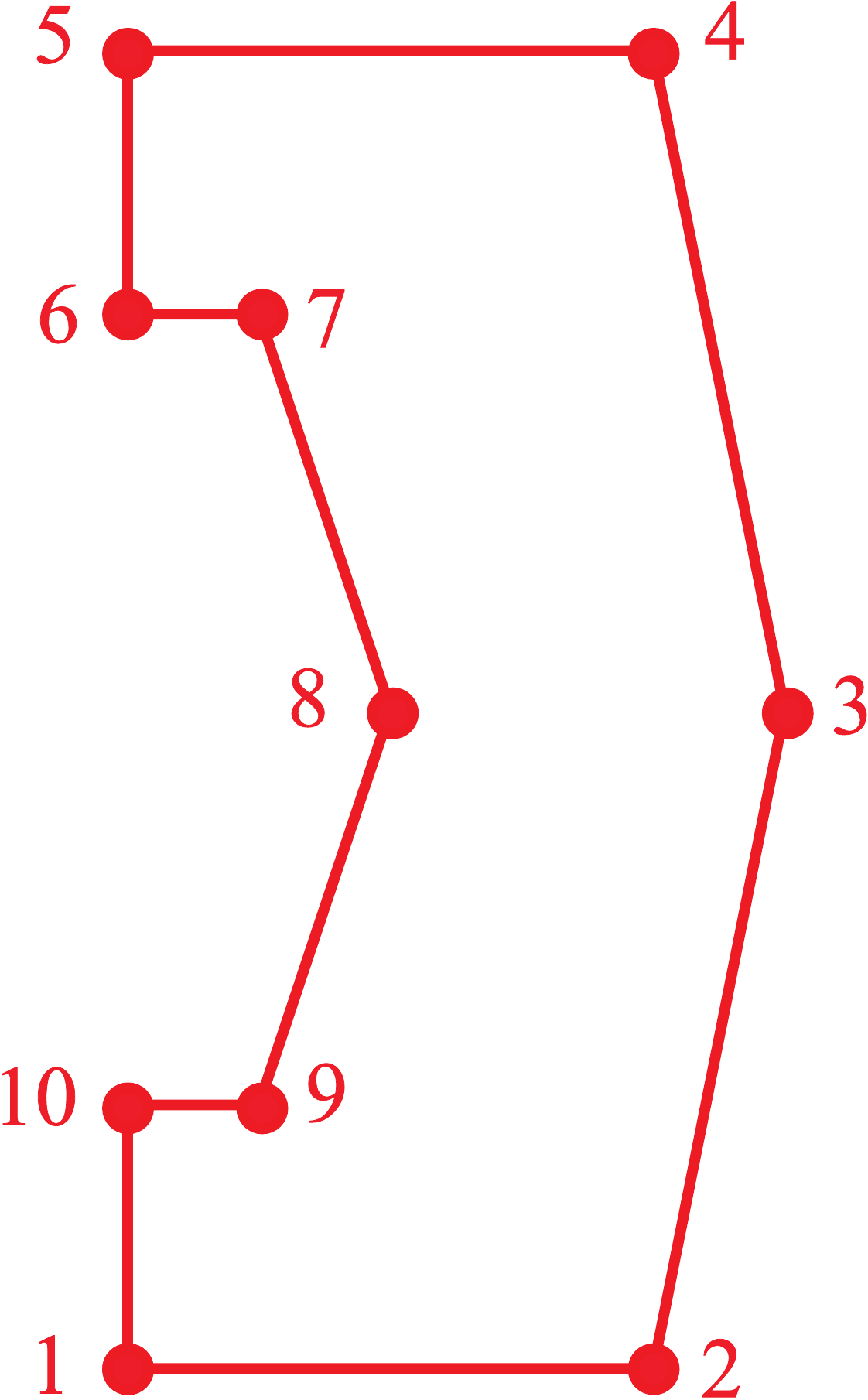, width=1.55 in} \qquad \qquad& \epsfig{file=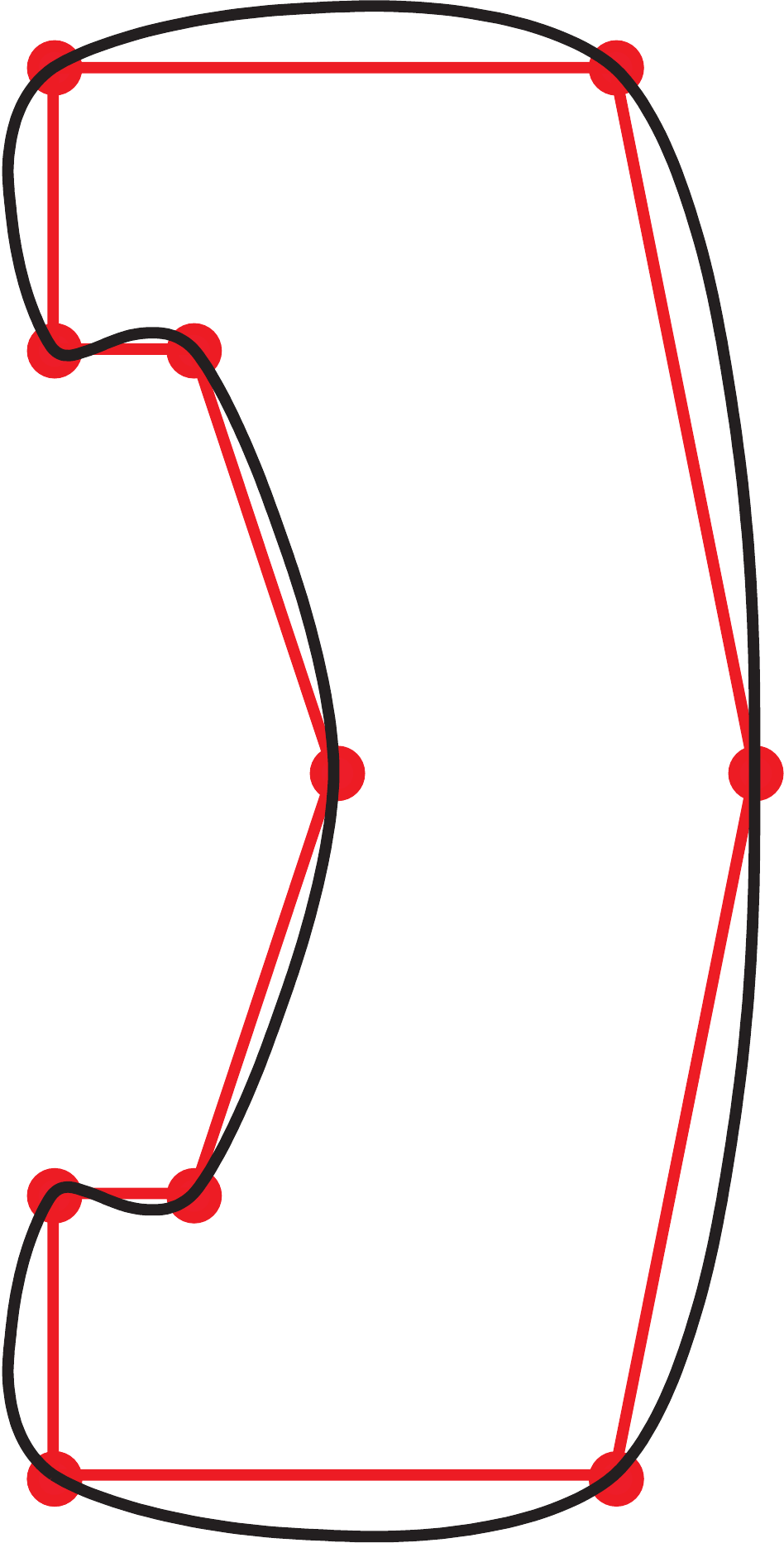, width=1.3 in} \qquad \qquad  & \epsfig{file=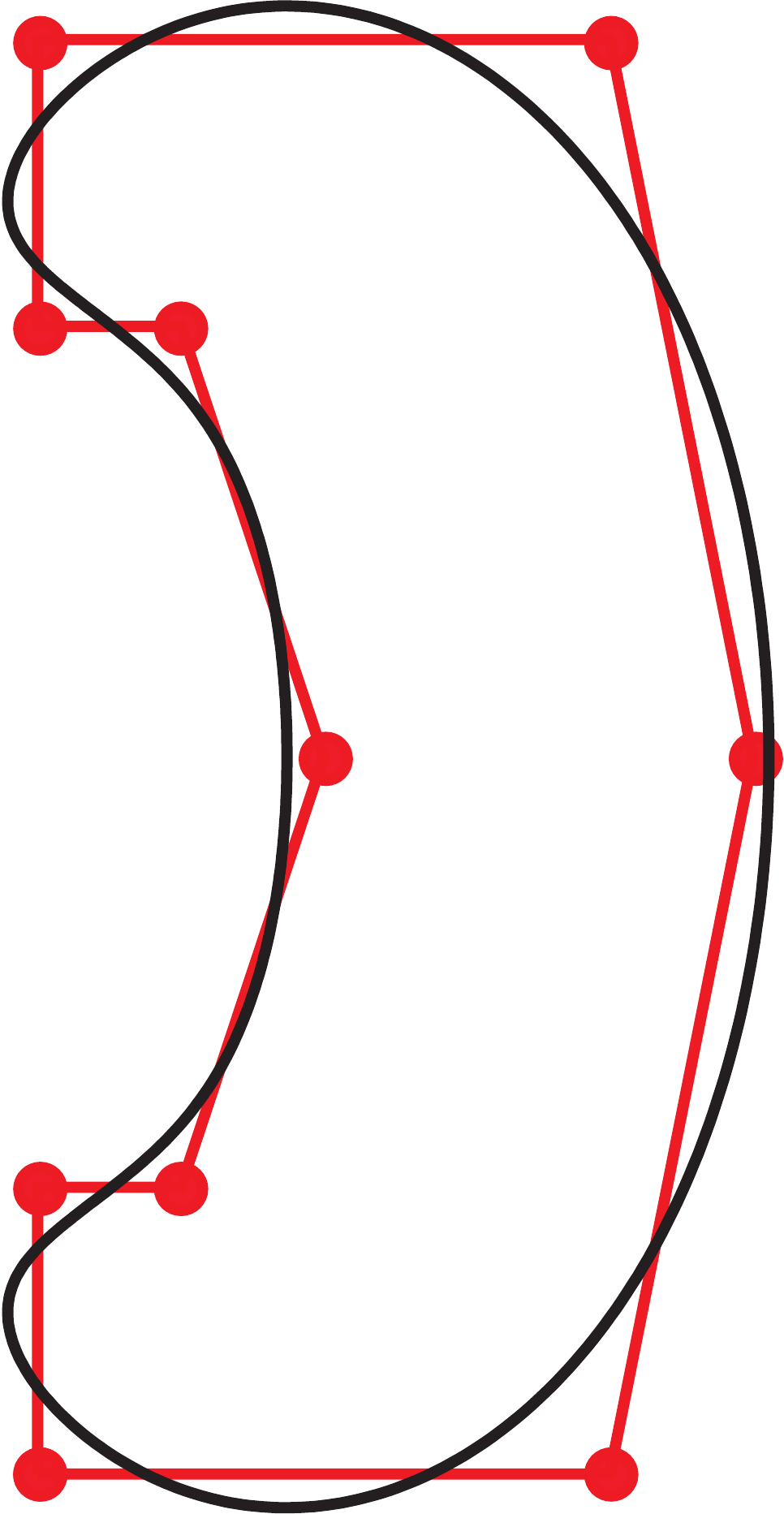, width=1.3 in}\\
(a)  \qquad \qquad& (b)  \qquad \qquad& (c) \\
\epsfig{file=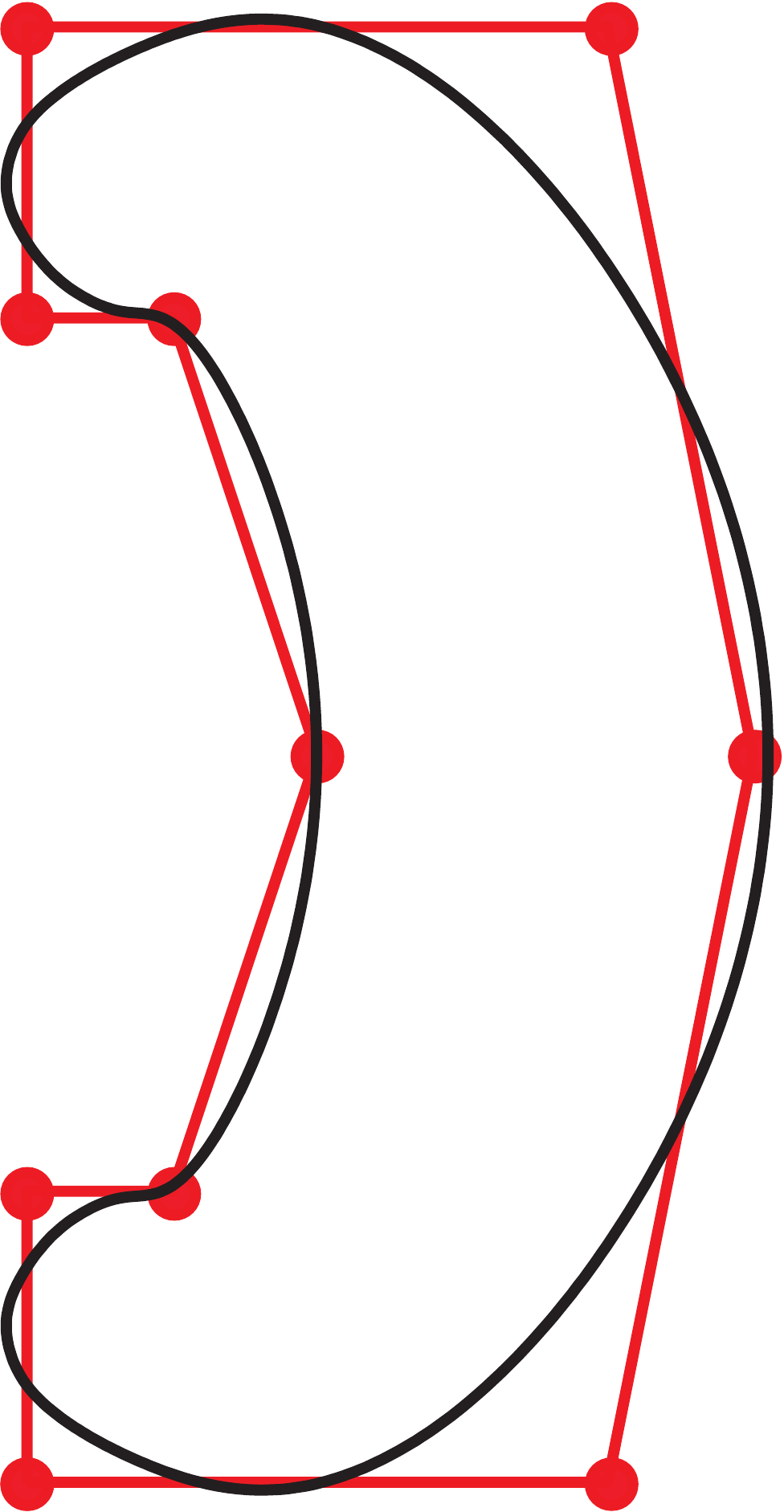, width=1.3 in} \qquad \qquad  & \epsfig{file=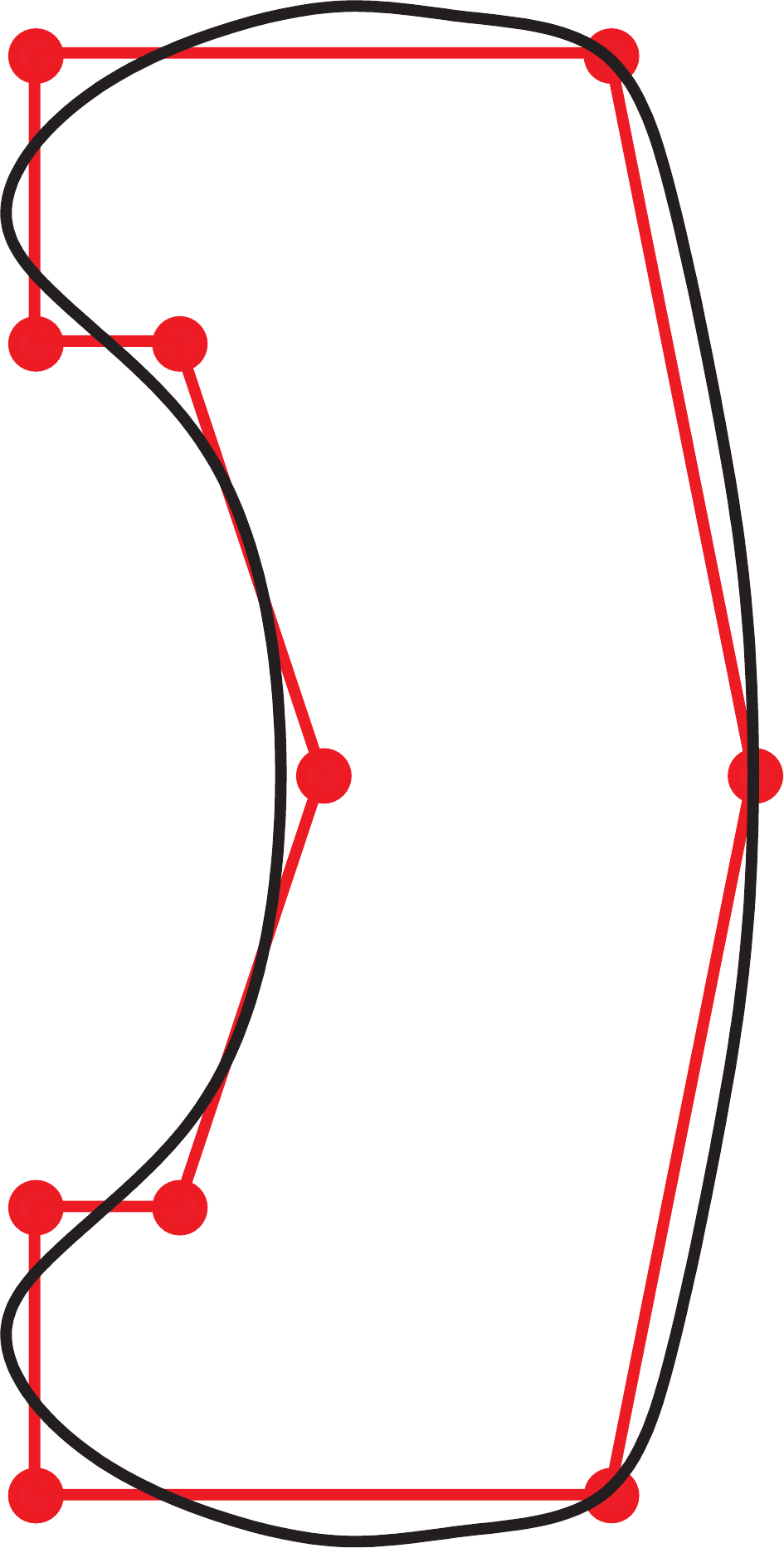, width=1.3 in} \qquad \qquad   & \epsfig{file=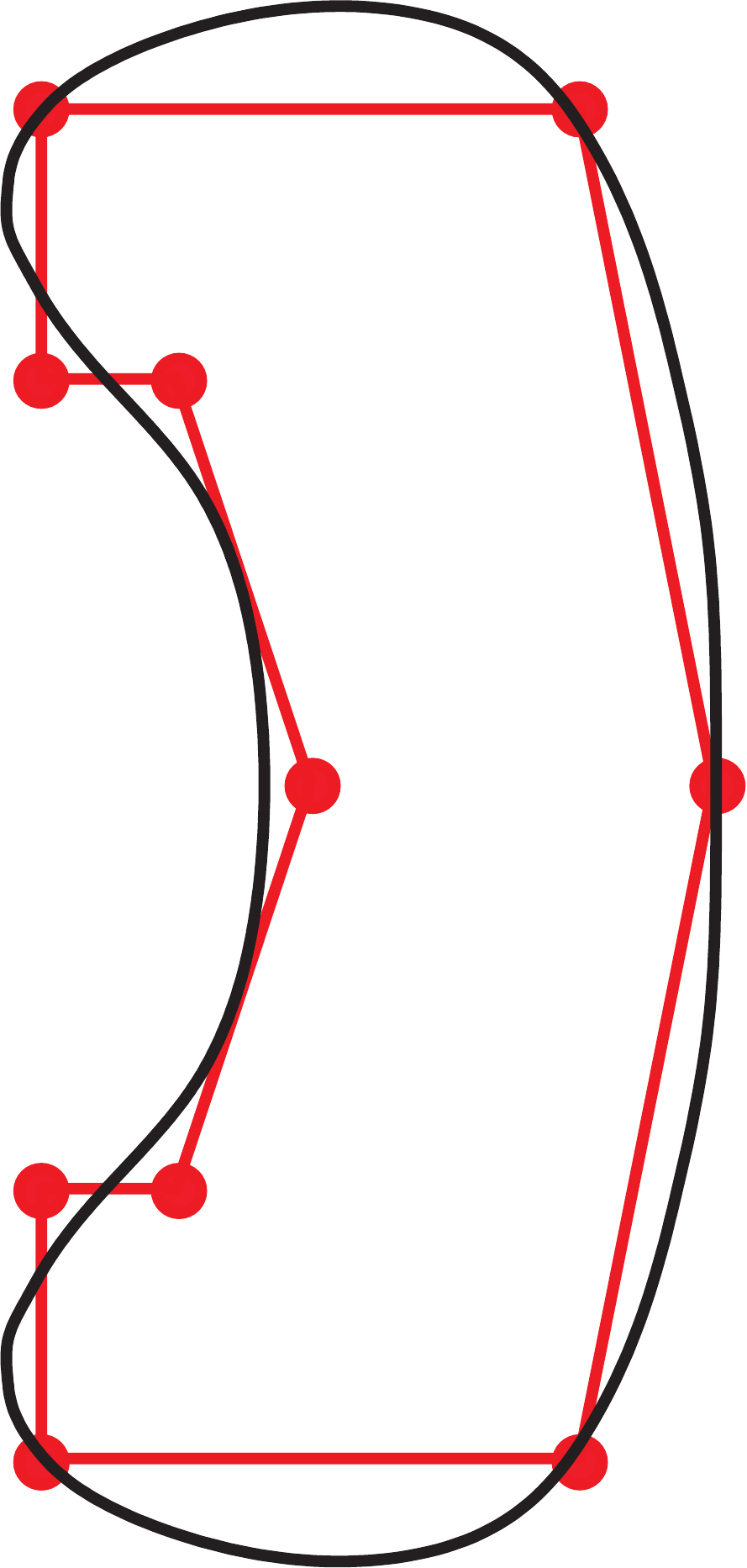, width=1.3 in}\\
 (d) \qquad \qquad  & (e) \qquad \qquad  & (f)
 \end{tabular}
\end{center}
\caption[Interproximate settings of scheme $S_{a_{5}}$.]{\label{interproximate-2-fig}\emph{The effect of local interpolation by our combined subdivision scheme $S_{a_{5}}$ with different $\alpha_{i}$ and $\beta_{i}$.}}
\end{figure}
\begin{figure}[!h] 
\begin{center}
\begin{tabular}{cccc}
\epsfig{file=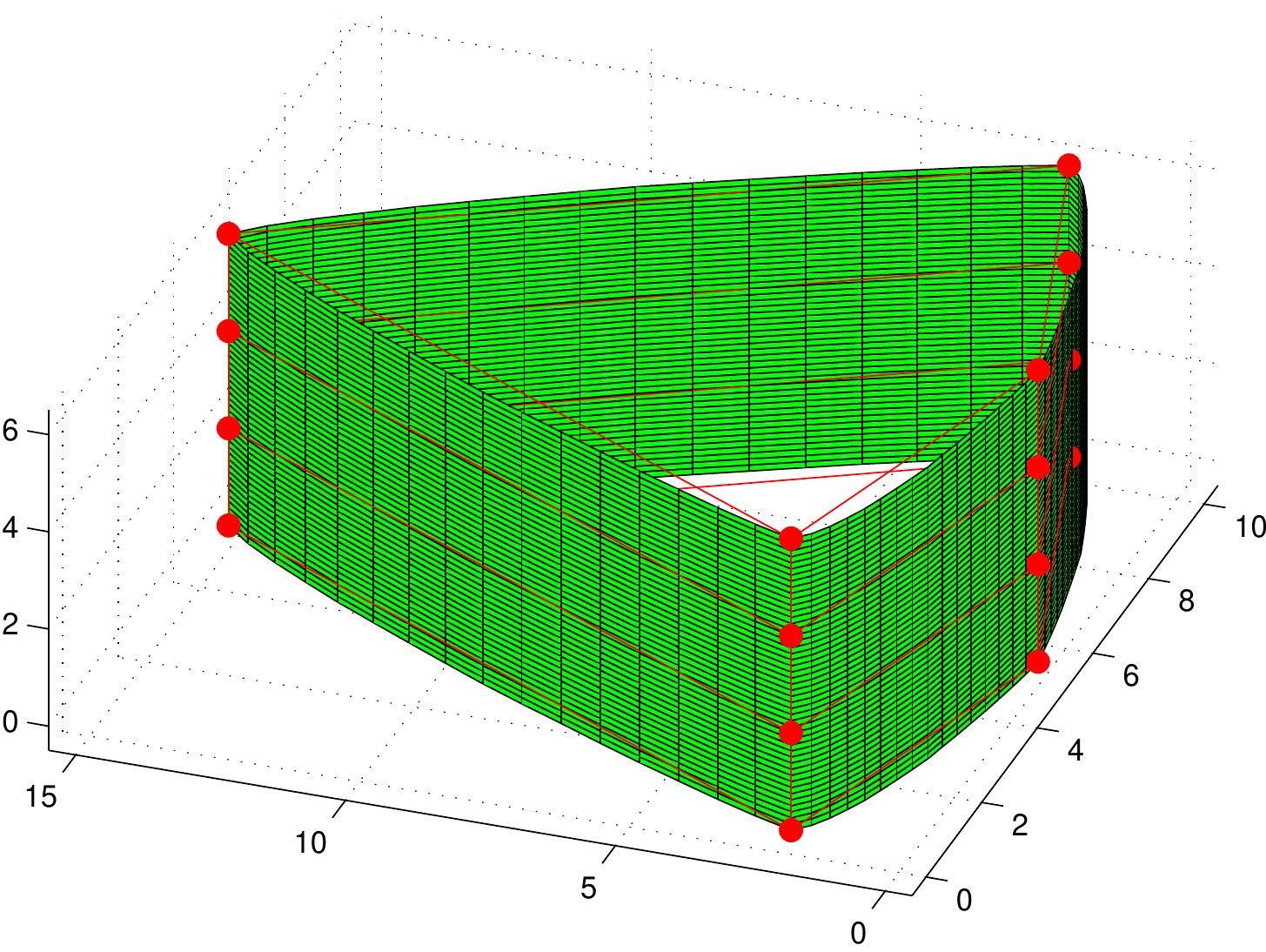, width=1.9 in} & \epsfig{file=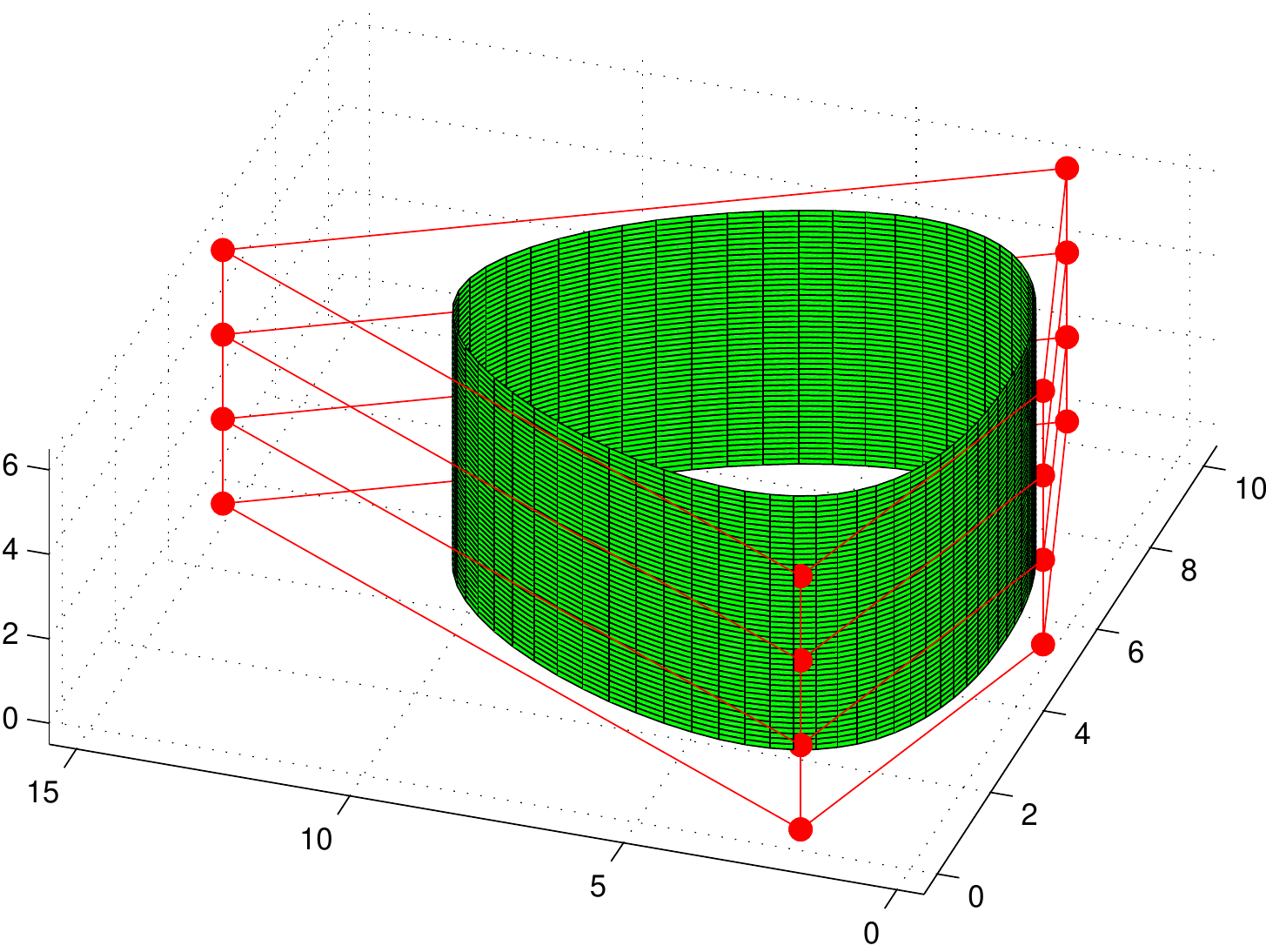, width=1.9 in}  & \epsfig{file=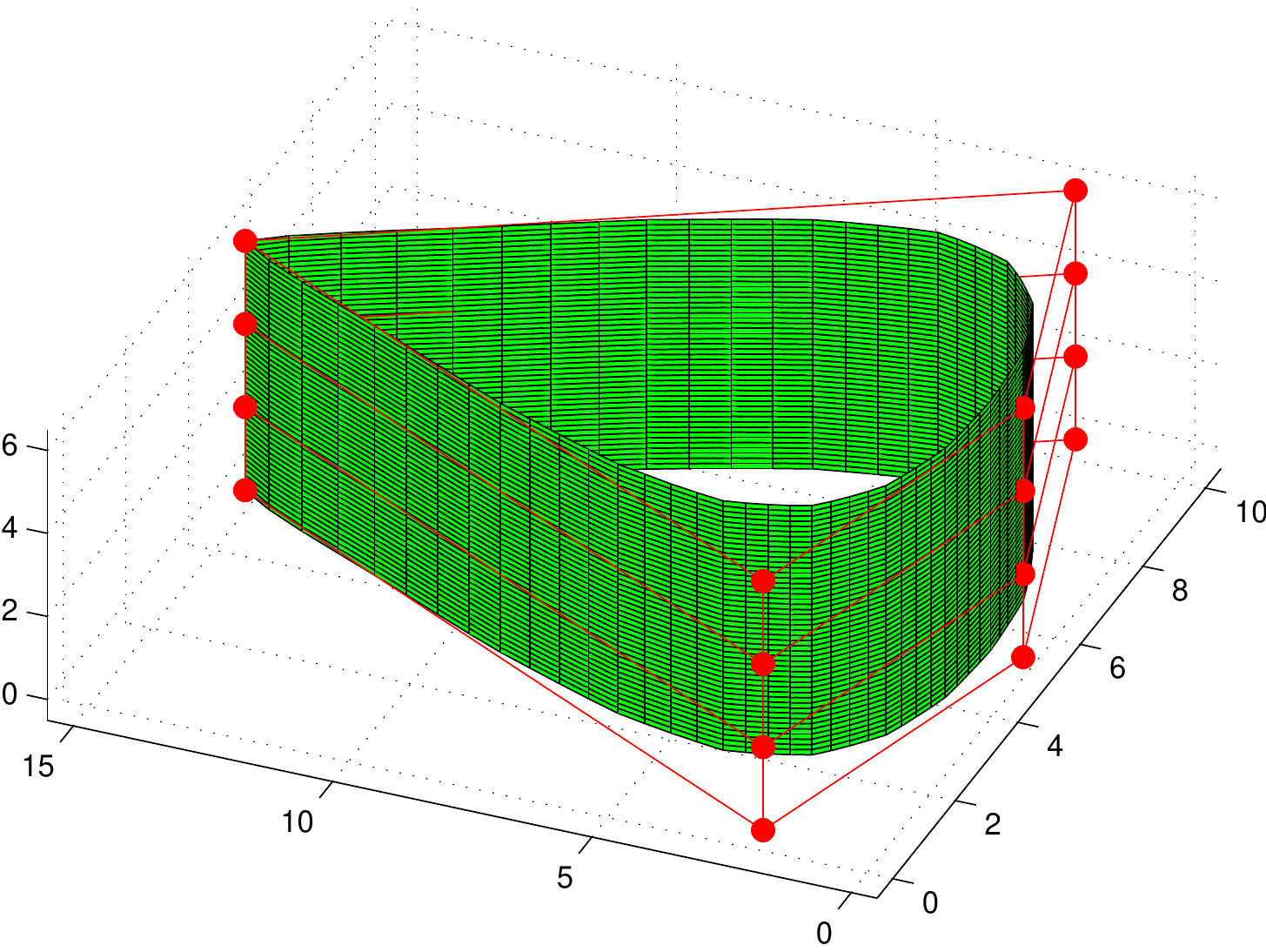, width=1.9 in}\\
(a)  & (b)  & (c)
 \end{tabular}
\end{center}
 \caption[Interproximate settings for the tensor product of scheme $S_{a_{3}}$.]{\label{interproximate-b-fig}\emph{The effect of local interpolation by tensor product of our combined subdivision scheme $S_{a_{3}}$ with different $\alpha_{i}$ and $\beta_{i}$.}}
\end{figure}

\section{Interproximate subdivision schemes}\label{Interproximate subdivision schemes}
In this section, we present a new family of subdivision schemes, that is the family of interproximate subdivision schemes, for generating curves that interpolate certain given initial control points and approximate the other initial control points. By the interproximate subdivision scheme, only the initial control points specified to be interpolated are fixed and the other points are updated at each refinement step.
The interproximate subdivision schemes can be defined by replacing $\alpha_{i}$ and $\beta_{i}$ as the substitution of $\alpha$ and $\beta$ in (\ref{schm2}) and (\ref{b2j+1}). In this interproximate subdivision process, the parameters $\alpha_{i}$ control the interpolating property of the subdivision schemes and parameters $\beta_{i}$ control the approximating property of the subdivision schemes. Figure \ref{interproximate-1-fig} shows the limit curves generated by the subdivision scheme $S_{a_{3}}$ using initial control points $(1,5),(1,2),(13,3.4),(14,2.5),(14,4.5),(13,3.6)$ and $(1,5)$. Figure \ref{interproximate-1-fig}(d)-\ref{interproximate-1-fig}(f) demonstrates that this scheme interpolates the control points in a local manner and uses a different value of tension parameter for each edge of the control polygon. Values of the tension parameters at first subdivision levels are shown in these figures. Whereas at the other subdivision levels we use same interpolating values for the control points which are interpolated at the first subdivision levels and use the approximating values of the tension parameters for the modified and new inserted points. Figure \ref{interproximate-2-fig} shows the limit curves generated by subdivision scheme $S_{a_{5}}$ using initial control points $(0,0),$ $(4,0),$ $(5,5),$ $(4,10),$ $(0,10),$ $(0,8),$ $(1,8),$ $(2,5),$ $(1,2)$ and $(0,2)$. In this figure:
\begin{itemize}
  \item (a) represents the initial polygon with indexed initial control points.
  \item (b) shows the limit curve that interpolates all the control points with $(\alpha,\beta)=(0,-0.001)$.
  \item (c) represents the limit curve that approximates all the initial control points with $(\alpha,\beta)=(\frac{1}{16},-\frac{1}{48})$.
  \item (d) shows the interproximate limit curve with $(\alpha_{i},\beta_{i})=[(\frac{1}{10},-\frac{49}{1152}),$ $(\frac{1}{10},$ $-\frac{49}{1152}),$ $(\frac{1}{10},$ $-\frac{49}{1152}),$ $(\frac{1}{10},$ $-\frac{49}{1152}),$ $(\frac{1}{10},-\frac{49}{1152}),$ $(\frac{1}{10},$ $-\frac{49}{1152}),$ $(0,\frac{1}{64}),$ $(0,\frac{1}{64}),$ $(0,\frac{1}{64}),$ $(\frac{1}{10},$ $-\frac{49}{1152})]$ at first subdivision level. Whereas at other subdivision levels, we use $(\alpha_{i},\beta_{i})=(0,\frac{1}{64})$ for the points $(1,8),$ $(2,5)$ $\&$ $(1,2)$ and $(\alpha_{i},\beta_{i})=(\frac{1}{10},-\frac{49}{1152})$ for all the other points.
  \item (e) shows the interproximate limit curve with $(\alpha_{i},\beta_{i})=[(\frac{1}{14},$ $-\frac{43}{1664}),$ $(0,$ $-\frac{2}{125}),$ $(0,-\frac{2}{125}),$ $(0,-\frac{2}{125}),$ $(\frac{1}{14},-\frac{43}{1664}),$ $(\frac{1}{14},-\frac{43}{1664}),$ $(\frac{1}{14},-\frac{43}{1664}),$ $(\frac{1}{14},$ $-\frac{43}{1664}),$ $(\frac{1}{14},-\frac{43}{1664}),$ $(\frac{1}{14},-\frac{43}{1664})]$ at first subdivision level. Whereas at other subdivision levels, we use $(\alpha_{i},\beta_{i})=(0,-\frac{2}{125})$ for the points $(4,0),$ $(5,5)$ $\&$ $(4,10)$ and $(\alpha_{i},\beta_{i})=(\frac{1}{14},-\frac{43}{1664})$ for all the other points.
  \item (f) shows the interproximate limit curve with $(\alpha_{i},\beta_{i})=[(0,\frac{1}{30}),(0,\frac{1}{30}),$ $(\frac{1}{11},-\frac{1}{30}),$ $(0,\frac{1}{30}),$ $(0,\frac{1}{30}),$ $(\frac{1}{11},-\frac{1}{30}),$ $(\frac{1}{11},-\frac{1}{30}),$ $(\frac{1}{11},-\frac{1}{30}),$ $(\frac{1}{11},-\frac{1}{30}),$ $(\frac{1}{11},$ $-\frac{1}{30})]$ at first subdivision level. Whereas at other subdivision levels, we use $(\alpha_{i},\beta_{i})=(0,\frac{1}{30})$ for the points $(0,0),$ $(4,0),$ $(4,10)$ $\&$ $(0,10)$ and $(\alpha_{i},\beta_{i})=(\frac{1}{11},-\frac{1}{30})$ for all the other points.
\end{itemize}
These figures show that proposed schemes can interpolate the initial control points which are chose by the programmers to be interpolated.

Figure \ref{interproximate-b-fig} shows the limit surfaces generated by tensor product subdivision scheme of scheme $S_{a_{3}}$ using initial control points $(0,2,0),$ $(5,0,0),$ $(10,2,0),$ $(5,15,0),$ $(0,2,0),$ $(0,2,2),$ $(5,0,2),$ $(10,2,2),$ $(5,15,2),$ $(0,2,2),$ $(0,2,$ $4),$ $(5,0,4),$ $(10,2,4),$ $(5,15,4),$ $(0,2,$ $4),$ $(0,2,6),$ $(5,0,6),$ $(10,2,6),$ $(5,15,6),$ $(0,2,6)$, where Figure \ref{interproximate-b-fig}(a) shows the limit surface that interpolates all the initial control points with $(\alpha,\beta)=(0,-\frac{1}{40})$, Figure \ref{interproximate-b-fig}(b) shows the limit surface that approximates all the initial control points with $(\alpha,\beta)=(\frac{1}{8},0)$ and Figure \ref{interproximate-b-fig}(c) shows the limit surface that interpolates only the control points $(5,15,0),$ $(5,15,2),$ $(5,15,4),$ $(5,15,6)$ at each level of subdivision with $(\alpha_{i},\beta_{i})=(0,-\frac{1}{40})$ and approximates all the other control points at each level of subdivision with $(\alpha_{i},\beta_{i})=(\frac{1}{8},0)$. Similarly, a programmer can choose other control points of his choice to be interpolated by using the tensor product schemes of the proposed schemes.

\section{Conclusion}\label{Conclusion3}
In this article, we have proposed a recursive method to generate the refinement rules of combined subdivision schemes. On the basis of that recursive refinement rules we have presented the family of $(2N+2)$-point relaxed primal combined schemes, the family of $(2N+3)$-point relaxed combined schemes and the family of $(2N+4)$-point interpolatory subdivision schemes with reproduction degrees $2N+1$, $2N+3$ and $2N+3$ respectively at certain values of the tension parameters. In fact, when value of $N$ is increased by one, polynomial reproductions of the proposed families of schemes are increased by two. Similarly, polynomial generations of the proposed families of schemes are $2N+3$, $2N+5$ and $2N+3$ respectively. $N$ is also directly proportional to the polynomial generations of the schemes. The continuity of the proposed families may be increased by increasing $N$. Our families of schemes not only give the flexibility in fitting limit curves/surfaces because of the involvement of tension parameters, but also give the optimal polynomial reproduction, polynomial generation and continuity than the existing primal schemes. Moreover, we converted the proposed family of $(2N+3)$-point combined subdivision schemes to the family of interproximate subdivision schemes by defining local parameters. One of these parameters is defined to control the interpolating property of the subdivision schemes and other one is defined to control the approximating property of the subdivision schemes. The interproximate subdivision schemes have applications in situations where some of the initial data points cannot be measured exactly. Future work is to do a theoretical study that how to choose values of tension parameters in an interproximate algorithm automatically.

\subsection*{Acknowledgement}
This work is supported by NRPU Project. No. 3183, Pakistan.

\end{document}